\documentclass[reqno]{amsart}
\usepackage[utf8]{inputenc}
\usepackage[margin=1in]{geometry}

\usepackage{pstricks} 

\usepackage[all,arc]{xy}
\usepackage{enumerate}
\usepackage{mathrsfs}

\usepackage{amsmath}
\usepackage{amsthm}
\usepackage{amsfonts}
\usepackage{amssymb}
\usepackage{amsbsy}
\usepackage{graphicx}
\usepackage{enumitem}
\usepackage{hyperref}
\usepackage{bbm}
\usepackage{longtable}
\usepackage{mhequ}

\newcommand{\ms}[1]{\mathscr{#1}}
\newcommand{\mc}[1]{\mathcal{#1}}

\newcommand{\ovl}[1]{\overline{#1}}
\newcommand{\varep}{ \varepsilon }
\newcommand{\sse} {\subseteq}

\newcommand{\Z}{\mathbb{Z}}
\newcommand{\R}{\mathbb{R}}

\newcommand{\C}{\mathbb{C}}
\newcommand{\E}{\mathbb{E}}
\newcommand{\N}{\mathbb{N}}
\newcommand{\T}{\mathbb{T}}

\newcommand{\icomplex}{\textup{i}}
\newcommand{\ra}{\rightarrow}
\newcommand{\toinf}{\ra \infty}
\newcommand{\ptl}{\partial}
\newcommand{\beq}{\begin{equation}}
\newcommand{\eeq}{\end{equation}}
\newcommand{\mbf}[1]{\mathbf{#1}}
\newcommand{\mrm}[1]{\mathrm{#1}}

\newcommand{\End}{\mrm{End}}

\newcommand{\frkg}{\mathfrak{g}}
\newcommand{\cycles}{\mrm{cycles}}

\newtheorem{theorem}{Theorem}
\newtheorem{prop}[theorem]{Proposition}
\newtheorem{lemma}[theorem]{Lemma}
\newtheorem{cor}[theorem]{Corollary}
\theoremstyle{definition}
\newtheorem{definition}[theorem]{Definition}
\newtheorem{notation}[theorem]{Notation}
\newtheorem{remark}[theorem]{Remark}

\newtheorem{problem}[theorem]{Problem}
\newtheorem{example}[theorem]{Example}
\newtheorem{assumptions}[theorem]{Assumptions}

\newcommand{\Var}[1]{\mathrm{Var}(#1)}
\newcommand{\Cov}{\mrm{Cov}}

\numberwithin{equation}{section}
\numberwithin{theorem}{section}
\newcommand{\ind}{\mathbbm{1}}
\newcommand{\Tr}{\mathrm{Tr}}
\newcommand{\torus}{\mathbb{T}}
\newcommand{\sym}{S_{\mrm{YM}}}
\newcommand{\scs}{S_{\mrm{CS}}}
\newcommand{\schwartz}{\mc{S}}
\newcommand{\manifold}{M}
\newcommand{\e}{\mathbf{e}}

\usepackage{accents}
\newcommand*{\bigdot}[1]{\accentset{\mbox{\large\bfseries .}}{#1}}

\newcommand{\liegroup}{G}

\newcommand{\groupid}{\textup{id}}
\newcommand{\unitary}{\mrm{U}}

\newcommand{\curl}{\mrm{curl}}
\newcommand{\codif}{d^*}
\newcommand{\gff}{\mrm{GFF}}
\newcommand{\fgf}{\mrm{FGF}}
\newcommand{\symgrp}{\mrm{S}}
\newcommand{\noise}{W}

\title{Fractional Gaussian forms and gauge theory: an overview}
\author{Sky Cao}
\address{Department of Mathematics, Massachusetts Institute of Technology, Cambridge, MA 02139}
\email{skycao@mit.edu}
\email{sheffield@math.mit.edu}
\author{Scott Sheffield}

\begin{document}

\begin{abstract}
Fractional Gaussian fields are scalar-valued random functions or generalized functions on an $n$-dimensional manifold $M$, indexed by a parameter $s$. They include white noise ($s = 0$), Brownian motion ($s=1, n=1$), the 2D Gaussian free field ($s = 1, n=2$) and the membrane model ($s = 2$). These simple objects are ubiquitous in math and science, and can be used as a starting point for constructing non-Gaussian theories.
They are sometimes parameterized by the {\em Hurst parameter} $H=s-n/2$. 

The {\em differential form} analogs of these objects are equally natural: for example, instead of considering an instance $h(x)$ of the GFF on $\mathbb R^2$, one might write $h_1(x)dx_1 + h_2(x) dx_2$ where $h_1$ and $h_2$ are independent GFF instances. This ``Gaussian-free-field-based $1$-form'' can be projected onto curl-free and divergence-free components, which in turn arise as gradients and dual-gradients of independent membrane models.

In general, given $k \in \{0,1,\ldots,n\}$, an instance of the {\em fractional Gaussian $k$-form} with parameter $s \in \mathbb R$ (abbreviated $\mathrm{FGF}_s^k(M)$) is given by $(-\Delta)^{-\frac{s}{2}} W_k,$
where $W_k$ is a $k$-form-valued white noise. We write
$$\textrm{FGF}_s^k(M)_{d=0}\,\,\,\,\, \textrm{and} \,\,\,\,\,\, \textrm{FGF}_s^k(M)_{d^*=0}$$
for the $L^2$ orthogonal projections of $\textrm{FGF}_s^k(M)$ onto the space of $k$-forms on which $d$ (resp.\ $d^*$) vanishes. We explain how $\mathrm{FGF}_s^k(M)$ and its projections transform under $d$ and $d^*$, as well as wedge/Hodge-star operators, subspace restrictions, and axial projections. We review ``massive'' variants, variants on lattices, and a variant involving the Chern-Simons action. We explain the role of the higher order cohomology groups of $M$.

The $1$-form $\textrm{FGF}_1^1(M)$ and its {\em gauge-fixed} projection $\textrm{FGF}_1^1(M)_{d^*=0}$ arise as low-temperature/small-scale limits of $\unitary(1)$ Yang-Mills gauge theory. When $\unitary(1)$ is replaced with another compact Lie group, the corresponding limit is a Lie-algebra-valued analog of $\textrm{FGF}_1^1(M)_{d^*=0}$, which can be interpreted as a random connection whose curvature is (a Lie-algebra valued analog of) $\textrm{FGF}_2^0(M)_{d=0}$.

Wilson loop observables of this connection are defined for sufficiently regular ``big loops'' (trajectories in the space of divergence-free $1$-forms obtainable as limits of finite-length loops) in a gauge invariant way, even in the non-abelian case. We define ``big surfaces'' (trajectories in the space of $2$-forms obtainable as limits of smooth surfaces with boundary) and note that Stokes' theorem converts big-loop integrals of $\textrm{FGF}_1^1(M)_{d^*=0}$ into big-surface integrals of $\textrm{FGF}_2^0(M)_{d=0}$ or $\textrm{FGF}_2^0(M)$. A type of exponential correlation decay and area law applies within the {\em slabs} $\mathbb R^2 \times [0,1]^m$ but not within $\mathbb R^n$ for $n > 2$. 

One may interpret $\textrm{FGF}_1^1(M)_{d^*=0}$ as a random divergence-free vector field, which is conjectured to be the fine-mesh scaling limit of the $n$-dimensional dimer model when $n>2$. (Kenyon proved this for $n=2$.) We formulate several conjectures and open problems about scaling limits, including possible off-critical/non-Gaussian limits, whose construction in the Yang-Mills setting is a famous open problem.
\end{abstract}

\maketitle

\tableofcontents
\section{Introduction}
\subsection{Fractional Gaussian fields} \label{sec:fgfoverview}
{\em Scalar-valued} fractional Gaussian fields are ubiquitous in physics, mathematics, finance and many other disciplines. Let us begin by briefly summarizing a few basic definitions and facts that are found e.g.\ in the overview~\cite{lodhia2016fractional}. Formally, the (scalar) fractional Gaussian free field FGF$_s(\mathbb R^n)$ on $\mathbb R^n$ with parameter $s$ has the form $$h = (-\Delta)^{-\frac{s}{2}} W$$ where $W$ is a complex or real white noise on $\mathbb R^n$. In the complex case, we note that the Fourier transform of complex white noise is complex white noise itself, and applying the Laplacian $(-\Delta)$ to a function has the effect of multiplying its Fourier transform by $|\cdot|^2$. Hence, the Fourier transform $\hat h$ of an FGF$_s(\mathbb R^n)$ instance $h$ can be understood as white noise times $|\cdot|^{-s}$. Replacing $W$ with a real valued white noise amounts to projecting $\hat h$ onto the subspace of $f$ satisfying $f(-\alpha) = \overline {f(\alpha)}$. We will deal mostly with the real-valued case in this paper.

Special cases include Brownian motion ($n=1, s=1$), fractional Brownian motion ($n=1,\,\, 1/2<s<3/2$), white noise ($s=0$), the Gaussian free field or GFF ($s=1$), the bi-Laplacian Gaussian field a.k.a.\ membrane model ($s=2$) and the log-correlated Gaussian field or LGF ($s=n/2$). These fields can also be parameterized by their {\em Hurst parameter} \begin{equation} H:=s-n/2.\end{equation} The Hurst parameter describes a scaling relation: namely, if $h$ is an instance of $\mathrm{FGF}_s(\mathbb R^n)$, then $h(ax)$ has the same law as $a^H h(x)$. For example, in the case of Brownian motion (where $s=n=1$) we have $H=1/2$, and this is the usual Brownian scaling. The Hurst parameter also encodes a two-point correlation function: if $H$ is a positive non-integer and $\phi_1$ and $\phi_2$ are appropriate test functions (see further discussion below) then
\begin{equation}\label{eq:cov} \mathrm{Cov}[(h,\phi_1), (h,\phi_2)] = C(s,n) \int_{\mathbb R^n} \int_{\mathbb R^n} |x-y|^{2H} \phi_1(x)\phi_2(y)dxdy,\end{equation} for some constant $C(s,n)$. Related statements apply when $H$ is negative or when $H$ is an integer, see \cite{lodhia2016fractional}. The following figure, which first appeared in \cite{lodhia2016fractional}, illustrates the $H$ values for a range of $s$ and $n$.
\vspace{.1in}
\begin{center}
       \includegraphics[width=.4\linewidth]{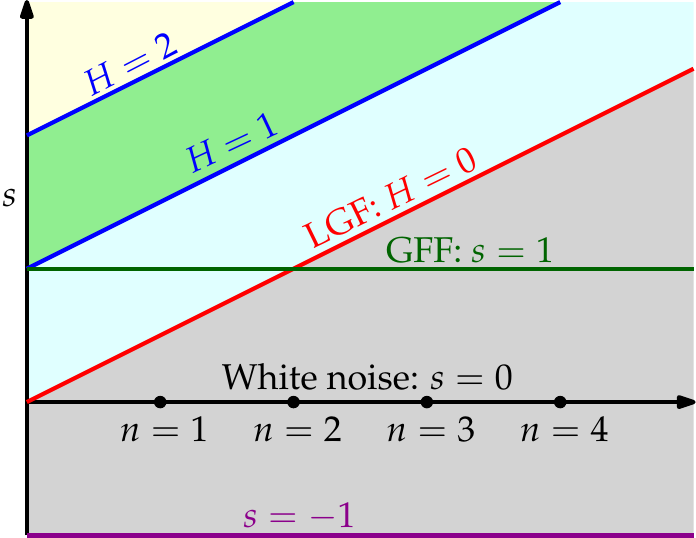}
\end{center} 
\vspace{.1in}
All of the $\mathrm{FGF}_s(\mathbb R^n)$ have translation invariant laws. The line $H=0$ (red) is special for several reasons:
\begin{enumerate}
\item When $H>0$, an instance of $\mathrm{FGF}_s(\mathbb R^n)$ is almost surely continuous (or ``has a continuous version'') but when $H \leq 0$ (grey region) it is a random {\em generalized function} and not defined pointwise.
\item When $H<0$, an instance of $\mathrm{FGF}_s(\mathbb R^n)$ is an a.s.\ well-defined generalized function, but when $H\geq 0$, an instance of $\mathrm{FGF}_s(\mathbb R^n)$ is only defined modulo the addition of a polynomial of degree $\lfloor H \rfloor$. This means that for our purposes, the quantities in \eqref{eq:cov} are only well-defined if \begin{equation} \int \phi_i(x)L(x)dx=0\end{equation} for any polynomial $L$ of degree $\lfloor H \rfloor$ in the coordinate variables $x_i$. For example, we treat a Brownian motion instance $B$ (which corresponds to $n=1, H=1/2$) as a random function defined only modulo additive constant.\footnote{One can of course ``fix'' an additive constant in the standard way by writing $\tilde B(t):=B(t) - B(0)$. One can then interpret $\tilde B$ as a Brownian motion with a ``fixed zero value'' given by $\tilde B(0) = 0$. But the law of $\tilde B$ is not translation invariant.} 
\item The value $H=0$ is therefore the ``most pathological'' in the sense that it is the only value for which an $\mathrm{FGF}_s(\mathbb R^n)$ is {\em both} only defined as a generalized function (no pointwise values) {\em and} only defined modulo additive constant. Put differently, if we try to formally define $h$ as the Fourier transform of \begin{equation} \label{eq:bigsmallfourier} |x|^{-s} W(x) = 1_{|x| < 1} |x|^{-s} W(x) + 1_{|x| \geq 1} |x|^{-s} W(x),\end{equation}
we find that $H=0$ is the unique value for which {\em neither} RHS term has a convergent Fourier transform that is defined pointwise in the usual sense: the Fourier transform of the first part (low frequencies) is defined only modulo additive constant while the Fourier transform of the second part (high frequencies) is only defined as a generalized function. When proving this, one starts by computing the variance of $h(0)$ (or $h(x)$ for some other $x$) when $\hat h$ is one of the terms on the RHS of \ref{eq:bigsmallfourier}: if $H=0$ (so $s=n/2$) then the variance computation gives $\int_U (|x|^{-s})^2dx = \int_U |x|^{-n}dx$ where $U$ is {\em either} the unit ball or its complement, and in both cases the integral diverges.
\item When $H=0$, the correlation formula \eqref{eq:cov} holds with $|x-y|^{2H}$ replaced by $-\log(|x-y|)$, hence the name {\em log-correlated Gaussian field} (LGF).
\item Using a suitable regularization, one can also construct the {\em exponential} of a multiple of the LGF; this object and its variants are well studied within the framework of {\em Gaussian multiplicative chaos}, {\em Liouville quantum gravity}, and {\em conformal field theory}  \cite{duplantier2017log}.
\end{enumerate}
Now let us discuss the ``test functions'' $\phi_1$ and $\phi_2$ in \eqref{eq:cov}. If $W$ is white noise and $\phi_1$ and $\phi_2$ are both in $L^2$ then by definition
\begin{equation}\label{eq:cov2} \mathrm{Cov}[(W,\phi_1), (W,\phi_2)] = \int_{\mathbb R^n} \int_{\mathbb R^n}  \phi_1(x)\phi_2(y)dxdy,\end{equation}
and hence the covariance in \eqref{eq:cov} can also be expressed (using integration by parts) as follows:
\begin{align} \label{eq:fgfcovariance} \mathrm{Cov}[(h,\phi_1), (h,\phi_2)] &=  \mathrm{Cov}[\bigl( (-\Delta)^{s/2} W,\phi_1\bigr), \bigl((-\Delta)^{-s/2}W,\phi_2\bigr)] \nonumber \\ 
&= \mathrm{Cov}[\bigl(  W,(-\Delta)^{-s/2}\phi_1\bigr), \bigl(W, (-\Delta)^{-s/2}\phi_2\bigr)] \nonumber \\ &= \Bigl((-\Delta)^{-s/2}\phi_1,(-\Delta)^{-s/2}\phi_2\Bigr) \nonumber\\
&= \Bigl((-\Delta)^{-s} \phi_1,\phi_2\Bigr).
\end{align}
The value $(h,\phi)$ is then defined as a random variable (with finite variance) if the following both hold: \begin{enumerate}
    \item \label{item:poly} either $H<0$ or $H \geq 0$ {\em and} $(\phi,L) = 0$ for any polynomial $L$ of degree $\lfloor H \rfloor$.
    \item \label{item:hilbert} $\|(-\Delta)^{-s/2}\phi\|^2 = \Bigl((-\Delta)^{-s/2}\phi,(-\Delta)^{-s/2}\phi\Bigr)  < \infty$. In other words, $\phi$ is in the $(-\Delta)^{s/2}$ image of $L^2$.
\end{enumerate}
We can treat \eqref{item:hilbert} as a norm and use it to define a Hilbert space $\dot{H}^{-s}$ (the Hilbert space closure w.r.t.\ this norm of the set of smooth functions satisfying \eqref{item:poly} and \eqref{item:hilbert}). The space $\dot H^{s}$ is called the {\em homogenous Sobolev space of order $s$}, and intuitively represents the space of functions whose $(s/2)$th Laplacian power (and $s$th derivatives) belong to $L^2$. See Definition~\ref{def:sobolev} and the surrounding discussion. To avoid confusion, we always formally interpret $(f,g)$ to mean $\int f(x) g(x)dx$ --- but sometimes this means the integral is defined in the usual sense, and sometimes it means $f$ is a (possibly random) generalized function and $g$ is {\em some} suitable test function. We do {\em not} write $(\cdot, \cdot)$ to denote the inner product corresponding to $\dot H^{-s}$ for $s \not = 0$, or to some other Hilbert space besides $L^2$.

Now if $h$ is an instance of FGF$_s(\mathbb R^n)$ then the set of random variables $(h, \phi)$ --- where $\phi \in \dot H^{-s}$ --- is a {\em Gaussian Hilbert space} (in the sense of \cite{janson1997gaussian}) with covariance inner product given by \eqref{eq:fgfcovariance}. See e.g.\ the discussion of Gaussian Hilbert spaces in the GFF context in \cite{sheffield2007gaussian}. Although $(h,\phi)$ is a.s.\ defined for {\em each given} $\phi \in \dot H^{-s}$, the map $\phi \to (h, \phi)$ is a.s.\ {\em not} a continuous function that makes sense for {\em all} such $\phi$ simultaneously. On the other hand, it is possible to specify a smaller (but dense in $\dot H^{-s}$) topological space of test functions $\phi$ on which the map $\phi \to (h, \phi)$ {\em is} a.s.\ everywhere-defined and continuous. Once this is done, $h$ can be understood as a {\em random distribution} or {\em generalized function} (again see \cite{sheffield2007gaussian} for a more detailed discussion in the GFF context).

There are many spaces of test functions one can choose for this purpose; the space $\Phi$ of {\em Lizorkin functions} discussed in Section~\ref{sec:preliminaries} is one example of a space that works for all $s$ simultaneously. Its dual $\Phi'$ is the space of tempered distributions modulo polynomials. Regardless of $s$, it turns out that one does not lose any {\em information} by simply defining $h$ as a random element of $\Phi'$. The reason is that for every $m$, the ``additive degree $m$ polynomial term'' is either undefined or a.s.\ {\em determined} by $h$ viewed as an element of $\Phi'$. For example, if $h$ is a 2D GFF then the additive constant is undefined, but if $h$ is 2D white noise, then the additive constant is determined (essentially because the mean value of $h$, on a centered radius-$R$ disk, should tend to zero $R \to \infty$, and there is a.s.\ exactly one choice of additive constant that makes this true). When $M$ is compact, we will see that one can alternatively define the map $\phi \to (h, \phi)$ as a continuous functional on (the $M$ analog of) $\dot H^{-s'}$ for any $s' < H$, which is equivalent to defining $h$ itself as a random element of $\dot H^{s'}$.

When $H \geq 0$ it is possible to make sense of the restriction of an instance of $\mathrm{FGF}_s(\mathbb R^n)$ to any lower-dimensional subspace, and the parameter $H$ is preserved under subspace restriction. For example, when $H=1/2$, the $\mathrm{FGF}_s(\mathbb R^n)$ is sometimes called {\em L\'evy's Brownian motion}: it is a random continuous function on $\mathbb R^n$ (defined modulo additive constant) whose restriction to any line in $\mathbb R^n$ has the law of an ordinary Brownian motion (defined modulo additive constant). More generally, for any $H \in (0,1)$ (the light blue band in the figure) the process $\mathrm{FGF}_s(\mathbb R^n)$ is {\em L\'evy's fractional Brownian motion}: a random continuous function on $\mathbb R^n$ (defined modulo additive constant) whose restriction to any line in $\mathbb R^n$ has the law of an ordinary fractional Brownian motion (defined modulo additive constant).

When $n=1$, differentiation is a ``square root of the Laplacian'' and integrating (resp.\ differentiating) an $\mathrm{FGF}_s(\mathbb R^n)$ instance has the effect of adding $1$ (resp.\ $-1$) to $s$, or equivalently to $H$. Hence any $\mathrm{FGF}_s(\mathbb R^n)$ can be obtained by starting with $H \in [0,1)$ (i.e.\ with a fractional Brownian motion or LGF instance) and then differentiating or integrating some number of times. In the context of the differential forms discussed in Section~\ref{sec:fracforms} below, the operators $d$ and $d^*$ (and their inverses) will somehow play the role of differentiation (and integration) in a manner that allows one to make similar statements for any $n>1$.

Variants of the fractional Gaussian free field can be defined on graphs, general manifolds, and domains with boundary conditions. Fractional Gaussian fields (and special cases such as Brownian motion, fractional Brownian motion, the GFF and the LGF) have been explored at length in many textbooks and survey articles, including some involving the second author of this paper \cite{mandelbrot1968fractional, sheffield2007gaussian, biagini2008stochastic,
morters2010brownian, nourdin2012selected,  freedman2012brownian, lodhia2016fractional, duplantier2017log, werner2021lecture, berestycki2024gaussian}. 

Throughout the remainder of the text, we will assume that $M$ is an oriented $n$-dimensional Riemannian manifold that is either compact (which implies that $-\Delta$ has a discrete spectrum) or equal to $\mathbb R^n$ (continuous spectrum). The theory can be easily extended to other non-compact manifolds (such as a cylinder). But we will avoid dealing with {\em fully} general {\em non-compact} manifolds because these would introduce some distracting complications, simply because the space of Laplacian eigenfunctions and eigenspaces can be quite complicated, and even determining whether the spectrum is discrete or continuous is non-trivial in general \cite{cianchi2011discreteness, cianchi2013bounds, colbois2017spectrum}.

\subsection{Fractional Gaussian forms} \label{sec:fracforms}

All the fractional Gaussian fields discussed in Section~\ref{sec:fgfoverview} have natural analogs in which random {\em scalar-valued (generalized) functions} are replaced by random {\em (generalized) differential forms}. These analogs arise naturally, for example, in many areas of mathematical physics, including gauge theory. Our aim in this work is to provide an accessible overview of their definitions and basic properties. In some respects, it is actually easier and {\em more} natural to work with fractional Gaussian forms than with their scalar analogs. When working with forms, we can use, for example, the fact that $(d+d^*)$ is a locally defined operator whose square is the Laplacian $(-\Delta)$, as we explain further below.

Fractional Gaussian $k$-forms on $\mathbb R^d$ are simple to define. First we recall that specifying a $k$-form on $\mathbb R^n$ is equivalent to specifying $\binom{n}{k}$ separate scalar functions, one for each of the $\binom{n}{k}$ independent wedge-products formed from the $n$ basis variables. (See the exterior calculus review in Section~\ref{section:exterior-calculus} for further details.) For example, if we parameterize $\mathbb R^2$ by $x=(x_1,x_2)$ then we can write a general $0$-form as a function $A(x)$, a general $1$-form as $B(x)dx_1 + C(x)dx_2$, and a general $2$-form as $D(x) dx_1 \wedge dx_2$. An instance of the {\em fractional Gaussian $k$-form} on $\mathbb R^n$ with parameter $s \in \mathbb R$ (abbreviated $\mathrm{FGF}_s^k(\mathbb R^n)$) is then described by $\binom{n}{k}$ independent random (generalized) functions, each of which {\em independently} has the law of a {\em scalar fractional Gaussian field} with parameter $s$.

On $\mathbb R^n$, applying the Laplacian to a $k$-form is equivalent to applying the Laplacian separately to each of the $\binom{n}{k}$ scalar functions. Hence, an instance of $\mathrm{FGF}_s^k(\mathbb R^n)$ can also be written as
\begin{equation}(-\Delta)^{-\frac{s}{2}} W_k,\end{equation}
where $W_k$ is a $k$-form-valued white noise. 

For a general manifold $M$, we write $\Omega^k(M)$ for the space of smooth and square-integrable $k$-forms on $M$ and we write $\Omega$ for the linear span of the $\Omega^k$. The smoothness ensures that $d$ and $d^*$ and $-\Delta$ are well defined in the traditional sense (as opposed to the distributional sense) and square integrability ensures that $\Omega$ has a natural Hilbert space closure $\overline \Omega$. However, we stress again that instances of $\mathrm{FGF}_s^k(\mathbb R^n)$ are a.s.\ neither smooth nor square integrable, and in some cases are defined only as generalized functions, not as functions.

An element of $\Omega(\mathbb R^n)$ is specified by $2^n = \sum_{k=0}^n {n \choose k}$ scalar functions in $L^2(\mathbb R^n)$. The boundary operator $d$ is formally defined on $\Omega$ via $d f = \sum_{i=1}^n \bigl( \frac{\partial}{\partial x_i} f \bigr) \wedge dx_i$ and its adjoint is written $d^*$ (see the more detailed explanation in Section~\ref{section:exterior-calculus}). These operators satisfy \begin{equation}\label{eq:ddzero} d d = d^*d^* = 0.\end{equation} One way to {\em define} the Laplacian $-\Delta$ on forms (equivalent to the definition above if $M=\mathbb R^n$) is by writing \begin{equation} \label{eq:squareddstar} -\Delta := (d+d^*)^2 = d d^* + d^*d.\end{equation} The space $\Omega$ has a natural inner product $(\cdot, \cdot)$ w.r.t.\ which the $\Omega^k$ are orthogonal subspaces, and by definition \begin{equation}\label{eq:adjoint} (du,v) = (u,d^*v),\end{equation} which implies
\begin{equation}\Bigl( (d+d^*) u, v \Bigr) = \Bigl( u, (d+d^*) v \Bigr),\end{equation}
and hence
\begin{equation}\Bigl( (d+d^*) u, (d+d^*) v \Bigr) = \Bigl(u, (d+d^*)^2 v \Bigr) = (u, -\Delta v),\end{equation} which is a more general form of the integration-by-parts formula $(\nabla u, \nabla v) = (u - \Delta v)$ for scalar $u$ and $v$ that is often used to study the Gaussian free field. It can be useful to think of $(d+d^*)$ as {\em one possible} square root of $(-\Delta)$, but when we write $(-\Delta)^{1/2}$ we mean the square root that acts on the scalar components separately, so in particular $(-\Delta)^{1/2}$ takes $k$-forms to $k$-forms, which is not true of $(d+d^*)$. We write
$$\textrm{FGF}_s^k(M)_{d=0}\,\,\,\,\, \textrm{and} \,\,\,\,\,\, \textrm{FGF}_s^k(M)_{d^*=0}.$$
for the $L^2$ orthogonal projections of $\textrm{FGF}_s^k(M)$ onto the space of $k$-forms on which $d$ (resp.\ $d^*$) vanishes. If $k=0$ then $\mathrm{FGF}_s^k(M)$ is a {\em scalar} fractional Gaussian field such as white noise ($s=0$), the Gaussian free field ($s=1$) and the membrane (or bi-Laplacian) model ($s=2$). 

Figure~\ref{fig:formchart} illustrates the so-called {\em Hodge decomposition} (see Section~\ref{sec:flatspace} for more details). Each space $\Omega^k$ of $k$-forms decomposes into orthogonal subspaces: the harmonic forms $\Omega^k_0$ on which both $d$ and $d^*$ vanish (grey boxes) and the images $\Omega^k_- := d(\Omega^{k-1})$ and $\Omega^k_+ := d^*(\Omega^{k+1})$ (red boxes) which are spaces on which only $d$ (resp.\ only $d^*$) vanishes. The map $d$ maps each red $\Omega^k_+$ box bijectively to the $\Omega^{k+1}_-$ box on its right (and sends all other boxes to zero). The map $d^*$ maps each red $\Omega^k_-$ box bijectively to the $\Omega^{k-1}_+$ on its left (and sends all other boxes to zero). The Laplacian $-\Delta = (dd^* + d^*d)$ maps each red box bijectively to itself and sends each grey box to zero. When $M$ is the $n$-dimensional torus, the forms in $\Omega^k_0$ (grey boxes) are precisely the constant forms. When $M=\mathbb R^n$ the grey boxes are empty, since there are no harmonic forms in $L^2$. For any compact manifold $M$, the $k$th grey box represents the $k$th cohomology group of $M$ (over $\mathbb R$ or $\mathbb C$). If $M$ is the $n$-dimensional sphere, for example, then $\Omega^k_0$ contains a one-dimensional space of constant forms if $k \in \{0,n\}$ and is empty otherwise. If $M$ is any compact manifold, then all of the ``grey box'' spaces (cohomology groups) must be finite dimensional \cite{Jost}.

The $k$-form white noise $W_k$ can be understood as a ``standard Gaussian'' on $\Omega^k$, equipped with the $L^2$ norm, and as such it can be written as a sum of standard Gaussians on each of $\Omega^k_-$, $\Omega^k_0$, and $\Omega^k_+$. When $s \not = 0$, however, the most straightforward way to define the FGF is to set the grey-box components to zero, since in general it does not make sense to ``apply a possibly-negative power of the Laplacian'' to a space on which the Laplacian is identically zero. In other words, for the purposes of defining general fractional Gaussian forms, one natural approach is to just forget about the grey boxes.

On the other hand, the grey boxes may be relevant for some applications, in particular when $s$ is a positive integer. For example, the scalar-valued Gaussian free field ($s=1, k=0$) can be understood as the $d$-preimage of the projection of $1$-form white noise $W_1$ onto $\Omega^1_-$. If $M$ is the two-dimensional torus, one could instead try to construct the $d$-preimage of white noise projected onto the {\em span} of $\Omega^1_-$ and $\Omega^1_0$ (this span is the $1$-form kernel of $d$). It is not hard to see that {\em this} preimage can be defined as a ``multi-valued'' generalized function on the torus. This ``multivalued height function'' is equivalent to the scaling limit of the height function of a random dimer model instance on the torus, provided that that one conditions on the ``height change'' around any non-trivial cycle being an integer \cite{kenyon2000conformal, kenyon2001dominos, dubedat2015asymptotics}. More generally, even though $d$ and $d^*$ map the grey boxes to zero, the $d$ and $d^*$ and $(-\Delta)$ {\em preimages} of elements in the grey boxes may become non-empty once the forms are lifted to the universal cover of $M$. This means that if $s$ is a positive integer and $W$ is the projection of a $k$-form-valued white noise onto one of the grey boxes, then $(-\Delta)^{-s/2}W = (d+d^*)^{-s}W$ can be defined as a random {\em multi-valued} form, which is defined only modulo the space of multi-valued forms in the kernel of $(d+d^*)^s$.
\begin{figure}
    \centering
    \includegraphics[width=.8\linewidth]{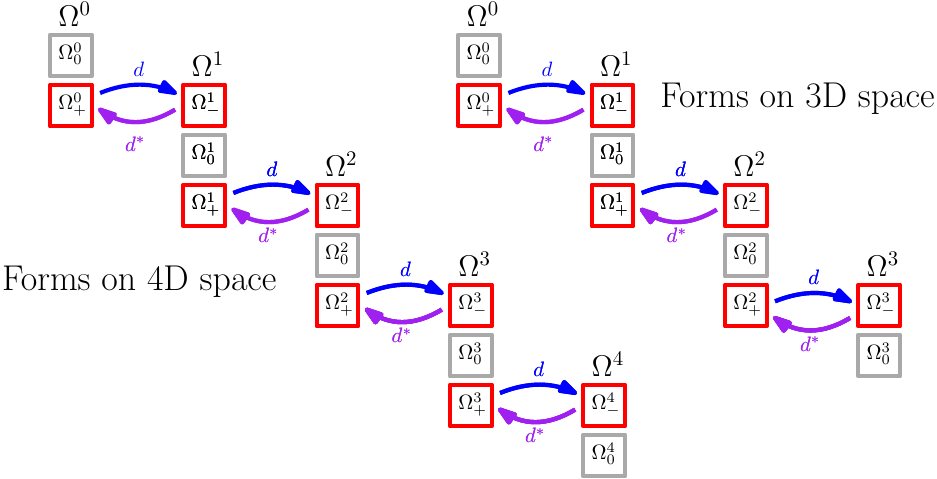}
    \caption{{\bf Hodge decomposition:} Each space $\Omega^k$ of $k$-forms decomposes into orthogonal subspaces: the harmonic forms $\Omega^k_0$ on which both $d$ and $d^*$ vanish (grey boxes), plus a space $\Omega^k_-$ on which only $d$ vanishes and a space $\Omega^k_+$ on which only $d^*$ vanishes (red boxes). By definition the elements of $\Omega^k$ are smooth, but similar orthogonal decompositions can be defined if we instead allow $\Omega^k$ to contain all $L^2$ forms, or to be a class of {\em generalized forms}.
    \label{fig:formchart}}
\end{figure}

\subsection{Fractional Gaussian forms on the torus} \label{sec:torusfgf}

Fractional Gaussian forms are perhaps most easily understood when the manifold $M$ is the $n$-dimensional torus $\mathbb T^n$ (understood as $\mathbb R^n$ modulo $2\pi \mathbb Z^n$).

To extend this construction to the space of forms, it is natural to consider for each $\alpha \in \mathbb Z^n$ the $2^n$ dimensional space ${}^\alpha\Omega$ of {\em $\alpha$-periodic forms} spanned by the forms written $e^{\icomplex \alpha \cdot x}$ times a wedge-product of basis elements. For example, if $n=3$ then ${}^\alpha\Omega$ is spanned by the following $8=2^3$ forms:
$$e^{\icomplex \alpha \cdot x}\cdot 1, \,\,\,\,  e^{\icomplex \alpha \cdot x} dx_1, \,\,\,\,e^{\icomplex \alpha \cdot x}dx_2, \,\,\,\,e^{\icomplex \alpha \cdot x} dx_3, \,\,\,\,e^{\icomplex \alpha \cdot x} dx_1 \wedge dx_2, \,\,\,\,e^{\icomplex \alpha \cdot x}dx_1 \wedge dx_3, \,\,\,\,e^{\icomplex \alpha \cdot x} dx_2 \wedge dx_3, \,\,\,\,e^{\icomplex \alpha \cdot x} dx_1 \wedge dx_2 \wedge dx_3.$$
In other words, an element of ${}^\alpha\Omega$ can be written as $e^{\icomplex \alpha \cdot x}$ times an element of the exterior algebra of $\mathbb R^3$, which is spanned by $\{1, dx_1, dx_2, dx_3, dx_1\wedge dx_2, dx_1 \wedge dx_3, dx_2 \wedge dx_3, dx_1 \wedge dx_2 \wedge dx_3 \}$.
One can check that for any $f$ in this span, the $d$ operator takes $f$ to $f \wedge \overline \alpha$ where $\overline \alpha = \sum_{i=1}^n \icomplex\alpha_idx_i$. In other words, applying $d$ corresponds to taking the wedge product with a particular vector $\overline \alpha$.

The {\bf Hodge star} operator (see Definition~\ref{def:hodge-star} for a formal definition) maps a wedge product of $k$ of the $\{dx_1, dx_2, \ldots, dx_n\}$ to the wedge product of the remaining $(n-k)$ elements. It extends linearly to a linear bijection between ${}^\alpha\Omega^k$ and ${}^\alpha\Omega^{n-k}$. Within $\mathbb T^3$, we can use this map to ``identify'' ${}^\alpha\Omega^0$ and ${}^\alpha\Omega^3$ (treating both as one-dimensional ``scalars'') and also identify ${}^\alpha\Omega^1$ and ${}^\alpha\Omega^2$, treating both as vectors in $\mathbb R^3$ if we focus on the real span (or $\mathbb C^3$ if we consider the complex span). The $d$ operator (on $\Omega^0$, $\Omega^1$, and  $\Omega^2$) encodes (respectively) the familiar gradient, curl, and divergence operators. When we restrict attention (to ${}^\alpha\Omega^0$, ${}^\alpha\Omega^1$, and ${}^\alpha\Omega^2$) the $d$ operator can be viewed as a linear function defined on these (respectively) one, three and three dimensional spaces. The operator naturally corresponds (respectively) to the familiar scalar, cross and dot products with the special vector $\overline{\alpha}$:
\begin{enumerate}
    \item {\bf Gradient} (apply $d$ to an element of ${}^\alpha\Omega^0$): multiply a scalar by the vector $\overline \alpha$.
    \item {\bf Curl} (apply $d$ to an element of ${}^\alpha\Omega^1$): take a cross product of a vector with $\overline \alpha$.
    \item {\bf Divergence} (apply $d$ to an element of ${}^\alpha\Omega^2$): take a dot product of a vector with $\overline \alpha$. (This refers to the dot product $x \cdot y = \sum x_j y_j$.)
\end{enumerate}
In each of these cases, applying $d^*$ to an element of ${}^\alpha\Omega^{3-k}$ can be understood (using the Hodge star duality mentioned above) the same way as applying $d$ to an element of ${}^\alpha\Omega^k$. It is clear that ${}^\alpha\Omega^1_-$ (resp.\ ${}^\alpha\Omega^2_{+}$) corresponds to the one-dimensional span of $\overline{\alpha}$ in $\mathbb R^3$, while  ${}^\alpha\Omega^1_+$ (resp.\  ${}^\alpha\Omega^2_-$) is the two-dimensional orthogonal complement.

The spaces ${}^\alpha\Omega$ are preserved by $d$ and $d^*$ and $(-\Delta)$, and they form an orthogonal decomposition of $\Omega$. So any fractional Gaussian form can be expanded as the sum of its (independent) projections onto ${}^\alpha \Omega$. We can also Hodge-decompose ${}^\alpha \Omega$ by writing ${}^\alpha\Omega^k_- := \Omega^k_- \cap {}^\alpha\Omega$, and similarly for $0$ and $+$ subscripts. When $\alpha \not = 0$, the space ${}^\alpha\Omega^k_0$ is empty, so only the ${}^\alpha\Omega^k_-$ and ${}^\alpha\Omega^k_+$ are relevant. These subspaces and their dimensions are illustrated in Figure~\ref{fig:alphaformchart}. Note that for $f \in {}^\alpha\Omega^k_+$ we have $(df, df) = (f, d^*df) = (f, - \Delta f) = |\alpha|^2(f,f)$. Applying a similar argument to $d^*$, it follows that two boxes in Figure~\ref{fig:alphaformchart} on the same horizontal level are isometric, with $d$ and $d^*$ both stretching the $L^2$ norm by a constant factor of $|\alpha|$. When $|\alpha| \not = 0$, one can also use induction on $k$ to show that the dimensions of ${}^\alpha\Omega^k_-$ and ${}^\alpha\Omega^k_+$ correspond to the ``Pascal's triangle'' relation: \begin{equation}{n \choose k} = \mathrm{dim}({}^\alpha\Omega^k) = \mathrm{dim}({}^\alpha\Omega^k_-)+ \mathrm{dim}({}^\alpha\Omega^k_+) \,\,\,\,\,\,\textrm{where}\,\,\,\,\,\, \mathrm{dim}({}^\alpha\Omega^k_-) = {n-1 \choose k-1}\,\,\,\,\,\, \textrm{and}\,\,\,\,\,\, \mathrm{dim}({}^\alpha\Omega^k_+) = {n-1 \choose k}. \end{equation} In particular, adding the red-box numbers in each column (for the 4D forms on the left) we find $(1,3,3,1,0) + (0,1,3,3,1)=(1,4,6,4,1)$, which relates the third and fourth rows of Pascal's triangle. Another way to compute the dimensions of ${}^\alpha\Omega^k_-$ and ${}^\alpha\Omega^k_+$ is to replace the basis $(dx_1,dx_2, \ldots, dx_n)$ with a new orthonormal basis $(dy_1, dy_2, \ldots, dy_n)$ spanning the same space such that $dy_n$ is a multiple of $\overline \alpha$. Then ${}^\alpha\Omega^k$ is spanned by terms obtained as a wedge product of $k$ elements from $\{dy_1, \ldots, dy_n\}$, while ${}^\alpha\Omega^k_-$ (resp.\ ${}^\alpha\Omega^k_+$) is spanned by those products that {\em include} $y_n$ plus $k-1$ of the remaining $n-1$ elements (resp.\ {\em do not include} $y_n$ but include $k$ of the remaining $n-1$ elements). Very informally, the $d$ operator ``appends $\wedge dy_n$'' and the $d^*$ operator ``removes $\wedge dy_n$'' (i.e., appends $\wedge dy_n$ to the dual form) and both $d$ and $d^*$ multiply magnitude by $|\alpha|$.

When $\alpha=0$, the only $\alpha$-periodic forms are the constant forms, which are all harmonic; in other words, the $\alpha=0$ version of the figure above would include only the grey boxes from Figure~\ref{fig:formchart} and none of the red boxes, with the corresponding dimension of ${}^\alpha\Omega^k={}^\alpha\Omega^k_0$ being ${n \choose k}$.

\begin{figure}
    \centering
    \includegraphics[width=.8\linewidth]{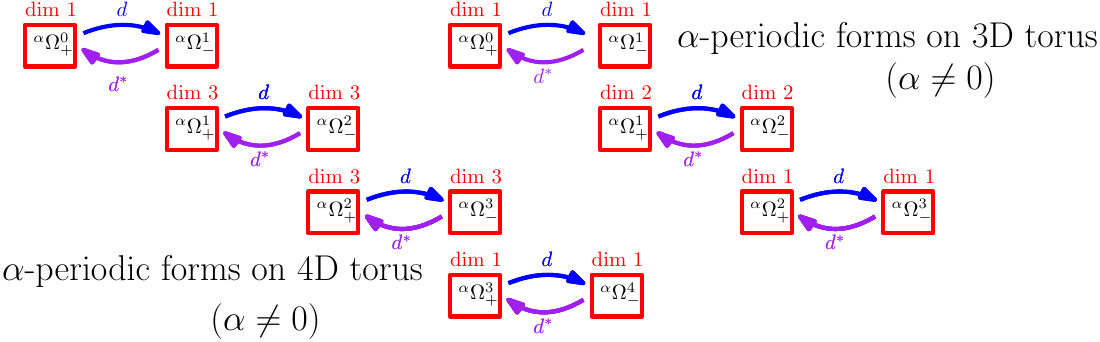}
    \caption{Decomposition of ${}^\alpha \Omega$ into ${}^\alpha \Omega^k_-$ and ${}^\alpha \Omega^k_+$ and corresponding dimensions when $\alpha \not = 0$.}
    \label{fig:alphaformchart}
\end{figure}

\subsection{Massive fields and Chern-Simons variants} \label{sec:introvariants}
Each of the red boxes in Figure~\ref{fig:formchart} or Figure~\ref{fig:alphaformchart}
can be assigned a natural ``standard Gaussian'' random variable $A$, where the density function for $A$ is formally proportional to \begin{equation}\exp\Bigl( -\frac{\|A\|^2}{2} \Bigr).\end{equation} This $A$ can be understood as a projection of a form-valued white noise onto that box. For example, the $A$ associated to the box $\Omega_-^1$ is the gradient of the Gaussian free field.

If we wished to instead describe the gradient of the {\em massive Gaussian free field} we would focus on the same box but use a different formal density function, in which we {\em combine} the $L^2$ norm on $\Omega_-^1$ with the $L^2$ of its preimage under $d$, so that the formal density is proportional to \begin{equation}\exp\Bigl( -\frac{\|A\|^2 + \lambda\|d^{-1} A\|^2}{2} \Bigr),\end{equation}
where $\lambda = m^2$ and $m>0$ is the so-called ``mass'' constant. On a fixed box of Figure~\ref{fig:alphaformchart}, this amounts to replacing $\exp\bigl(-\|A\|^2/2\bigr)$ with $\exp\Bigl(-\bigl(1+\frac{\lambda}{|\alpha|^2}\bigr)\frac{\|A\|^2}{2}\Bigr)$. There are many variants of this construction, as one may begin with any red box of Figure~\ref{fig:formchart} and replace $\bigl(1+\frac{\lambda}{|\alpha|^2}\bigr)$ with any linear combination of (possibly fractional) powers of $|\alpha|$, or some other function of $\alpha$. All of these variants give rise to different Gaussian probability measures on the space of forms.

One of the more interesting variants involves the {\em Chern-Simons} action, which involves the case $n=3$ and $k=1$. If $A_1$ and $A_2$ are two ``degenerate $1$-forms'' that transport a unit of current around simple non-intersecting oriented loops $\gamma_1$ and $\gamma_2$, respectively, then one can define a {\em linking number} of the pair $(A_1, A_2)$. This function on pairs $(A_1, A_2)$ can then be extended bilinearly to a quadratic form on $\Omega^1_+ \times \Omega^1_+$ (the set of pairs of divergence-free $1$-forms). It turns out that one way to express this quadratic form is as a constant times $(A_1, \curl^{-1} A_2)$.

Note that $\curl$ is {\em a} square root of $(-\Delta)$, but it is not {\em the} canonical square root denoted $(-\Delta)^{1/2}$ because $\curl$ is not positive-definite. It has a space of positive and a space of negative eigenvalues: the box ${}^\alpha \Omega^1_+$ (on the right side of Figure~\ref{fig:alphaformchart}) has complex dimension $2$, and can be decomposed into two orthogonal one-dimensional eigenspaces of $\curl$ with eigenvalues $|\alpha|$ and $-|\alpha|$.

The fact that the operator $(A_1, \curl^{-1} A_2)$ fails to be positive definite complicates its use in constructing a Gaussian probability measure. But one can consider a linear combination of  the operator $(A_1, \curl^{-1} A_2)$ (on the torus, say) with another operator, and such a combination may be positive definite.

Let us suggest one such construction which is particularly natural. We recall from \cite{chandgotia2023large} that the scaling limit of the flow induced by an instance of the 3D dimer model is conjectured to be $\textrm{FGF}_1^1(M)_{d^*=0}$, see Problem~\ref{prob:dimer}. If we consider this random flow {\em weighted} by (a constant times) the exponential of the {\em twist} parameter associated to the dimer instance, we obtain a random discretized $1$-form whose scaling limit we can conjecture (see Problem~\ref{prob:twist}) to be the standard Gaussian w.r.t.\ to the quadratic form on $\Omega^1_+$ given by $\|A\|^2+\lambda (A, \curl^{-1} A)$. Essentially, one is replacing the $(1 + \frac{\lambda}{|\alpha|^2})$ from before with the coefficient $(1 \pm \frac{\lambda}{|\alpha|})$, where the $\pm$ represents $+$ on one of the eigenspaces of ${}^\alpha \Omega^1_+$ and $-$ on the other eigenspace. These coefficients are always positive if $\lambda<1$ because $|\alpha|$ is always at least $1$ (unless $\alpha = 0$).

\subsection{Differential forms in gauge theory}\label{section:differential-forms-in-gauge-theory}
A more detailed introduction to Yang-Mills random connections and lattice Yang-Mills theory appears in Section~\ref{sec:yangmills}, but we will roughly outline some of the basic ideas here and explain the sense in which the Gaussian forms described in this paper arise as low-temperature-fine-mesh limits of those that appear in lattice Yang-Mills gauge theory.

We have thus far discussed $1$-forms defined on a manifold $M$ for which the ``target space'' is $\mathbb R$ or $\mathbb C$. We can also replace $\mathbb R$ or $\mathbb C$ with any other finite-dimensional vector space $V$. An $\mathbb R^m$-valued $k$-form, for example, is equivalent to an ordered list of $m$ distinct $\mathbb R$-valued $k$-forms, one for each component of $\mathbb R^m$. Using this approach, all of the random forms FGF$_k^s(M)$ have obvious analogs in which the target space is $V$ instead of $\mathbb R$ or $\mathbb C$. As an important example, a $1$-form that takes values in a {\em Lie algebra} (understood as a finite-dimensional vector space) is called a {\em connection}. Given a smooth path on $M$, one can ``integrate'' the $1$-form along that path to produce a trajectory in the corresponding Lie group (see 
Section~\ref{sec:wilsonloops} for more details).

In lattice Yang-Mills gauge theory, one considers ``discrete connections'' defined as follows. Given a graph $\Lambda$, a {\em discrete connection} assigns to each oriented edge $e$ an element $Q_e$ of some group $G$, under the constraint that if $e'$ is the orientation reversal of $e$ then $Q_{e'} = Q_e^{-1}$. For example, in the image below, each edge is labeled with an element from the subgroup of $O(2)$ generated by axis-reflections and $90$-degree rotations. The matrix associated to a directed edge (left) is the inverse of the matrix associated to the orientation-reversed edge (right).
\vspace{.1in}
\begin{center}
       \includegraphics[width=.3\linewidth]{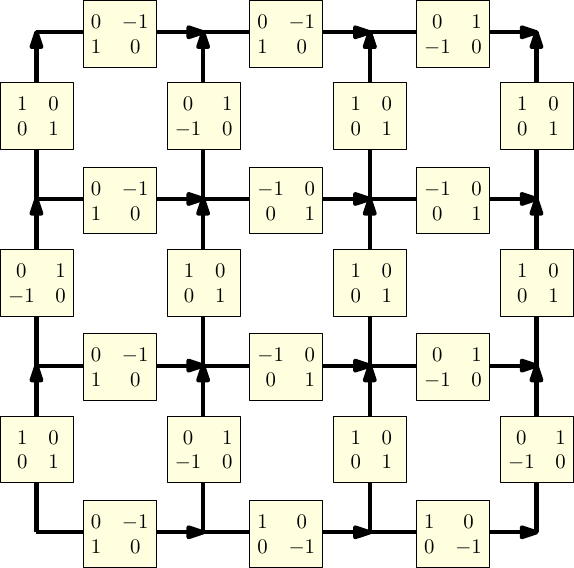} \hspace{1in}        \includegraphics[width=.3\linewidth]{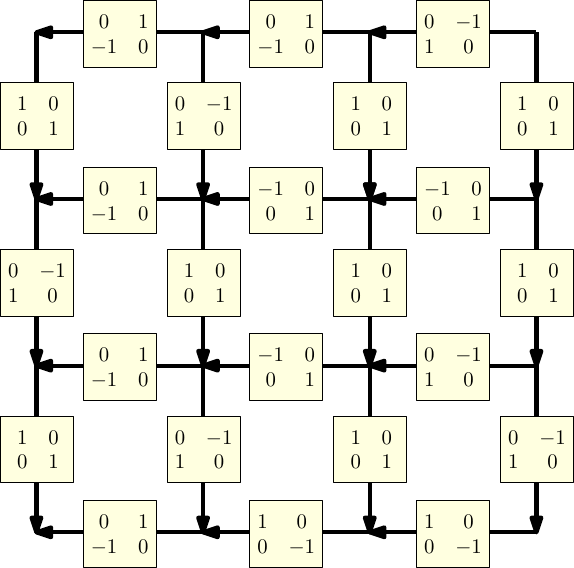}
\end{center} 
\vspace{.1in}
Given any directed path $p = (e_1, \ldots, e_n)$ of edges in $\Lambda$, we can define $Q_p := Q_{e_1} Q_{e_2} \cdot \ldots \cdot Q_{e_n}$ as a discrete analog of ``integrating a connection along a path $p$.'' In lattice Yang-Mills theory, one chooses a {\em random connection} $Q$ where the law of the $Q_e$ is given by i.i.d.\ Haar measure {\em weighted} in a way that is biased toward configurations for which $Q_p$ is close to the identity whenever $p$ is a small (length four) loop.

If $e_1, e_2, \ldots, e_{2d}$ are the oriented edges that terminate at a fixed vertex $v$, one can perform a ``gauge transformation'' by replacing each $Q_{e_j}$ with $Q_{e_j}g$ for some fixed group element $g$. Since $(Q_{e_i} g ) (Q_{e_j} g)^{-1} = Q_{e_i} Q_{e_j}^{-1}$, applying a gauge transformation does not change $Q_p$ for any closed cycle $p$ that starts and ends away from $v$. It turns out (see Section~\ref{sec:yangmills} for more details) that one can always ``gauge fix'' in a way that guarantees that each edge $e$ within a given spanning tree of $\Lambda$ will have $Q_e$ equal to the identity. Having done this, one can then argue that at sufficiently low temperatures, all of the remaining $Q_e$ are with high probability {\em close} to the identity.

In fact, it turns out that if we consider gauge-fixed Yang-Mills on $\Lambda$ at low enough temperature, then we can treat all of the matrices as tiny perturbations of the identity, and their products become approximately abelian, simply because $(I+a)(I+b) \approx (I + a + b)$ if $a$ and $b$ are small. If we consider the very low temperature limit (where everything becomes abelian) and take a subsequent fine-mesh, we obtain a continuum connection. In this particular limiting setting (where everything is abelian) the {\em curvature} of a $1$-form $A$ is simply the $2$-form $dA$. There is also another way to gauge fix, which allows one to assume that $d^*A=0$ (i.e.\ the divergence of the $1$-form is zero) instead of assuming that $1$-form vanishes in directions associated to the limit of the spanning tree.

Using this interpretation, the $\mathbb R$-valued $1$-form $\textrm{FGF}_1^1(M)$ and its {\em gauge-fixed} projection $\textrm{FGF}_1^1(M)_{d^*=0}$ arise as small-scale limits of $\unitary(1)$ Yang-Mills gauge theory, since the Lie algebra corresponding to $\unitary(1)$ is $\mathbb R$ (see Section~\ref{section:U(1)-theory}). When $\unitary(1)$ is replaced with another compact Lie group, the corresponding small-scale limit is a Lie-algebra-valued analog of $\textrm{FGF}_1^1(M)_{d^*=0}$. If $A$ is an instance of $\textrm{FGF}_1^1(M)_{d^*=0}$, then the curvature $dA$ has the law of (a Lie-algebra valued analog of) $\textrm{FGF}_2^0(M)_{d=0}$.

Although Wilson loop observables of this connection are undefined for smooth loops when $n > 2$, it is possible (see Section~\ref{sec:wilsonloops}) to make sense of these these integrals for sufficiently regular ``big loops'' (trajectories in the space of divergence-free $1$-forms obtainable as limits of finite-length loops). We will also define ``big surfaces'' (trajectories in the space of $2$-forms obtainable as limits of smooth surfaces with boundary) and show that Stokes' theorem converts big-loop integrals of $\textrm{FGF}_1^1(M)_{d^*=0}$ into big-surface integrals of $\textrm{FGF}_2^0(M)_{d=0}$ or $\textrm{FGF}_2^0(M)$. In this limiting setting, ``gauge invariance'' refers to the fact that one can add a gradient form to an instance of $\textrm{FGF}_1^1(M)_{d^*=0}$ without impacting the Wilson loop integrals.

Let us conclude this subsection with some Yang-Mills references. The subject is too large for us to properly survey here (an online search for ``Yang-Mills'' or ``gauge theory'' turns up hundreds of thousands of references) but we can suggest a few papers for a mathematician getting started in the subject. For an overview of lattice and continuum gauge theory, the reader may see e.g.\ the recent work of Chatterjee \cite{Chatterjee2024} on Gaussian connections, as well as recent overviews by Chatterjee and others on Yang-Mills gauge theory more broadly
\cite{chatterjee2019yang, chatterjee2016leading,chatterjee2021probabilistic}. See also the overview of the Clay Millenium Prize problem concerning the construction and analysis of non-Gaussian Yang-Mills theories in dimension four \cite{Jaffe2006a}.

If $M$ is a two-dimensional manifold, the theory is somewhat simpler, see e.g.\ recent work by L\'evy and others \cite{Levy2008, Levy2010, Levy2011a, Levy2019a, driver2017a, BG2018, Chevyrev2019, DN2020, DL2023, DL2022}. See \cite{Wilson1974,wilson2001mathematical,wilson2004origins} for the history of {\em Wilson loop expectations}. Their expression in terms of surface expansions goes back at least to early work of Gross, King and Sengupta, and Gross and Taylor \cite{Gross1989, Gross1993, Gross1993a}. Chatterjee and Jafarov developed a representation of Wilson loop expectations in terms of vanishing string trajectories \cite{Chatterjee2019a, chatterjee2016, jafarov2016}. A key step in deriving their representation was a recursion known as the master loop equation, see \cite{shen2022new, DM2022, AN2023, cao2023random} for new derivations of this equation. There are many other recent results on lattice gauge theories, see \cite{Cao2020, FLV2022, GS2023, For2021, FLV2022b, For2022, FLV2023, Adh2024, AC2022, ChaChe2024, CG2024}

For recent progress on the stochastic PDE approach to Yang--Mills, see \cite{shen2018, CCHS2020, CCHS2022, Che2022, shen2023stochastic, CS23, BC2023, BCG2023, BC24, SZZ2024}. A state space for 3D Yang--Mills was proposed in \cite{CC23, CC24, CCHS2022}. See recent work on Brownian loops and random surfaces \cite{ang2022brownian} and recent work of Magee and Puder \cite{Magee2015}. Finally, let us mention that the current authors were part of more recent work using sums over random surfaces to explain both 2D gauge theory (in the continuum) and higher-dimensional gauge theory (on a lattice) \cite{park2023wilson, cao2023random}.

\subsection{Overview of remaining sections}
We will introduce preliminary notation in Sections~\ref{sec:preliminaries} and review the exterior calculus in greater detail \ref{section:exterior-calculus}. This will include more formal definitions of the wedge and Hodge star operators on general compact manifolds, along with proofs of several fundamental identities. We introduce a notion of generalized function (dual of the Lizorkin functions) that is closed under the fractional Laplacian and its inverses.

We construct fraction Gaussian forms formally in Section~\ref{sec:fgfconstruction}. As is the case with scalar fractional Gaussian fields, there is both a \textit{Gaussian Hilbert space} interpretation and a \textit{random generalized function} interpretation. We review Gaussian $1$-forms in the Yang-Mills context in Section~\ref{sec:yangmills}.

In Section~\ref{sec:chernsimon} we discuss the Chern-Simons action and the corresponding random form constructions. We discuss gauge transformations in Section~\ref{sec:gaugetransformations} and subspace restrictions in Section~\ref{sec:subspacerestrictions}.  We discuss Wilson loop expectations and the so-called surface expansion in  Section~\ref{sec:wilsonloops}. We proceed with lattice Gaussian forms in Section~\ref{sec:latticeforms} and the confinement and mass gap problems in Section~\ref{sec:confinement}. Finally we present a list of open problems in Section~\ref{sec:openproblems}, including questions about possible Gaussian scaling limits of discrete models, as well possible off-Gaussian continuum theories constructed via scaling limits or otherwise.

\subsection{Acknowledgements}

We thank Bjoern Bringmann, Sourav Chatterjee, Hao Shen, and Xin Sun for helpful discussions. This work was started while both authors were at the Institute for Advanced Study; we thank IAS for its hospitality during our stay. S.C.\ was partially supported by a Minerva Research Foundation membership while at IAS, as well as by the NSF under Grant No.\ DMS 2303165. S.S.\ was partially supported by NSF Grant No.\ DMS 2153742.

\section{Preliminaries} \label{sec:preliminaries}

In this section, we introduce some notation and definitions that will be used throughout the remainder of the paper. We identify $\T^n \cong [-\pi, \pi)^n$ in this paper.

\begin{notation}
We denote by $\N(0, \sigma^2)$ the normal distribution with mean zero and variance $\sigma^2$. We denote by $\N_\C(0, \sigma^2)$ the law of a complex-valued random variable $Z = X + \icomplex Y$, where $X, Y \stackrel{i.i.d.}{\sim} \N(0, \sigma^2/2)$.
\end{notation}

\begin{notation}
We will implicitly sum over repeated indices, unless otherwise noted. Thus an expression like $x_k y_k$ really means
\[ \sum_k x_k y_k, \]
where $k$ ranges over some index set, which is to be understood from context.
\end{notation}

\begin{notation}[Symmetric group]
For $n \geq 1$, we let $\symgrp_n$ be the symmetric group on $n$ elements.
\end{notation}

\begin{notation}
Given $f, g \colon \manifold \ra \C$, we write
\begin{equs}\label{eq:L2-inner-product}
(f, g) := \int_M f(x) \ovl{g(x)} \mrm{vol}(x),
\end{equs}
where $\mrm{vol}$ is the volume form on $\manifold$. In the case $\manifold = \R^n, \T^n$, this definition just reduces to the usual expression $\int_{\manifold} f(x) \ovl{g(x)} dx$. More generally, given a finite-dimensional inner product space $(V, \langle \cdot, \cdot \rangle_V)$ and $f, g \colon \manifold \ra V$, we write
\begin{equs}
(f, g) = \int_M \langle f(x), g(x) \rangle_V \mrm{vol}(x).
\end{equs}
\end{notation}

\subsection{Distribution spaces}\label{section:distribution-spaces}

We first discuss distributions on $\R^n$. As in \cite[Section 2]{lodhia2016fractional}, let $\mc{S} = \mc{S}(\R^n)$ be the real Schwartz space (i.e. the space of real-valued Schwartz functions), and $\mc{S}'$ be its topological dual, the space of tempered distributions. We denote by $\schwartz_\C$ the complex Schwartz space, and $\schwartz_\C'$ its topological dual.

\begin{remark}[Real vs. complex-valued test functions]
This convention is somewhat nonstandard, as most references will take $\schwartz$ itself to be the space of complex-valued Schwartz functions. Because we will mostly discuss real-valued Gaussian fields in  this survey, we find it convenient to take our spaces of test functions to also be real-valued (as is done in the previous survey \cite{lodhia2016fractional}). Nevertheless, we still need to briefly mention and work with the complex Schwartz space $\schwartz_\C$, since the Fourier transform (as we discuss next) is only closed on $\schwartz_\C$, as the Fourier transform of a real-valued function is in general complex-valued.
\end{remark}

Let $\mc{P} \sse \mc{S}'$ be the space of polynomials. Given $\phi \in \mc{S}'$, $f \in \mc{S}$, we will write $(\phi, f) := \phi(f)$. Let $\mc{F} : \schwartz_\C \ra \schwartz_\C$ be the Fourier transform. Explicitly, for $f \in \schwartz_\C$, 
let
\begin{equs}\label{eq:fourier-transform} (\mc{F} f)(p) := \frac{1}{(2\pi)^{\frac{n}{2}}} \int_{\R^n} f(x) e^{-\icomplex p \cdot x} dx. 
\end{equs}
The inverse Fourier transform is given by
\[ (\mc{F}^{-1} f)(x) := \frac{1}{(2\pi)^{\frac{n}{2}}}\int_{\R^n} f(p) e^{\icomplex p \cdot x} dp. \]
For $f \in \schwartz_\C$, we have that $\mc{F} \mc{F}^{-1} f = f$. We also have the Plancherel identity
\begin{equs}\label{eq:plancherel}
(f, g) = (\mc{F} f, \mc{F} g), ~~ f, g \in \schwartz_\C.
\end{equs}
The Fourier transform extends by duality to an operator $\mc{F} : \schwartz'_\C \ra \schwartz'_\C$ on the space of tempered distributions.

\begin{notation}
We will often use the common notation $\widehat{f}(p) = \mc{F} f(p)$. However, occasionally we will find it advantageous to write $\mc{F} f$ instead of $\widehat{f}$. Also, we will often write $|p|^s \mc{F} f$ as a shorthand for the function $p \mapsto |p|^s (\mc{F} f)(p)$.
\end{notation}

\begin{remark}[Fourier transform of real-valued functions]
Note that for $f \in \schwartz$ (so that $f$ is real-valued), then its Fourier transform satisfies
\begin{equs}
\widehat{f}(-p) = \ovl{\widehat{f}(p)} \text{ for all $p \in \R^n$.}
\end{equs}
The converse is also true.
\end{remark}

\begin{definition}[Fractional Laplacian]
For $s \geq 0$, we define the fractional Laplacian $(-\Delta)^s : \mc{S} \ra \mc{S}'$ by
\begin{equs}
(-\Delta)^s f := \mc{F}^{-1}\big( |\cdot|^{2s} \mc{F} f \big).
\end{equs}
\end{definition}


As mentioned above, a {\em tempered distribution} (a.k.a.\ {\em generalized function}) is a continuous linear map on the Schwartz space. The idea behind the definition is that even if a $\phi \in \mc S'$ is ``too rough'' to be defined as a function at individual points, we can characterize $\phi$ completely by specifying $(\phi, f)$ for $f$ in the space $\mc S$ of ``extraordinarily well-behaved'' functions.

When working with fractional Gaussian fields, however, the Schwartz space $\mc{S}$ is not the most convenient space for this purpose, because $\mc{S}$ is not closed under taking negative powers of the Laplacian. An alternative is the Lizorkin space $\Phi \sse \mc{S}$, which we define next, and which has the key property that $(-\Delta)^s : \Phi \ra \Phi$ is an isomorphism for any $s \in \R$.

\begin{definition}[Lizorkin space]
A {\bf Lizorkin function} is an element of the set $\Phi := \bigcap_{m \geq 0} \Delta^m \mc{S}$. We equip $\Phi \sse \mc{S}$ with the subspace topology induced by the topology of $\mc{S}$. Let $\Psi := \mc{F}(\Phi)$ be the image of the Fourier transform on $\Phi$. Let $\Phi'$ be the topological dual of $\Phi$. Elements of $\Phi'$ are referred to as Lizorkin distributions.
\end{definition}

\begin{remark}
When working with fractional Gaussian fields, it is often necessary to work with a subset of Schwartz space. The most common example is the 2D GFF, which is only defined modulo additive constant, or equivalently, is only defined on mean zero test functions. In this survey, our default space of test functions will be the space $\Phi$ of Lizorkin functions, and our default notion of generalized function is an element of the dual space $\Phi'$. This is done mostly for technical convenience, to ensure that arbitrary powers of the Laplacian are well defined, and so that we have a simple space that works for all $s$ simultaneously (see Proposition \ref{prop:fractional-laplacian-isomorphism-on-lizorkin-space}). As mentioned in Section~\ref{sec:fgfoverview}, one does not lose any {\em information} by restricting to test functions in $\Phi$. The reason is that Lizorkin space is dense in the relevant function spaces (Lemma \ref{lemma:lizorkin-space-dense-fractional-sobolev-space})---in particular, if $\phi$ is in the Hilbert space of functions for which $(h, \phi)$ can be defined as a Gaussian random variable (see Section~\ref{sec:fgfoverview}) then $\phi$ is a Hilbert-space limit of elements of $\Phi$, and $(h,\phi)$ can be defined as an a.s.\ limit of a sequence of $(h,\psi)$ with $\psi \in \Phi$. As mentioned in Section~\ref{sec:fgfoverview}, there are many alternative spaces one can use, see e.g.\ the approaches in previous surveys in this area \cite{sheffield2007gaussian, lodhia2016fractional, duplantier2017log}.
\end{remark}

For the reader's convenience, we give some equivalent characterizations of the spaces $\Phi, \Psi$. The proofs may be found in \cite[Section 4]{troyanov2007hodge}. 
\begin{lemma}\label{lemma:phi-psi-characterizations}
Let $f \in \mc{S}$. We have that $f \in \Phi$ if and only if $(P, f) = 0$ for any polynomial $P \in \mc{P}$. Next, let $g \in \mc{S}$. The following are equivalent:
\begin{enumerate}
    \item $g \in \Psi$,
    \item $(\ptl^\alpha g)(0) = 0$ for all multi-indices $\alpha$,
    \item $(\ptl^\alpha g)(p) = o(|p|^t)$ as $|p| \ra 0$ for any multi-index $\alpha$ and any $t > 0$,
    \item $|p|^{-2m} g \in \mc{S}$ for any $m \geq 0$.
\end{enumerate}
\end{lemma}

\begin{remark}\label{remark:lizorkin-space-closed-under-differentiation}
From this lemma, we see that $\Phi$ is closed under taking derivatives. Indeed, for $f \in \Phi$, $P \in \mc{P}$, and a multi-index $\alpha$, we have that (by integration by parts)
\[ (P, \ptl_\alpha f) = (-1)^{|\alpha|} (\ptl_\alpha P, f) = 0.\]
\end{remark}

We also mention that $\Phi'$ may naturally be interpreted as the space of tempered distributions modulo polynomials:

\begin{lemma}[Proposition 4.8 of \cite{troyanov2007hodge}]
The space $\Phi'$ is canonically isomorphic to $\mc{S}' / \mc{P}$.
\end{lemma}

\begin{remark}
Concretely, this lemma says that given $\phi_1, \phi_2 \in \mc{S}'$ such that $\phi_1 = \phi_2$ on $\Phi$, there exists a polynomial $P \in \mc{P}$ such that $\phi_1 = \phi_2 + P$.
\end{remark}

Lemma \ref{lemma:phi-psi-characterizations} also allows us to define arbitrary powers of the fractional Laplacian on Lizorkin space.

\begin{definition}[Fractional Laplacian on Lizorkin space]
For $s \in \R$, we define the fractional Laplacian $(-\Delta)^s : \Phi \ra \Phi$ by
\begin{equs}
(-\Delta)^s f := \mc{F}^{-1}\big( |\cdot|^{2s} \mc{F} f\big).
\end{equs}
\end{definition}

The next result follows from the discussion in \cite[Section 5]{troyanov2007hodge}. As previously alluded to, this result is the main reason why we work with $\Phi, \Phi'$ instead of $\mc{S}, \mc{S}'$.

\begin{prop}\label{prop:fractional-laplacian-isomorphism-on-lizorkin-space}
For any $s \in \R$, $(-\Delta)^s : \Phi \ra \Phi$ is an isomorphism with inverse $(-\Delta)^{-s}$.
\end{prop}

Next, we define the homogeneous Sobolev spaces, which are natural spaces of test functions for fractional Gaussian fields.

\begin{definition}[Sobolev spaces]\label{def:sobolev}
Let $s \in \R$. Define the pre-Hilbert space 
\[ \mathring{H}^s = \mathring{H}^s(\R^n) := \{f \in \mc{S} : (-\Delta)^{\frac{s}{2}} f \in L^2(\R^n)\}, \]
equipped with the inner product (recall the Plancherel identity \eqref{eq:plancherel})
\[ (f, g)_{\dot{H}^s(\R^n)} := ((-\Delta)^{\frac{s}{2}} f, (-\Delta)^{\frac{s}{2}} g) = (|p|^{s} \mc{F}f, |p|^{s} \mc{F}g). \]
Let $\dot{H}^s = \dot{H}^s(\R^n)$ be the Hilbert space completion of $(\mathring{H}^s(\R^n), (\cdot, \cdot)_{\dot{H}^s})$. The space $\dot{H}^s(\R^n)$ is the homogeneous Sobolev space of order $s$.

We define the inhomogeneous Sobolev space $H^s(\R^n)$ in the same manner, except we use the inner product
\begin{equs}
(f, g)_{H^s(\R^n)} := \big((-\Delta + 1)^{\frac{s}{2}} f, (-\Delta + 1)^{\frac{s}{2}} g\big) = \big ((|p|^2 + 1)^{\frac{s}{2}} \mc{F} f, (|p|^2 + 1)^{\frac{s}{2}} \mc{F} g\big ), 
\end{equs}
and for the pre-Hilbert space we use $\schwartz$ (note that $(-\Delta + 1)^{\frac{s}{2}}  f \in L^2(\R^n)$ for all $f \in \schwartz$).
\end{definition}

The following lemma shows that $\Phi$ is dense in every homogeneous Sobolev space. This implies that (at least from some ``information'' point of view) we lose nothing by working with Lizorkin space instead of Schwartz space.

\begin{lemma}\label{lemma:lizorkin-space-dense-fractional-sobolev-space}
For any $s \in \R$, we have that $\Phi$ is dense in $\dot{H}^s(\R^n)$.
\end{lemma}
\begin{proof}
The fact that $\Phi$ is dense in $\dot{H}^0 = L^2$ follows by \cite[Proposition 6.1]{troyanov2007hodge}. Now let $s \in \R$ be general. It suffices to show that $\Phi$ is dense in $\mathring{H}^s$ under $(\cdot, \cdot)_{\dot{H}^s}$.  Let $f \in \mathring{H}^s$. Let $g := (-\Delta)^{\frac{s}{2}} f$. Observe that $\widehat{g} \in L^2$, and thus also $g \in L^2$. Moreover $\|f\|_{\dot{H}^s} = \|g\|_{L^2}$. By the density of $\Phi$ in $L^2$, there exists a sequence $\{g_n\}_{n \geq 1} \sse \Phi$ such that $g_n \stackrel{L^2}{\ra} g$. Define $f_n := (-\Delta)^{-\frac{s}{2}} g_n$. By Proposition \ref{prop:fractional-laplacian-isomorphism-on-lizorkin-space}, we have that  $f_n \in \Phi$. We also have that
\[ \|f_n - f\|_{\dot{H}^s} = \|(-\Delta)^{\frac{s}{2}} f_n - (-\Delta)^{\frac{s}{2}} f\|_{L^2} = \|g_n - g\|_{L^2} \ra 0. \qedhere\]
\end{proof}

The fractional Laplacian relates homogeneous Sobolev spaces of different orders, as the next lemma shows.

\begin{lemma}\label{lemma:fractional-laplacian-isometry-fractional-sobolev-space}
Let $s, s' \in \R$. We have that $(-\Delta)^{s'} : \dot{H}^s(\R^n) \ra \dot{H}^{s - 2s'}(\R^n)$ is an isometry.
\end{lemma}
\begin{proof}
It suffices to verify the isometry on the dense subspace $\Phi \sse \dot{H}^s$ (recall Lemma \ref{lemma:lizorkin-space-dense-fractional-sobolev-space}). For $f \in \Phi$, we have that
\[\begin{split}
\|(-\Delta)^{s'} f\|_{\dot{H}^{s - 2s'}}^2 &= (|p|^{s - 2s'} \mc{F}((-\Delta)^{s'} f), |p|^{s - 2s'} \mc{F}((-\Delta)^{s'} f)) \\
&= (|p|^{s - 2s'} |p|^{2s'} \mc{F}f, |p|^{s - 2s'} |p|^{2s'} \mc{F}f) \\
&= (|p|^s \mc{F} f , |p|^s \mc{F} f) = \|f\|_{\dot{H}^s}^2,  
\end{split}\]
as desired.
\end{proof}

Next, we define the massive fractional Laplacian, which is simpler to deal with. It maps Schwartz space to Schwartz space, so there is no need to deal with Lizorkin distributions.

\begin{definition}[Massive fractional Laplacian]
For $\lambda > 0$, $s \in \R$, define $(-\Delta + \lambda)^s : \mc{S} \ra \mc{S}$ by
\begin{equs}
(-\Delta + \lambda)^s f := \mc{F}^{-1}\big( (|\cdot|^2 + \lambda)^s \mc{F} f\big).
\end{equs}
\end{definition}

Next, we discuss distributions on the torus. Recall that we identify $\T^n$ with $[-\pi, \pi)^n$. Since $\T^n$ is compact, it suffices to work with the space $C^\infty(\T^n)$ of real-valued smooth functions. We equip this space with the topology of uniform convergence of all derivatives. Let $\mc{D}'(\T^n)$ be the topological dual of $C^\infty(\T^n)$. For $\alpha \in \Z^n$, let $\e_\alpha \in C^\infty(\T^n)$ be the function $\e_\alpha(x) := e^{\icomplex \alpha \cdot x}$. For $f \in L^1(\T^n)$, define the Fourier transform
\begin{equs}
(\mc{F} f)(\alpha) = \widehat{f}(\alpha) := \frac{1}{(2\pi)^{\frac{n}{2}}} \int_{\T^n} f(x) \ovl{\e_{\alpha}(x)} dx, ~~ \alpha \in \Z^n.
\end{equs}
The Fourier inversion formula on the torus is then given by (at least if $f \in C^\infty(\T^n)$)
\begin{equs}
f = \frac{1}{(2\pi)^{\frac{n}{2}}} \sum_{\alpha \in \Z^n} \widehat{f}(\alpha) \e_{\alpha}. 
\end{equs}
The Plancherel identity reads
\begin{equs}\label{eq:torus-plancherel}
(f, g) = \sum_{\alpha \in \Z^n} \widehat{f}(\alpha) \ovl{\widehat{g}(\alpha)}.
\end{equs}
For a positive power $s \geq 0$, we may define the fractional Laplacian $-\Delta : C^\infty(\T^n) \ra C^\infty(\T^n)$ in Fourier space by
\begin{equs}
\widehat{(-\Delta)^s f}(\alpha) := |\alpha|^{2s} \widehat{f}(\alpha), ~~ \alpha \in \Z^n.
\end{equs}
On the other hand, negative powers of the Laplacian are only defined on the subspace of $C^\infty(\T^n)$ consisting of mean-zero functions.

\begin{definition}
Let $\bigdot{C}^\infty(\T^n)$ be the subspace of $C^\infty(\T^n)$ consisting of mean-zero functions, i.e. those functions $f \in C^\infty(\T^n)$ such that $\int_{\T^n} f(x) dx = 0$. Let $\bigdot{\mc{D}}' = \bigdot{\mc{D}}'(\T^n)$ be the topological dual of $\bigdot{C}^\infty(\T^n)$.
\end{definition}

For $f \in \bigdot{C}^\infty(\T^n)$, the Fourier representation of $f$ is given by
\begin{equs}
f = \frac{1}{(2\pi)^{\frac{n}{2}}} \sum_{0 \neq \alpha \in \Z^n} \widehat{f}(\alpha) \e_\alpha.
\end{equs}
Thus for $s < 0$, we may define $(-\Delta)^s \colon \bigdot{C}^\infty(\T^n) \ra \bigdot{C}^\infty(\T^n)$
\begin{equs}
\widehat{(-\Delta)^s f}(\alpha) := |\alpha|^{2s} \widehat{f}(\alpha), ~~ 0 \neq \alpha \in \Z^n.
\end{equs}
Similar to the case of $\R^n$, powers of the massive fractional Laplacian may be defined directly on $C^\infty(\T^n)$. For $\lambda > 0$, $s \in \R$, we may define the massive fractional Laplacian $(-\Delta + \lambda)^s \colon C^\infty(\T^n) \ra C^\infty(\T^n)$ in Fourier space by
\begin{equs}
\widehat{(-\Delta + \lambda)^s f}(\alpha) := (|\alpha|^2 + \lambda)^s \widehat{f}(\alpha), ~~ \alpha \in \Z^n.
\end{equs}
Next, for $s \in \R$, we define the homogeneous Sobolev spaces $\dot{H}^s(\T^n)$ as the completion of $\bigdot{C}^\infty(\T^n)$ with respect to the inner product
\begin{equs}
(f, g)_{\dot{H}^s(\T^n)} := ((-\Delta)^{\frac{s}{2}} f, (-\Delta)^{\frac{s}{2}} g) = \sum_{0 \neq \alpha \in \Z^n} |\alpha|^{2s} \widehat{f}(\alpha) \ovl{\widehat{g}(\alpha)}.
\end{equs}
We state the following analog of Lemma \ref{lemma:fractional-laplacian-isometry-fractional-sobolev-space}. The proof is omitted.

\begin{lemma}
Let $s, s' \in \R$. We have that $(-\Delta)^{s'} : \dot{H}^s(\T^n) \ra \dot{H}^{s - 2s'}(\T^n)$ is an isometry.
\end{lemma}

For $s \in \R$, we also define the inhomogeneous Sobolev spaces $H^s(\T^n)$ as the completion of $C^\infty(\T^n)$ with respect to the inner product
\begin{equs}
(f, g)_{H^s(\T^n)} := \big((-\Delta + 1)^{\frac{s}{2}} f, (-\Delta + 1)^{\frac{s}{2}} g\big) =  \sum_{\alpha \in \Z^n} (|\alpha|^2 + 1)^{s} \widehat{f}(\alpha) \ovl{\widehat{g}(\alpha)}.
\end{equs}

\begin{remark}[Values in a general vector space]\label{remark:values-in-a-general-vector-space}
The previous discussion directly extends to setting where all functions and distributions take values in a finite dimensional inner product space $(V, \langle \cdot, \cdot \rangle_V)$. One simply replaces products of real numbers by inner products of elements in $V$. When discussing Fourier transforms, in general the Fourier coefficients will take values in the complexified space $V^\C := \{v + \icomplex w : v, w \in V\}$, which we may endow with the inner product
\begin{equs}
\langle v_1 + \icomplex w_1, v_2 + \icomplex w_2 \rangle_{V^\C} := \langle v_1, v_2 \rangle_V +  \langle w_1, w_2 \rangle_V  + \icomplex \big(\langle w_1, v_2 \rangle_V - \langle v_1, w_2 \rangle_V\big).
\end{equs}
Thus, products of Fourier coefficients should be replaced by inner products defined using the above formula.
\end{remark}

\begin{notation}[Vector-valued function spaces]
We will denote $V$-valued function spaces by adding ``$(V)$" or ``$;V$", e.g. $\Phi(V)$ is the $V$-valued Lizorkin space, $\dot{H}^s(\manifold; V)$ is the $V$-valued homogeneous Sobolev space, etc.
\end{notation}

\section{Review of exterior calculus}\label{section:exterior-calculus}

In this section, we introduce/review the basic concepts associated to differential forms on a manifold $\manifold$. We start in the cases $\manifold = \R^n, \T^n$ in Section \ref{sec:flatspace}. In these cases, we will make our discussion as concrete as possible, in order to make this part of the survey accessible to as many people as possible. We then review the additional geometric concepts that come into play on more general manifolds in Section \ref{section:general-manifold}. Here, the discussion may be a bit less accessible to readers without a differential geometry background, but it will serve to set the notation that we will later use. The reader who is less familiar with differential topology may prefer to skip this section and focus on the cases $\manifold = \R^n, \T^n$ throughout this survey. This section contains a lot of routine calculation that some readers may prefer to skim or skip on a first read. In that case, one may move on to the construction of fractional Gaussian forms in Section~\ref{sec:fgfconstruction} and refer back to this section for reference.

\subsection{Flat space} \label{sec:flatspace}

Throughout this subsection, let $\manifold = \R^n$ or $\T^n$. 

\begin{definition}
Let $n \geq 1$. Let $\varep_{i_1 \cdots i_n} = \varep^{i_1 \cdots i_n}$ be the totally antisymmetric tensor. I.e., $\varep_{i_1 \cdots i_n} = \varep^{i_1 \cdots i_n} = (-1)^{\mrm{sgn}(\sigma)}$, where $\sigma$ is the permutation which takes $(i_1, \ldots, i_n)$ to $(1, \ldots, n)$ if the $i_1, \ldots, i_n$ are distinct, and $\varep_{i_1 \cdots i_n} = 0$ otherwise. Note that $\varep$ is characterized by the following two conditions:
\begin{enumerate}
    \item $\varep_{1 \cdots n} = 1$
    \item for any transposition $\tau \in \symgrp_n$, $\varep_{i_{\tau(1)} \cdots i_{\tau(n)}} = - \varep_{i_1 \cdots i_n}$.
\end{enumerate}
\end{definition}

\begin{definition}[Differential forms]\label{def:k-form}
A $0$-form $f$ on $\manifold$ is simply a function $f \colon M \ra \R$. Let $1 \leq k \leq n$. A $k$-form can be expressed as 
\beq\label{eq:k-form-basic-expression} f = \sum_{1 \leq i_1 < \cdots < i_k \leq n} f_{i_1 \cdots i_k} dx^{i_1} \wedge \cdots \wedge dx^{i_k}. \eeq
For general collections of indices $i_1, \ldots, i_k \in [n]$, the wedge symbols $dx^{i_1} \wedge \cdots \wedge dx^{i_k}$ satisfy the antisymmetry condition 
\begin{equs}
dx^{i_{\sigma(1)}} \wedge \cdots dx^{i_{\sigma(k)}} = \varep^{\sigma(1) \cdots \sigma(k)} dx^{i_1} \wedge \cdots dx^{i_k}. 
\end{equs}
In particular, $dx^{i_1} \wedge \cdots \wedge dx^{i_k} = 0$ if the $i_1, \ldots, i_k$ are not distinct. Using this convention, we may express $k$-forms as linear combinations of the wedge symbols where the $i_1, \ldots, i_k$ are not in increasing order. 

Given a property P (e.g., smooth, $L^2$, etc.), we will say that a $k$-form $f$ satisfies P if each $f_{i_1 \cdots i_k}$ satisfies P. Given a space $\mc{F}$ of functions $\manifold \ra \R$ (e.g. $C^\infty(\manifold)$), we let $\Omega^k \mc{F}$ be the space of $k$-forms $f$ such that each $f_{i_1 \cdots i_k} \in \mc{F}$. So $\Omega^k C^\infty(\manifold)$ is the space of smooth $k$-forms, $\Omega^k L^p(\manifold)$ the space of $L^p$ $k$-forms, etc. By default, $\Omega^k(\manifold) = \Omega^k C^\infty(\manifold)$, i.e. the space of smooth $k$-forms. When $\manifold = \T^n$, we will also use the notation $\bigdot{\Omega}^k(\T^n) = \Omega^k \bigdot{C}^\infty(\T^n)$, i.e. the space of smooth $k$-forms with zero mean.
\end{definition}


\begin{remark}\label{remark:k-form-alternating-multilinear-map}
One may be wondering at this point about the meaning of the wedge terms $dx^{i_1} \wedge \cdots \wedge dx^{i_k}$ in \eqref{eq:k-form-basic-expression}. For this survey, we essentially treat these as formal symbols which satisfy the antisymmetry condition from Definition \ref{def:k-form}. Thus we will often equivalently regard $k$-forms as collections of functions $(f_{i_1 \cdots i_k}, i_1, \ldots, i_k \in [n])$ satisfying the same antisymmetry conditions, that is
\begin{equs}
f_{i_{\sigma(1)} \cdots i_{\sigma(k)}} = \varep_{\sigma(1) \cdots \sigma(k)} f_{i_1 \cdots i_k}, ~~ i_1, \ldots, i_k \in [n], \sigma \in \symgrp_k.
\end{equs}
We just briefly remark here that a more correct way to view $k$-forms is as follows. The term $dx^{i_1} \wedge \cdots \wedge dx^{i_k}$ may be interpreted as the alternating multilinear map $(\R^n)^k \ra \R$ given by
\[ (v_1, \ldots, v_k) \mapsto \det(\langle v_j, e_{i_\ell}\rangle_{j, \ell = 1}^k),\]
where $e_1, \ldots, e_n$ is the standard basis of $\R^n$. Actually, this is not quite right, as really one should say that given $x \in \manifold$, $dx^{i_1} \wedge \cdots \wedge dx^{i_k}$ evaluated at $x$ is an alternating multilinear map $(T_x\manifold)^k \ra \R$ specified by the above display. However, because we are in the geometrically trivial setting of flat space, all tangent spaces $T_x \manifold$ may be canonically identified with $\R^n$ (and in particular, the basis $e_1, \ldots, e_n$ is an orthonormal basis of $T_x \manifold$ for each $x \in \manifold$), which is the reason why we can just interpret $dx^{i_1} \wedge \cdots \wedge dx^{i_k}$ itself as a map on $(\R^n)^k$. Note that the antisymmetry property of $dx^{i_1} \wedge \cdots \wedge dx^{i_k}$ arises from the usual antisymmetry property of the determinant of a matrix.

Given this discussion, the expression \eqref{eq:k-form-basic-expression} for a general $k$-form $f$ should be interpreted as a function on $\manifold$ which associates to each $x \in \manifold$ an alternating multilinear map $f(x) \colon (\R^n)^k \ra \R$ given by
\[\begin{split}
(v_1, \ldots, v_k) &\mapsto \sum_{1 \leq i_1 < \cdots \leq i_k \leq n} f_{i_1 \cdots i_k}(x) (dx^{i_1} \wedge \cdots dx^{i_k})(v_1, \ldots, v_k) \\
&= \sum_{1 \leq i_1 < \cdots \leq i_k \leq n} f_{i_1 \cdots i_k}(x) \det(\langle v_j, e_{i_\ell}\rangle_{j, \ell = 1}^k).
\end{split}\]
\end{remark}



\begin{definition}[Inner product on $k$-forms]
We define an inner product $(\cdot, \cdot)$ on $\Omega^k L^2(\manifold)$ by
\[ (f, g) := \sum_{1 \leq i_1 < \cdots \leq i_k \leq n} (f_{i_1 \cdots i_k}, g_{i_1 \cdots i_k}), \]
where $(f_{i_1 \cdots i_k}, g_{i_1 \cdots i_k})$ denotes the usual inner product for functions introduced in \eqref{eq:L2-inner-product}.
\end{definition}


\begin{definition}[Exterior derivative]\label{def:exterior-derivative}
Let $0 \leq k \leq n-1$.
The exterior derivative of a $k$-form $f$ is a $(k+1)$-form $df$ defined by
\[ df := \sum_{1 \leq i_1 < \cdots < i_k \leq n} \sum_{j=1}^n (\ptl_j f_{i_1 \cdots i_k}) dx^j \wedge dx^{i_1} \wedge \cdots \wedge dx^{i_k}.\]
Note that if $i_1, \ldots, i_{k+1} \in [n]$ are distinct indices, then in the interpretation of differential forms as a tuple of functions, we have that
\begin{equs}\label{eq:exterior-derivative-coordinatewise-formula}
(df)_{i_1, \ldots, i_{k+1}} = \sum_{\ell=1}^{k+1} (-1)^{\ell-1} \ptl_{i_\ell} f_{i_1 \cdots \hat{i}_\ell \cdots i_{k+1}},
\end{equs}
where $\hat{i}_\ell$ denotes omission of that index.
If $f$ is an $n$-form, then we define $df : = 0$.
\end{definition}


\begin{remark}
One could in principle add subscripts and denote $d_k : \Omega^k\ra \Omega^{k+1}$. But we will follow the usual convention of treating $d$ as a single operator acting on all of $\Omega$. When we focus on a specific domain, it should be clear from context.
\end{remark}

\begin{definition}[Codifferential]
Let $1 \leq k \leq n$. The codifferential of a $k$-form $f$ is a $(k-1)$-form $\codif f$ defined by 
\begin{equs}
\codif f = -\sum_{1 \leq i_1 < \cdots < i_{k-1} \leq n} \sum_{j=1}^n \ptl_j g_{j i_1 \cdots i_{k-1}} dx^{i_1} \wedge \cdots \wedge dx^{i_{k-1}}.
\end{equs}
In the interpretation of differential forms as a tuple of functions, we have that (note the implicit summation over the index $j$)
\beq\label{eq:codifferential-coordinatewise-expression} (\codif f)_{i_1 \cdots i_{k-1}} = -\ptl_j f_{j i_1 \cdots i_{k-1}} .\eeq
\end{definition}

An alternative formula for the codifferential is
\begin{equs}\label{eq:codif-alternative-formula}
\codif f = \sum_{1 \leq i_1 < \cdots < i_k \leq n} \sum_{\ell=1}^k (-1)^\ell \ptl_{i_\ell} f_{i_1 \cdots i_k} dx^{i_1} \wedge \cdots \wedge \widehat{dx}^{i_\ell} \wedge \cdots \wedge dx^{i_k},
\end{equs}
where $\widehat{dx}^{i_\ell}$ denotes omission of that term. The codifferential $\codif$ is defined precisely so that $d$ and $\codif$ are adjoint, as the next lemma shows.

\begin{lemma}[Exterior derivative and codiferential are adjoint]
Let $1 \leq k \leq n$. Let $f \in \Omega^{k-1}(\manifold) \cap \Omega^{k-1} L^2(\manifold)$, $g \in \Omega^k (\manifold) \cap \Omega^k L^2(\manifold)$. We have that
\[ (df, g) = (f, \codif g).\]
\end{lemma}
\begin{proof}
We have that (using \eqref{eq:exterior-derivative-coordinatewise-formula} in the second identity, integration by parts and the assumption that $f, g$ are compactly supported in the third identity, and \eqref{eq:codifferential-coordinatewise-expression} in the second-to-last identity)
\[\begin{split}
(df, g) &= \sum_{1 \leq i_1 < \cdots < i_{k+1} \leq n} ((df)_{i_1 \cdots i_{k+1}}, g_{i_1 \cdots i_{k+1}}) \\
&=\sum_{1 \leq i_1 < \cdots < i_{k+1} \leq n}  \sum_{\ell=1}^{k+1} (-1)^{\ell-1} (\ptl_{i_\ell} f_{i_1 \cdots \hat{i}_\ell \cdots i_{k+1}}, g_{i_1 \cdots i_{k+1}}) \\
&=  \sum_{1 \leq i_1 < \cdots < i_{k+1} \leq n} \sum_{\ell=1}^{k+1} (-1)^\ell (f_{i_1 \cdots \hat{i}_\ell \cdots i_{k+1}}, \ptl_{i_\ell} g_{i_1 \cdots i_{k+1}}) \\
&= \sum_{1 \leq j_1 < \cdots \leq j_k \leq n} \sum_{j=1}^n (f_{j_1 \cdots j_k}, \ptl_j g_{j j_1 \cdots j_{k+1}}) \\
&= \sum_{1 \leq j_1 < \cdots \leq j_k \leq n} (f_{j_1 \cdots j_k}, (\codif g)_{j_1 \cdots j_{k}}) = (f, \codif g). \qedhere
\end{split}\]
\end{proof}

\begin{lemma}\label{lemma:d-squared-zero}
Let $0 \leq k \leq n$. Let $f$ be a $k$-form. We have that $ddf = 0$.
\end{lemma}
\begin{proof}
We have that 
\begin{equs}
d df &=  \sum_{1 \leq i_1 < \cdots < i_k \leq n} \sum_{j=1}^n d\Big( \ptl_j f_{i_1 \cdots i_k} dx^j \wedge dx^{i_1} \wedge \cdots \wedge dx^{i_k}\Big) \\
&= \sum_{1 \leq i_1 < \cdots < i_k \leq n} \sum_{j=1}^n \sum_{j' = 1}^n \ptl_{j'} \ptl_j f_{i_1 \cdots i_k} dx^{j'} \wedge dx^j \wedge dx^{i_1} \wedge \cdots \wedge dx^{i_k}.
\end{equs}
Using that $\ptl_{j'} \ptl_j = \ptl_j \ptl_{j'}$ and that $dx^{j'} \wedge dx^j = - dx^j \wedge dx^{j'}$, we see that the above is zero, as desired.
\end{proof}

\begin{definition}[Hodge star]\label{def:hodge-star}
Let $0 \leq k \leq n$. Given a $k$-form $f$, the Hodge star $\star f$ of $f$ is a $(n-k)$-form defined by
\[ \star f := \sum_{1 \leq i_1 < \cdots < i_k \leq n} \sum_{1 \leq j_1 < \cdots < j_{n-k} \leq n}  \varep_{i_1 \cdots i_k j_1 \cdots j_{n-k}} f_{i_1 \cdots i_k} dx^{j_1} \wedge \cdots \wedge dx^{j_{n-k}}.  \]
Note in particular that
\beq\label{eq:hodge-star-coordinatewise-formula} (\star f)_{j_1 \cdots j_{n-k}} =  \sum_{1 \leq i_1 < \cdots < i_k \leq n} \varep_{i_1 \cdots i_k j_1 \cdots j_{n-k}} f_{i_1 \cdots i_k}.\eeq
\end{definition}

\begin{remark}\label{remark:hodge-star-j}
Let $1 \leq j_1 < \cdots < j_{n-k} \leq n$. Let $1 \leq i_1 < \cdots < i_k \leq n$ be the unique indices such that $\{j_1, \ldots, j_{n-k}, i_1, \ldots, i_k\} = [n]$. Then
\begin{equs}
\star (dx^{j_1} \wedge \cdots dx^{j_{n-k}}) = \varep_{j_1 \cdots j_{n-k} i_1 \cdots i_k} dx^{i_1} \wedge \cdots \wedge dx^{i_k}.
\end{equs}
\end{remark}



\begin{lemma}\label{lemma:double-hodge-star}
Let $0 \leq k \leq n$. Let $f$ be a $k$-form. Then $\star \star f = (-1)^{k(n-k)} f$.
\end{lemma}
\begin{proof}
By definition and Remark \ref{remark:hodge-star-j}, we have that
\begin{equs}
\star (\star f) &=  \sum_{1 \leq i_1 < \cdots < i_k \leq n} \sum_{1 \leq j_1 < \cdots < j_{n-k} \leq n} \varep_{i_1 \cdots i_{k} j_1 \cdots j_{n-k}} f_{i_1 \cdots i_k} \star( dx^{j_1} \wedge \cdots \wedge dx^{j_k}) \\
&= \sum_{1 \leq i_1 < \cdots < i_k \leq n} \sum_{1 \leq j_1 < \cdots < j_{n-k} \leq n} \varep_{i_1 \cdots i_{k} j_1 \cdots j_{n-k}} \varep_{j_1 \cdots j_{n-k} i_1 \cdots i_k} f_{i_1 \cdots i_k} dx^{i_1} \wedge \cdots \wedge dx^{i_k}.
\end{equs}
Observe that for any fixed $1 \leq i_1 < \cdots < i_k \leq n$, there is a unique ordered collection $1 \leq j_1 < \cdots < j_{n-k}$ such that $\varep_{i_1 \cdots i_k j_1 \cdots j_{n-k}} \neq 0$, and for this collection, 
\begin{equs}
\varep_{i_1 \cdots i_{k} j_1 \cdots j_{n-k}} \varep_{j_1 \cdots j_{n-k} i_1 \cdots i_k} = (-1)^{k(n-k)} \varep_{i_1 \cdots i_{k} j_1 \cdots j_{n-k}}^2 = (-1)^{k(n-k)}.
\end{equs}
From this, it follows that
\[
\star (\star f) = (-1)^{k(n-k)} \sum_{1 \leq i_1 < \cdots < i_k \leq n}  f_{i_1 \cdots i_k} dx^{i_1} \wedge \cdots \wedge dx^{i_k}. \qedhere
\]
\end{proof}

\begin{lemma}
Let $0 \leq k \leq n$. Let $f \in \Omega^k L^2(\manifold)$, $g \in \Omega^{n-k} L^2(\manifold)$. Then
\[ (\star f, g) = (-1)^{k(n-k)} (f, \star g).\]
\end{lemma}
\begin{proof}
We have that (using \eqref{eq:hodge-star-coordinatewise-formula} in the second and fourth identities)
\begin{align*}
(\star f, g) &= \sum_{1 \leq j_1 < \cdots < j_{n-k} \leq n} \sum_{1 \leq i_1 < \cdots < i_k \leq n} (\varep_{i_1 \cdots i_k j_1 \cdots j_{n-k}} f_{i_1 \cdots i_k}, g_{j_1 \cdots j_{n-k}}) \\
&= (-1)^{k(n-k)} \sum_{1 \leq i_1 < \cdots < i_k \leq n} \sum_{1 \leq j_1 < \cdots < j_{n-k} \leq n} (f_{i_1 \cdots i_k}, \varep_{j_1 \cdots j_{n-k} i_1 \cdots i_k} g_{j_1 \cdots j_{n-k})} \\
&= (-1)^{k(n-k)} (f, \star g). \qedhere
\end{align*}
\end{proof}

The following lemma relates $d$ and $\codif$ via $\star$.

\begin{lemma}\label{lemma:codif-dif-hodge-star}
We have that $\codif = (-1)^{n(k+1) + 1} \star d \star$.
\end{lemma}
\begin{proof}
We have that
\begin{equs}
d (\star f) &= \sum_{1 \leq i_2 < \cdots < i_{n-k+1} \leq n} \sum_{i_1=1}^n \ptl_{i_1} (\star f) _{i_2 \cdots i_{n-k+1}} dx^{i_1} \wedge \cdots \wedge  dx^{i_{n-k+1}} \\
&= \frac{1}{(n-k)!} \sum_{i_2, \ldots, i_{n-k+1} \in [n]} \sum_{i_1 \neq i_2, \ldots, i_{n-k+1}} \ptl_{i_1} (\star f) _{i_2 \cdots i_{n-k+1}} dx^{i_1} \wedge \cdots \wedge dx^{i_{n-k+1}} .
\end{equs}
Next, by Remark \ref{remark:hodge-star-j} we have that
\begin{equs}
(\star f)_{i_2 \cdots i_{n-k+1}} &= \sum_{1 \leq \ell_1 < \cdots < \ell_k \leq n} \varep_{\ell_1 \cdots \ell_k i_ 2 \cdots i_{n-k+1}} f_{\ell_1 \cdots \ell_k} \\
&= \frac{1}{k!} \sum_{\ell_1, \ldots, \ell_k \in [n]} \varep_{\ell_1 \cdots \ell_k i_ 2 \cdots i_{n-k+1}} f_{\ell_1 \cdots \ell_k},
\end{equs}
and thus 
\begin{equs}
d (\star f) = \frac{1}{k! (n-k)!} \sum_{i_2, \ldots, i_{n-k+1} \in [n]} \sum_{i_1 \neq i_2, \ldots, i_{n-k+1}} \sum_{\ell_1, \ldots, \ell_k \in [n]} \varep_{\ell_1 \cdots \ell_k i_ 2 \cdots i_{n-k+1}} \ptl_{i_1} f_{\ell_1 \cdots \ell_k}  dx^{i_1} \wedge \cdots dx^{i_{n-k+1}}.
\end{equs}
Combining this with the fact that for any $i_1, \ldots, i_{n-k+1}$, we have that
\begin{equs}
\star (dx^{i_1} \wedge \cdots \wedge dx^{i_{n-k+1}}) &= \sum_{1 \leq j_1 <  \ldots <  j_{k-1} \in [n]} \varep_{i_1 \cdots i_{n-k+1} j_1 \cdots j_{k-1}} dx^{j_1} \wedge \cdots \wedge dx^{j_{k-1}} \\
&= \frac{1}{(k-1)!} \sum_{j_1, \ldots, j_{k-1} \in [n]} \varep_{i_1 \cdots i_{n-k+1} j_1 \cdots j_{k-1}} dx^{j_1} \wedge \cdots \wedge dx^{j_{k-1}},
\end{equs}
we thus obtain (there is a sum over all indices since all indices are repeated, but we stop explicitly writing the summation symbols)
\[\begin{split}
\star d \star f = \frac{1}{(k-1)!} \frac{1}{k! (n-k)!} \varep_{i_1 \cdots i_{n-k+1} j_1 \cdots j_{k-1}} &\varep_{\ell_1 \cdots \ell_k i_2 \cdots i_{n-k+1}} ~\times \\
&\ptl_{i_1} f_{\ell_1 \cdots \ell_k} dx^{j_1} \wedge \cdots \wedge dx^{j_{k-1}}.
\end{split}\]
To finish, it suffices to show (by \eqref{eq:codifferential-coordinatewise-expression}) that for fixed $j_1, \ldots, j_{k-1}$,
\[ \frac{1}{k! (n-k)!} \varep_{i_1 \cdots i_{n-k+1} j_1 \cdots j_{k-1}} \varep_{\ell_1 \cdots \ell_k i_2 \cdots i_{n-k+1}} \ptl_{i_1} f_{\ell_1 \cdots \ell_k} = (-1)^{n(k+1)} \ptl_i f_{i j_1 \cdots j_{k-1}}. \]
By antisymmetry, we have that
\[\begin{split}
\varep_{i_1 \cdots i_{n-k+1} j_1 \cdots j_{k-1}}& \varep_{\ell_1 \cdots \ell_k i_2 \cdots i_{n-k+1}} =  \\
&(-1)^{(k-1)(n-k)} \varep_{i_1 j_1 \cdots j_{k -1} i_2 \cdots i_{n-k+1}} \varep_{\ell_1 \cdots \ell_k i_2 \cdots i_{n-k+1}}. 
\end{split}\]
For fixed $i_1, \ldots, i_{n-k+1}$, $\ell_1, \ldots, \ell_k$, we have that
\[ \varep_{i_1 j_1 \cdots j_{k-1} i_2 \cdots i_{n-k+1}} \varep_{\ell_1 \cdots \ell_k i_2 \cdots i_{n-k+1}} \neq 0 \iff \{i_1, j_1, \ldots, j_{k-1}\} = \{\ell_1, \ldots, \ell_k\}. \]
Fixing $j_1, \ldots, j_{k-1}$, $i_1, \ldots, i_{n-k+1}$, observe that since $\varep$ and $f$ are both antisymmetric, we have that for $\ell_1, \ldots, \ell_k$ satisfying the condition on the right hand side above,
\[ \varep_{\ell_1 \cdots \ell_k i_2 \cdots i_{n-k+1}} \ptl_{i_1} f_{\ell_1 \cdots \ell_k} = \varep_{i_1 j_1 \cdots j_{k-1} i_2 \cdots i_{n-k+1}} \ptl_{i_1} f_{i_1 j_1 \cdots j_{k-1}}. \]
Combining the previous few observations, we have that
\[\begin{split}
\frac{1}{k! (n-k)!}&\varep_{i_1 \cdots i_{n-k+1} j_1 \cdots j_{k-1}} \varep_{\ell_1 \cdots \ell_k i_2 \cdots i_{n-k+1}} \ptl_{i_1} f_{\ell_1 \cdots \ell_k} \\
&= (-1)^{(k-1)(n-k)} \frac{k!}{k! (n-k)!} \varep_{i_1 j_1 \cdots j_{k-1} i_2 \cdots i_{n-k+1}}^2 \ptl_{i_1} f_{i_1 j_1 \cdots j_{k-1}} =  (-1)^{(k-1)(n-k)} \ptl_{i_1} f_{i_1 j_1 \cdots j_{k-1}}.
\end{split}\]
To finish, observe that $(-1)^{(k-1)(n-k)} = (-1)^{n(k+1)}$.
\end{proof}

\begin{cor}\label{cor:dif-codif-hodge-star-commute}
We have that $\star d \codif = \codif d \star$ and $\star \codif d = d \codif \star$.
\end{cor}
\begin{proof}
We show the first identity; the proof of the second is similar. By Lemmas \ref{lemma:double-hodge-star} and \ref{lemma:codif-dif-hodge-star}, we have that
\[ (-1)^{k(n-k)} \star d = \star d \star \star = (-1)^{n(k+1) + 1} \codif \star. \]
Thus
\[ \star d \codif = (-1)^{n(k+1) + 1 + k(n-k)} \codif \star \codif. \]
Again applying Lemmas \ref{lemma:double-hodge-star} and \ref{lemma:codif-dif-hodge-star}, we have that
\[ \star \codif = (-1)^{n(k+1) + 1} \star \star d \star = (-1)^{n(k+1) + 1 + k(n-k)} d \star. \]
Combining the previous few displays, we thus obtain
\[ \star d \codif = (-1)^{n(k+1) + 1 + k(n-k)} (-1)^{n(k+1) + 1 + k(n-k)} \codif d \star = \codif d \star, \]
as desired.
\end{proof}

\begin{definition}[Hodge Laplacian on differential forms]
Let $0 \leq k \leq n$. We define the $k$-form Laplacian $\Delta : \Omega^k \ra \Omega^k$ by $\Delta := -(d\codif + \codif d)$. This is usually referred to as the Hodge Laplacian.
\end{definition}

\begin{remark}
We note here that this definition is off by a sign compared to the usual textbook definition of the Laplacian on forms, e.g. in \cite{Jost}. In particular, we have that
\[ (f, \Delta f) = - (f, \codif d f + d\codif f) = -(df, df) - (\codif f, \codif f).\]
Thus $\Delta$ is negative definite, just like the usual Laplacian on functions. 
\end{remark}

\begin{remark}
Let $\Omega(\manifold) = \bigoplus_{k = 0}^n \Omega^k(\manifold)$. We may view $d, \codif, \Delta$ as operators on $\Omega(\manifold)$. Then by the facts that $d^2 = 0 = (\codif)^2$, we have that $\Delta = -(d + \codif)^2$.
\end{remark}

In fact, the Hodge Laplacian on differential forms when $\manifold = \R^n, \T^n$ is closely related to the Laplacian on functions, as the next lemma shows.

\begin{lemma}
We have that
\begin{equs}\label{eq:hodge-laplacian-coordinatewise-formula}
\Delta f = \sum_{1 \leq i_1 < \cdots < i_k \leq n} \Delta f_{i_1 \cdots i_k} dx^{i_1} \wedge \cdots \wedge dx^{i_k}, 
\end{equs}
where $\Delta f_{i_1 \cdots i_k}$ is the usual Laplacian on functions. 
\end{lemma}
\begin{proof}
Given a $k$-form $f$, we compute (using \eqref{eq:codif-alternative-formula})
\begin{equs}
d \codif f &= \sum_{1 \leq j_1 < \cdots \leq j_{k-1} \leq n} \sum_{m=1}^n \ptl_m  (\codif f)_{j_1 \cdots j_{k-1}} dx^m \wedge dx^{j_1} \wedge \cdots \wedge dx^{j_{k-1}} \\
&= \frac{1}{(k-1)!} \sum_{j_1, \ldots, j_{k-1} \in [n]}\sum_{m=1}^n \ptl_m  (\codif f)_{j_1 \cdots j_{k-1}} dx^m \wedge dx^{j_1} \wedge \cdots \wedge dx^{j_{k-1}}  \\
&= - \frac{1}{(k-1)!} \sum_{j_1, \ldots, j_{k-1} \in [n]} \sum_{m\neq j_1, \ldots, j_{k-1}} \sum_{j \neq j_1,\ldots, j_{k-1}} \ptl_m \ptl_j f_{j j_1 \cdots j_{k-1}} dx^m \wedge dx^{j_1} \wedge \cdots \wedge dx^{j_{k-1}} .
\end{equs}
We may further split the above sum into the cases $m = j$, $m \neq j$, in which we obtain that the above is equal to $I_1 + I_2$, where
\begin{equs}
I_1 &:= -\frac{1}{(k-1)!} \sum_{j_1, \ldots, j_{k-1} \in [n]} \sum_{j \neq j_1, \ldots, j_{k-1}} \ptl_j \ptl_j f_{j j_1 \cdots j_{k-1}} dx^j \wedge dx^{j_1} \wedge \cdots \wedge dx^{j_{k-1}} \\
&= -\frac{1}{k!} \sum_{j_1, \ldots, j_k} \sum_{\ell=1}^k \ptl_{j_\ell j_\ell} f_{j_1 \cdots j_{k}}  dx^{j_1} \wedge \cdots \wedge dx^{j_k}, \\
I_2 &:= -\frac{1}{(k-1)!} \sum_{j_1, \ldots, j_{k-1} \in [n]} \sum_{j \neq j_1 \cdots j_{k-1}} \sum_{m \neq j, j_1, \ldots, j_{k-1}} \ptl_m \ptl_j f_{j j_1 \cdots j_{k-1}} dx^m \wedge dx^{j_1} \wedge \cdots \wedge dx^{j_{k-1}}.
\end{equs}
Note in the identity for $I_1$, we used the antisymmetry of both the wedge terms as well as the $f_{j_1 \cdots j_k}$, so that 
\begin{equs}
\ptl_j \ptl_j f_{j j_1 \cdots j_{k-1}} dx^j \wedge dx^{j_1} \wedge \cdots \wedge dx^{j_{k-1}} = \ptl_j \ptl_j f_{j_1 j j_2 \cdots j_{k-1}} dx^{j_1} \wedge dx^j \wedge dx^{j_2} \wedge \cdots \wedge dx^{j_{k-1}},
\end{equs}
and more generally the index $j$ can be placed in any of the $k$ possible spots. Using a similar argument based on antisymmetry, we may also express (the index $m$ becomes $j_\ell$ below, and the sum over $\ell$ arises because we can place $m$ into one of $k$ possible spots)
\begin{equs}\label{eq:intermedaite-I2-identity}
I_2 = \frac{1}{k!} \sum_{j_1, \ldots, j_k \in [n]} \sum_{\ell=1}^k (-1)^\ell \sum_{j \neq j_1, \ldots, j_k} \ptl_{j_\ell} \ptl_j f_{j j_1 \cdots \hat{j}_\ell \cdots j_k} dx^{j_1} \wedge \cdots \wedge dx^{j_k}.
\end{equs}
Next, we compute (using \eqref{eq:exterior-derivative-coordinatewise-formula})
\begin{equs}
\codif d f &= - \sum_{1 \leq i_1 < \cdots < i_k \leq n} \sum_{m =1}^n \ptl_m (df)_{m i_1 \cdots i_k} dx^{i_1} \wedge \cdots \wedge dx^{i_k} \\
&= -\frac{1}{k!} \sum_{i_1, \ldots, i_k \in [n]} \sum_{m \neq i_1, \ldots, i_k} \ptl_m (df)_{m i_1 \cdots i_k} dx^{i_1} \wedge \cdots \wedge dx^{i_k} \\
&= -\frac{1}{k!} \sum_{i_1, \ldots, i_k \in [n]} \sum_{m \neq i_1, \ldots, i_k} \bigg(\ptl_m \ptl_m f_{i_1 \cdots i_k} + \sum_{\ell=1}^k (-1)^\ell \ptl_m \ptl_{i_\ell} f_{m i_1 \cdots \hat{i}_\ell \cdots i_k} \bigg) dx^{i_1} \wedge \cdots \wedge dx^{i_k}.
\end{equs}
We may split this sum into $I_3 + I_4$, where
\begin{equs}
I_3 &:= -\frac{1}{k!} \sum_{i_1, \ldots, i_k \in [n]} \sum_{m \neq i_1, \ldots, i_k} \ptl_m \ptl_m f_{i_1 \cdots i_k} dx^{i_1} \wedge \cdots \wedge dx^{i_k} \\
I_4 &:= -\frac{1}{k!} \sum_{i_1, \ldots, i_k \in [n]} \sum_{m \neq i_1, \ldots, i_k} \sum_{\ell=1}^k (-1)^\ell \ptl_m \ptl_{i_\ell} f_{m i_1 \cdots \hat{i}_\ell \cdots i_k}  dx^{i_1} \wedge \cdots \wedge dx^{i_k}.
\end{equs}
By \eqref{eq:intermedaite-I2-identity}, we have that $I_2 + I_4 = 0$, and thus we have that
\begin{equs}
\Delta f &= -(d\codif + \codif d) f = -(I_1 + I_3) \\
&= \frac{1}{k!} \sum_{i_1, \ldots, i_k \in [n]} \Delta f_{i_1 \cdots i_k} dx^{i_1} \wedge \cdots \wedge dx^{i_k} \\
&= \sum_{1 \leq i_1 < \cdots < i_k \leq n} \Delta f_{i_1 \cdots i_k} dx^{i_1} \wedge \cdots \wedge dx^{i_k} ,
\end{equs}
as desired.
\end{proof}

It follows directly from Corollary \ref{cor:dif-codif-hodge-star-commute} that $\Delta$ and $\star$ commute:

\begin{cor}\label{cor:hodge-star-commutes-with-laplacian}
We have that $\star \Delta = \Delta \star$.
\end{cor}

Having made several abstract definitions, we now give a concrete example which recovers classical facts from vector calculus.

\begin{example}[Vector calculus identities]
Let $A$ be a vector field on $\R^3$. One may identify $A$ with the 1-form $A_1 dx^1 + A_2 dx^2 + A_3 dx^3$. Then, one may check that 
\begin{equs}
\curl A = (\ptl_2 A_3 - \ptl_3 A_2, \ptl_3 A_1 - \ptl_1 A_3, \ptl_1 A_2 - \ptl_2 A_1) = \star d A.
\end{equs}
By Lemma \ref{lemma:codif-dif-hodge-star}, we then have that
\begin{equs}
\curl^2 A = \star d \star d A =  d^* d A = - dd^* A -\Delta A .
\end{equs}
Since $d^* A = -\ptl_j A_j = -\nabla \cdot A$, we see that
\begin{equs}\label{eq:curl-squared}
\curl^2 A = d (\nabla \cdot A)  - \Delta A =  \nabla(\nabla \cdot A) - \Delta A,
\end{equs}
which is a classical vector calculus identity. 
From Lemmas \ref{lemma:d-squared-zero}, \ref{lemma:double-hodge-star}, and \ref{lemma:codif-dif-hodge-star}, we also see that the divergence of the curl is zero:
\begin{equs}
\nabla \cdot (\curl A) = -d^* \star d A = (-1)^{n(k+1)} d \star \star d \star A = (-1)^{n(k+1) + k(n-k)} d^2 \star A = 0.
\end{equs}
\end{example}

Next, we state the Hodge decomposition for forms on $\R^n$ and $\T^n$, which was previously mentioned in Section \ref{sec:fracforms}. First, we make the following definition.

\begin{definition}[Spaces of exact and coexact forms]
Let $E \Omega^k L^p(\R^n) := \Omega^k L^p(\R^n) \cap d \Omega^{k-1} \schwartz'(\R^n)$, and $E^* \Omega^k L^p(\R^n) := \Omega^k L^p(\R^n) \cap \codif \Omega^{k+1} \schwartz'(\R^n)$.
\end{definition}

The following theorem is contained in \cite[Theorem 1.2]{troyanov2007hodge}. We will only use the case $p = 2$.

\begin{theorem}[$L^p$ Hodge decomposition on $\R^n$]\label{thm:hodge-decomposition}
Let $1 < p < \infty$. For $\manifold = \R^n$, we have the following direct sum decomposition
\[ \Omega^k L^p(\R^n) = E \Omega^k L^p(\R^n) \oplus E^* \Omega^k L^p(\R^n). \]
Moreover, $E \Omega^k L^p$, $E^* \Omega^k L^p$ are closed subspaces of $L^p$.
\end{theorem}

Next, we state the Hodge decomposition for smooth $k$-forms on $\T^n$. An $L^2$ version would then directly follow by taking closures. In the following, let $\Omega^k_0(\T^n)$ be the set of harmonic $k$-forms, i.e. the set of $k$-forms $f$ such that $\Delta f = 0$.

\begin{prop}[Hodge decomposition]
The space $\Omega^k(\T^n)$ admits the orthogonal decomposition
\begin{equs}
\Omega^k(\T^n) = \Omega^k_0(\T^n) \oplus  d \Omega^{k-1}(\T^n)\oplus \codif \Omega^{k+1}(\T^n) .
\end{equs}
\end{prop}

\begin{remark}
To relate to the notation in Section \ref{sec:fracforms}, recall that $\Omega^k_- = d \Omega^{k-1}$ and $\Omega^k_+ = \codif \Omega^{k+1}$.
\end{remark}

\begin{definition}[Wedge product of forms]
Let $0 \leq k, \ell \leq n$, and let $f$ be a $k$-form, $g$ be an $\ell$-form. The wedge product $f \wedge g$ is defined to be zero if $k + \ell > n$, and otherwise is defined to be
\begin{equs}\label{eq:wedge-product}
f \wedge g := \sum_{1 \leq r_1 < \cdots < r_{k+\ell} \leq n}  \bigg(\frac{1}{k!\ell!} \sum_{\sigma \in \symgrp_{k+\ell}} \varep_{\sigma(1) \cdots \sigma(k+\ell)} f_{r_{\sigma(1)} \cdots r_{\sigma(k)}} g_{r_{\sigma(k+1)} \cdots r_{\sigma(k+\ell)}} \bigg) dx^{r_1} \wedge \cdots \wedge dx^{r_{k+\ell}}.
\end{equs}
\end{definition}

The formula for $f \wedge g$ arises by defining the wedge product on ``elementary" forms by
\begin{equs}
\big( dx^{i_1} \wedge \cdots \wedge  dx^{i_k}\big) \wedge \big(dx^{j_1} \wedge \cdots \wedge dx^{j_\ell}\big) := dx^{i_1} \wedge \cdots \wedge dx^{i_k} \wedge dx^{j_1} \wedge \cdots \wedge dx^{j_\ell},
\end{equs}
and then extending linearly.

\begin{definition}[Vector space-valued forms]
Let $(V, \langle \cdot, \cdot \rangle_V)$ be a finite dimensional inner product space. For $0 \leq k \leq n$, A $V$-valued $k$-form on $\manifold$ is given by the same equation \eqref{eq:k-form-basic-expression}, except now the coordinate functions $f_{i_1 \cdots i_k} \colon \manifold \ra V$ are $V$-valued. We denote the space of smooth $V$-valued $k$-forms by $\Omega^k(\manifold, V)$.
\end{definition}

Given this definition, all of the previous discussion extends to $V$-valued differential forms, except for the wedge product of forms. 
This is because in general, vectors spaces do not come with a product. On the other hand, the main examples of vector spaces $V$ that we have in mind are Lie algebras $\frkg$ of a compact Lie group $\liegroup$ (see Section \ref{sec:yangmills}). Moreover, we will always take $\liegroup$ to be a matrix Lie group, and thus $\frkg$ is a vector space of matrices. In this setting, we may define the wedge product of $\frkg$-valued differential forms exactly as in \eqref{eq:wedge-product}. However, note that in general, the product of $\frkg$-valued elements is not in $\frkg$. Thus, in general $f \wedge g$ may not actually be a $\frkg$-valued $k$-form. Another way to define the wedge product of $\frkg$-valued differential forms is by using the Lie bracket, as we discuss next.

\begin{definition}[Wedge product of $\frkg$-valued differential forms]
Let $0 \leq k, \ell \leq n$, and let $f$ be a $\frkg$-valued $k$-form, $g$ be a $\frkg$-valued $\ell$-form. The wedge product $[f \wedge g]$ is defined to be zero if $k + \ell > n$, and otherwise is defined to be
\begin{equs}
~[f \wedge g] := \sum_{1 \leq r_1 < \cdots < r_{k+\ell} \leq n}  \bigg(\frac{1}{k!\ell!} \sum_{\sigma \in \symgrp_{k+\ell}} \varep_{\sigma(1) \cdots \sigma(k+\ell)} [f_{r_{\sigma(1)} \cdots r_{\sigma(k)}}, g_{r_{\sigma(k+1)} \cdots r_{\sigma(k+\ell)}}] \bigg) dx^{r_1} \wedge \cdots \wedge dx^{r_{k+\ell}}.
\end{equs}
\end{definition}

We will refer to both $f \wedge g$ and $[f \wedge g]$ as ``wedge-products". In the special case where $f, g$ are $\frkg$-valued $1$-forms, note that
\begin{equs}
~[f \wedge g] &= \sum_{1 \leq i < j \leq n} \big([f_i, g_j] - [f_j, g_i]\big) dx^i \wedge dx^j, \\
f \wedge g &= \sum_{1 \leq i < j \leq n} \big(f_i g_j - f_j g_i\big) dx^i \wedge dx^j.
\end{equs}
By taking $g = f$, we have that
\begin{equs}
~[f \wedge f] = 2 f \wedge f.
\end{equs}
In particular, note that due to the possible non-commutativity of the elements of $\frkg$, in general the wedge product of a $\frkg$-valued form with itself is not zero, unlike in the real-valued case.

\subsection{Exterior calculus in Fourier space}

In this subsection, we describe how the various operators $d, \codif, \Delta, E, E^*$ look in Fourier space. Let $f \in \Omega^k (\manifold) \cap \Omega^k L^1(\manifold)$ be a smooth and integrable $k$-form. We previously defined the Fourier transform on functions -- recall \eqref{eq:fourier-transform}. By taking the Fourier transform component-wise, we also obtain a definition of $\mc{F} f$, where now $(\mc{F} f)(p) = ((\mc{F} f_{i_1 \cdots i_k})(p), i_1, \ldots, i_k \in [n]) \in \C^{n^k}$ takes values in a tensor space. More precisely, $(\mc{F} f)(p) \in \Lambda^k(\C^n)$, i.e. the $k$th exterior power of $\C^n$.

Let $p \in \R^n$ (if $\manifold = \T^n$ then we would further restrict $p \in \Z^n$). From \eqref{eq:exterior-derivative-coordinatewise-formula}, \eqref{eq:codifferential-coordinatewise-expression}, and \eqref{eq:hodge-laplacian-coordinatewise-formula}, we see that upon applying $d, d^*$, or $\Delta$, the Fourier coefficients transform as follows:
\begin{equs}
\big(\mc{F} (df)_{i_1 \cdots i_{k+1}}\big)(p) &= \icomplex \sum_{\ell=1}^{k+1} (-1)^{\ell-1} p_{i_\ell} (\mc{F} f_{i_1 \cdots \hat{i}_\ell \cdots i_{k+1}})(p), \\
\big(\mc{F} (d^* f)_{i_1 \cdots i_{k-1}}\big)(p) &= -\icomplex p_j (\mc{F} f_{j i_1 \cdots i_{k-1}})(p), \\
\big( \mc{F} (\Delta f)_{i_1 \cdots i_k}\big)(p) &= - |p|^2 (\mc{F} f_{i_1 \cdots i_k})(p).
\end{equs}
To those familiar with exterior algebra (for an overview, see e.g. \cite[Chapter 1]{GH2019}), note that we may write
\begin{equs}
(\mc{F} (df))(p) &= \icomplex (k+1) p \wedge (\mc{F} f)(p), \\
(\mc{F} (d^* f))(p) &= -\icomplex \iota_p ((\mc{F} f)(p)),
\end{equs}
where $\iota_p$ denotes the interior product. 

Note that in the special case $k = 1$, we have that
\begin{equs}
\big(\mc{F} (\codif f)\big)(p) = -\icomplex (\mc{F} f)(p) \cdot p,
\end{equs}
so that the codifferential in Fourier space essentially just acts by inner product with the Fourier mode $p$. Additionally, when $k = 1$, the projections $E, E^*$ act as
\begin{equs}
\big(\mc{F} (Ef)\big)(p) &= \big(\mc{F} f (p) \cdot p \big) \frac{p}{|p|^2} \\
\big(\mc{F} (E^*f)\big)(p) &= \mc{F}f (p) - \big(\mc{F}f (p) \cdot p \big) \frac{p}{|p|^2}.
\end{equs}
In words, $\big(\mc{F}(Ef)\big)(p)$ is the projection of $\mc{F} f(p)$ onto the span of $p$, while $\big(\mc{F}(E^* f)\big)(p)$ is the projection of $\mc{F} f(p)$ onto the orthogonal complement of $p$.

\subsection{General manifold}\label{section:general-manifold}

In this subsection, we review exterior calculus on a general compact oriented Riemannian manifold $(\manifold, g)$. Our intention is not to give a detailed introduction to exterior calculus in this general geometric setting, so we will be quite brief. For more discussion and a textbook introduction, see e.g. \cite[Sections 1.8 and 2.1]{Jost}. The reader who is content with working on $\R^n$ or $\T^n$ may skip this section. The main outcome of this discussion is that the case of general $\manifold$ will be very similar to the case of $\T^n$, since we are assuming that $\manifold$ is compact.

As before, let $n$ be the dimension of $\manifold$. For $0 \leq k \leq n$, let $\Lambda^k(T^* M)$ be the vector bundle over $\manifold$ such that for each $x \in \manifold$, the fiber above $x$ is $\Lambda^k(T_x^* M) = T_x^*\manifold \wedge \cdots \wedge T_x^* \manifold$ (i.e. the $k$-th exterior power of $T_x^* M$). A smooth $k$-form $f$ is simply a smooth section of $\Lambda^k(T^* M)$. The space of smooth $k$-forms is denoted by $\Omega^k(\manifold)$. Recalling Remark \ref{remark:k-form-alternating-multilinear-map}, one can think of a $k$-form as an assignment of each point $x \in \manifold$ to an alternating multilinear map $(T_x \manifold)^k \ra \R$. Note that such maps are in canonical bijection with elements of $T_x^* \manifold \wedge \dots \wedge T_x^* \manifold$.

In local coordinates $(x^1, \ldots, x^n)$, a general $k$-form $f$ may be written exactly as in \eqref{eq:k-form-basic-expression}. The exterior derivative $d f$ is then defined locally exactly as in Definition \ref{def:exterior-derivative}. One then checks that this definition is independent of the choice of coordinates, which amounts to checking that the exterior derivative transforms appropriately under change of coordinates.

The definitions of differential forms and exterior derivative only require that $M$ is a differentiable manifold, i.e. these definitions did not involve the metric $g$, nor an orientation for $M$. By contrast, the Hodge star $\star$ and codifferential $\codif$ do require the metric and an orientation. The Hodge star $\star : \Omega^k(\manifold) \ra \Omega^{n-k}(\manifold)$ is defined by applying a fiberwise linear map which we will also denote $\star$, which maps $\star : \Lambda^k(T_x^* M) \ra \Lambda^{n-k}(T_x^* M)$. For $w \in \Lambda^k(T_x^* M)$, $\star w \in \Lambda^{n-k}(T_x^* M)$ is characterized by
\begin{equs}
v \wedge \star w = \langle v, w \rangle \mrm{vol}(x) \text{ for all $v \in \Lambda^k(T_x^* M)$,}
\end{equs}
where recall $\mrm{vol}$ is the volume form of $(\manifold, g)$.
Given local coordinates $(x^1, \ldots, x^n)$, this may be written (here $g = (g_{ij})_{i, j=1}^n$ is the metric tensor)
\begin{equs}
\mrm{vol}(x) = \sqrt{\mrm{det}(g_{ij}(x))_{i, j =1}^n} dx^1 \wedge \cdots \wedge dx^n.
\end{equs}
Given $k$-forms $f, g \in \Omega^k(\manifold)$, we may define their $L^2$ inner product by
\begin{equs}
(f, g) := \int_{\manifold} \langle f, g \rangle \mrm{vol}(x) = \int_{\manifold} f \wedge \star g.
\end{equs}
The codifferential $\codif$ is defined as the formal adjoint of $d$, i.e. it is characterized by the property
\begin{equs}
(\codif f, g) = (f, d g) \text{ for all $f \in \Omega^{k-1}(\manifold)$, $g \in \Omega^k(\manifold)$.}
\end{equs}
We denote by $\Omega^k L^p(\manifold)$ the space of $k$-forms in $L^p$.

Similar to before, we define the Hodge Laplacian on differential forms by
\begin{equs}
\Delta := -(d\codif + \codif d).
\end{equs}
With the minus sign, $\Delta$ is negative definite, and it coincides with the Laplacian on functions. This is somewhat non-standard, as many textbooks define $\Delta$ to be positive definite. Our preference is to define $\Delta$ to be negative definite, so that $\Delta$ reduces to the usual function Laplacian when viewing it as an operator on $0$-forms.

We summarize the properties that $d, \codif, \star, \Delta$ satisfy, which we have previously seen in the case of flat space.

\begin{enumerate}
    \item $d^2 = 0$ and $(\codif)^2 = 0$.
    \item $(df, g) = (f, \codif g)$, i.e. $\codif$ is the $L^2$ adjoint of $d$.
    \item $\codif = (-1)^{n(k+1) + 1} \star d \star$.
    \item $\star \star = (-1)^{k(n-k)}$.
    \item $(\star f, g) = (-1)^{k(n-k)} (f, \star g)$.
    \item $\Delta \star = \star \Delta$.
\end{enumerate}

A smooth $k$-form $f$ is an eigenform of $\Delta$ if $\Delta f = \lambda f$ for some $\lambda \in \R$. Note that
\begin{equs}
(f, \Delta f) = -(df, df) - (\codif f, \codif f),
\end{equs}
and so $\lambda \leq 0$. We say that $f$ is harmonic if $\Delta f = 0$. We state the following result about the eigenforms of $\Delta$, which can be found as \cite[Chapter 2, Exercise 7]{Jost}.

\begin{prop}\label{prop:hodge-laplacian-eigendecomposition}
Let $(\manifold, g)$ be a compact oriented Riemannian manifold of dimension $n$. Let $0 \leq k \leq n$. All eigenspaces of $\Delta$ as an operator on $\Omega^k(\manifold)$ are finite-dimensional, and the eigenforms form an $L^2$-orthonormal basis for $\Omega^k L^2(\manifold)$. Moreover, there are infinitely many eigenvalues $0 \geq \lambda_1 \geq \lambda_2 \geq \cdots$, the set of eigenvalues is locally finite, and eigenforms corresponding to different eigenvalues are orthogonal.
\end{prop}

Based on this proposition, we make the following definition.

\begin{definition}
For $0 \leq k \leq n$, let $0 \geq \lambda_1^k \geq \lambda_2^k \cdots$ be the eigenvalues of the Hodge Laplacian $\Delta$ on $k$-forms (allowing multiplicity). Let $\{f^k_m\}_{m \geq 1}$ be the associated eigenvectors.
\end{definition}

Because we have an eigendecomposition of the Hodge Laplacian $\Delta$, the following discussion mirrors the case of $\T^n$ very closely. Given any smooth $k$-form $f \in \Omega^k(\manifold)$, we may decompose
\begin{equs}
f = \sum_{m=1}^\infty (f, f_m^k) f_m^k,
\end{equs}
with convergence taking place in $L^2$. For any integer power $r \geq 0$, we have that $(-\Delta)^r f$ is smooth, and thus also in $L^2$ (since $\manifold$ is compact). From this, it follows that 
\begin{equs}
\sum_{m=1}^\infty (-\lambda^k_m)^{2r} (f, f_m^k)^2 = ((-\Delta)^r f, (-\Delta)^r f) < \infty.
\end{equs}
Next, using the eigendecomposition of the Hodge Laplacian, we define the fractional Laplacian and Sobolev spaces, mirroring the case of $\manifold = \T^n$.

\begin{definition}
Let $\bigdot{\Omega}^k(\manifold)$ be the space of smooth $k$-forms such that $(f, f_m^k) = 0$ for all $m$ such that $\lambda_m^k = 0$. 
\end{definition}

In words, $\bigdot{\Omega}^k(\manifold)$ is the space of $k$-forms with no harmonic component.

\begin{definition}[Fractional Laplacian]
Let $s \in \R$. Define the fractional Laplacian $(-\Delta)^s : \bigdot{\Omega}^k(\manifold) \ra \bigdot{\Omega}^k(\manifold)$ by
\begin{equs}
(-\Delta)^s f := \sum_{\substack{m \geq 1 \\ \lambda_m^k \neq 0}} (-\lambda_m^k)^{s} (f, f_m^k) f_m^k.
\end{equs}
Given $\lambda > 0$, define also the massive fractional Laplacian $(-\Delta + \lambda)^s : \Omega^k(\manifold) \ra \Omega^k(\manifold)$ by
\begin{equs}
(-\Delta + \lambda)^s f := \sum_{m=1}^\infty (-\lambda_m^k + \lambda)^s (f, f_m^k) f_m^k.
\end{equs}
\end{definition}

\begin{definition}[Sobolev spaces]
Let $s \in \R$. On $\bigdot{\Omega}^k(\manifold)$, define the inner product
\begin{equs}
(f, g)_{\Omega^k\dot{H}^{s}(\manifold)} := \big((-\Delta)^{\frac{s}{2}} f, (-\Delta)^{\frac{s}{2}} g\big) = \sum_{\substack{m \geq 1 \\ \lambda_m^k \neq 0}} (-\lambda_m^k)^s (f, f_m^k) (g, f_m^k).
\end{equs}
Define the homogeneous Sobolev space $\Omega^k \dot{H}^s(\manifold)$ as the Hilbert space completion of $\bigdot{\Omega}^k(\manifold)$ with respect to this inner product.

Similarly, define the inhomogeneous Sobolev space $\Omega^kH^s (\manifold)$ as the completion of $\Omega^k(\manifold)$ with respect to the inner product
\begin{equs}
(f, g)_{\Omega^k H^{s} (\manifold)} := \big((-\Delta + 1)^{\frac{s}{2}} f, (-\Delta + 1)^{\frac{s}{2}} g\big) = \sum_{m=1}^\infty (-\lambda_m^k + 1)^s (f, f_m^k) (g, f_m^k).
\end{equs}
\end{definition}



Next, we define distributional forms, slightly anticipating the topic of Section \ref{section:distributional-forms}.

\begin{definition}[Distributional forms]
Let $\Omega^k \mc{D}'(\manifold)$ be the topological dual of $\Omega^k(\manifold)$, and let $\Omega^k \bigdot{\mc{D}}'(\manifold)$ be the topological dual of $\bigdot{\Omega}^k(\manifold)$.
\end{definition}

\subsection{Spaces of distributional forms}\label{section:distributional-forms}


In this subsection, we discuss spaces of distributional forms that appear when working with fractional Gaussian forms. In places where we consider a general compact oriented Riemannian manifold $\manifold$, the reader who skipped Section \ref{section:general-manifold} can simply take $\manifold = \T^n$.

We first discuss the case $\manifold = \R^n$. Having defined the spaces $\Phi, \Phi', \dot{H}^s(\R^n)$, we can consider the spaces $\Omega^k \Phi, \Omega^k \Phi'$, and $\Omega^k \dot{H}^s(\R^n)$. Observe that $\Omega^k \Phi'$ may also be thought of as $(\Omega^k \Phi)'$, i.e. the topological dual of the space of Lizorkin $k$-forms. The previous results from Section \ref{section:distribution-spaces} (Proposition \ref{prop:fractional-laplacian-isomorphism-on-lizorkin-space}, Lemmas \ref{lemma:lizorkin-space-dense-fractional-sobolev-space} and \ref{lemma:fractional-laplacian-isometry-fractional-sobolev-space}) regarding $\Phi, \dot{H}^s(\R^n), (-\Delta)^s$ extend to analogous results for differential forms with values in $\Phi$ or $\dot{H}^s(\R^n)$.

\begin{notation}
When $\manifold = \R^n, \T^n$, for $f, g \in \Omega^k \dot{H}^s(\manifold)$, we write
\begin{equs}
(f, g)_{\Omega^k \dot{H}^s(\manifold)} := \sum_{1 \leq i_1 < \cdots < i_k \leq n} (f_{i_1 \cdots i_k}, g_{i_1 \cdots i_k})_{\dot{H}^s(\manifold)},
\end{equs}
and analogously for the inhomogeneous Sobolev space $\Omega^k H^s(\manifold)$. Note that for $f, g \in \Omega^k \Phi \sse \Omega^k \dot{H}^s(\manifold)$, we have that
\begin{equs}
(f, g)_{\Omega^k \dot{H}^s(\manifold)} = (f, (-\Delta)^{-s} g) = ((-\Delta)^{-\frac{s}{2}} f, (-\Delta)^{-\frac{s}{2}} g).
\end{equs}
\end{notation}

Recall (by Remark \ref{remark:lizorkin-space-closed-under-differentiation}) that the space $\Phi$ is closed under taking derivatives. Recalling also the explicit coordinatewise formulas for $d$ \eqref{eq:exterior-derivative-coordinatewise-formula} and $\codif$ \eqref{eq:codifferential-coordinatewise-expression}, it thus follows that $d : \Omega^k \Phi \ra \Omega^{k+1} \Phi$ and $\codif : \Omega^k \Phi \ra \Omega^{k-1} \Phi$. We may thus define $d : \Omega^k \Phi' \ra \Omega^{k+1} \Phi'$ and $\codif : \Omega^k \Phi' \ra \Omega^{k-1} \Phi'$ by duality:
\begin{equs}
(d f, \phi) &:= (f, \codif \phi), ~~ \phi \in \Omega^{k+1} \Phi, \\
(\codif f, \eta) &:= (f, d \eta), ~~ \eta \in \Omega^{k-1} \Phi.
\end{equs}
We may also define the Hodge star $\star : \Omega^k \Phi' \ra \Omega^{n-k} \Phi'$ by
\[ (\star f, \phi) := (-1)^{k(n-k)} (f, \star \phi), ~~ \phi \in \Omega^{n-k} \Phi.\]
The previously proved identities relating $d, \codif, \star$ extend to elements of $\Omega^k \Phi'$. Note that this entire discussion could have also taken place with $\Phi, \Phi'$ replaced by $\schwartz, \schwartz'$, but later we will need to restrict to $\Phi, \Phi'$ in order to take negative powers of the Laplacian.

\begin{remark}
We could have also defined $df, \codif f, \star f$ by using the corresponding coordinatewise formulas \eqref{eq:exterior-derivative-coordinatewise-formula}, \eqref{eq:codifferential-coordinatewise-expression}, \eqref{eq:hodge-star-coordinatewise-formula}. Of course, these coordinatewise definitions coincide with the given definitions.
\end{remark}

The previous discussion also applies to a compact oriented Riemannian manifold $\manifold$, except we replace the spaces $\Phi$ and $\Phi'$ by $\bigdot{\Omega}^k(\manifold)$ and $\Omega^k\bigdot{\mc{D}}'(\manifold)$. 

\begin{definition}[Projections]\label{def:E-E-star-projections}
Let $0 \leq k \leq n$. For $\manifold = \R^n$, define $E : \Omega^k \Phi \ra  \Omega^k \Phi$ by $E := d\codif (-\Delta)^{-1}$. Similarly, define $E^* : \Omega^k \Phi \ra \Omega^k \Phi$ by $E^* :=  \codif d (-\Delta)^{-1}$. Note by duality we may also define operators $E : \Omega^k \Phi' \ra  \Omega^k \Phi'$ and $E^* : \Omega^k \Phi' \ra \Omega^k \Phi'$ using the same formulas.

For a compact oriented Riemannian manifold $\manifold$, define $E : \bigdot{\Omega}^k (\manifold) \ra  \bigdot{\Omega}^k(\manifold)$ and $E^* : \bigdot{\Omega}^k (\manifold) \ra  \bigdot{\Omega}^k(\manifold)$ by the same formulas. Again by duality, we also define
$E : \Omega^k \bigdot{\mc{D}}'(\manifold) \ra  \Omega^k \bigdot{\mc{D}}'(\manifold)$ and $E^* : \Omega^k \bigdot{\mc{D}}'(\manifold) \ra  \Omega^k \bigdot{\mc{D}}'(\manifold)$,
\end{definition}

We state the following result about the projections $E, E^*$. The case $\manifold = \R^n$ is part of \cite[Proposition 7.4 and Corollary 7.5]{troyanov2007hodge}, while the case where $\manifold$ is a compact oriented Riemannian manifold follows by direct verification. 


\begin{prop}\label{prop:hodge-decomp}
Let $\manifold = \R^n$. Let $0 \leq k \leq n$. The operators $E, E^* : \Omega^k \Phi' \ra \Omega^k \Phi'$, are self-adjoint and satisfy $E(\Omega^k \Phi), E^*(\Omega^k \Phi) \sse \Omega^k \Phi$, $E + E^* = I$, $E^2 = E$, $(E^*)^2 = E^*$, $E E^* = E^* E = 0$.
For $\manifold$ a compact oriented Riemannian manifold, the operators $E, E^* : \Omega^k \bigdot{\mc{D}}'(\manifold) \ra \Omega^k \bigdot{\mc{D}}'(\manifold)$ satisfy $E(\bigdot{\Omega}^k(\manifold)), E^*(\bigdot{\Omega}^k(\manifold)) \sse \bigdot{\Omega}^k(\manifold)$, and $E, E^*$ again satisfy $E + E^* = I$, $E^2 = E$, $(E^*)^2 = E^*$, and $EE^* = E^* E = 0$.
\end{prop}


For $k$-forms $f$ living in the domain of $E$ and $E^*$, we say that $E f$ is the exact part of $f$ and $E^* f$ the co-exact part. The next lemma is a direct consequence of Corollaries \ref{cor:dif-codif-hodge-star-commute} and \ref{cor:hodge-star-commutes-with-laplacian}.

\begin{lemma}
We have that $\star E = E^* \star$ and $\star E^* = E \star$.
\end{lemma}

Later on, we will need the following lemma.

\begin{lemma}\label{lemmma:projections-bounded-sobolev-space}
For all $s \in \R$, $0 \leq k \leq n$, $E, E^*$ are bounded operators on $\Omega^k \dot{H}^{-s}(\manifold)$.
\end{lemma}
\begin{proof}
We prove the result for $\manifold = \R^n$. Let $\phi \in \Omega^k \Phi$. One may check that for any $i_1, \ldots, i_k$, we have that $\big(\mc{F} (E \phi)_{i_1 \cdots i_k}\big)(p)$ is a linear combination of terms of the form
\begin{equs}
\frac{p_k p_\ell}{|p|^2} (\mc{F} \phi_{j_1 \cdots j_k})(p).
\end{equs}
Note that
\begin{equs}
\int_{\R^n} dp |p|^{2s} \Big| \frac{p_k p_\ell}{|p|^2} (\mc{F} \phi_{j_1 \cdots j_k})(p)\Big|^2 \lesssim \int_{\R^n} dp |p|^{2s} |(\mc{F} \phi_{j_1 \cdots j_k})(p)|^2 = \|\phi_{j_1 \cdots j_k}\|_{\dot{H}^{s}(\R^n)}^2,
\end{equs}
and the desired boundedness now follows from the density of $\Omega^k \Phi$ in $\Omega^k \dot{H}^{-s}(\R^n)$ (Lemma \ref{lemma:lizorkin-space-dense-fractional-sobolev-space}).
\end{proof}


\section{Construction of fractional Gaussian forms} \label{sec:fgfconstruction}



In this section, we construct fractional Gaussian forms and discuss their various basic properties. As usual, we assume that $\manifold$ is either $\R^n$ or a compact oriented Riemannian manifold. The reader who skipped the review of exterior calculus on a general manifold (Section \ref{section:general-manifold}) can take $\manifold = \T^n$ in the latter case throughout this section.

\begin{definition}[White noise]
We say that $\noise$ is a $k$-form white noise on $\R^n$ if it is a random variable taking values in $\Omega^k \Phi'$ such that $((\noise, f), f \in \Omega^k \Phi)$ is a centered Gaussian process with covariance
\begin{equs}\label{eq:white-noise-covariance}
\Cov((\noise, f), (\noise, g)) = (f, g). 
\end{equs}
Let $\manifold$ be a compact oriented Riemannian manifold $\manifold$. We say that $\noise$ is a $k$-form white noise on $\manifold$ if it is a random variable taking values in $\Omega^k \bigdot{\mc{D}}'(\manifold)$ such that $((\noise, f), f \in \bigdot{\Omega}^k (\manifold))$ is a centered Gaussian process with the above covariance.
\end{definition}

\begin{remark}\label{remark:larger-space}
The existence of such a random variable follows by the Bochner-Minlos theorem (see \cite[Section 3]{lodhia2016fractional}). Also, we could have constructed $k$-form white noise on $\R^n$ as a random variable taking values in $\Omega^k \schwartz'$, and similarly on compact $\manifold$ as a random variable taking values in $\Omega^k \mc{D}'(\manifold)$. However, in later discussion we will want to apply arbitrary powers of $(-\Delta)$ to white noise, and thus we find it convenient to work directly on $\Omega^k \Phi'$ and $\Omega^k \bigdot{\mc{D}}'(\manifold)$, where such powers are indeed always well-defined. In the case of $\R^n$, due to the density of $\Phi$ in the relevant spaces (Lemma \ref{lemma:lizorkin-space-dense-fractional-sobolev-space}), we essentially do not lose anything by working on these smaller spaces of distributions -- see Remark \ref{remark:Ts-space}. In the case of compact $\manifold$, this essentially amounts to setting all harmonic forms to be zero -- recall the discussion from Section \ref{sec:fracforms}.
\end{remark}

\begin{remark}\label{remark:white-noise-induces-gaussian-hilbert-space}
By the usual Hilbert space isometry arguments,
given a $k$-form white noise $\noise$ on $\R^n$, we may obtain a centered Gaussian process $((\noise, f), f \in \Omega^k L^2(\R^n))$ indexed by $L^2(\R^n)$ $k$-forms which has covariance given by \eqref{eq:white-noise-covariance} (here we use the fact that $\Omega^k \Phi$ is dense in $\Omega^k L^2(\R^n)$ -- recall Lemma \ref{lemma:lizorkin-space-dense-fractional-sobolev-space}). This process $((\noise, f), f \in \Omega^k L^2(\R^n))$ is a Gaussian Hilbert space (see \cite[Definition 2.5]{lodhia2016fractional}). Similar considerations hold for compact $\manifold$, where instead of $L^2$ we use $\dot{H}^0$ (we could still use $L^2$ if we define white noise as a random distribution on $\Omega^k \mc{D}'(\manifold)$, as mentioned in Remark \ref{remark:larger-space}).
\end{remark}

\begin{remark}
Concretely, when $\manifold = \R^n$ or $\T^n$, note that $\noise$ is a $k$-form white noise if and only if its components
\[ (\noise_{i_1 \cdots i_k}, 1 \leq i_1 < \cdots < i_k \leq n)\]
are i.i.d. ($0$-form) white noises.
\end{remark}

\begin{remark}[White noise on $\T^n$ as a random Fourier series]\label{remark:white-noise-torus-random-fourier-series}
On the torus, $0$-form white noise $\noise$ has the following explicit representation as a random Fourier series:
\begin{equs}
\noise = \frac{1}{(2\pi)^{\frac{n}{2}}}\sum_{0 \neq \alpha \in \Z^n} Z_\alpha \e_\alpha,
\end{equs}
where the collection $(Z_\alpha, 0 \neq \alpha \in \Z^d)$ are i.i.d. $\N_\C(0, 1)$, modulo the constraint that $Z_{-\alpha} = \ovl{Z_\alpha}$ (which ensures that the distribution defined by the Fourier series is $\R$-valued). To be precise, this series can be seen to a.s. converge in the space $\dot{H}^{-\frac{d}{2}-\varep}(\T^n)$, for any $\varep > 0$. 

To see why $\noise$ has this explicit representation, note that given $f \in \bigdot{C}^\infty(\T^n)$, we have that
\begin{equs}
(\noise, f) = \sum_{0 \neq \alpha \in \Z^n} Z_\alpha \ovl{\widehat{f}(\alpha)},
\end{equs}
and thus we see that (recalling the Plancherel identity \eqref{eq:torus-plancherel})
\begin{equs}
\Cov((\noise, f), (\noise, g)) = \sum_{0 \neq \alpha \in \Z^n} \widehat{f}(\alpha) \ovl{\widehat{g}(\alpha)} = (f, g).
\end{equs}
\end{remark}

\begin{definition}[$k$-form fractional Gaussian field]\label{def:fgf}
Let $\manifold = \R^n$. We say that $A$ is a $k$-form $\fgf_s(\R^n)$ if it is a $\Omega^k \Phi'$-valued random variable such that $((A, f), f \in \Omega^k \Phi)$ is a centered Gaussian process with covariance
\begin{equs}\label{eq:fgf-covariance}
\mrm{Cov}((A, f), (A, g)) = (f, g)_{\Omega^k \dot{H}^{-s}(\manifold)}. 
\end{equs}
Alternatively, let $\manifold$ be a compact oriented Riemannian manifold. We say that $A$ is a $k$-form $\fgf_s(\manifold)$ if it is a $\Omega^k \bigdot{\mc{D}}'(\manifold)$-valued random variable such that $((\noise, f), f \in \bigdot{\Omega}^k(\manifold))$ is a centered Gaussian process with the above covariance. In either case of $\manifold$, we refer to a $k$-form $\fgf_s(\manifold)$ as a $\fgf_s^k(\manifold)$. By default, we say that $A$ is a $\fgf_s(\manifold)$ if it is a $\fgf_s^0(\manifold)$.
\end{definition}

\begin{remark}
One should think of the $s$ parameter as tracking how many derivatives smoother $\fgf_s^k$ is compared to white noise. This is formalized in Lemma \ref{lemma:laplacian-of-fgf}, where $\fgf_s^k$ is obtained by applying $(-\Delta)^{-\frac{s}{2}}$ to $\fgf_0^k$. Heuristically, note that the inverse Laplacian $(-\Delta)^{-1}$ smooths by two derivatives, and thus $(-\Delta)^{-\frac{s}{2}}$ should smooth by $s$ derivatives.
\end{remark}

\begin{remark}\label{remark:gaussian-hilbert-space}
As in Remark \ref{remark:white-noise-induces-gaussian-hilbert-space}, given a $\fgf_s^k(\manifold)$, we may obtain a Gaussian Hilbert space $((A, f), f \in \Omega^k \dot{H}^{-s}(\manifold))$ with covariance given by \eqref{eq:fgf-covariance}.
\end{remark}

\begin{remark}\label{remark:Ts-space}
In \cite[Section 3.1]{lodhia2016fractional}, the $\fgf_s(\R^n)$ is associated to a Gaussian Hilbert space indexed by a class $T_s$ of test functions, defined to be the closure of a certain subspace $\mc{S}_H$ of $\mc{S}$ with respect to the norm $(\cdot, \cdot)_{\dot{H}^{-s}}$. In fact, this space $T_s$ is equal to $\dot{H}^{-s}(\R^n)$. This follows from the fact that $\Phi \sse \mc{S}_H$ combined with the density of $\Phi$ in $\dot{H}^{-s}(\R^n)$ (Lemma \ref{lemma:lizorkin-space-dense-fractional-sobolev-space}).
\end{remark}

\begin{remark}\label{remark:fgf-fourier-series}
Continuing Remark \ref{remark:white-noise-torus-random-fourier-series}, an $\fgf_s^k(\T^n)$ has the following explicit representation: 
\begin{equs}
A = \frac{1}{(2\pi)^{\frac{n}{2}}} \sum_{0 \neq \alpha \in \Z^n} |\alpha|^{-s} Z_\alpha \e_\alpha,
\end{equs}
where the $(Z_\alpha, 0 \neq \alpha \in \Z^d)$ are as in Remark \ref{remark:white-noise-torus-random-fourier-series}. Moreover, note that for any $\varep > 0$, we have that (recall here $H = s -n/2$ is the Hurst parameter)
\begin{equs}\label{eq:fgf-in-sobolev-space}
\E\Big[ \|A\|_{\dot{H}^{H -\varep}(\T^n)}^2\Big] \lesssim \sum_{0 \neq \alpha \in \Z^n} |\alpha|^{2H - 2s - 2\varep} = \sum_{0 \neq \alpha \in \Z^n} |\alpha|^{-n - 2\varep} < \infty.
\end{equs}
From this, it follows that a.s., $A \in \dot{H}^{H - \varep}(\T^n)$ for all $\varep > 0$.

More generally, if $\manifold$ is a compact oriented Riemannian manifold, we have the following explicit representation of an $\fgf_s^k(\manifold)$:
\begin{equs}
A = \sum_{\substack{m \geq 1 \\ \lambda_m^k \neq 0}} (-\lambda_m^k)^{-\frac{s}{2}} Z_m^k f_m^k,
\end{equs}
where recall $(\lambda_m^k)_{m \geq 1}, (f_m^k)_{m \geq 1}$ are respectively the eigenvalues and eigenforms of the Hodge Laplacian $\Delta$ as an operator on $k$-forms, and $(Z_m^k)_{m \geq 1}$ is a collection of i.i.d. $\N(0, 1)$ random variables. To obtain a result analogous to \eqref{eq:fgf-in-sobolev-space}, we need a version of Weyl's law for the Hodge Laplacian. For $\lambda \leq 0$, let $N^k(\lambda) := \#\{k : 0 \geq \lambda_m^k > \lambda\}$. By \cite[Corollary 2.43]{BGV2004}, we have that
\begin{equs}
\lim_{\lambda \ra -\infty} \frac{N^k(\lambda)}{\lambda^{\frac{n}{2}}} = C(k, \manifold),
\end{equs}
where $C(k, \manifold)$ is some explicit constant depending on $k$ and $\manifold$.
By applying this result, we obtain
\begin{equs}
\E\Big[ \|A\|_{\dot{H}^{H -\varep}(\manifold)}^2\Big] = \sum_{\substack{m \geq 1 \\ \lambda_m^k \neq 0}} (-\lambda_m^k)^{-\frac{n}{2} - \varep} &\lesssim C + \sum_{j=1}^\infty 2^{-j(\frac{n}{2}+\varep)} \#\{m : -2^{j-1} \geq \lambda_m^k > -2^j\} \\
&\lesssim C + \sum_{j=1}^\infty 2^{-j\varep} < \infty,
\end{equs}
and thus we have that a.s., $A \in \dot{H}^{H-\varep}(\manifold)$. In the above, $C$ absorbs the contribution coming from the eigenvalues $\lambda_m^k \in (-1, 0)$, of which there are only finitely many, by Proposition \ref{prop:hodge-laplacian-eigendecomposition}.
\end{remark}

\begin{remark}[Other distribution spaces]
The $\fgf_s^k(\manifold)$ can also be shown to lie in other distribution spaces, not just Sobolev spaces. For instance, $\fgf_s^k(\T^n)$ is a.s. in the Besov-H\"{o}lder space $\mc{B}^{H-\varep}_{\infty, \infty}(\T^n)$ for any $\varep > 0$ (see \cite[Chapter 2]{BCD2011} for a textbook introduction to the Besov spaces $\mc{B}^{\alpha}_{p, q}$). This is actually a stronger statement, since one has the embedding $\mc{B}^{H-\varep}_{\infty, \infty}(\T^n) \hookrightarrow \dot{H}^{H-2\varep}(\T^n)$ for any $\varep > 0$. Besov spaces are commonly used in stochastic PDE (in this setting one would now have a process of fractional Gaussian fields indexed by a time parameter) where the choice of distribution space is crucial in the analytic arguments -- see e.g. the surveys \cite{Hairer2014survey, GP2015, CW2017} for more.
\end{remark}

\begin{definition}[$k$-form Gaussian free field]
We say that $A$ is a $k$-form $\gff(\manifold)$, or a $\gff^k(\manifold)$, if it is a $\fgf_1^k(\manifold)$.
\end{definition}

\begin{remark}
In principle, when $n \geq 3$ is is possible to construct the $\gff^k(\R^n)$ as a random variable taking values in $\Omega^k \schwartz'$, and when $n = 2$ it is possible to construct $\gff^k(\R^n)$ as a random variable taking values on a space $\Omega^k \schwartz_0'$, where $\schwartz_0 \sse \schwartz$ is the subspace of mean zero Schwartz functions. See \cite[Section 3]{lodhia2016fractional}. Again, we choose to work on the space $\Omega^k \Phi'$ so that we may take arbitrary powers of the Laplacian.
\end{remark}

\begin{remark}\label{remark:projection-fgf-gaussian-hilbert-space}
Given an $\fgf_s^k(\manifold)$ instance $A$ and its associated Gaussian Hilbert space $((A, \phi), \phi \in \Omega^k \dot{H}^{-s}(\manifold))$, we may obtain Gaussian Hilbert spaces $((EA, \phi), \phi \in \Omega^k \dot{H}^{-s}(\R^n))$, $((E^* A, \phi), \phi \in \Omega^k \dot{H}^{-s}(\manifold))$, by defining
\begin{equs}
(EA, \phi) := (A, E \phi), ~~ (E^* A, \phi) := (A, E^* \phi).
\end{equs}
Recall by Lemma \ref{lemmma:projections-bounded-sobolev-space} that $E\phi, E^* \phi \in \Omega^k \dot{H}^{-s}(\manifold)$ if $\phi \in \Omega^k \dot{H}^{-s}(\manifold)$, so this definition makes sense. We have that
\begin{equs}\label{eq:covariance-projected-fgf}
\Cov((EA, \phi), (EA, \psi)) = (\phi, E \psi)_{\Omega^k \dot{H}^{-s}(\manifold)}, \quad \Cov((E^*A, \phi), (E^*A, \psi)) = (\phi, E^* \psi)_{\Omega^k \dot{H}^{-s}(\manifold)}.
\end{equs}
Observe also that these processes are independent, since
\begin{equs}
\Cov((EA, \phi), (E^* A, \psi)) = (E \phi, E^* \psi)_{\Omega^k \dot{H}^{-s}(\manifold)} = (E^* E \phi, \psi)_{\Omega^k \dot{H}^{-s}(\manifold)} = 0.
\end{equs}
\end{remark}

We note that $EA$, $E^* A$ are precisely the random distributional forms $\fgf^k_s(\manifold)_{d = 0}$, $\fgf^k_s(\manifold)_{d^* = 0}$ which were mentioned in Section \ref{sec:fracforms}. We thus make the following definition.

\begin{definition}[Projections of fractional Gaussian forms]\label{def:projection-fgf}
We say that $B$ is an $\fgf^k_s(\manifold)_{d = 0}$ if it has the same law as $EA$, where $A$ is an $\fgf^k_s(\manifold)$. Similarly, we say that $B$ is an $\fgf^k_s(\manifold)_{d^* = 0}$ if it has the same law as $E^* A$.
\end{definition}

The next lemma shows that the fractional Laplacian relates fractional Gaussian forms with different $s$ parameters.

\begin{lemma}\label{lemma:laplacian-of-fgf}
Let $\noise$ be a $k$-form white noise on $\manifold$. Then $(-\Delta)^{-\frac{s}{2}} \noise$ is a $\fgf_s^k(\manifold)$. Consequently, if $A$ is a $\fgf_s^k(\manifold)$, then $(-\Delta)^{s'} A$ is a $\fgf_{s-2s'}^k(\manifold)$.
\end{lemma}
\begin{proof}
We show the claim for $M = \R^n$, as the proof for other $\manifold$ is almost the same (one just considers the appropriate space of test functions, which is different on compact $\manifold$ than on $\R^n$). For $f, g \in \Omega^k \Phi$, we have that
\[\begin{split}
\Cov\big(((-\Delta)^{-\frac{s}{2}} \noise, f), ((-\Delta)^{-\frac{s}{2}} \noise, g)\big) &= \Cov((\noise, (-\Delta)^{-\frac{s}{2}} f), (\noise, (-\Delta)^{-\frac{s}{2}} g)) \\
&= ((-\Delta)^{-\frac{s}{2}} f, (-\Delta)^{-\frac{s}{2}} g) \\
&= (f, g)_{\Omega^k \dot{H}^{-s}(\manifold)}.
\end{split}\]
For the second claim, we note that by using the first claim,
\[ (-\Delta)^{s'} A \stackrel{d}{=} (-\Delta)^{s'} (-\Delta)^{-\frac{s}{2}} \noise = (-\Delta)^{-(s - 2s')/2} \noise \stackrel{d}{=} \fgf_{s - 2s'}^k(\R^n). \qedhere\]
\end{proof}

Thus far, the discussion about fractional Gaussian forms has not been so different from the scalar case. What separates the former from the latter is that the operators $d, d^*, \star$ etc. relate fractional Gaussian forms of different degrees $k$. We proceed to discuss various results along these lines.

\begin{lemma}
Let $\noise$ be a $k$-form white noise on $\manifold$. Then $\star \noise$ is an $(n-k)$-form white noise on $\manifold$.
\end{lemma}
\begin{proof}
We show the claim for $M = \R^n$. Recall that for $f \in \Omega^{n-k} \Phi$, we have that $(\star \noise, f) = (-1)^{k(n-k)} (\noise, \star f)$. We thus have that
\[ \mrm{Var}((\star \noise, f)) = \mrm{Var}((\noise, \star f)) = (\star f, \star f) = (f, f).\]
This shows that $\star \noise$ is an $(n-k)$-form white noise.
\end{proof}

The $s = 1$ case of the next lemma says that the exterior derivative of a $(k-1)$-form GFF is the exact part of a $k$-form white noise, and the codifferential of a $(k+1)$-form GFF is the co-exact part of a $k$-form white noise.

\begin{lemma}\label{lemma:projected-white-noise}
Let $0 \leq k \leq n$, $s \in \R$, and let $A$ be an $\fgf_s^k(\manifold)$. We have that $dA$ is an $(\fgf_{s-1}^{k+1}(\manifold))_{d = 0}$, and $\codif A$ is an $(\fgf_{s-1}^{k-1}(\manifold))_{\codif = 0}$. 
\end{lemma}
\begin{remark}
Since $d$ and $\codif$ are differential operators, applying either of them to an $\fgf_s^k$ should decrease the $s$ parameter by $1$, while also increasing/decreasing the $k$ parameter by $1$. The projection onto the $d = 0$ or $\codif = 0$ subspace arises since $d^2 = 0 = (\codif)^2$.
\end{remark}
\begin{proof}
We show the claim for $M = \R^n$. For $\phi \in \Omega^{k+1} \Phi$, we have that
\[ (dA, \phi) = (A, \codif \phi) \sim \N\big(0, \|\codif \phi\|_{\Omega^k \dot{H}^{-s}(\R^n)}^2\big). \]
Observe that
\[ \|\codif \phi\|_{\Omega^k \dot{H}^{-s}(\R^n)}^2 = ((-\Delta)^{-s/2} \codif \phi, (-\Delta)^{-s/2} \codif \phi) = (\phi, d\codif (-\Delta)^{-s} \phi) = (\phi, (-\Delta)^{-(s-1)} E \phi). \]
To finish, note that $(\phi, (-\Delta)^{-(s-1)} E \phi) = (E\phi, (-\Delta)^{-(s-1)} E \phi) = \|E \phi\|^2_{\Omega^{k+1} \dot{H}^{s-1}(\R^n)}$. This shows the result for $dA$ (recall \eqref{eq:covariance-projected-fgf}). The result for $\codif A$ follows by similar calculations. 
\end{proof}

The next lemma says that the Hodge star of the exact part of a $k$-form white-noise is the co-exact part of an $(n-k)$-form white noise, and vice versa.

\begin{lemma}\label{lemma:hodge-star-fgf}
Let $1 \leq k \leq n$. Let $\noise, \eta$ be respectively $k$-form and $(n-k)$-form white noises on $\manifold$. Then $\star E \noise \stackrel{d}{=} E^* \eta$, and $\star E^* \eta \stackrel{d}{=} E \noise$.
\end{lemma}
\begin{proof}
For the first identity, we have that $\star E \noise = E^* \star \noise \stackrel{d}{=} E^* \eta$. For the second identity, note that by the first identity, we have $\star E^* \eta \stackrel{d}{=} \star \star E \noise = (-1)^{k(n-k)} E \noise \stackrel{d}{=} E \noise$.
\end{proof}

Next, we give a concrete application of the previous results which relates the curl of a divergence-free $1$-form GFF to another fractional Gaussian form.

\begin{lemma}
Let $n = 3$. Let $A$ be a $\gff^1(\R^3)_{d^* = 0}$ and $\noise$ be a $1$-form white noise on $\R^3$. Then $\curl A \stackrel{d}{=} E^* \noise \stackrel{d}{=} (-\Delta)^{1/2} A$.
\end{lemma}
\begin{proof}
By combining Lemmas \ref{lemma:projected-white-noise} and \ref{lemma:hodge-star-fgf}, we obtain
\begin{equs}
\curl A = \star d A \stackrel{d}{=} E^* \noise.
\end{equs}
We may further write
\begin{equs}
E^* \noise 
= (-\Delta)^{\frac{1}{2}} E^* (-\Delta)^{-\frac{1}{2}} \noise.
\end{equs}
By Lemma \ref{lemma:laplacian-of-fgf}, we have that $(-\Delta)^{-\frac{1}{2}} \noise$ is a $\gff^1(\R^3)$, and thus $E^* (-\Delta)^{-\frac{1}{2}} \noise$ has the same distribution as $A$, i.e. it is a $\gff^1(\R^3)_{d^* = 0}$. The desired result now follows.
\end{proof}

Next, we define massive variants of fractional Gaussian forms.

\begin{definition}[Massive $k$-form fractional Gaussian field]
Let $\lambda > 0$. We say that $A$ is a $k$-form $\fgf_s(\R^n; \lambda)$ if it is a $\Omega^k \mc{S}'$-valued random variable such that $((A, f), f \in \Omega^k \mc{S})$ is a centered Gaussian process with covariance
\begin{equs}\label{eq:massive-fgf-covariance}
\Cov((A, f), (A, g)) = (f, (-\Delta + \lambda)^{-s} g).
\end{equs}
Let $M$ be a compact oriented Riemannian manifold. We say that $A$ is a $k$-form $\fgf_s(\manifold; \lambda)$ if it is a $\Omega^k \mc{D}'(\manifold)$-valued random variable such that $((A, f), f \in \Omega^k(\manifold))$ is a centered Gaussian process with the above covariance. We will denote a $k$-form $\fgf_s(\manifold; \lambda)$ by $\fgf_s^k(\manifold; \lambda)$. In the case $s = 1$, we will say that $A$ is a $k$-form $\gff(\manifold)$ with mass $\lambda$, or alternatively a $\gff^k(\manifold; \lambda)$.
\end{definition}

\begin{remark}
By the usual isometry arguments, given $A$ which is a $\fgf_s^k(\manifold; \lambda)$, we may obtain a Gaussian process $((A, f), f \in \Omega^k H^{-s}(\manifold))$ indexed by elements of the inhomogeneous Sobolev space $\Omega^k H^{-s}(\manifold)$, with covariance given by \eqref{eq:massive-fgf-covariance}.
\end{remark}

\begin{definition}[Vector space valued fractional Gaussian forms]
Let $(V, \langle \cdot, \cdot \rangle_V)$ be a finite-dimensional inner product space. We say that $A$ is a $V$-valued $k$-form $\fgf_s^k(\R^n)$ if it is an $\Omega^k \Phi'(V)$-valued random variable such that $((A, \phi), \phi \in \Phi(V))$ is a mean zero Gaussian process with covariance
\begin{equs}
\Cov((A, \phi), (A, \psi)) = (\phi, \psi)_{\Omega^k \dot{H}^{-s}(\R^n; V)}.
\end{equs}
Let $\manifold$ be a compact oriented Riemannian manifold. We say that $A$ is a $V$-valued $k$-form $\fgf^k_s(\manifold)$ if it is an $\Omega^k \bigdot{\mc{D}}'(\manifold; V)$-valued random variable such that $((A, \phi), \phi \in \bigdot{\Omega}^k (\manifold; V))$ is a mean zero Gaussian process with the above covariance (where $\R^n$ is replaced by $\manifold$ in the inner product).
\end{definition}

Similar to Remark \ref{remark:values-in-a-general-vector-space}, all the previous discussion extends to this more general case of $V$-valued fractional Gaussian forms. Inherently, this has to be the case, because we can always expand a given $V$-valued fractional Gaussian form into an orthonormal basis and obtain a collection of i.i.d. $\R$-valued fractional Gaussian forms.

\subsection{Covariance kernel of FGF}

In this section, we review the explicit form of the covariance kernel of $\fgf_s^0(\R^n)$, as in \cite[Section 3.2]{lodhia2016fractional}. Let $h$ be a $0$-form $\fgf_s$. We want a kernel $G^s(x, y)$ such that for all $\phi_1, \phi_2 \in \Phi$, 
\[ \mrm{Cov}((h, \phi_1), (h, \phi_2)) = (\phi_1, \phi_2)_{\dot{H}^{-s}(\R^n)}= (\phi_1, (-\Delta)^{-s} \phi_2) = \int_{\R^n} \int_{\R^n} \phi_1(x) G^s(x, y) \phi_2(y) dx dy.\]
This formula can be extended to $\phi_1, \phi_2$ lying in a larger class of test functions (see Remark \ref{remark:loosen-phi-restriction}).
When $0 < 2s < n$, we have that $G^s$ is given by
\[ G^s(x, y) = k_{2s}(x-y) = \frac{1}{\gamma_n(2s)} |x-y|^{2s - n}, \]
where
\begin{equs}\label{eq:gamma-n-def} \gamma_n(\alpha) := 2^\alpha \pi^{n/2} \frac{\Gamma(\alpha / 2)}{\Gamma((n-\alpha)/2)}, ~~ \alpha \in (0, n).\end{equs}
This follows because for $\alpha \in (0, n)$, we have that $(-\Delta)^{-\frac{\alpha}{2}} \phi_2$ is explicitly given by the convolution
\[ (-\Delta)^{-\frac{\alpha}{2}} \phi_2 = k_{\alpha} * \phi_2, \]
where $k_\alpha$ is the Riesz potential
\[ k_\alpha(x) := \frac{1}{\gamma_n(\alpha)} |x|^{\alpha - n}.\]
We now seek to extend this to the case of general $\alpha \in \R$. The following proposition encapsulates all the possible cases. First, let
\beq\label{eq:H-j-def} H_0 := 1, ~~ H_j := \frac{1}{2^j j! \prod_{i=0}^{j-1} (n + 2i)}, ~~ j \geq 1. \eeq

\begin{remark}
There is a typo in \cite[Theorem 3.3]{lodhia2016fractional} in the definition of $H_j$ -- there should be no $\Omega_d$ (which in that paper is the surface area of the $d$-dimensional hypersphere) in the numerator. This is a result of a corresponding typo on \cite[Page 47]{landkof1972foundations} in the definition of $H_k$, where there should be no $\omega_p$ in the numerator. This typo is corrected in the above definition.
\end{remark}

\begin{prop}\label{prop:fractional-laplacian-formula}
Let $\phi \in \Phi$. We have that $((-\Delta)^{-\frac{\alpha}{2}} \phi)(x)$ is equal to:
\begin{enumerate}
    \item\label{item:fractional-laplacian-positive-non-integer} $\alpha > 0$ and $\alpha \notin n + 2\N$
    \[ \frac{1}{\gamma_n(\alpha)}\int_{\R^n} \phi(y) |x-y|^{\alpha - n} dy, \]
    \item \label{item:fractional-laplacian-positive-integer-log-correction}$\alpha = n + 2m$, $m \geq 0$
    \[ 2\frac{(-1)^{m+1} 2^{-(n+2m)}\pi^{-n/2}}{ m! \Gamma((n+2m)/2)} \int_{\R^n} \phi(y) |x-y|^{2m} \log|x-y| dy, \]
    \item $\alpha < 0$, $\alpha \in (-2m, -2(m-1))$ for some $m \geq 1$
    \[ \int_{\R^n} \bigg(\phi(y) - \sum_{j=0}^{m-1} H_j (\Delta^j \phi)(x) |x-y|^{2j}\bigg) |x-y|^{\alpha - n} dy, \] 
    \item $\alpha = -2m$, some $m \geq 0$
    \[ \int_{\R^n} \phi(y) ((-\Delta)^{m} \delta_x)(y) dy .\]
\end{enumerate}
\end{prop}

Because of the typo in the definition of $H_j$ in \cite{landkof1972foundations} (and consequently also in \cite{lodhia2016fractional}), we prove Proposition \ref{prop:fractional-laplacian-formula} in Appendix \ref{section:covariance-kernel-proof} with the corrected definition of $H_j$.

\begin{remark}\label{remark:loosen-phi-restriction}
As will be clear from the proof, we may loosen the restriction that $\phi \in \Phi$ to varying degrees, depending on the case. Cases (1) and (2) continue to hold as long as $\phi \in \mc{S}$ is such that $\widehat{\phi}(p) = o(|p|^{\alpha - n + \varep})$ for some $\varep > 0$ as $p \ra 0$, while cases (3) and (4) actually hold for any $\phi \in \mc{S}$.
\end{remark}

Given this proposition, we may directly obtain the formula for $G^s(x, y)$ in all the possible cases. 

\begin{cor}\label{cor:fgf-covariance-kernel}
Let $s \in \R$. We have that $G^s(x, y)$ is equal to:
\begin{enumerate}
    \item $2s > 0$ and $2s \notin n + 2\N$
    \[ \frac{1}{\gamma_n(2s)} |x-y|^{2s - n},  \]
    \item $2s = n + 2m$ for some $m \geq 0$
    \[ 2\frac{(-1)^{m+1} 2^{-(n+2m)}\pi^{-n/2}}{ m! \Gamma((n+2m)/2)}  |x-y|^{2m} \log|x-y|,\]
    \item $2s \in (-2m, -2(m-1))$ for some $m \geq 1$
    \[ \bigg(1 - \sum_{j=0}^{m-1} H_j (\Delta^j \delta_x)(y) |x-y|^{2j}\bigg)|x-y|^{\alpha - n} .\]
    \item $2s = -2m$ for some $m \geq 0$
    \[ ((-\Delta)^m \delta_x)(y). \]
\end{enumerate}
\end{cor}

\subsection{Covariance kernel of the $k$-form FGF}


In this section, we discuss covariance kernels of various cases of $\fgf_s^k(\R^n)$. First, we say that $K(x, y) = (K_{\mbf{i} \mbf{j}}(x, y))$, where $\mbf{i}, \mbf{j}$ range over increasing sequences of length $k$ with values in $[n]$, is the covariance kernel of a random distributional $k$-form $A$ if for all $\phi, \psi \in \Omega^k \Phi$, we have that
\begin{equs}
\Cov((A, \phi), (A, \psi)) = \sum_{1 \leq i_1 < \cdots < i_k \leq n} \sum_{1 \leq j_1 < \cdots < j_k \leq n} \int_{\R^n} \int_{\R^n} dx dy \phi_{\mbf{i}}(x) K_{\mbf{i} \mbf{j}}(x, y) \psi_{\mbf{j}}(y),
\end{equs}
where $\mbf{i} = (i_1, \ldots, i_k), \mbf{j} = (j_1, \ldots, j_k)$.

Later, we will see in the case of a projected white noise that this definition is slightly ill-defined, and as a result we will need to slightly relax the defining condition for $K$. 

\begin{example}
It follows directly from its definition that the covariance kernel of $\fgf^k_s(\R^n)$ is given by 
\begin{equs}
K_{\mbf{i}\mbf{j}}(x, y) = \delta_{\mbf{i} \mbf{j}} G^s(x, y), 
\end{equs}
where $\delta_{\mbf{i} \mbf{j}} = \prod_{\ell =1}^k \delta_{i_\ell j_\ell}$,
and where $G^s$ is as in Corollary \ref{cor:fgf-covariance-kernel}. 
\end{example}

More interesting examples of covariance kernels arise if we apply the projections $E$ or $E^*$ to a $\fgf_s^k(\R^n)$, i.e. if we consider covariance kernels of $\fgf_s^k(\R^n)_{d = 0}$ or $\fgf_s^k(\R^n)_{\codif = 0}$. In the following proposition, we derive the covariance kernel in the most interesting (at least, from the point of view of Yang--Mills) case $s = 1, k = 1$ (i.e. $1$-form GFF), as well as other cases $s > 0, k = 1$. In the following, let $\omega_n$ be the surface area of the unit sphere in $n$ dimensions $\{x \in \R^n : |x| = 1\}$, which is explicitly given by:
\begin{equs}\label{eq:omega-n}
\omega_n := \frac{2\pi^{n/2}}{\Gamma(n/2)}
\end{equs}

\begin{prop}\label{prop:1-form-gff-projection-correlation-functions}
The covariance kernel of $\fgf_1^1(\R^n)_{d = 0}$ is
\begin{equs}
K^{d = 0}_{ij}(x, y) := \frac{\delta_{ij}}{2} G^1(x, y) - \frac{1}{2\omega_n} \frac{(x_i - y_i)(x_j - y_j)}{|x-y|^n}.
\end{equs}
The covariance kernel of $\fgf_1^1(\R^n)_{\codif = 0}$ is
\begin{equs}
K^{\codif = 0}_{ij}(x, y) := \frac{\delta_{ij}}{2} G^1(x, y) + \frac{1}{2\omega_n} \frac{(x_i - y_i)(x_j - y_j)}{|x-y|^n}.
\end{equs}
Put more succinctly, in matrix notation, we have that
\begin{equs}
K^{d = 0}(x, y) &= \frac{1}{2} G^1(x, y) I - \frac{1}{2\omega_n} \frac{(x-y) (x-y)^T}{|x-y|^n}, \\
K^{\codif = 0}(x, y) &= \frac{1}{2} G^1(x, y) I + \frac{1}{2\omega_n} \frac{(x-y) (x-y)^T}{|x-y|^n}.
\end{equs}
For other values of $s > 0$, we have that the respective covariance kernels $K^{d=0}_s, K^{\codif = 0}_s$ of $\fgf_s^1(\R^n)_{d = 0}, \fgf_s^1(\R^n)_{\codif = 0}$ are given by:
\begin{enumerate}
    \item $0 < s < 1$
    \begin{equs}
    K^{d=0}_s(x, y) &= \frac{1}{2s} G^s(x, y) I - \frac{\Gamma((n-2(s-1))/2)}{2^{2s} \pi^{n/2} \Gamma(s+1)} |x-y|^{2(s-1)-n} (x-y)(x-y)^T , \\
    K^{\codif = 0}_s(x, y) &= \bigg(1 - \frac{1}{2s}\bigg) G^s(x, y) I + \frac{\Gamma((n-2(s-1))/2)}{2^{2s} \pi^{n/2} \Gamma(s+1)} |x-y|^{2(s-1)-n} (x-y)(x-y)^T .
    \end{equs}
    \item $s > 1$
    \begin{equs}
    K^{d=0}_s(x, y) &= \frac{1}{2s} G^s(x, y) I - \frac{\Gamma(s-1)}{4 \Gamma(s+1)}G^{s-1}(x, y) (x - y)(x-y)^T, \\
    K^{\codif = 0}_s(x, y) &= \bigg(1 - \frac{1}{2s}\bigg) G^s(x, y) I +  \frac{\Gamma(s-1)}{4 \Gamma(s+1)} G^{s-1}(x, y) (x - y)(x-y)^T.
\end{equs}
\end{enumerate}

\end{prop}

This proposition follows directly from the next lemma, whose proof is deferred to Appendix \ref{section:covariance-kernel-proof}.

\begin{lemma}\label{lemma:projection-laplace-inverse-formula}
Let $\phi \in \Omega^1 \Phi$. For $i \in [n]$, $x \in \R^n$, $s > 0$, we have that
\begin{enumerate}
    \item $0 < s < 1$
    \begin{equs}
    (E (-\Delta)^{-s} \phi)_i(x) &= \frac{\delta_{ij}}{2s} \int_{\R^n} dy G^{s}(x, y) \phi_j(y) \\
    &\quad - \frac{\Gamma((n-2(s-1))/2)}{2^{2s} \pi^{n/2} \Gamma(s+1)} \int_{\R^n} dy |x-y|^{2(s-1) -n} (x_i - y_i)(x_j - y_j) \phi_j(y), \\
    (E^* (-\Delta)^{-s} \phi)_i(x) &= \delta_{ij}\bigg(1 - \frac{1}{2s}\bigg) \int_{\R^n} dy G^{s}(x, y) \phi_j(y) \\
    &\quad +  \frac{\Gamma((n-2(s-1))/2)}{2^{2s} \pi^{n/2} \Gamma(s+1)} \int_{\R^n} dy |x-y|^{2(s-1) -n} (x_i - y_i)(x_j - y_j) \phi_j(y).
    \end{equs}
    \item $s = 1$
    \begin{equs}
    (E (-\Delta)^{-1} \phi)_i(x) &= \frac{1}{2} \int_{\R^n} dy \Big( \delta_{ij} G^1(x, y) - \frac{1}{\omega_n} \frac{(x_i - y_i) (x_j - y_j)}{|x-y|^n}\Big)\phi_j(y) , \\
    (E^* (-\Delta)^{-1} \phi)_i(x) &= \frac{1}{2} \int_{\R^n} dy  \Big(\delta_{ij} G^1(x, y) + \frac{1}{\omega_n} \frac{(x_i - y_i) (x_j - y_j)}{|x-y|^n} \Big) \phi_j(y).
\end{equs}
    \item $s > 1$ 
    \begin{equs}
    (E (-\Delta)^{-s} \phi)_i(x) &= \frac{\delta_{ij}}{2s} \int_{\R^n} dy G^{s}(x, y) \phi_j(y) ~\\
    &\quad - \frac{\Gamma(s-1)}{4 \Gamma(s+1)} \int_{\R^n} dy G^{s-1}(x, y) (x_i - y_i)(x_i - y_j) \phi_j(y), \\
    (E^* (-\Delta)^{-s} \phi)_i(x) &= \delta_{ij}\bigg(1 - \frac{1}{2s}\bigg) \int_{\R^n} dy G^{s}(x, y) \phi_j(y) \\
    &\quad +\frac{\Gamma(s-1)}{4 \Gamma(s+1)} \int_{\R^n} dy G^{s-1}(x, y) (x_i - y_i)(x_i - y_j) \phi_j(y).
    \end{equs}
\end{enumerate}
\end{lemma}

Next, we discuss the covariance kernel of projected white noise, which is the other interesting case from the point of view of Yang--Mills, since $d \fgf_1^1(\R^n) \stackrel{d}{=} \fgf_0^1(\R^n)_{d = 0}$
(recall Lemma \ref{lemma:projected-white-noise}). This can be thought of as the $s = 0$ analog of Proposition \ref{prop:1-form-gff-projection-correlation-functions}. Recall the projection $E 
= dd^* (-\Delta)^{-1}$. In coordinates, given a $2$-form $\phi \in \Omega^2 \Phi$, we have that 
\begin{equs}
(E \phi)_{ij} &= \ptl_i (d^* (-\Delta)^{-1} \phi)_j - \ptl_j (d^* (-\Delta)^{-1} \phi)_i \\
&= - \ptl_i \ptl_k (-\Delta)^{-1} \phi_{k j} + \ptl_j \ptl_k (-\Delta)^{-1} \phi_{ki}.
\end{equs}
Define the Riesz transforms by $\mc{R}_i$, $i \in [n]$ by $\mc{R}_i := -\ptl_j (-\Delta)^{-1/2}$. Note that $\mc{R}_i : \Phi \ra \Phi$ maps the Lizorkin space to itself. Note also that $\mc{R}_j \phi = \mc{F}^{-1}(\icomplex p_j |p|^{-1} \mc{F} \phi)$. Define the iterated Riesz transform $\mc{R}_{ij} := \mc{R}_i \mc{R}_j$. We have that
\begin{equs}
(E\phi)_{ij} = -\mc{R}_{ik} \phi_{kj} + \mc{R}_{jk} \phi_{ki}.
\end{equs}
Thus if $\noise$ is a $\fgf_0^2(\R^n)$, then the projection $E \noise$ satisfies (the $\frac{1}{2}$ factor arises from a summation over all $i, j$, and not just $i < j$)
\begin{equation}\label{eq:projected-white-noise-covariance}
\begin{aligned}
\Cov((E\noise, \phi), (E\noise, \psi)) &= (\phi, E\psi) \\
&= \frac{1}{2} (\phi_{ij}, (E\psi)_{ij}) \\
&= \frac{1}{2} \big(\phi_{ij}, -\mc{R}_{ik} \psi_{kj} + \mc{R}_{jk} \psi_{ki}\big).
\end{aligned} 
\end{equation}
Thus to obtain the covariance kernel for $E \noise$, we want to find a kernel $G_{ik}$ such that
\begin{equs}
(\phi_{ij}, \mc{R}_{ik} \psi_{kj}) = \int dx \int dy \phi_{ij}(x) G_{ik}(x, y) \psi_{kj}(y) .
\end{equs}
This is a classical topic in analysis (in particular the theory of singular integral operators), and it is well known that unfortunately such a kernel $G_{ik}$ does not exist. See \cite{Stein1970} for more on the theory of singular integral operators. Instead, what is true is that there exists a kernel $G_{ik}$ such that
\begin{equs}
(\phi_{ij}, \mc{R}_{ik} \psi_{kj}) = \lim_{\varep \downarrow 0} \int \int_{|x-y| > \varep} \phi_{ij}(x) G_{ik}(x, y) \psi_{kj}(y) .
\end{equs}
This is the content of the next lemma. 

\begin{lemma}\label{lemma:iterated-riesz-kernel}
For $\phi_1, \phi_2 \in \Phi$, $i, j \in [n]$, we have that
\[ \begin{split}
(\phi_1, \mc{R}_{ij} \phi_2) = ~~~~&\frac{n}{\omega_n} \lim_{\varep \downarrow 0} \int \int_{|x-y| > \varep} \phi_1(x) \phi_2(y) \frac{(x_i - y_i)(x_j - y_j)}{|x-y|^{n+2}} dx dy \\
&- \frac{\delta_{ij}}{\omega_n}\lim_{\varep \downarrow 0} \int \int_{|x-y| > \varep} \phi_1(x) \phi_2(y) |x-y|^{-n} dx dy  \\
&- \frac{\delta_{ij}}{n} (\phi_1, \phi_2).
\end{split}\]
\end{lemma}
\begin{remark}
As a sanity check, note that the formula gives
\begin{equs}
(\phi_1, \mc{R}_{ij} \phi_2) \delta^{ij} = - (\phi_1, \phi_2),
\end{equs}
which is consistent with the fact that $\mc{R}_{ij} \delta^{ij} = \ptl_i \ptl_i (-\Delta)^{-1} = \Delta (-\Delta)^{-1} = -I$.
\end{remark}

The proof of Lemma \ref{lemma:iterated-riesz-kernel} is deferred to Appendix \ref{section:covariance-kernel-proof}. Using this lemma and \eqref{eq:projected-white-noise-covariance}, we immediately obtain the ``covariance kernel" for $E\noise$.

\begin{prop}\label{prop:projected-white-noise-correlation-functions}
Let $\noise$ be a $\Omega_0^2(\R^n)$. For $\phi, \psi \in \Omega^2 \Phi$, we have that
\begin{equs}
\Cov((E\noise, \phi), (E\noise, \psi)) = &~\frac{1}{n}(\phi_{ij}, \psi_{ij}) + \lim_{\varep \downarrow 0} \frac{1}{\omega_n} \int \int_{|x-y| > \varep} dx dy \phi_{ij}(x) \psi_{ij}(y) |x-y|^{-n} \\
&- \frac{n}{2\omega_n} \lim_{\varep \downarrow 0} \int \int_{|x-y| > \varep} dx dy \phi_{ij}(x) \frac{x_k - y_k}{|x-y|^{n+2}} \Big( \psi_{kj}(y) (x_i - y_i) - \psi_{ki}(y) (x_j - y_j) \Big).
\end{equs}
\end{prop}
\begin{proof}
For $i, j \in [n]$, we have by Lemma \ref{lemma:iterated-riesz-kernel} that
\begin{equs}
(\phi_{ij}, -\mc{R}_{ik} \psi_{kj} &+ \mc{R}_{jk} \psi_{ki}) = \frac{\delta_{ik}}{n} (\phi_{ij}, \psi_{kj}) - \frac{\delta_{jk}}{n} (\phi_{ij}, \psi_{ki}) \\
&+ \frac{1}{\omega_n} \lim_{\varep \downarrow 0} \int \int_{|x-y| > \varep} dx dy \phi_{ij}(x) (\delta_{ik} \psi_{kj}(y) - \delta_{jk} \psi_{ki}(y)) |x-y|^{-n}\\
&- \frac{n}{\omega_n} \lim_{\varep \downarrow 0} \int \int_{|x-y| > \varep} dx dy \phi_{ij}(x) \frac{(x_k - y_k)}{|x-y|^{n+2}} \Big(\psi_{kj}(y) (x_i - y_i) - \psi_{ki}(y)(x_j - y_j)\Big).
\end{equs}
Upon performing the summation in $k$ and using that $\psi_{ki} = -\psi_{ik}$, we further obtain 
\begin{equs}
(\phi_{ij}, -\mc{R}_{ik} \psi_{kj} &+ \mc{R}_{jk} \psi_{ki}) = \frac{2}{n} (\phi_{ij}, \psi_{ij}) + \lim_{\varep \downarrow 0} \frac{2}{\omega_n} \int \int_{|x-y| > \varep} dx dy \phi_{ij}(x) \psi_{ij}(y) |x-y|^{-n} \\
&- \frac{n}{\omega_n} \lim_{\varep \downarrow 0} \int \int_{|x-y| > \varep} dx dy \phi_{ij}(x) \frac{(x_k - y_k)}{|x-y|^{n+2}} \Big(\psi_{kj}(y) (x_i - y_i) - \psi_{ki}(y)(x_j - y_j)\Big).
\end{equs}
The desired result now follows by combining this with \eqref{eq:projected-white-noise-covariance}.
\end{proof}



\section{Gaussian 1-forms and Yang-Mills} \label{sec:yangmills}


Various versions of fractional Gaussian forms are related to Yang--Mills. Indeed, this was one of the main motivations for writing this survey. In this section, we discuss Yang--Mills and its relation to fractional Gaussian forms in various cases. 

\subsection{The \texorpdfstring{$\unitary(1)$}{U(1)} theory}\label{section:U(1)-theory}

The $\unitary(1)$ pure Yang--Mills measure on $\R^n$ is the formal probability measure on the space of $1$-forms with density given by
\begin{equs}\label{eq:u(1)-measure}
Z^{-1} \exp\Big(- \frac{1}{2}\|dA\|_2^2\Big) dA, 
\end{equs}
where $dA$ is formal Lebesgue measure on the space of $1$-forms. To be clear, here $\|dA\|_2^2 = (dA, dA)$. The observables of interest are the ``gauge-invariant" observables, which in this setting are functions of $A$ only through $dA$, e.g. $(dA, \phi)$ for a test $2$-form $\phi$. Thus, one may define the $\unitary(1)$ pure Yang--Mills measure as the law of a random distributional two form $F$ such that $((F, \phi), \phi \in \Omega^2 \mc{S})$ is a Gaussian process with
\begin{equs}\label{eq:law-of-curvature}
(F, \phi) \sim \N\big(0, (\codif \phi, (-\Delta)^{-1} \codif \phi)\big).
\end{equs}
This is how the measure is defined in \cite{Gross1983, Driver1987} (these papers refer to the field $F$ as the ``free electromagnetic field"). We note that 
\begin{equs}
(\codif \phi, (-\Delta)^{-1} \codif \phi) = (\phi, d(-\Delta)^{-1} d^* \phi) = (\phi, E\phi).
\end{equs}
This immediately leads to the following characterization of the $\unitary(1)$ pure Yang--Mills measure as 2-form white noise, projected onto the space of closed 2-forms.

\begin{lemma}
The law of $\fgf_0^2(\R^n)_{d = 0}$ is the $\unitary(1)$ pure Yang--Mills measure.
That is, instances $F$ of $\fgf_0^2(\R^n)_{d = 0}$ satisfy \eqref{eq:law-of-curvature}. 
\end{lemma}

\begin{remark}\label{remark:random-distribution-vs-stochastic-process}
To be clear, recall that from our definition of fractional Gaussian forms (Definition \ref{def:fgf}), instances $F$ of $\fgf_0^2(\R^n)_{d = 0}$ are random variables taking values in $\Omega^2 \Phi'$, so that $F$ induces a stochastic process $((F, \phi), \phi \in \Omega^2 \Phi)$. Then by Hilbert space isometry (Remark \ref{remark:gaussian-hilbert-space}), we extend the index set of this stochastic process to $((F, \phi), \phi \in \Omega^2 \dot{H}^{-1}(\R^n))$. When $n \geq 3$, we have that $\schwartz \sse \dot{H}^{-1}(\R^n)$ and thus also $\Omega^2 \schwartz \sse \Omega^2 \dot{H}^{-1}(\R^n)$. Therefore $F$ can indeed be viewed as a stochastic process indexed by Schwartz $2$-forms, which is how e.g. \cite{Gross1983} works with the $\unitary(1)$ Yang--Mills measure. This is slightly weaker than defining $F$ as a random tempered distribution, but for most practical purposes this difference does not matter.

When $n = 2$, the $\unitary(1)$ Yang--Mills measure is simply $2$-form white noise, since the projection $E$ acts as the identity in this dimension. Put another way, we have that $\fgf_0^2(\R^2)_{d = 0} \stackrel{d}{=} \fgf_0^2(\R^2)$.
\end{remark}

If instead one wants to construct the measure \eqref{eq:u(1)-measure} on a space of distributional 1-forms, then an immediate obstacle is that the action $\|dA\|_2^2$ is zero on the subspace of closed $1$-forms $\{A : dA = 0\}$. We may circumvent this in one of two ways: either ``gauge-fix" and restrict to the space of divergence-free $1$-forms $\{A : \codif A  = 0\}$, or modify the action from $\|dA\|_2^2$ to $\|dA\|_2^2 + \|\codif A \|_2^2 = (A, (-\Delta) A)$. The latter approach gives a $\fgf_1^1(\R^n)$, while the former approach gives a $\fgf_1^1(\R^n)_{\codif = 0}$. Both approaches are equivalent in that their exterior derivatives have the same law, which is characterized by \eqref{eq:law-of-curvature}. We state this as the following lemma. The proof is omitted as it follows directly from the results of Section \ref{sec:fgfconstruction}.

\begin{lemma}
We have that $d \fgf_1^1(\R^n) \stackrel{d}{=} d \fgf_1^1(\R^n)_{\codif = 0} \stackrel{d}{=} \fgf_0^1(\R^n)_{d = 0}$.
\end{lemma}

Another way to construct the $\unitary(1)$ pure Yang--Mills measure as a random distributional 1-form is to start with 2-form white noise, apply the inverse Laplacian, and then take the codifferential. This gives an instance of $\fgf_1^1(\R^n)_{\codif = 0}$.
We state this as the following lemma. Again, the proof is omitted.

\begin{lemma}
We have that $\codif (-\Delta)^{-1} \fgf_0^1(\R^n) \stackrel{d}{=} \fgf_1^1(\R^n)_{\codif = 0}$.
\end{lemma}

Note that $(-\Delta)^{-1} \fgf_0(\R^n)$ is usually referred to as the membrane model. Thus, it is natural to refer to $(-\Delta)^{-1} \fgf_0^1(\R^n)$ as the $1$-form membrane model.

We also note that $\fgf_1^1(\R^n)_{\codif = 0}$ is also often referred to as the massless Euclidean Proca field; see \cite[Section 2.2]{Chatterjee2024} and the references therein. That paper also discusses the massive Euclidean Proca field, which we now discuss. Let $\lambda > 0$. In \cite[equation (2.3)]{Chatterjee2024}, the operator $Q_\lambda$ is defined as
\begin{equs}
Q_\lambda := -\Delta + \lambda - d d^* = d^* d + \lambda.
\end{equs}
The operator $R_\lambda$, defined in \cite[equation (2.2)]{Chatterjee2024} is shown in \cite[Lemma 2.2]{Chatterjee2024} to be the inverse of $Q_\lambda$. By using the projections $E, E^*$, we have the following alternative formula for the inverse of $Q_\lambda$.

\begin{lemma}\label{lemma:Q-lambda-inverse}
On $\Omega^1 \Phi$, we have that
\begin{equs}
Q_\lambda^{-1} = \lambda^{-1} E + (-\Delta + \lambda)^{-1} E^*.
\end{equs}
\end{lemma}
\begin{proof}
Since $d^* d E^* = (-\Delta) E^*$, we may write $Q_\lambda = (-\Delta) E^* + \lambda$ = $(-\Delta + \lambda) E^* + \lambda E$. The formula for the inverse now directly follows.
\end{proof}

By definition, the massive Euclidean Proca field is the random distributional 1-form $X$ such that $((X, \phi), \phi \in \Omega^1 \mc{S})$ is a Gaussian process with
\begin{equs}\label{eq:proca-field}
(X, \phi) \sim \N\big(0, (\phi, Q_\lambda^{-1} \phi)\big).
\end{equs}
Thus, the law of $X$ has the formal representation
\begin{equs}
Z^{-1} \exp\Big(-\frac{1}{2}(X, Q_\lambda X)\Big) dX = Z^{-1} \exp\Big(-\frac{1}{2}\|d X\|_{2}^2 - \frac{\lambda}{2}\|X\|_2^2\Big) dX.
\end{equs}
Note that the action behaves differently on the subspace of closed 1-forms compared to the subspace of divergence-free 1-forms. On the former subspace, the action is that of a massive GFF, while on the latter subspace, the action is that of a white noise. In particular, this decomposition points the way towards constructing the massive Euclidean Proca field, which is the content of the next lemma.

\begin{lemma}\label{lemma:proca-field-construction}
Let $\noise$ be a 1-form white noise. Then
\begin{equs}
X := \frac{1}{\sqrt{\lambda}} E \noise + (-\Delta + \lambda )^{-\frac{1}{2}} E^* \noise
\end{equs}
is the massive Euclidean Proca field.
\end{lemma}
\begin{remark}
\cite{Chatterjee2024} constructs the massive Euclidean Proca field as the scaling limit of analogous lattice fields, see \cite[Section (4.3)]{Chatterjee2024}. Lemma \ref{lemma:proca-field-construction} is an alternative construction directly in the continuum.
\end{remark}
\begin{proof}
Let $\phi \in \Omega^1 \Phi$. We may compute the variance (recall from Remark \ref{remark:projection-fgf-gaussian-hilbert-space} that $E \noise, E^* \noise$ are independent)
\begin{equs}
\mrm{Var}((X, \phi)) = \frac{1}{\lambda}(\phi, E\phi) + (\phi, (-\Delta + \lambda)^{-1} E^* \phi) &= \Big(\phi, \big(\lambda^{-1} E + (-\Delta + \lambda)^{-1} E^* \big) \phi\Big) \\
&= (\phi, Q_\lambda^{-1} \phi),
\end{equs}
where we applied Lemma \ref{lemma:Q-lambda-inverse} in the last identity. The desired result now follows.
\end{proof}

\begin{remark}
Similar to Remark \ref{remark:random-distribution-vs-stochastic-process}, to be precise, the $X$ in Lemma \ref{lemma:proca-field-construction} is a random variable taking values in $\Omega^1 \Phi'$ which induces a mean zero Gaussian process $((X, \phi), \phi \in \Omega^1 \Phi)$ characterized by \eqref{eq:proca-field}. By the usual Hilbert space isometry arguments, the index set of this stochastic process can be extended to a larger space which contains $\Omega^1 \schwartz$. This then matches with the definition of the massive Euclidean Proca field in \cite{Chatterjee2024} as a stochastic process indexed by $\Omega^1 \schwartz$.
\end{remark}

As a corollary of Lemma \ref{lemma:proca-field-construction}, we see that the exterior derivative of the massive Euclidean Proca field has the same law as the exterior derivative of a massive 1-form GFF.

\begin{cor}
Let $\lambda > 0$, $X$ be the massive Euclidean Proca field with parameter $\lambda$, and $A$ be a $1$-form $\gff$ with mass $\lambda$. Then $dX \stackrel{d}{=} dA$.
\end{cor}

\subsection{The non-Abelian theory}\label{section:non-abelian-theory}

In this subsection, we discuss the basics of the general non-Abelian Yang--Mills theory. Let $G$ be a compact Lie group with Lie algebra $\frkg$. The Lie algebra $\frkg$ is a vector space, and we assume that it is equipped with an $\mrm{Ad}$-invariant inner product $\langle \cdot, \cdot \rangle_\frkg$. I.e., for all $X, Y \in \frkg$ and $g \in \liegroup$, we have that $\langle \mrm{Ad}_g X, \mrm{Ad}_g Y \rangle_{\frkg} = \langle X, Y \rangle_{\frkg}$ (in other words, $\mrm{Ad}_g : \frkg \ra \frkg$ is an isometry). In the following, we will always assume that $\liegroup$ is a matrix Lie group, in which case $\frkg$ is a vector space of matrices, and $\mrm{Ad}_g X = g X g^{-1}$.

Let $A$ be a connection on the principal $G$-bundle $P = \R^n \times G$, which can be identified with a $\frkg$-valued 1-form $A : \R^n \ra \frkg^n$. The curvature of $A$ is the $\frkg$-valued 2-form defined by 
\begin{equs}
F_A := dA + \frac{1}{2}[A \wedge A], \quad \text{or in coordinates,}\quad (F_A)_{ij} := \ptl_i A_j - \ptl_j A_i + [A_i, A_j], ~~ i, j \in [n].
\end{equs}
The Yang--Mills action is defined as (here $|F_A(x)|^2 = \frac{1}{2} \langle (F_A)_{ij}(x), (F_A)_{ij}(x) \rangle_{\frkg}$)
\begin{equs}\label{eq:sym}
\sym(A) := \frac{1}{2} \int_{\R^n} |F_A(x)|^2 dx ,
\end{equs}
and the Yang--Mills measure is the formal probability measure
\begin{equs}\label{eq:muym}
d\mu_{\mrm{YM}}(A) = Z^{-1} \exp(-\sym(A)) dA.
\end{equs}
The Yang--Mills action is a natural geometric quantity, and as such, is invariant under changes of coordinates. In the setting of gauge theory, changes of coordinates are known as {\bf gauge transformations}. In our current setup, a gauge transformation is a function $g : \R^n \ra \liegroup$. Supposing that $g$ has enough regularity, the action of $g$ on a given connection $A$ is given by
\begin{equs}\label{eq:gauge-transformation}
A^g := \mrm{Ad}_g A - (dg) g^{-1} \text{ or, in coordinates, } A^g_i = g A_i g^{-1} - (\ptl_i g) g^{-1}.
\end{equs}
One can check that under this action,
\begin{equs}
F_{A^g} = \mrm{Ad}_g F_A, \quad \text{and thus} \quad
\sym(A^g) = \sym(A),
\end{equs}
i.e. $\sym$ is gauge invariant. Thus the formal measure $\mu_{\mrm{YM}}$ exhibits gauge symmetry, in the sense that its action is gauge-invariant. This is one reason why constructing the Yang--Mills measure is difficult. For more discussion on the role of gauge symmetry, see \cite{CC24, CCHS2022}.

We expect that the Yang--Mills measure is supposed to be related to a $\frkg$-valued $1$-form GFF. In particular, it is expected that the short-distance behavior of the Yang--Mills measure is governed by the free field. In physics, this is known as {\bf asymptotic freedom}, which was shown by physicists to hold for 4D Yang--Mills in the Nobel prize-winning work of Gross and Wilczek \cite{GW1973} and Politzer \cite{Pol1973}.

In mathematics, a precise way to state this phenomenon is on the lattice, which we will discuss in Section \ref{section:lattice-gauge-theory}. Another way which works directly in the continuum is the following. One approach towards constructing measures of the form \eqref{eq:muym} is via {\bf stochastic quantization}. First introduced by Parisi and Wu \cite{PW1981}, stochastic quantization aims to construct such measures by using a Gaussian source of randomness. For more background, see the course by Gubinelli \cite{GubSQcourse}. One common approach towards stochastic quantization is via the {\bf Langevin dynamics}, which is the stochastic gradient flow associated to the formal measure that one is trying to construct. The hope is to understand the dynamics well enough to be able to prove that it possesses a unique invariant measure, which is then taken to be the definition of the formal probability measure. The analysis of the Langevin dynamics is a stochastic PDE (SPDE) problem, and the literature on SPDEs is by now too vast so we will not attempt to give a complete survey here. See the recent SPDE survey \cite{CS2020}. We only mention that for Yang--Mills, local existence (and gauge-covariance in law) of the Langevin dynamics was shown in two and three dimensions \cite{CCHS2020, CCHS2022} (the 2D case was also revisited by \cite{BC2023} using a different approach). These references all worked on the torus, and as a result the following discussion will also be on the torus.

The more recent work \cite{CS23} was able to show global existence and invariance (as well as universality) for the Langevin dynamics in dimension two. As a corollary, the authors showed that (see \cite[Corollary 2.21]{CS23}) the random $1$-form $A$ having law \eqref{eq:muym} has a decomposition 
\begin{equs}
A = A_{\mrm{gff}} + X,
\end{equs}
where $A_{\mrm{gff}}$ is a $\frkg$-valued $1$-form GFF, and $X$ is a random function which is $\alpha$-H\"{o}lder for any $\alpha < 1$. We remark that $A_{\mrm{gff}}$ and $X$ are dependent, with complicated dependence structure. This is a precise mathematical formulation of the statement that the free field governs the short distance behavior of the Yang--Mills measure.

It is an open problem to carry out the stochastic quantization approach to Yang--Mills in dimension three (as well as dimension four, where even local existence of the Langevin dynamics is a major open problem). Similar to the 2D case, if the stochastic quantization approach can be carried out in 3D, then this would  result in a decomposition
\begin{equs}
A = A_{\mrm{gff}} + Y,
\end{equs}
where $A_{\mrm{gff}}$ is a $\frkg$-valued 1-form $\gff$, and $Y$ is now a random distribution with regularity just below $0$. In 4D, it is less clear what the decomposition should look like, and this is related to the fact that four is the critical dimension for Yang--Mills, and thus none of the existing local solutions theories for stochastic PDEs apply -- see the survey \cite{CS2020} for more discussion.

In 3D, what remains is to show global existence for the Langevin dynamics (to be precise, one needs to show a more quantitative form involving uniform-in-time bounds). Currently, the only known nontrivial\footnote{By nontrivial, we mean that the SPDE is nonlinear, because the linear SPDEs in this setting are completely understood.} results for global existence of the Langevin dynamics in the setting of Yang--Mills are in 2D: the aforementioned work \cite{CS23}, and the recent work \cite{BC24} which deals with Abelian Higgs.


\subsection{Lattice gauge theory}\label{section:lattice-gauge-theory}

We now discuss a lattice discretization of Yang--Mills known as lattice gauge theory. Let $\Lambda \sse \Z^n$ be a finite box (or more generally, a finite cell subcomplex of $\Z^n$). In anticipation of notation that will be reintroduced later (see Section \ref{section:lattice-gaussian-forms}), let $\overrightarrow{C}^1(\Lambda), C^1(\Lambda)$ be respectively be the sets of oriented and unoriented edges of $\Lambda$. Similarly, let $\overrightarrow{C}^2(\Lambda), C^2(\Lambda)$ be respectively the sets of oriented and unoriented plaquettes (i.e. unit squares) of $\Lambda$. Given an oriented edge $e$ or plaquette $p$, we will denote by $-e$, $-p$ the edge and plaquette of opposite orientations.

As in Section \ref{section:non-abelian-theory}, let $\liegroup$ be a compact matrix Lie group. We further assume that $\liegroup \sse \unitary(N)$ for some $N \geq 1$.

\begin{definition}[Lattice gauge fields]
A lattice gauge field is a function $U \colon \overrightarrow{C}^1(\Lambda) \ra \liegroup$ such that $U(-e) = U(e)^{-1}$ for all $e \in \overrightarrow{C}^1(\Lambda)$. We identify the set of lattice gauge fields with the product space $\liegroup^{C^1(\Lambda)}$.
\end{definition}

\begin{remark}
One should think of the lattice gauge field $U$ as a discretization of a continuum connection $A$. In particular, for each lattice edge $e$, $U(e)$ should be thought of as the parallel transport of $A$ along $e$. Due to properties of parallel transport, this imposes that $U(-e) = U(e)^{-1}$.
\end{remark}

\begin{definition}[Plaquette variables]
Given a lattice gauge field $U \colon \overrightarrow{C}^1(\Lambda) \ra \liegroup$ and a plaquette $p \in \overrightarrow{C}^2(\Lambda)$, which we represent by its four boundary edges $p = e_1 e_2 e_3 e_4$, define the plaquette variable (in a slight abuse of notation)
\begin{equs}
U_p := U_{e_1} U_{e_2} U_{e_3} U_{e_4}.
\end{equs}
\end{definition}

\begin{remark}
Technically, plaquette variables are only defined up to a conjugate, because given $p = e_1 e_2 e_3 e_4$, we could have also written $p = e_2 e_3 e_4 e_1$ (say), which would result in a slightly different plaquette variable in general. However, we note that these possibly different plaquette variables are always conjugates of each other. We will soon see that only conjugacy classes of plaquette variables matter to us, and thus we will ignore this indeterminacy in the definition of plaquette variables.
\end{remark}

\begin{definition}[Lattice gauge theory]
Given $\beta \geq 0$, define the probability measure $\mu_{\Lambda, \beta}$ on $\liegroup^{C^1(\Lambda)}$ by
\begin{equs}\label{eq:lgt}
d\mu_{\Lambda, \beta}(U) := Z_{\Lambda, \beta}^{-1} \prod_{p \in C^2(\Lambda)} \exp\Big( \beta \mrm{Re} \Tr(U_p)\Big) \prod_{e \in C^1(\Lambda)} dU_e.
\end{equs}
Here, $\prod_{e \in C^1(\Lambda)} dU_e$ is the product Haar measure on $\liegroup^{C^1(\Lambda)}$, and $Z_{\Lambda, \beta}$ is the normalizing constant so that $\mu_{\Lambda, \beta}$ is a probability measure. We refer to $\beta$ as the inverse temperature, and we refer to $\mu_{\Lambda, \beta}$ as a lattice gauge theory.
\end{definition}

\begin{remark}
On the space of $N \times N$ complex matrices $M$, we may define the Frobenius norm
\begin{equs}
\|M\| := \Tr(M^* M)^{\frac{1}{2}}.
\end{equs}
Note since $\liegroup \sse \unitary(N)$, we have that
\begin{equs}
\|\groupid - U\|^2 = 2 (N - \mrm{Re} \Tr(U)) \text{ for all $U \in \liegroup$.}
\end{equs}
In particular, up to an inconsequential change of the normalizing constant $Z_{\Lambda, \beta}$, an equivalent definition of lattice gauge theory is
\begin{equs}\label{eq:lgt-alternative}
d\mu_{\Lambda, \beta}(U) := Z_{\Lambda, \beta}^{-1} \prod_{p \in C^2(\Lambda)} \exp\Big(-\frac{\beta}{2}\|\groupid - U_p\|^2\Big) \prod_{e \in C^1(\Lambda)} dU_e.
\end{equs}
\end{remark}

See \cite[Section 3]{chatterjee2019yang} for how lattice gauge theory is formally derived from continuum Yang--Mills theory. See other parts of the same survey for more background and many references. We remark that gauge symmetry on the lattice takes the following form. A lattice gauge transformation is a function $g : \Lambda \ra \liegroup$, and its action on a given lattice gauge field $U$ is defined by
\begin{equs}
U^g(x, y) := g(x) U(x, y) g(y)^{-1}.
\end{equs}
We say that two lattice gauge fields $U, \tilde{U}$ are gauge-equivalent if there exists a gauge transformation $g$ such that $\tilde{U} = U^g$. One can show that this is indeed an equivalence relation. We refer to the gauge equivalence classes as gauge orbits.

It is clear that for any plaquette $p$, the plaquette variable $U^g_p$ is a conjugate of the original plaquette variable $U_p$, and thus we always have
\begin{equs}
\Tr(U^g_p) = \Tr(U_p),\quad \text{and also} \quad \|\groupid - U^g_p\|^2 = \|\groupid - U^g_p\|^2.
\end{equs}
This implies that for any gauge transformation $g : \Lambda \ra \liegroup$, given a random variable $U \sim \mu_{\Lambda, \beta}$, we also have that $U^g \sim \mu_{\Lambda, \beta}$. In other words, the measure $\mu_{\Lambda, \beta}$ is gauge invariant. This large amount of symmetry is what sets lattice gauge theories apart from other statistical mechanical models such as the Ising model or spin $O(n)$ model. This gauge symmetry is an indication that the gauge field $U$ is in fact not the fundamental quantity in gauge theory. Rather, the gauge orbit of $U$ is really what is fundamental. Thus, one ultimately seeks to understand not $\mu_{\mrm{YM}}$ itself, but rather the law it induces on the space of gauge orbits. For more discussion, see \cite{CC23, CCHS2022}.

In gauge theory, one often wants to gauge-fix, which amounts to selecting a unique representative of each gauge orbit. On the lattice, this may be done as follows. As mentioned in Section \ref{section:differential-forms-in-gauge-theory}, given a spanning tree $T$ of $\Lambda$, for any lattice gauge field $U$, there exists a gauge transformation $g$ such that $U^g_e = \groupid$ for all edges $e \in T$. Indeed, $g$ can be constructed by ``integrating the configuration $U$ along $T$" as follows. Fix a root $x_0$ of $T$, and for all $x \in T$, let $x_0, \ldots, x_n = x$ be the unique path in $T$ from $x_0$ to $x$. The condition
\begin{equs}
g(x_{n-1}) U(x_{n-1}, x_n) g(x_n)^{-1} = U^g_{(x_{n-1}, x_n)} = \groupid 
\end{equs}
then imposes that
\begin{equs}
g(x_n) = g(x_{n-1}) U(x_{n-1}, x_n),
\end{equs}
which implies that
\begin{equs}
g(x) = U(x_0, x_1) \cdots U(x_{n-1}, x_n).
\end{equs}
Indeed, one may verify that with this choice of $g$, $U^g_e = \groupid$ for all $e \in T$. In this way, one may associate to each gauge field $U$ a gauge-equivalent field $\tilde{U}$ which is identity on $T$. The choice of $\tilde{U}$ is unique once one fixes the root $x_0$ of $T$.

A natural question that arises is whether control of plaquette variables leads to control of edge variables. I.e., if we know that $U_p$ is close to $\groupid$ for all $p$, does this imply the same for the $U_e$? In general, this is manifestly not true -- start with the all-identity configuration $U_0 \equiv \groupid$, and then take an arbitrary gauge transformation $U_0^g$. All plaquette variables of $U_0^g$ are still the identity, however, the edge variables may be anything. One benefit of gauge-fixing is that it does ensure that one can go from control of plaquette variables to control of edge variables. See e.g. \cite[Lemma 10.2]{chatterjee2016leading} for an example result which shows that if one further assumes that $U_e = \groupid$ for all edges $e$ in a certain spanning tree $T$, then the action $\sum_{p \in C^2(\Lambda)} \|\groupid - U_p\|^2$ does control $\|\groupid - U_e\|$ for all $e$. Another result of this flavor is \cite[Theorem 4.5]{Chevyrev2019}. Combining this discussion with the previous one on gauge-fixing, we have that if all plaquette variables $U_p$ are close to $\groupid$, then there exists a gauge transformation $g$ such that all edge variables $U^g_e$ are close to $\groupid$.

\begin{remark}
As mentioned just before \cite[Lemma 10.2]{chatterjee2016leading}, one should think of a result of this type as a lattice version of a nonlinear Sobolev inequality. Indeed, the continuum analog of such a result would be the claim that for a $1$-form $A$, the Yang--Mills action $\sym(A)$ \eqref{eq:sym} controls $L^p$ norms of $A$. This result goes back to the classic work of Uhlenbeck \cite{Uhl1982}, see also the book \cite{Wehr2004}. One should think of $\sym(A)$ as a nonlinear version of the $H^1$ norm of $A$, and thus such a result can be thought of as a nonlinear Sobolev inequality. Of course, in the continuum it is also essential to gauge-fix, otherwise Uhlenbeck's result would not be true.
\end{remark}

With this discussion in hand, next we heuristically discuss how Gaussian theories appear from lattice gauge theories. One natural approach towards constructing continuum Yang--Mills theories is by taking a scaling limit of lattice gauge theories. Given a mesh size $\varep > 0$, we would now consider lattice gauge theories defined on boxes $\Lambda_\varep \sse \varep \Z^n$ of mesh size $\varep$. As discussed in \cite[Section 3]{chatterjee2019yang} it is generally expected that one must also vary the inverse temperature $\beta$ according to the mesh size like
\begin{equs}
\beta_\varep \sim \begin{cases} \varep^{n-4} & n = 2, 3 \\
\log (1/\varep) & n = 4 \end{cases}.
\end{equs}
In particular, $\beta_\varep \toinf$ as $\varep \downarrow 0$. From the equivalent definition \eqref{eq:lgt-alternative}, it is clear that the density of $\mu_{\Lambda, \beta}$ is maximized when $U_p = I$ (the identity matrix) for all plaquettes $p \in C^2(\Lambda)$. Moreover, as $\beta \toinf$, we would perhaps expect that the typical plaquette variable $U_p$ satisfies $\|I - U_p\| \lesssim \beta^{-\frac{1}{2}}$. If this is the case, then one might further hope that after gauge-fixing, each edge variable $U_e$ is close to the identity. Once we know this, we can hope to realize $U_e = \exp(A_e)$, where $A_e \in \frkg$ and $\exp \colon \frkg \ra \liegroup$ is the exponential map from the Lie algebra to the Lie group. Given this, one has that
\begin{equs}\label{eq:plaquette-first-order-expansion}
U_p \approx I + A_{e_1} + A_{e_2} + A_{e_3} + A_{e_4} + \text{l.o.t.},
\end{equs}
and thus to leading order, the action $\exp(-\frac{\beta}{2}\|I - U_p\|^2)$ would look quadratic in $(A_e, e \in C^1(\Lambda))$.

The outcome of this heuristic discussion is that as $\beta \toinf$, we would expect to be able to approximate lattice gauge theory by a lattice $\frkg$-valued Gaussian $1$-form model. We will discuss such models in detail in Section \ref{sec:latticeforms}. Naturally, this approximation should only hold on boxes which are small enough relative to $\beta^{-1}$ (though the precise quantitative dependence of the size of the box on $\beta^{-1}$ is less clear), since $\beta^{-1}$ governs the size of the approximation error coming from the lower order terms in \eqref{eq:plaquette-first-order-expansion}. Since continuum Yang--Mills theories should arise as scaling limits of lattice gauge theories (moreover with the inverse temperature diverging in the mesh size), this means that we would expect continuum Yang--Mills theories to be approximated by some version of a $\frkg$-valued $1$-form $\gff$ at small scales. 

The approximation of large-$\beta$ lattice gauge theories by Gaussian models has been made precise in several settings. Chatterjee \cite{chatterjee2016leading} has shown that a Gaussian model gives the leading term of the free energy $\frac{1}{|\Lambda|}\log Z_{\Lambda, \beta}$, essentially by carrying out the approximation procedure that we sketched. In more recent work \cite{Chatterjee2024}, he has been able to show that a certain scaling limit of $\mrm{SU}(2)$ Yang--Mills-Higgs theory is Gaussian. The benefit of the additional Higgs field is that it automatically forces each edge variable $U_e$ to be close to the identity, and so there is no need to gauge fix. 


\section{Chern-Simons action} \label{sec:chernsimon}


In this section, we specialize to three dimensions, i.e. $n = 3$. We discuss how at a formal level, the Abelian Chern-Simons theory gives rise to a link invariant. The Abelian Chern-Simons action is given by
\[ \scs(A) := \int_{\R^3} A \wedge dA = (A, \curl A),\]
where here $A$ is an $\R$-valued $1$-form on $\R^3$. Let $\beta \in \R$, and consider the formal measure
\begin{equs}\label{eq:CS}
Z^{-1} \exp\bigg(\frac{\icomplex \beta}{4\pi} \scs(A)\bigg) dA
\end{equs}
Let $J$ be a divergence-free vector field on $\R^3$. For notational convenience, let $\Sigma$ be such that
\[ \frac{\icomplex \beta}{4\pi} \scs(A) = \frac{1}{2}(A, \Sigma^{-1} A), \quad \text{i.e.}\quad \Sigma = -\icomplex \frac{2\pi}{\beta} \curl^{-1}. \] 
(Here, we are considering $\Sigma$ restricted to the space of divergence free vector fields, so that $\curl$ is invertible.) Then a formal Gaussian computation shows that
\[\begin{split}
Z^{-1} \int \exp\bigg(-\frac{1}{2} (A, \Sigma^{-1} A)\bigg) \exp(\icomplex(A, J)) dA &=  \exp\bigg(-\frac{1}{2}(J, \Sigma J)\bigg) \\
&= \exp\bigg(\frac{\icomplex \pi}{\beta} (J, \curl^{-1} J)\bigg).
\end{split}\]
The quadratic form $(J, \curl^{-1} J)$ is intimately related to the linking number from topology. To see this, first we give an explicit integral formula for $(J, \curl^{-1} J)$. Let $B$ be the solution to
\[ \curl B = J, ~~ \codif B = 0. \]
Note then $(J, \curl^{-1} J) = (J, B)$. To solve for $B$, we use that $\curl^2 B = -\Delta B$, so that it suffices to solve for
\[ -\Delta B = \curl J. \]
On $\R^3$, we have the explicit solution
\[ B = (-\Delta)^{-1} \curl J = \frac{1}{4\pi} \int_{\R^3} dy \frac{1}{|\cdot - y|} (\curl J)(y).\]
Coordinate-wise, this reads
\[ B_i(x) = \frac{1}{4\pi} \int_{\R^3} \frac{1}{|x - y|} \varep_{ijk} \ptl_j J_k(y) = \frac{1}{4\pi} \int_{\R^3} \varep_{ijk} \frac{x_j - y_j}{|x-y|^3} J_k(y), \]
where we applied integration by parts in the second equality. We thus obtain
\[(J, \curl^{-1} J) = (J, B) = \frac{1}{4\pi} \int_{\R^3} dx \int_{\R^3} dy \varep_{ijk}  \frac{x_j - y_j}{|x-y|^3} J_i(x) J_k (y).  \]
More generally, given divergence-free vector fields $J, \tilde{J}$, we have that
\begin{equs}
(J, \curl^{-1} \tilde{J}) = \frac{1}{4\pi} \int_{\R^3} dx \int_{\R^3} dy \varep_{ijk}  \frac{x_j - y_j}{|x-y|^3} J_i(x) \tilde{J}_k (y),
\end{equs}
and this is the familiar Gauss linking integral for $J, \tilde{J}$. 

In particular, if we take $J, \tilde{J}$ to be loops $\gamma, \tilde{\gamma}$ (which may be thought of as distributional divergence-free vector fields), then we obtain integral 
\begin{equs}
\frac{1}{4\pi} \int_0^1 dt \int_0^1 ds \varep_{ijk} \frac{\gamma_j(t) - \tilde{\gamma}_j(s)}{|\gamma(t) - \tilde{\gamma}(s)|^3} \gamma_i'(t) \tilde{\gamma}_k'(s).
\end{equs}
This may be written in the more familiar form
\begin{equs}
-\frac{1}{4\pi} \int_0^1 dt \int_0^1 ds \frac{\mrm{det}(\gamma'(t), \tilde{\gamma}'(s), \gamma(t) - \tilde{\gamma}(s))}{|\gamma(t) - \tilde{\gamma}(s)|^3},
\end{equs}
which is (up to a sign) the Gauss linking integral of $\gamma, \tilde{\gamma}$.

The take home message is that, at least at a formal level, topological invariants may be obtained by evaluating expectations of certain observables in Chern--Simons theories. The reason why one might expect this is because the Chern--Simons action is independent of the metric $g$ of the base manifold (as it only involves the wedge product and not the Hodge star), and thus the resulting Chern--Simons theories should encode only topological (rather than geometric) information of the base manifold. This idea was introduced in \cite{Witten1989}, who related non-Abelian Chern-Simons theories to the Jones polynomial, which is a knot invariant. This sparked an enormous amount of activity in the wider mathematics literature, which we will not attempt to survey here. Rather, we cite here various papers which have tried to make rigorous probabilistic sense of Chern--Simons theories:~\cite{FK1989, AS1997, Hahn2004, Hahn2005, Hahn2008, Weits2024}. For a physics introduction to Chern--Simons, see the lectures notes by Moore \cite{MooreCSNotes}. We also remark that it may be possible to make rigoroous probabilistic sense of Abelian Chern--Simons using a complex Langevin equation -- see \cite{WZ1993}.

The immediate barrier to viewing Abelian Chern--Simons \eqref{eq:CS} as the law of a fractional Gaussian form is the fact that $\curl$ is not positive semidefinite -- in fact, its spectrum is symmetric about $0$, see \cite{Bar2019}. Thus, even if we change $\beta$ to be imaginary, so that $\icomplex \beta$ is real, \eqref{eq:CS} does not give the law of a fractional Gaussian form. One possible thing to do is to restrict (say) to the span of positive eigenspaces of $\curl$ and take $\beta = \icomplex \alpha$ with $\alpha > 0$, in which case $\icomplex \beta \curl$ is negative definite on this subspace, and thus defines the law of some fractional Gaussian form.

\section{Gauge transformations}\label{sec:gaugetransformations}


In this section, we discuss gauge transformations of fractional Gaussian $1$-forms. By ``gauge transformation", we mean something slightly different than the true definition \eqref{eq:gauge-transformation}. Rather, for most of this paper, a gauge transformation is a function $\lambda \colon \manifold \ra \frkg$ (recall from Section \ref{section:non-abelian-theory} that $\frkg$ is the Lie algebra of some Lie group $\liegroup$) which acts on $\frkg$-valued $1$-forms $A$ by
\begin{equs}
A \mapsto A + d\lambda.
\end{equs}
Thus, two $1$-forms $A, A'$ are gauge-equivalent if they differ by a gradient. We think of this modified version of gauge symmetry as the one that arises in the small-scale limit of Yang--Mills where essentially everything is Abelian and Gaussian. Recall from Section \ref{section:lattice-gauge-theory} the discussion of how Gaussian theories arise from lattice gauge theories by taking $\beta \toinf$, and in particular, recall the expansion \eqref{eq:plaquette-first-order-expansion}. The lower order terms which are hidden in this expansion reflect the non-commutativity of $\liegroup$ and $\frkg$. By throwing these terms away, we essentially put ourselves into an Abelian setting. Because of this, for the rest of this section, we will take $A$ to just be an $\R$-valued $1$-form.

Gauge-fixing is a procedure that associates to any given $1$-form $A$ a gauge-equivalent $1$-form satisfying some pre-specified condition. We call such conditions ``gauge conditions". Next, we define our two main examples of gauge conditions.

\begin{definition}[Coulomb gauge]
Let $A$ be a $1$-form. We say that $A$ is in the Coulomb gauge if $d^* A = 0$, or equivalently, if $\ptl_j A_j = 0$. In other words, $A$ is in the Coulomb gauge if it is divergence-free.
\end{definition}

\begin{definition}[Axial gauge]
Let $A$ be a $1$-form and let $j \in [n]$. We say that $A$ is in the axial gauge in direction $j$ if $A_j = 0$.
\end{definition}


For each of these conditions, there exists $\lambda$ such that $A + d\lambda$ satisfies the condition. For instance, for the Coulomb gauge, we may set $\lambda = - d^* (-\Delta)^{-1} A$, so that $A + d\lambda = A - dd^* (-\Delta)^{-1} A = A - EA$. Here, $E$ is as in Definition \ref{def:E-E-star-projections}. By Proposition \ref{prop:hodge-decomp}, we have that $A - EA = E^* A$, where $E^*$ is precisely the projection of $A$ onto the space of divergence-free 1-forms. This shows that indeed $d^* (A + d\lambda) = 0$. This discussion all works if $A$ is an $\fgf_s^1(\R^n)$ since $d^* (-\Delta)^{-1} A$ is well-defined. In particular, we see that if we put an $\fgf_s^1(\R^n)$ in the Coulomb gauge, the resulting fractional Gaussian form is precisely an $\fgf_s^1(\R^n)_{\codif = 0}$. The covariance kernel of $\fgf_1^1(\R^n)_{\codif = 0}$ was previously worked out in Proposition \ref{prop:1-form-gff-projection-correlation-functions}, and thus this also gives the covariance kernel for a $\fgf_1^1(\R^n)$ in the Coulomb gauge.

Next, let us consider the axial gauge condition. Without loss of generality, fix the direction $n$. We seek to find a function $\lambda : \R^n \ra \R$ such that
\begin{equs}
A_n + \ptl_n \lambda = (A + d\lambda)_n = 0.
\end{equs}
This is a simple ODE for $\lambda$, which may be integrated to obtain
\begin{equs}
\lambda(x) - \lambda(x', 0) = -\int_0^{x_n} A_n(x', u) du.
\end{equs}
Here and in the following, given $x \in \R^n$ we shall write $x' = (x_1, \ldots, x_{n-1}) \in \R^{n-1}$. This shows that
\begin{equs}
\lambda(x) = \lambda(x', 0) - \int_0^{x_n} A_n(x', u) du.
\end{equs}
In particular, $\lambda$ is not uniquely specified, as it may take arbitrary values on the $x_n = 0$ hyperplane. This implies that the axial gauge is not a complete gauge condition -- by changing the value of $\lambda$ on the $x_n = 0$ hyperplane, we may obtain different gauge-equivalent 1-forms which are in the axial gauge.

More importantly, the line integral in the equation for $\lambda$ becomes problematic if now $A$ is a $\fgf_s^1(\R^n)$. Depending on the value of $s$ (which affects the regularity of $A$), line integrals of $A$ may not be well-defined. For instance, when $s = 1$ so that $A$ is a $\gff^1(\R^n)$, line integrals of $A$ are well-defined when $n = 2$, but not when $n \geq 3$. On the other hand, since line integrals of $A$ are linear in $A$, we can hope to define them as distributions. We discuss this next.

We seek to define
\begin{equs}
\lambda(x) = -\int_0^{x_n} A_n(x', u) du
\end{equs}
as a distribution. Let $\phi \in \schwartz$. Formally, by exchanging integration, we would expect that
\begin{equs}
(\lambda, \phi) &= \int_{\R^n} dx \lambda(x) \phi(x) \\
&= - \int_{\R^n} dx \phi(x)  \int_0^{x_n} du A_n(x', u) \\
&= - \int_{\R^{n-1}} dx' \int_\R du A_n(x', u) \bigg(\ind(u \geq 0) \int_u^\infty dx_n \phi(x', x_n) - \ind(u < 0) \int_{-\infty}^u dx_n \phi(x', x_n)\bigg).
\end{equs}
Thus, we can simply try to make the definition
\begin{equs}
(\lambda, \phi) := (A_n, \psi_\phi),
\end{equs}
where
\begin{equs}\label{eq:psi-phi}
\psi_\phi(x', u) := - \bigg(\ind(u \geq 0) \int_u^\infty dx_n \phi(x', x_n) - \ind(u < 0) \int_{-\infty}^u dx_n \phi(x', x_n)\bigg).
\end{equs}
Observe that if $\phi$ satisfies
\begin{equs}\label{eq:phi-vertical-integral-zero}
\int_\R dx_n \phi(x', x_n) = 0 \text{ for all $x' \in \R^{n-1}$},
\end{equs}
then in fact $\psi_\phi \in \schwartz$. The only worry is the possible discontinuity of $\psi_\phi$ at $u = 0$, but the above condition ensures that $\psi_\phi$ is indeed continuous. Now when $n \geq 3$, one may check that $\schwartz \sse \dot{H}^{-1}(\R^n)$. Recall that $A_n$ is a $\gff(\R^n)$, which in particular gives a stochastic process $((A_n, \phi), \phi \in \dot{H}^{-1}(\R^n))$ by Remark \ref{remark:gaussian-hilbert-space}.

\begin{example}
Examples of $\phi \in \schwartz$ satisfying \eqref{eq:phi-vertical-integral-zero} include $\phi(x', x_n) = \psi(x') \eta(x_n)$, where $\psi \in \schwartz(\R^{n-1})$ and $\eta \in \schwartz(\R)$ is mean zero.
\end{example}

Thus when $n \geq 3$, we may define $\lambda$ as a stochastic process $(\lambda, \phi)$ indexed by $\phi \in \schwartz$ satisfying \eqref{eq:phi-vertical-integral-zero}. Having defined $\lambda$, we can define $\ptl_j \lambda$ as the stochastic process 
\begin{equs}
(\ptl_j \lambda, \phi) := -(\lambda, \ptl_j \phi),
\end{equs}
where for $1 \leq j \leq n-1$, the index set is again the set of $\phi \in \schwartz$ satisfying \eqref{eq:phi-vertical-integral-zero} (note if $\phi$ satisfies \eqref{eq:phi-vertical-integral-zero}, then so does $\ptl_j \phi$). When $j = n$, $\ptl_n \phi$ automatically satisfies \eqref{eq:phi-vertical-integral-zero} because $\phi$ is Schwartz and so it decays at infinity. Thus when $j = n$, the stochastic process $(\ptl_n \lambda, \phi)$ may be indexed by $\phi \in \schwartz$. With this discussion at hand, we may now define the $\gff^1(\R^n)$ in the axial gauge.

\begin{definition}[$\gff^1(\R^n)$ in the axial gauge]\label{def:axial-gauge}
Let $n \geq 3$. Let $A$ be a $\gff^1(\R^n)$. We define the stochastic process (note the implicit summation over $j$):
\begin{equs}
(\tilde{A}, \phi) := (A, \phi) + (\ptl_j \lambda, \phi_j),
\end{equs}
indexed by $\phi \in \Omega^1 \schwartz$ such that $\phi_1, \ldots, \phi_{n-1}$ satisfies \eqref{eq:phi-vertical-integral-zero}. We say that $\tilde{A}$ is an instance of $\gff^1(\R^n)$ in the axial gauge.
\end{definition}

Observe that for $\phi \in \Omega^1 \schwartz$ of the form $\phi = (0, \ldots, 0, \phi_n)$, we have that
\begin{equs}
(\tilde{A}, \phi) = (A_n, \phi_n) + (\ptl_n \lambda, \phi_n) = (A_n, \phi_n) - (\lambda, \ptl_n \phi_n) = (A_n, \phi_n) - (A_n, \phi_n) = 0.
\end{equs}
Thus $\tilde{A}$ is indeed in the axial gauge in the direction $n$.

\begin{remark}
In Fourier space, adding a gradient to $A$ amounts to changing each Fourier mode $\widehat{A}(p)$ by a multiple of $p$. In order to set an axial gauge, we thus want that
\begin{equs}
(\widehat{A}(p) + z_p p)_n = 0 \text{ for all $p \in \R^n$.}
\end{equs}
For $p$ such that $p_n \neq 0$, we may solve for $z_p$:
\begin{equs}
z_p = - \frac{\widehat{A}_n(p)}{p_n}.
\end{equs}
Thus, applying Fourier inversion, we informally have that $A$ in the axial gauge is given by
\begin{equs}
\tilde{A}(x) = \int_{\R^n} dp \bigg(\widehat{A}(p) - \frac{\widehat{A}_n(p)}{p_n} p \bigg) e^{\icomplex p \cdot x},
\end{equs}
and the function $\lambda$ is given by:
\begin{equs}
\lambda(x) = -\frac{1}{\icomplex} \int_{\R^n} dp \frac{\widehat{A}_n(p)}{p_n} e^{\icomplex p \cdot x}.
\end{equs}
However, this integral for $\lambda$ is problematic because of the $p_n^{-1}$ factor which is singular on the entire $p_n = 0$ hyperplane. This can be made sense of if we test against $\phi \in \schwartz$ satisfying the following condition on its Fourier transform:
\begin{equs}
\widehat{\phi}(p) = 0 \text{ for all $p \in \R^n$ such that $p_n = 0$.}
\end{equs}
Note that this condition is precisely \eqref{eq:phi-vertical-integral-zero} written in Fourier space.
\end{remark}

\begin{remark}
Recall from \eqref{remark:fgf-fourier-series} the explicit Fourier series representation of a $\gff^1(\T^n)$:
\begin{equs}
A = \sum_{0 \neq \alpha \in \Z^n} \frac{Z(\alpha)}{|\alpha|} \e_\alpha,
\end{equs}
To try to put $A$ in the Coulomb gauge, we can add a multiple of $\alpha$ to each Fourier mode whenever $\alpha_n \neq 0$:
\begin{equs}
\tilde{A} &= \sum_{\substack{0 \neq \alpha \in \Z^n \\ \alpha_n \neq 0}} \frac{1}{|\alpha|} \bigg(Z(\alpha) - \frac{Z_n(\alpha)}{\alpha_n} \alpha\bigg)\e_\alpha + \sum_{\substack{0 \neq \alpha \in \Z^n \\ \alpha_n = 0}} \frac{1}{|\alpha|} Z_n(\alpha) \e_\alpha .
\end{equs}
The first term is indeed in the axial gauge, i.e. its $n$th coordinate is zero. However, the second term in general is not in the axial gauge. On the other hand, the second term is constant in the $x_n$ direction. 

The failure to put a general $1$-form on the torus in the axial gauge is a reflection of the following basic obstruction. If we change $A$ by adding a gradient, there is no way to change the line integral of $A$ along a straight line parallel to one of the coordinate axes. 
\end{remark}

Next, we discuss the covariance kernel of the $\gff^1(\R^n)$ in the axial gauge. It turns out that the covariance kernel can only be interpreted in a limiting sense, similar to the case of projected white noise (Proposition \ref{prop:projected-white-noise-correlation-functions}). For $\varep > 0$, $1 \leq i, j \leq n-1$, define the kernel (here $x = (x', x_n)$, $y = (y', y_n)$)
\begin{equation}\label{eq:axial-gauge-covariance-kernel}
\begin{aligned}
G^{\mrm{ax}, \varep}_{ij}(x, y) :=~ &\delta_{ij} G^1(x, y) + \frac{\delta_{ij}}{n}  \delta(x' - y') \big(\ind(x_n, y_n > 0) + \ind(x_n, y_n < 0)\big) \min(|x_n|, |y_n|) \\
&+ \frac{\delta_{ij}}{\omega_n} \int_\R du \int_\R dv \ind(|(x', u) - (y', v)| > \varep) |(x', u) - (y', v)|^{-n}  ~\times \\
&\quad \quad \quad \big(\ind(0 < u < x_n) - \ind(x_n < u < 0)\big) \big(\ind(0 < v < y_n) - \ind(y_n < v < 0)\big) \\
&- \frac{n}{\omega_n} \int_\R du \int_\R dv \ind(|(x', u) - (y', v)| > \varep) \frac{(x'_i - y'_i)(x'_j -y'_j)}{|(x', u) - (y', v)|^{n+2}}  ~\times \\
&\quad \quad \quad \big(\ind(0 < u < x_n) - \ind(x_n < u < 0)\big)\big(\ind(0 < v < y_n) - \ind(y_n < v < 0)\big)
\end{aligned}
\end{equation}

\begin{prop}\label{prop:axial-gauge-correlation-function}
Let $\tilde{A}$ be a $\gff^1(\R^n)$ in the axial gauge. Let $\phi = (\phi_1, \ldots, \phi_n)$, $\eta = (\eta_1, \ldots, \eta_n)$ be such that $\phi_j, \eta_j \in \schwartz(\R^n)$ satisfy \eqref{eq:phi-vertical-integral-zero} for $1 \leq j \leq n-1$, and $\phi_n, \eta_n \in \schwartz(\R^n)$. We have that
\begin{equs}
\Cov((\tilde{A}, \phi), (\tilde{A}, \eta)) = \lim_{\varep \downarrow 0} \sum_{i, j=1}^{n-1} \int_{\R^n} dx \int_{\R^n} dy \phi_i(x) G^{\mrm{ax}, \varep}_{ij}(x, y) \eta_j(y).
\end{equs}
\end{prop}
\begin{proof}
Let $\tilde{A}$ be constructed from $A$, a $\gff^1(\R^n)$, as in Definition \ref{def:axial-gauge}. We have that
\begin{equs}
(\tilde{A}, \phi) &= (A, \phi) - (\lambda, \psi_{\ptl_j \phi_j}) \\
&= \sum_{j=1}^{n-1} \big( (A_j, \phi_j) - (A_n, \ptl_j \psi_{\phi_j})\big) + (A_n, \phi_n) - (A_n, \phi_n) \\
&= \sum_{j=1}^{n-1} \big( (A_j, \phi_j) - (A_n, \ptl_j \psi_{\phi_j})\big).
\end{equs}
In the second identity, we used that 
\begin{equs}
\psi_{\sum_{j=1}^n \ptl_j \phi_j} = \sum_{j=1}^{n-1} \ptl_j \psi_{\phi_j} + \psi_{\ptl_n \phi_n} \quad \text{and} \quad \psi_{\ptl_n \phi_n} = \phi_n,
\end{equs}
which both follow directly from the definition \eqref{eq:psi-phi}. Since $A_n$ is independent of $(A_1, \ldots, A_{n-1})$, we have that 
\begin{equs}
\Cov((\tilde{A}, \phi), (\tilde{A}, \eta)) &= \sum_{j=1}^{n-1} \Cov((A_j, \phi_j), (A_j, \eta_j)) + \sum_{i, j=1}^{n-1} \Cov((A_n, \ptl_i \psi_{\phi_i}), (A_n, \ptl_j \psi_{\eta_j})) \\
&= \sum_{j=1}^{n-1} (\phi_j, \eta_j)_{\dot{H}^{-1}(\R^n)} + \sum_{i, j=1}^{n-1} (\ptl_i \psi_{\phi_i}, \ptl_j \psi_{\eta_j})_{\dot{H}^{-1}(\R^n)}.
\end{equs}
The first term gives the $\delta_{ij}G^1(x, y)$ term in the definition \eqref{eq:axial-gauge-covariance-kernel} of $G^{\mrm{ax}, \varep}_{ij}$. 
We thus focus on the second term. For fixed $1 \leq i, j \leq n-1$, we have that
\begin{equs}
(\ptl_i \psi_{\phi_i}, \ptl_j \psi_{\eta_j})_{\dot{H}^{-1}(\R^n)} &= -(\psi_{\phi_i}, \ptl_i \ptl_j \psi_{\eta_j})_{\dot{H}^{-1}(\R^n)} \\
&= - (\psi_{\phi_i}, \ptl_i \ptl_j (-\Delta)^{-1} \psi_{\eta_j}).
\end{equs}
By Lemma \ref{lemma:iterated-riesz-kernel}, we have that 
\begin{equs}
- (\psi_{\phi_i}, \ptl_i \ptl_j (-\Delta)^{-1} \psi_{\eta_j}) = \frac{\delta_{ij}}{n}(\psi_{\phi_i}, \psi_{\eta_j}) &+ \frac{\delta_{ij}}{\omega_n} \lim_{\varep \downarrow 0} \int \int_{|x-y| > \varep} \psi_{\phi_i}(x) \psi_{\eta_j}(y)|x-y|^{-n} dx dy ~\\
&- \frac{n}{\omega_n} \lim_{\varep \downarrow 0} \int \int_{|x-y| > \varep} \psi_{\phi_i}(x) \psi_{\eta_j}(y) \frac{(x_i - y_i)(x_j - y_j)}{|x-y|^{n+2}} dx dy. 
\end{equs}
The desired result now follows by applying Lemma \ref{lemma:psi-phi-l2-norm} to the first term and Lemma \ref{lemma:psi-phi-kernel} to the last two terms, in order to turn each of the three terms into integrals directly involving $\phi_i, \eta_j$.
\end{proof}

\section{Subspace restrictions}\label{sec:subspacerestrictions}

In this section, we discuss restrictions of various fractional Gaussian forms on $\R^n$ to subspaces (which for us will always be taken to be $(n-1)$-dimensional hyperplanes). First, we review the scalar case, and then we go on to discuss the more general case of $k$-forms.

\subsection{Review of the scalar case}

In this subsection, we discuss restrictions of (scalar) fractional Gaussian fields to subspaces. We first review the discussion in \cite[Section 7]{lodhia2016fractional}, and then we will present an alternative point of view, based on the fractional stochastic heat equation. Such a viewpoint has previously been used in the probabilistic approach to Liouville Conformal Field Theory, where the GFF on the unit disk in 2D is seen as the solution to a fractional stochastic heat equation on the 1D torus $\torus$. See \cite[Section 5.1]{GKR2024}.

Fix $x_0 \in \R$. Given $\phi \in \schwartz(\R^{n-1})$, we may define $\phi^{\uparrow, x_0} \in \schwartz'(\R^n)$ by
\begin{equs}
(\phi^{\uparrow, x_0}, \psi) := \int_{\R^{n-1}} \phi(x') \psi(x', x_0) dx' \text{ for all $\psi \in \schwartz(\R^n)$.}
\end{equs}
We first show the following relation between negative Sobolev norms of $\phi$ and $\phi^{\uparrow, x_0}$.

\begin{lemma}\label{lemma:restriction-sobolev-bound}
Let $x_0 \in \R$, $s > \frac{1}{2}$. For all $\phi \in \schwartz(\R^{n-1}) \cap \dot{H}^{\frac{1}{2}-s}(\R^{n-1})$, we have that $\phi^{\uparrow, x_0} \in \dot{H}^{-s}(\R^n)$, and moreover
\begin{equs}
\big\|\phi^{\uparrow, x_0}\big\|_{\dot{H}^{-s}(\R^n)}^2 = \frac{1}{2\pi} \int_\R du (1+u^2)^{-s} \big\| \phi\big\|_{\dot{H}^{\frac{1}{2}-s}(\R^{n-1})}^2.
\end{equs}
\end{lemma}
\begin{proof}
To show that $\phi^{\uparrow, x_0} \in \dot{H}^{-s}(\R^n)$, it suffices to produce a Cauchy sequence $\{\psi_k\}_{k \geq 1} \sse \dot{H}^{-s}(\R^n)$ such that $\psi_k \ra \phi^{\uparrow, x_0}$ in $\mc{S}'(\R^n)$. We first record the following preliminary calculation that will be used later.
\begin{equs}\label{eq:restriction-sobolev-bound-intermediate}
\hspace{-5mm}\int_{\R^{n-1}} dp' |\widehat{\phi}(p')|^2 \int_{\R} dp_n (|p'|^2 + p_n^2)^{-s} = \int_\R du (1+u^2)^{-s} \int_{\R^{n-1}} dp' |\widehat{\phi}(p')|^2 |p'|^{1 - 2s} \lesssim_s \big\|\phi\big\|_{\dot{H}^{\frac{1}{2}-s}(\R^{n-1})}.
\end{equs}
The identity follows from the change of variables formula:
\begin{equs}
\int_\R dp_1 (a + p_1^2)^{-s} = a^{\frac{1}{2}-s} \int_\R du (1 + u^2)^{-s}, ~~ \text{ for all $a \in \R$,}
\end{equs}
and the inequality follows from the fact that $\int_\R du (1+u^2)^{-s} < \infty$ if $s > \frac{1}{2}$.

Now, to produce the Cauchy sequence, let $\rho \in C^\infty_c(\R^n)$ be a non-negative compactly supported function which is 1 on a ball around the origin, and such that $\int_{\R^n} dx \rho(x) = 1$, i.e. $\rho$ is $L^1$-normalized. For $k \geq 1$, define the re-scaled version $\rho_k(x) := (2^k)^n \rho(2^k x)$, which is still $L^1$-normalized. Define $\psi_k := \rho_k * \phi^{\uparrow, x_0}$. We have that $\psi_k \in \schwartz(\R^n)$ and $\psi_k \ra \phi^{\uparrow, x_0}$ in $\schwartz'(\R^n)$. Next, we show that $\{\psi_k\}_{k \geq 1}$ is Cauchy in $\dot{H}^{-s}(\R^n)$. Towards this end, by direct calculation we have that
\begin{equs}
\widehat{\psi}_k(p) = e^{-\icomplex p_n x_0} (2\pi)^{\frac{n-1}{2}} \widehat{\phi}(p') \widehat{\rho}_k(p), ~~ p = (p', p_n) \in \R^n.
\end{equs}
We have that $\|\widehat{\rho}_k\|_{L^\infty} \lesssim \|\rho\|_{L^1} = 1$. Moreover, since $\rho_k \ra \delta_0$ as $k \toinf$, one may verify that $\widehat{\rho}_k \ra (2\pi)^{-\frac{n}{2}}$ pointwise as $k \toinf$. From this, the bound \eqref{eq:restriction-sobolev-bound-intermediate}, and the dominated convergence theorem, it follows that
\begin{equs}
\lim_{k \toinf} \int_{\R^n} dp |p|^{-2s} \big|\widehat{\psi}_k(p) - e^{-\icomplex p_n x_0} (2\pi)^{-\frac{1}{2}} \widehat{\phi}(p')\big|^2 = 0,
\end{equs}
which implies that the sequence $\{\psi_k\}_{k \geq 1}$ is Cauchy in $\dot{H}^{-s}(\R^n)$. Moreover, we have from the above display that
\begin{equs}
\big\|\phi^{\uparrow, x_0}\big\|_{\dot{H}^{-s(\R^n)}}^2 = \lim_{k \toinf} \big\|\psi_k\big\|_{\dot{H}^{-s}(\R^n)}^2 = \frac{1}{2\pi}  \int_{\R^n} dp |\widehat{\phi}(p')|^2 |p|^{-2s} = \frac{1}{2\pi}  \int_\R du (1 + u^2)^{-s} \big\|\phi\big\|_{\dot{H}^{\frac{1}{2}-s}(\R^{n-1})}^2,
\end{equs}
where in the last identity, we applied \eqref{eq:restriction-sobolev-bound-intermediate}.
\end{proof}

Using Lemma \ref{lemma:restriction-sobolev-bound} and Remark \ref{remark:gaussian-hilbert-space}, we may define the restriction of an $\fgf_s(\R^n)$ when $s > \frac{1}{2}$ as follows.

\begin{definition}[Restriction of fractional Gaussian fields]\label{def:restriction-fgf}
Let $A$ be a $\fgf_s(\R^n)$. Then the restriction of $A$ to the hyperplane $\{(x', x_0) : x' \in \R^{n-1}\} \sse \R^n$, which we denote by $A^{\downarrow, x_0}$, is the stochastic process $(A^{\downarrow, x_0}, \phi) : \phi \in \schwartz(\R^{n-1}) \cap \dot{H}^{\frac{1}{2}-s}(\R^{n-1})\}$ defined by 
\begin{equs}
(A^{\downarrow, x_0}, \phi) := (A, \phi^{\uparrow, x_0}).
\end{equs}
This is well-defined by Remark \ref{remark:gaussian-hilbert-space} and Lemma \ref{lemma:restriction-sobolev-bound}. Moreover, by the same lemma, we have that
\begin{equs}
\Var{(A^{\downarrow, x_0}, \phi)} = \big\|\phi^{\uparrow, x_0}\big\|_{\dot{H}^{-s}(\R^n)}^2 = \frac{1}{2\pi} \int_\R du (1+u^2)^{-s} \big\|\phi\big\|_{\dot{H}^{\frac{1}{2}-s}(\R^{n-1})}^2. 
\end{equs}
By the usual $L^2$ isometry arguments, we may then extend this to a stochastic process $((A^{\downarrow, x_0}, \phi) : \phi \in \dot{H}^{\frac{1}{2} - s}(\R^{n-1}))$, and we also see that $A^{\downarrow, x_0}$ is a (multiple of) $\fgf_{s - \frac{1}{2}}(\R^{n-1})$.
\end{definition}

\begin{remark}
The reason why the $s$ parameter decreases by $\frac{1}{2}$ upon restriction can be seen as follows. Given $A \sim \fgf_s(\R^n)$, the Hurst parameter $H = 2s - n$ governs the regularity of $A$. Now, if we restrict $A$ to an $(n-1)$-dimensional hyperplane, the regularity should not change. Denoting by $s^{\downarrow}$ the new $s$ parameter corresponding to the restriction of $A$, we see that
\begin{equs}
2s - n = 2 s^{\downarrow} - (n-1),
\end{equs}
which readily gives that $s^{\downarrow} = s - \frac{1}{2}$.

Another heuristic way to see how $s$ should transform is to note that at least in the simple range $0 < 2s < n$, the covariance kernel is
\begin{equs}
G^s(x, y) \sim |x-y|^{2s - n}, ~~ x, y \in \R^n.
\end{equs}
Now if we restrict $x, y$ to lie in the hyperplane $\{(x',x_0) : x' \in \R^{n-1}\}$, then the covariance kernel specializes to
\begin{equs}
G^{s}((x', x_0), (y', x_0)) = |x' - y'|^{2s-n} = |x'-y'|^{2s^\downarrow - (n-1)}.
\end{equs}
This leads to the same equation $2s - n = 2s^\downarrow - (n-1)$ for $s^{\downarrow}$ as before.
\end{remark}

We begin with some preliminary discussion on fractional stochastic heat equations. Our intention is not to give a comprehensive introduction to stochastic heat equations (for that, see e.g. the lecture notes \cite{HairerSPDEcourse}), and thus the reader unfamiliar with this topic may feel free to skip the following discussion.

We work on $\R^n$, and we view the first coordinate as the time direction, and the last $n-1$ coordinates as the spatial directions. Given a function $\phi : \R^n \ra \R$, we will write $\ptl_t \phi$ to denote the time-derivative, and $\Delta_x \phi$ to denote the Laplacian in only the spatial directions. We will write $\Delta = \ptl_{tt} + \Delta_x$ to denote the Laplacian on $\R^n$.

Let $\xi$ be a white noise\footnote{In the following discussion, we will use $\xi$ rather than $\noise$ because this is what is typically used in the stochastic PDE literature.} on $\R^n$. In stochastic PDE literature, $\xi$ is usually referred to as a space-time white noise, since we are viewing one of the directions of $\R^n$ as the time direction. A fractional stochastic heat equation is a stochastic PDE of the form
\begin{equs}\tag{FSHE}\label{eq:FSHE}
\ptl_t A = -(-\Delta_x)^{\frac{s}{2}} A + \xi.
\end{equs}
We will be primarily interested with stationary solutions $(A(t), t \in \R)$ to \eqref{eq:FSHE}, which are formally written 
\begin{equs}
A(t) = \int_{-\infty}^t e^{-(t-s)(-\Delta_x)^{\frac{s}{2}}} \xi(s) ds.
\end{equs}
More precisely, the above can be interpreted as a stochastic process $\big((A(t), \phi), \phi \in \schwartz(\R^{n-1})\big)$, defined as
\begin{equs}
(A(t), \phi) := \int_{-\infty}^t ds \bigg(\xi(s), e^{-(t-s)(-\Delta_x)^{\frac{s}{2}}} \phi\bigg) = \Big(\xi, F_t(\phi)\Big),
\end{equs}
where the middle term is still only formal, and the precise interpretation is given by the right-most term. Here, $F_t(\phi) \in L^2(\R^n)$ is defined as $F_t(u, x) := \big(e^{-(t-u)(-\Delta_x)^{\frac{s}{2}}} \phi\big)(x) \ind(u \leq t)$, for $u \in \R$, $x \in \R^{n-1}$. One may check in Fourier space that $F_t(\phi)$ is indeed in $L^2(\R^n)$, and thus $(\xi, F_t(\phi))$ is well-defined.

The following proposition shows that in the case $s = 1$, a stationary solution to \eqref{eq:FSHE} is in fact a GFF.

\begin{prop}\label{prop:fractional-she}
Let $A$ be a stationary solution to the $(n-1)$-dimensional fractional stochastic heat equation
\[ \ptl_t A = -(-\Delta_x)^{1/2} A + \xi. \]
When viewed as a distribution on $\R^n$, $A$ is a $\gff$.
\end{prop}
\begin{proof}
We start with the formal representation of a stationary solution
\[ A(t) = \int_{-\infty}^t e^{-(t-u)(-\Delta_x)^{1/2}} \xi(u) du. \]
Here, we write $\Delta_x$ to denote the Laplacian in the spatial variables. Let $\phi \in \Phi(\R^n)$ be a test function on $\R^n$. We can view $\phi$ as a function $\phi : \R \ra \mc{S}(\R^{n-1})$. We want to show that
\[ \mrm{Var}((A, \phi)) = (\phi, (-\Delta)^{-1} \phi) = (\phi, (-\ptl_{tt} - \Delta_x)^{-1} \phi).\]
Towards this end, write
\[\begin{split}
(A, \phi) &= \int_\R dt (A(t), \phi(t)) = \int_\R dt \int_{-\infty}^t du (e^{-(t-u)(-\Delta_x)^{1/2}} \xi(u), \phi(t))  \\
&= \int_\R dt \int_{-\infty}^t (\xi(u), e^{-(t-u) (-\Delta_x)^{1/2}} \phi(t)) \\
&= \int_\R du \bigg(\xi(u), \int_u^\infty dt e^{-(t-u)(-\Delta_x)^{1/2}} \phi(t) \bigg).
\end{split}\] 
Now, let $\psi = (-\Delta)^{-1} \phi = (-\ptl_{tt} -\Delta_x)^{-1} \phi$, so that $\phi = (-\ptl_{tt} -\Delta_x ) \psi$. Fix $u \in \R$. By integration by parts, we have that
\begin{equs}
-\int_u^\infty dt e^{-(t-u)(-\Delta_x)^{\frac{1}{2}}} \ptl_{tt} \psi(t) &= -\int_u^\infty dt (-\Delta_x)^{\frac{1}{2}} e^{-(t-u)(-\Delta_x)^{\frac{1}{2}}} \ptl_t \psi(t) + \ptl_u \psi(u) \\
&= \int_u^\infty dt \Delta_x e^{-(t-u)(-\Delta_x)^{\frac{1}{2}}} \psi(t) + (-\Delta_x)^{\frac{1}{2}}\psi(u) + \ptl_u \psi(u) .
\end{equs}
Inserting this identity and using that $\phi = (-\ptl_{tt} - \Delta_x) \psi$, we obtain
\begin{equs}
\int_u^\infty dt e^{-(t-u)(-\Delta_x)^{\frac{1}{2}}} \phi(t) &= (-\Delta_x)^{\frac{1}{2}} \psi(u) + \ptl_u \psi(u).
\end{equs}
From this, we obtain
\[ (A, \phi) = \int_\R du \big(\xi(u), (\ptl_u + (-\Delta_x)^{1/2}) \psi(u)\big). \]
It follows that
\[\begin{split}
\mrm{Var}((A, \phi)) &= \big((\ptl_u + (-\Delta_x)^{1/2}) \psi, (\ptl_u + (-\Delta_x)^{1/2}) \psi\big) \\
&= \big(\psi, (-\ptl_{uu}) \psi\big) + \big(\psi, (-\Delta_x) \psi\big) + 2 \big(\ptl_u \psi, (-\Delta_x)^{1/2} \psi\big).
\end{split}\]
Observe that
\[ 2 \big(\ptl_u \psi, (-\Delta_x)^{1/2} \psi\big) = \int_\R du \frac{d}{du} \big(\psi(u), (-\Delta_x)^{1/2} \psi(u)\big) = 0,  \]
since $\psi$ decays at infinity. We thus have that
\[ \mrm{Var}((A, \phi)) = (\psi, (-\ptl_{uu} - \Delta_x) \psi) = (\psi, (-\Delta) \psi).\]
The desired result now follows upon recalling that $\psi = (-\Delta)^{-1} \phi$.
\end{proof}

\begin{lemma}\label{lemma:stationary-fractional-she}
Let $s > 0$. Let $A$ be a stationary solution to the $(n-1)$-dimensional fractional stochastic heat equation
\[ \ptl_t A = - (-\Delta)^s A +  \xi. \]
Then for any $t \in \R$, $\sqrt{2} A(t) \sim \fgf_s(\R^{n-1})$.
\end{lemma}
\begin{proof}
We use the explicit formula for a stationary solution given by
\[ A(t) = \int_{-\infty}^t e^{-(t - u )(-\Delta)^s} \xi(u) du. \]
Let $\phi \in \Phi(\R^{n-1})$ be a test function. We have that
\[ (\sqrt{2} A(t), \phi) = \sqrt{2} \int_{-\infty}^t (e^{-(t-u)(-\Delta)^s} \xi(u), \phi) du = \sqrt{2} \int_{-\infty}^t (\xi(u), e^{-(t-u)(-\Delta)^s} \phi) du.\]
Thus
\[\begin{split}
\mrm{Var}((\sqrt{2} A(t), \phi)) &= 2 \int_{-\infty}^t (e^{-(t-u)(-\Delta)^s} \phi, e^{-(t-u)(-\Delta)^s} \phi) du \\
&= 2\bigg(\phi, \int_{-\infty}^t du e^{-2(t-u)(-\Delta)^s} \phi\bigg).
\end{split}\]
To finish, observe that
\[ \int_{-\infty}^t du e^{-2(t-u)(-\Delta)^s} = \frac{1}{2} (-\Delta)^{-s}. \qedhere \]
\end{proof}

To connect back to our previous definition of restriction (Definition \ref{def:restriction-fgf}), note in that definition, when $s = 1$, we have that the restriction $A^{\downarrow, x_0}$ of $A \sim \fgf_1(\R^n)$ satisfies
\begin{equs}
\Var{(A^{\downarrow, x_0}, \phi)} = \frac{1}{2 \pi} \int_\R du (1 + u^2)^{-1} \big\|\phi\big\|_{\dot{H}^{-\frac{1}{2}}(\R^{n-1})}^2.
\end{equs}
An explicit computation shows that $\int_{\R} du (1+u^2)^{-1} = \pi$, and thus we see that
\begin{equs}
\Var{(\sqrt{2} A^{\downarrow, x_0}, \phi)} = \big\|\phi\big\|_{\dot{H}^{-\frac{1}{2}}(\R^{n-1})}^2 = (\phi, (-\Delta)^{-\frac{1}{2}} \phi).
\end{equs}
This is consistent with with Proposition \ref{prop:fractional-she} and Lemma \ref{lemma:stationary-fractional-she}.

\subsection{Restrictions of fractional Gaussian forms}

Similar to before, fix $x_0 \in \R$. Given $\phi \in \Omega^k \schwartz(\R^{n-1})$, we may define $\phi^{\uparrow, x_0}$ by
\begin{equs}\label{eq:form-restriction-def}
(\phi^{\uparrow, x_0}, \psi) := \sum_{1 \leq i_1 < \cdots < i_k \leq n-1} \int_{\R^{n-1}} dx' \phi_{i_1 \cdots i_k} (x') \psi_{i_1 \cdots i_k}(x', x_0).
\end{equs}
Here, the indices $i_1, \ldots, i_k \in [n-1]$, which may also be thought of adjoing zeros to $\phi$ to form a tensor $\phi_{i_1 \cdots i_k}$ with $i_1, \ldots, i_k \in [n]$. By applying Lemma \ref{lemma:restriction-sobolev-bound} coordinatewise, we obtain the following result.

\begin{lemma}\label{lemma:schwartz-forms-restriction}
Let $0 \leq k \leq n$, $x_0 \in \R$, $s > \frac{1}{2}$. For all $\phi \in \Omega^k\schwartz(\R^{n-1}) \cap \dot{H}^{\frac{1}{2}-s}(\R^{n-1})$, we have that $\phi^{\uparrow, x_0} \in \Omega^k \dot{H}^{-s}(\R^n)$, and moreover
\begin{equs}
\|\phi^{\uparrow, x_0} \|_{\Omega^k\dot{H}^{-s}(\R^n)} \lesssim \|\phi\|_{\Omega^k \dot{H}^{\frac{1}{2}-s}(\R^{n-1})}.
\end{equs}
\end{lemma}
We may then define restrictions of fractional Gaussian $k$-forms.

\begin{definition}[Restriction of fractional Gaussian $k$-forms]
Let $A$ be a $\fgf_s^k(\R^n)$. Then the restriction of $A$ to the hyperplane $\{(x', x_0) : x' \in \R^{n-1}\} \sse \R^n$, which we denote by $A^{\downarrow, x_0}$, is the stochastic process $(A^{\downarrow, x_0}, \phi) : \phi \in \Omega^k \schwartz(\R^{n-1}) \cap \dot{H}^{\frac{1}{2}-s}(\R^{n-1})\}$ defined by 
\begin{equs}
(A^{\downarrow, x_0}, \phi) := (A, \phi^{\uparrow, x_0}).
\end{equs}
\end{definition}

Note that this definition is equivalent to restricting each of the independent $(A_{i_1, \ldots, i_k}, 1 \leq i_1 < \cdots < i_k \leq n-1)$ to the hyperplane. It follows that $A^{\downarrow, x_0}$ is a (multiple of a) $\fgf_{s-\frac{1}{2}}^{k-1}(\R^{n-1})$.

Thus far the restriction of fractional Gaussian forms is essentially the same as in the scalar case. Next, we discuss a more interesting example, where we restrict $\fgf_s^1(\R^n)_{d=0}$ or $\fgf_s^1(\R^n)_{\codif = 0}$. We will need the following lemma. Recall the kernels $K^{d = 0}_s, K^{\codif = 0}_s$ from Proposition \ref{prop:1-form-gff-projection-correlation-functions}. When $s = 1$, by we default we take $K^{d = 0}_1 = K^{d = 0}, K^{\codif = 0}_1 = K^{\codif = 0}$.

\begin{lemma}\label{lemma:restriction-projection-kernel}
Let $s > \frac{1}{2}$. Let $\phi, \psi \in \Omega^1\schwartz(\R^{n-1}) \cap \dot{H}^{\frac{1}{2}-s}(\R^{n-1})$. Then 
\begin{equs}
(\phi^{\uparrow, x_0}, E^* \psi^{\uparrow, x_0})_{\Omega^1 \dot{H}^{-s}(\R^n)} = \int_{\R^{n-1}} \int_{\R^{n-1}} dx' dy' \phi_i(x') \Big(K^{\codif = 0}_s\Big)_{ij}((x', x_0), (y', x_0)) \psi_j(y').
\end{equs}
The same result holds with $E^*$ replaced by $E$ and ``$\codif = 0$" replaced by ``$d = 0$" everywhere. To be clear, here the kernels $K^{d = 0}_s, K^{\codif = 0}_s$ are defined on $\R^n$, but we only sum over indices $i, j \in [n-1]$.
\end{lemma}
\begin{proof}
Observe that $(-\Delta)^{-s} : \dot{H}^{-s}(\R^n) \ra \dot{H}^s(\R^n)$ is an isometry (Lemma \ref{lemma:fractional-laplacian-isometry-fractional-sobolev-space}) and $E^*$ is bounded on $\Omega^1 \dot{H}^s(\R^n)$ (Lemma \ref{lemmma:projections-bounded-sobolev-space}). Combining this with Lemma \ref{lemma:schwartz-forms-restriction}, we have that $E^* (-\Delta)^{-s} \psi^{\uparrow, x_0} \in \Omega^1 \dot{H}^s(\R^n)$, and
\begin{equs}
(\phi^{\uparrow, x_0}, E^* \psi^{\uparrow, x_0})_{\dot{H}^{-s}(\R^n)} = (\phi^{\uparrow, x_0}, E^* (-\Delta)^{-s} \psi^{\uparrow, x_0}),
\end{equs}
where in the RHS, $(\cdot, \cdot)$ denotes the dual pairing $\Omega^1 \dot{H}^{-s}(\R^n) \times \Omega^1 \dot{H}^s(\R^n) \ra \R$. By general distribution theory arguments, if $E^* (-\Delta)^{-s} \psi^{\uparrow, x_0}$ is also a bounded continuous function, then we have that
\begin{equs}
(\phi^{\uparrow, x_0},  E^* (-\Delta)^{-s} \psi^{\uparrow, x_0}) = \int_{\R^{n-1}} dx' \phi_i (x') (E^* (-\Delta)^{-s} \psi^{\uparrow, x_0})_i(x', x_0).
\end{equs}
It remains to show that
\begin{equs}
(E^* (-\Delta)^{-s} \psi^{\uparrow, x_0})_i(x) = \int_{\R^{n-1}} dy' \Big(K^{\codif = 0}_s\Big)_{ij}(x, (y', x_0)) \psi_j(y'),
\end{equs}
and moreover this is bounded and continuous in $x \in \R^n$. First, observe that the RHS above is absolutely convergent for any $x \in \R^{n}$, which follows from the definition of $K^{\codif = 0}_s$ (Proposition \ref{prop:1-form-gff-projection-correlation-functions}). Moreover, the RHS defines a bounded continuous function of $x \in \R^{n}$, and thus it can be interpreted as a tempered distribution. To finish, we observe that for $\eta \in \Omega^1 \Phi$,
\begin{equs}
(E^* (-\Delta)^{-s} \psi^{\uparrow, x_0}, \eta) &= (\psi^{\uparrow, x_0}, E^* (-\Delta)^{-s} \eta) \\
&= \int_{\R^{n-1}} dy' \psi(y') (E^* (-\Delta)^{-s} \eta)(x_0, y').
\end{equs}
We now finish by expressing $E^* (-\Delta)^{-s} \eta$ as a convolution against $K^{\codif = 0}_s$ (by Lemma \ref{lemma:projection-laplace-inverse-formula}), and then exchanging order of integration to obtain
\begin{equs}
(E^* (-\Delta)^{-s} \psi^{\uparrow, x_0}, \eta)  = \int_{\R^n} dx \eta_i (x) \int_{\R^{n-1}} dy' \Big(K^{\codif = 0}_s\Big)_{ij}(x, (y', x_0)) \psi_j(y'),
\end{equs}
which shows the desired identity.
\end{proof}

\begin{definition}[Restriction of $\fgf_s^k(\R^n)_{d = 0}$ and $\fgf_s^k(\R^n)_{d^* = 0}$]
Let $s > \frac{1}{2}$ and $A$ be an $\fgf_s^k(\R^n)$. We define the restrictions $(EA)^{\downarrow, x_0}, (E^* A)^{\downarrow, x_0}$ as the stochastic processes $((EA)^{\downarrow, x_0}, \phi)$ and $((E^* A)^{\downarrow, x_0}, \phi)$ indexed by $\phi \in \Omega^k \schwartz(\R^{n-1}) \cap \dot{H}^{\frac{1}{2}-s}(\R^{n - 1})$, which are defined by
\begin{equs}
((EA)^{\downarrow, x_0}, \phi) := (EA, \phi^{\uparrow, x_0}), \quad ((E^* A)^{\downarrow, x_0}, \phi) := (E^* A, \phi^{\uparrow, x_0}).
\end{equs}
This is well-defined by Lemma \ref{lemma:schwartz-forms-restriction} and Remark \ref{remark:projection-fgf-gaussian-hilbert-space}.
\end{definition}

Lemma \ref{lemma:restriction-projection-kernel} gives the covariance kernel for the restrictions of $\fgf^1_s(\R^n)_{d = 0}, \fgf^1_s(\R^n)_{\codif = 0}$. We have that
\begin{equs}
\Cov\big(((E^* A)^{\downarrow, x_0}, \phi), ((E^* A)^{\downarrow, x_0}, \psi)\big) &= \Cov((E^*A, \phi^{\uparrow, x_0}), (E^* A, \psi^{\uparrow, x_0})) \\
&= (\phi^{\uparrow, x_0}, E^* \psi^{\uparrow, x_0})_{\Omega^1 \dot{H}^{-s}(\R^n)}\\
&= \int_{\R^{n-1}} \int_{\R^{n-1}} dx' dy' \phi_i(x') \Big(K^{\codif = 0}_s\Big)_{ij}((x', x_0), (y', x_0)) \psi_j(y'),
\end{equs}
and analogously for $(EA)^{\downarrow, x_0}$.

\section{Wilson loop expectations and surface expansions}\label{sec:wilsonloops}


In this section, we discuss Wilson loop expectations, which are the primary observables of interest in gauge theories. See the references at the end of Section \ref{section:differential-forms-in-gauge-theory} for more discussion of  these observables. We begin by defining these observables. Throughout this section, take $\manifold = \R^n$ or $\T^n$. 

As in Section \ref{section:non-abelian-theory}, let $\liegroup \sse \unitary(N)$ be a compact Lie group and $\frkg$ be its Lie algebra. Let $A \in \Omega^1(\manifold, \frkg^n)$ be a smooth $\frkg$-valued $1$-form. Let $\gamma \colon [0, 1] \ra \manifold$ be a piecewise smooth path in $\manifold$. Let $h \colon [0, 1] \ra \liegroup$ be the solution to the following ODE:
\begin{equs}\label{eq:holonomy-ode}
h'(t) = h(t) A(\gamma(t)) \cdot \gamma'(t), ~~ h(0) = \groupid. 
\end{equs}
We say that $h(1)$ is the holonomy of $A$ along $\gamma$. We note that $h$ has the following explicit series expansion
\begin{equs}
h(t) = \sum_{m \geq 0} \int_0^t dt_1 \int_0^{t_1} dt_2 \cdots \int_0^{t_{m-1}} dt_m  \big(A(\gamma(t_m)) \cdot \gamma'(t_m)\big) \cdots \big(A(\gamma(t_1)) \cdot \gamma'(t_1)\big) .  
\end{equs}
To be clear, the $m = 0$ term in the above sum is simply $I$, the $N \times N$ identity matrix. This can be seen by explicitly verifying that $h$ solves the ODE \eqref{eq:holonomy-ode}. 

\begin{notation}
For $t \geq 0$, $m \geq 1$, let \begin{equs}
\Delta_m(t) := \{ s = (s_1, \ldots, s_m) \in [0, t]^m : 0 < s_1 < \cdots < s_m < t\}.
\end{equs}
\end{notation}

With this notation, we may write
\begin{equs}\label{eq:holonomy-series-expansion}
h(t) = I + \sum_{m=1}^\infty \int_{\Delta_m(t)} ds \big(A(\gamma(s_m)) \cdot \gamma'(s_m)\big) \cdots \big(A(\gamma(s_1)) \cdot \gamma'(s_1)\big).
\end{equs}
In the Abelian case $G = \unitary(1)$, we have that $\frkg = \icomplex \R$, so that in fact $h$ is the exponential of the line integral:
\begin{equs}
h(t) = \exp\bigg(\int_{\gamma |_{[0, t]}} A\bigg) = \exp\bigg(\int_0^t ds A(\gamma(s)) \cdot \gamma'(s)\bigg).
\end{equs}
This follows since 
\begin{equs}
\int_{\Delta_m(t)} ds \big(A(\gamma(s_m)) \cdot \gamma'(s_m)\big) \cdots \big(A(\gamma(s_1)) \cdot \gamma'(s_1)\big)  = \frac{1}{m!} \bigg(\int_0^t ds A(\gamma(s)) \cdot \gamma'(s)\bigg)^m.
\end{equs}
In the general case, the problem is that the matrices of $\frkg$ do not necessarily commute, so that the above identity is not necessarily true.

\begin{definition}[Wilson loop observable]
Let $\gamma \colon [0, 1] \ra \manifold$ be a piecewise smooth loop. Define the function $W_\gamma \colon \Omega^1(\manifold, \frkg^n) \ra \C$ 
\begin{equs}
W_\gamma(A) := \Tr(h(1)),
\end{equs}
where $h$ is the holonomy of $A$ along $\gamma$. This function is known as a Wilson loop observable.
\end{definition}

It is important that $\gamma$ is a loop and not just a path in the definition. This is to ensure gauge invariance of $W_\gamma$, i.e. once can check that for $\gamma$ a loop, we have that
\begin{equs}
W_\gamma(A^g) = W_\gamma(A),
\end{equs}
where $g : \manifold \ra \liegroup$ and $A^g$ is as in \eqref{eq:gauge-transformation}. To see this, one may check that given a solution $h$ to \eqref{eq:holonomy-ode}, the function
\begin{equs}
h_g(t) := g(\gamma(0)) h(t) g(\gamma(t))^{-1}
\end{equs}
solves the ODE
\begin{equs}
h_g'(t) = h_g(t) A^g(\gamma(t)) \cdot \gamma'(t), ~~ h_g(0) = \groupid,
\end{equs}
and thus $h_g(1) = g(\gamma(0)) h(1) g(\gamma(1))^{-1}$ is the holonomy of $A^g$ along $\gamma$. Thus if $\gamma(1) = \gamma(0)$, then $\Tr(h_g(1)) = \Tr(h(1))$, which gives the desired gauge invariance of $W_\gamma$.

\subsection{Regularized Wilson loops}

In this subsection, we discuss regularized versions of Wilson loop observables introduced in \cite{CC23, CCHS2022}. To start the discussion, suppose that $A$ is a $\frkg$-valued 1-form $\gff$. Can we still define the holonomy of $A$ around a loop $\gamma$? In the special case $G = \unitary(1)$, this essentially boils down to whether we can define line integrals of a $\R$-valued 1-form $\gff$. As discussed in \cite[Section 3.1]{Chevyrev2019}, a variance calculation shows that this is possible when $n = 2$, but not when $n \geq 3$. Thus when $n \geq 3$, we must somehow regularize. The problem is that the loop is supported on a one-dimensional set, and thus a natural idea is to replace the loop by a ``rod", i.e. a smoothed version of the loop which is supported on a three-dimensional set. However, the resulting observable would not be gauge-invariant (in the original sense of gauge symmetry, i.e. \eqref{eq:gauge-transformation}, not in the sense of Section \ref{sec:gaugetransformations}).

In 3D, a solution to this problem proposed independently in \cite{CC24, CCHS2022} is to apply a gauge-covariant smoothing procedure to $A$. The natural smoothing procedure to take is the Yang--Mills heat flow, which is a nonlinear PDE given by:
\begin{equs}\label{eq:ymhf}\tag{YMHF}
\ptl_t A = -d_A^* F_A, ~~ A(0) = A_0.
\end{equs}
This PDE arises as the gradient flow of the Yang--Mills functional $\sym$ defined in \eqref{eq:sym}. This PDE posesses two crucial properties that make it a valid smoothing procedure:
\begin{enumerate}
    \item Letting $\Phi_t(A_0)$ denote the map which takes the initial data $A_0$ to the solution of \eqref{eq:ymhf} at time $t$, we have that $\Phi_t(A_0^g) = \big(\Phi_t(A_0)\big)^g$ for any gauge transformation $g \in C^\infty(\R^n, G)$, i.e. $\Phi_t$ is commutes with gauge transformations. Because of this, we say that \eqref{eq:ymhf} is gauge-covariant. 
    \item Solutions to \eqref{eq:ymhf} are smooth at positive times\footnote{Technically, this is only true up to gauge.}, and thus holonomies are well-defined at all positive times. Here, the time parameter should be thought of as the smoothing parameter -- the larger the time, the more the smoothing.
\end{enumerate}

Given these two properties, a gauge-invariant observable $W_{\gamma, t}$ may be defined for loops $\gamma$ and times $t > 0$ as:
\begin{equs}
W_{\gamma, t}(A) := W_\gamma(\Phi_t(A)). 
\end{equs}
It was shown in \cite{CC23, CCHS2022} that in 3D, \eqref{eq:ymhf} is well-posed if the $\frkg$-valued 1-form $\gff$ is taken as the initial data. Thus, regularized Wilson loop observables may be defined for the $\frkg$-valued 1-form $\gff$. 

\begin{remark}[Higher dimensions]
In dimension four and higher, even this approach breaks down, as it is unclear whether \eqref{eq:ymhf} is even well-posed with GFF initial data.
\end{remark}


Having defined regularized Wilson loop observables, one would like to understand their probability distributions. Unfortunately, unlike in the exactly solvable 2D case, it seems too much to hope for an exact characterization of the law of these observables in general. When $G = \unitary(1)$, we do have exact formulas -- see \cite[Section 2]{CC24}.

On the other hand, it was shown in \cite{CC23, CCHS2022} that at small times, $\Phi_t(A)$ looks like the solution to the heat equation started from $A$, which we write as $e^{t \Delta} A$. Moreover, the approximation becomes better in a precise sense as the smoothing parameter $t$ gets small. Indeed, this observation is the starting point to showing that \eqref{eq:ymhf} is well-posed with GFF initial data. 
Thus, while $e^{t \Delta} A$ is not gauge-covariant, we may still try to understand the distribution of $W_\gamma(e^{t \Delta} A)$ as a first approximation towards the actual distribution of $\Phi_t(A)$.

In the following, we derive a surface sum representation for $\E W_\gamma(e^{t\Delta} A)$, where $A$ is a $1$-form $\frkg$-valued GFF. Such expansions were recently derived in various Yang--Mills contexts in the recent papers \cite{park2023wilson, cao2023random}. However, the precise form of the following expansion is new. We work out the details for $G = \unitary(N)$, but similar considerations hold when $G$ is another classical compact Lie group, such as $\mrm{O}(N)$ (the cited papers do consider other such groups). 

\begin{remark}[Connection to matrix-map story for GUE/GOE]
The ensuing arguments bear resemblance to the classical matrix-map (or i.e. genus expansion) story for GUE or GOE matrices. There is a good reason for this: if $X$ is a GUE matrix, then $\icomplex X$ is a Gaussian random variable taking values in $\mathfrak{u}(N)$, the Lie algebra of $\unitary(N)$. Later on, we will need to compute expectations of traces of products of $\mathfrak{u}(N)$-valued Gaussian random matrices. Due to the relation to GUE, such expectations are (up to sign) the same as expectations of traces of products of GUE matrices. Classical matrix-map stories give surface sum representations of the latter. 
\end{remark}
 
We first need some facts about the $\unitary(N)$ and its Lie algebra $\mathfrak{u}(N)$. The latter is the vector space $\{X \in \End(\C^N) : X^* = -X\}$ of complex skew-Hermitian matrices. We define an inner product $\langle \cdot, \cdot \rangle$ on $\mathfrak{u}(N)$ by
\begin{equs}
\langle X, Y \rangle := -N \Tr(X Y) = N \Tr(X Y^*).
\end{equs}
By the last identity, we see that this inner product is (up to the factor $N$) the Frobenius inner product on matrices. The vector space $\mathfrak{u}(N)$ has dimension $N^2$. We fix an orthonormal basis $(E_a, a \in [N^2])$ of $\mathfrak{u}(N)$.

Before stating the surface sum representation, we discuss the types of surfaces which will appear. 

\begin{definition}[Pairings]
For $m \geq 1$ even, let $\mc{P}(m)$ be the set of pairings on $[m]$, i.e. the set of partitions of $[m]$ into two element sets.
\end{definition}

\begin{definition}[Maps built out of pairings]
To each pairing $\pi \in \mc{P}(m)$ we may associate a unicellular map $S(\pi)$ obtained by gluing the edges of an $m$-gon according to $\pi$. See Figure \ref{figure:surface_from_pairing-1} for an example. Let $V(\pi)$ be the number of vertices of the resulting map $S(\pi)$. 

More generally, let $\mu = (\mu_1, \ldots, \mu_k)$ be a partition of $m$. Then to each pairing $\pi$, we may associate the map $S(\pi, \mu)$ obtained by gluing the edges of a collection of $\mu_1, \mu_2, \ldots, \mu_k$-gons according to $\pi$. See Figure \ref{figure:surface_from_pairing-2} for an example. Let $V(\pi, \mu)$ be the number of vertices of the resulting map $S(\pi, \mu)$.
\end{definition}

\begin{notation}
We identify pairings $\pi \in \mc{P}(m)$ with elements of $\symgrp_m$ given by the corresponding product of transpositions. I.e., if $\pi = \{ \{i_1, j_1\}, \ldots, \{i_{m/2}, j_{m/2}\}\}$, then we identify $\pi$ with $(i_1 ~ j_1) \cdots (i_{m/2} ~ j_{m/2})$. We will also denote the latter by $\pi$.
\end{notation}

\begin{figure}
    \centering
    \includegraphics[width=.5\linewidth]{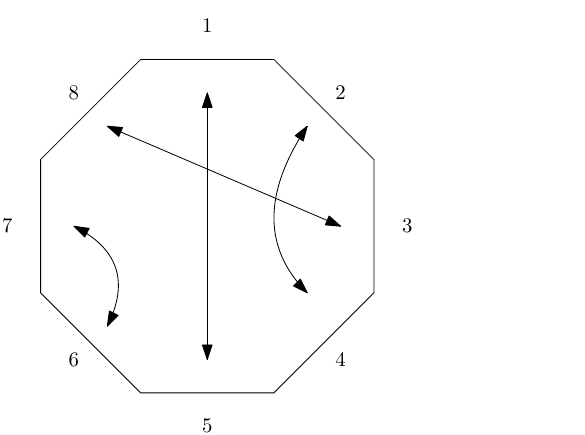}
    \caption{In this example, $m = 8$ and $\pi = \{\{1, 5\}, \{2, 4\}, \{3, 8\}, \{6, 7\}\}$. The edges of the $8$-gon are labeled 1 through 8, and the edges are glued together according to $\pi$.}
    \label{figure:surface_from_pairing-1}
\end{figure}

\begin{figure}
    \centering
    \includegraphics[width=.5\linewidth, page=2]{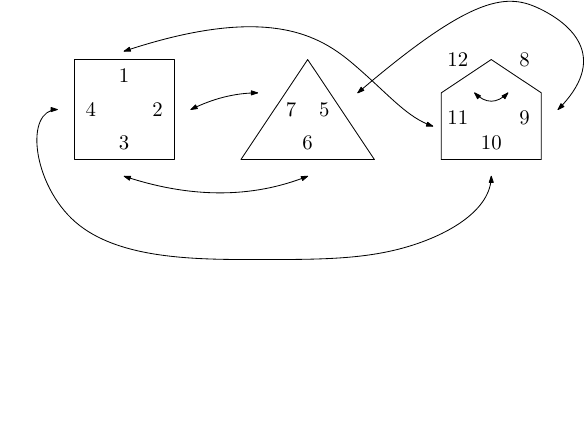}
    \vspace{-25mm}
    \caption{In this example, $m = 12$, $\mu = (4, 3, 5)$, and $\pi = \{\{1, 11\}, \{2, 7\}, \{3, 6\}, \{4, 10\}, \{5, 9\}, \{8, 12\}\}$.}
    \label{figure:surface_from_pairing-2}
\end{figure}

In the following, we will work in a slightly more general setup than previously described. Let $A = (A_1, \ldots, A_n)$ be a $\mathfrak{u}(N)^n$-valued mean zero Gaussian process with continuous sample paths. 

\begin{assumptions}\label{assumptions:smoothed-gff}
We make the following assumptions on $A$.
\begin{enumerate}
    \item $A \stackrel{d}{=} -A$.
    \item There is a function $G : \R^n \times \R^n \ra \R$ such that for $X, Y \in \mathfrak{u}(N)$, $j, k \in [n]$, we have that  $\E[\langle A_j(x), X \rangle \langle A_k(y), Y\rangle] = \delta_{jk} G(x, y) \langle X, Y \rangle$.
\end{enumerate} 
\end{assumptions}

One should think of $A$ as being obtained by smoothing a $\mrm{u}(N)$-valued $1$-form GFF (which we also previously denoted by $A$), for instance by applying the heat operator $e^{t \Delta}$ as in our previous discussion, or by convolving against a smooth bump function. The covariance function $G$ would then be a smoothed version of the Green's function.

\begin{prop}[Surface sum representation]\label{prop:approximate-regularized-wilson-loops-genus-expansion}
Let $A$ be as in Assumptions \ref{assumptions:smoothed-gff}. Let $\gamma \colon [0, 1] \ra \manifold$ be a piecewise smooth loop. We have that
\begin{equs}
\E\big[\Tr(W_\gamma(A))\big] = \sum_{m = 0}^\infty (-1)^m \sum_{\pi \in \mc{P}(2m)} N^{V(\pi) - E(\pi)} \int_{\Delta_{2m}(1)} dt \prod_{\{i, j\} \in \pi} G(\gamma(t_i), \gamma(t_j)) \gamma'(t_i) \cdot \gamma'(t_j).
\end{equs}
More generally, given piecewise smooth loops $\gamma_1, \ldots, \gamma_k$, we have that
\begin{equs}
\E&\bigg[\prod_{\ell \in [k]} \Tr(W_{\gamma_\ell}(A)) \bigg] = \\
&\sum_{m=0}^\infty (-1)^m \sum_{\substack{m_1, \ldots, m_k \geq 0 \\ m_1 + \cdots + m_k = 2m}} \sum_{\pi \in \mc{P}(2m)} N^{V(\pi, \mu) - E(\pi, \mu)} \int_{\Delta_{m_1}(1)} dt^1 \cdots \int_{\Delta_{m_k}(1)} dt^k \prod_{\{i, j\} \in \pi} G(\gamma(t_i), \gamma(t_j)) \gamma'(t_i) \cdot \gamma'(t_j).
\end{equs}
Here, the partition $\mu$ appearing in the sum is $\mu = (m_1, \ldots, m_k)$. Also, given $t^1 \in \Delta_{m_1}(1), \ldots, t^k \in \Delta_{m_k}(1)$, we concatenate $t = (t^1, \ldots, t^k) \in [0, 1]^{2m}$ to obtain a vector of length $2m$. The variables $t_i, t_j$ which appear in the above display are the coordinates of this vector.
\end{prop}

To begin towards the proof of Proposition \ref{prop:approximate-regularized-wilson-loops-genus-expansion}, we first show the following interpretation of $V(\pi)$.

\begin{lemma}\label{lemma:vertices-equals-cycles}
For $\pi \in \mc{P}(m)$, we have that
\begin{equs}
\cycles((1 \cdots m) \pi) = V(\pi).
\end{equs}
More generally, given a partition $\mu = (\mu_1, \ldots, \mu_k)$ of $m$, let $\sigma_\mu \in \symgrp_m$ be the permutation with cycles $c_1, \ldots, c_k$, where $c_\ell := (m_1 + \cdots + m_{\ell-1} + 1  \cdots m_1 + \cdots + m_{\ell})$. We have that
\begin{equs}
\cycles(\sigma_\mu \pi) = V(\pi, \mu).
\end{equs}
\end{lemma}
\begin{proof}
We begin with some concrete examples. First, consider the setting of Figure \ref{figure:surface_from_pairing-1}, where $m = 8$ and $\pi =  \{\{1, 5\}, \{2, 4\}, \{3, 8\}, \{6, 7\}\}$. We may also label the vertices of the polygon by 1 through 8, this time in blue. We work out this example in Figure \ref{figure:surface_from_pairing-3}.
\begin{figure}
    \centering
    \includegraphics[width=.5\linewidth, page=3]{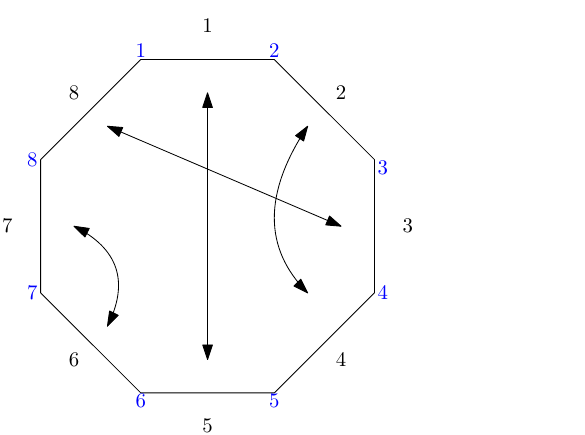}
    \caption{With $\pi =  \{\{1, 5\}, \{2, 4\}, \{3, 8\}, \{6, 7\}\}$, observe that $(1 \cdots 8) \pi = (1~ 6~ 8~ 4~ 3) (2 ~5)$, and moreover the cycles of the product are precisely given by the vertices of the polygon which are identified together into a single vertex of the resulting map.}
    \label{figure:surface_from_pairing-3}
\end{figure}
Next, consider the setting of Figure \ref{figure:surface_from_pairing-2}. We work out this example in Figure \ref{figure:surface_from_pairing-4}.
\begin{figure}
    \centering
    \includegraphics[width=.5\linewidth, page=4]{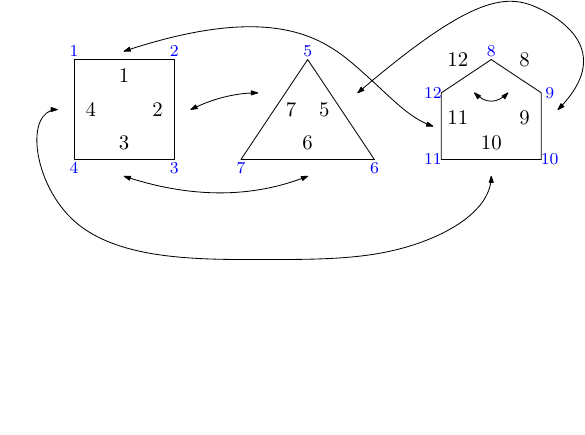}
    \vspace{-25mm}
    \caption{Recall from Figure \ref{figure:surface_from_pairing-2} that $m = 12$, $\mu = (4, 3, 5)$, and $\pi = \{\{1, 11\}, \{2, 7\}, \{3, 6\}, \{4, 10\}, \{5, 9\}, \{8, 12\}\}$. Observe that $(1 ~ 2 ~ 3 ~ 4) (5 ~ 6 ~ 7) (8 ~ 9 ~ 10 ~ 11 ~ 12) \pi$ = $(1 ~ 12 ~ 9 ~ 6 ~ 4 ~ 11 ~ 2 ~ 5 ~ 10) (3 ~ 7) (8)$. As in Figure \ref{figure:surface_from_pairing-3}, the cycles of the product correspond to the vertices of the resulting map after performing all the gluings.}
    \label{figure:surface_from_pairing-4}
\end{figure}
From these two examples, we see that more generally, the cycles of $\sigma_\mu \pi$ are precisely given by the vertices of the original polygons which are identified together into a single vertex in the resulting map. From this, it follows that $\cycles(\sigma_\mu \pi) = V(\pi, \mu)$.
\end{proof}

Next, we discuss some elements of representation theory which enter into the proof of Proposition \ref{prop:approximate-regularized-wilson-loops-genus-expansion}. Much of this discussion may be found in \cite[Section 2]{Levy2008}. Let $N, m \geq 1$. The vector space $\C^N$ has standard basis given by $(e_i, i \in [N])$, and inner product given by $\langle v, w \rangle = v_i \bar{w}_i$. We can form the $m$th tensor power $(\C^N)^{\otimes m}$, which is a vector space with basis given by $e_{i_1} \otimes \cdots \otimes e_{i_m}$, $i_1, \ldots, i_m \in [N]$. The inner product on $\C^N$ induces an inner product on $(\C^N)^{\otimes m}$ via the following definition on pure tensors:
\begin{equs}
\langle v_1 \otimes \cdots \otimes v_m, w_1 \otimes \cdots w_m \rangle := \langle v_1, w_1\rangle \cdots \langle v_m, w_m \rangle, ~~ v_i, w_i \in \C^N \text{ for $i \in [m]$}.
\end{equs}
Define the representation $\rho : \symgrp_m \ra \End((\C^N)^{\otimes m})$ by
\begin{equs}
\rho(\sigma)(v_1 \otimes \cdots \otimes v_m) = v_{\sigma(1)} \otimes \cdots \otimes  v_{\sigma(m)}. 
\end{equs}
Given operators $M_1, \ldots, M_m \in \End(\C^N)$, we may obtain $M_1 \otimes \cdots \otimes M_m \in \End((\C^N)^{\otimes m})$ which is defined on pure tensors by
\begin{equs}
(M_1 \otimes \cdots \otimes M_m)(v_1 \otimes \cdots v_m) := (M_1 v_1) \otimes \cdots \otimes (M_m v_m).
\end{equs}

\begin{lemma}\label{lemma:trace-product-and-tensor-product}
Let $M_1, \ldots, M_m \in \End(\C^N)$. We have that
\begin{equs}
\Tr(M_1 \cdots M_m) = \Tr\big((M_1 \otimes \cdots \otimes M_m) \rho((1 \cdots m))\big).
\end{equs}
Here, $(1 \cdots m) \in \symgrp_m$ is written in cycle notation. More generally, for $\sigma \in \symgrp_m$, we have that
\begin{equs}
\prod_{\substack{c = (i_1, \ldots, i_k) \\ \text{a cycle of $\sigma$}}} \Tr(M_{i_1} \cdots M_{i_k}) = \Tr\big((M_1 \otimes \cdots \otimes M_m) \rho(\sigma)\big).
\end{equs}
\end{lemma}
\begin{proof}
We have that
\begin{equs}
\Tr\big((M_1 \otimes \cdots \otimes M_m) \rho((1 \cdots m))\big) &= \big\langle e_{i_1} \otimes \cdots \otimes e_{i_m}, (M_1 \otimes \cdots \otimes M_m) \rho((1 \cdots m)) e_{i_1} \otimes \cdots \otimes e_{i_m}\big \rangle \\
&= \big\langle e_{i_1} \otimes \cdots \otimes e_{i_m}, (M_1 \otimes \cdots \otimes  M_{m-1} \otimes M_m) (e_{i_2} \otimes \cdots \otimes e_{i_m} \otimes e_{i_1}) \big \rangle \\
&= \langle e_{i_1}, M_1 e_{i_2} \rangle \cdots \langle e_{i_{m-1}}, M_{m-1} e_{i_m} \rangle \langle e_{i_m}, M_m e_{i_1} \rangle \\
&= (M_1)_{i_1 i_2} \cdots (M_{m-1})_{i_{m-1} i_m} (M_m)_{i_m i_1} \\
&= \Tr(M_1 \cdots M_m),
\end{equs}
as desired. The more general result follows similarly.
\end{proof}

By taking all the matrices $M_i = I_N$, $i \in [m]$, we obtain the following direct corollary of the previous result. 

\begin{cor}\label{cor:trace-rho-sigma}
We have that
\begin{equs}
\Tr(\rho(\sigma)) = N^{\cycles(\sigma)}, \quad \sigma \in \symgrp_m.
\end{equs}
\end{cor}

The reason that the representation $\rho$ enters is due to the following result, which computes $E_a \otimes E_a$ as a matrix in $\End((\C^N)^{\otimes 2})$. Recall $(E_a, a \in [N^2])$ is an (arbitrary) orthonormal basis of $\mathfrak{u}(N)$.

\begin{lemma}\label{lemma:casimir-result}
We have that
\begin{equs}\label{eq:casimir-identity}
E_a \otimes E_a = -\frac{1}{N} \rho((1, 2)).
\end{equs}
More generally, for $m \geq 1$ even and $\pi \in \mc{P}(m)$, we have that
\begin{equs}\label{eq:casimir-identity-higher-tensor}
E_{a_1} \otimes \cdots \otimes E_{a_n} \prod_{\{i, j\} \in \pi} \delta^{a_i a_j} = \frac{(-1)^{\frac{m}{2}}}{N^{\frac{m}{2}}} \rho(\pi).
\end{equs}
\end{lemma}
\begin{proof}
We first observe that $E_a \otimes E_a$ is independent of the choice of orthonormal basis $(E_a, a \in [N^2])$. Towards this end, observe that for $X, Y \in \mathfrak{u}(N)$, we have that
\begin{equs}
\langle E_a \otimes E_a, X \otimes Y \rangle = \langle E_a, X \rangle \langle E_a, Y \rangle = \langle X, Y \rangle.
\end{equs}
Thus $E_a \otimes E_a$ is uniquely characterized as the element of $\mathfrak{u}(N)^{\otimes 2}$ which defines the linear functional that sends $X \otimes Y \mapsto \langle X, Y \rangle$. This shows that $E_a \otimes E_a$ is independent of the choice of orthonormal basis.

Next, we exhibit an explicit orthonormal basis with which we may perform explicit computations. For $i, j \in [N]$, let $E_{ij}$ be the elementary matrix with $1$ in its $(i, j)$ entry, and $0$ everywhere else. Note that we may write $E_{ij} = e_i e_j^T$, where $e_1, \ldots, e_n$ are the standard basis vectors of $\R^n$. One may check that the following collection forms an orthonormal basis of $\mathfrak{u}(N)$:
\begin{equs}
\Big( \frac{1}{\sqrt{2N}} (E_{ij} - E_{ji}), ~ 1 \leq i < j \leq N\Big), ~~ \Big(\frac{\icomplex}{\sqrt{2N}} (E_{ij} + E_{ji}), ~ 1 \leq i < j \leq N\Big), ~~ \Big(\frac{\icomplex}{\sqrt{N}} E_{ii}, ~i \in [N]\Big).
\end{equs}
We explicitly compute
\begin{equs}
\frac{1}{2N} \sum_{1 \leq i < j \leq N} (E_{ij} - E_{ji}) \otimes (E_{ij} - E_{ji}) &= \frac{1}{2N} \sum_{i \neq j} E_{ij} \otimes E_{ij} - \frac{1}{2N} \sum_{i \neq j} E_{ij} \otimes E_{ji}, \\
-\frac{1}{2N} \sum_{1 \leq i < j \leq N} (E_{ij} + E_{ji}) \otimes (E_{ij} + E_{ji}) &= -\frac{1}{2N} \sum_{i \neq j} E_{ij} \otimes E_{ij} - \frac{1}{2N} \sum_{i \neq j} E_{ij} \otimes E_{ji} .
\end{equs}
Summing the two displays and adding in $-\frac{1}{N} \sum_{i \in [N]} E_{ii} \otimes E_{ii}$, we obtain
\begin{equs}
E_a \otimes E_a = -\frac{1}{N} E_{ij} \otimes E_{ji}.
\end{equs}
To finish the proof of \eqref{eq:casimir-identity}, we claim that $E_{ij} \otimes E_{ji} = \rho((1 ~ 2))$. To see this, let $v, w \in \C^N$, and compute
\begin{equs}
(E_{ij} \otimes E_{ji})(v \otimes w) = (e_i v_j) \otimes (e_j w_i) = w_i v_j (e_i \otimes e_j) = w_i e_i \otimes v_j e_j = w \otimes v = \rho((1 ~ 2))(v \otimes w),
\end{equs}
as desired.

The identity \eqref{eq:casimir-identity-higher-tensor} follows by applying \eqref{eq:casimir-identity} a total of $\frac{m}{2}$ times.
\end{proof}

\begin{lemma}\label{lemma:product-of-A-varep}
Let $A$ be as in Assumptions \ref{assumptions:smoothed-gff}. Given a piecewise smooth loop $\gamma \colon [0, 1] \ra \R^n$ and times $t_1, \ldots, t_m \in (0, 1)$, with $m \geq 1$ even, we have that
\begin{equs}
\E\big[(A(\gamma(t_1)) \cdot \gamma'(t_1)) \otimes \cdots \otimes (A(\gamma(t_m)) \cdot \gamma'(t_m))\big] = \frac{(-1)^{\frac{m}{2}}}{N^{\frac{m}{2}}} \sum_{\pi \in \mc{P}(m)} \rho(\pi) \prod_{\{i, j\} \in \pi} G(\gamma(t_i), \gamma(t_j)) \gamma'(t_i) \cdot \gamma'(t_j).
\end{equs}
\end{lemma}
\begin{proof}
For brevity, let $X_j = A(\gamma(t_j)) \cdot \gamma'(t_j)$. Further, for an orthonormal basis $(E_a, a \in [N^2])$ of $\mrm{u}(N)$, decompose $X_j = X_j^a E_a$. We have that
\begin{equs}
X_1 \otimes \cdots \otimes X_m = E_{a_1} \otimes \cdots \otimes E_{a_m} X_1^{a_1} \cdots X_m^{a_m}, 
\end{equs}
and that $(X_j^a, j \in [m], a \in [N^2])$ is a jointly Gaussian collection of $\R$-valued random variables. Thus by Wick's formula, we have that
\begin{equs}
\E[X_1 \otimes \cdots \otimes X_m] = \sum_{\pi \in \mc{P}(m)} E_{a_1} \otimes \cdots \otimes E_{a_m} \prod_{\{i, j\} \in \pi} \E[X_i^{a_i} X_j^{a_j}].
\end{equs}
By Lemma \ref{lemma:casimir-result}, to finish, it suffices to show that for any $a, b \in [N^2]$ and $i, j \in [m]$, we have that
\begin{equs}
\E[X_i^{a} X_j^{b}] = \delta^{a b} G(\gamma(t_i), \gamma(t_j)) \gamma'(t_i) \cdot \gamma'(t_j).
\end{equs}
Towards this end, we may write $A_{k} = A_{k}^a E_a$ (where $A_k^a = \langle A_k, E_a \rangle$) so that
\begin{equs}
X^a_i = A_{k}^a(\gamma(t_i)) \gamma_k'(t_i).
\end{equs}
By the second item in Assumptions \ref{assumptions:smoothed-gff}, we have that
\begin{equs}
\E[X^a_i X^b_j] = \E[\langle A_k (\gamma(t_i)), E_a \rangle \langle A_\ell(\gamma(t_j)), E_b \rangle] \gamma_k'(t_i) \gamma_\ell'(t_j) &= \delta^{ab} \delta_{k\ell} G(\gamma(t_i), \gamma(t_j)) \gamma_k'(t_i) \gamma_\ell'(t_j) \\
&= \delta^{ab} G(\gamma(t_i), \gamma(t_j)) \gamma'(t_i) \cdot \gamma'(t_j),
\end{equs}
as desired.
\end{proof}

\begin{proof}[Proof of Proposition \ref{prop:approximate-regularized-wilson-loops-genus-expansion}]
We show the first claim, as the second follows by a straightforward generalization of the argument. For $m \geq 1$, define
\begin{equs}
I_m := \Tr \E\bigg[\int_{\Delta_{m} (1)} dt \big(A(\gamma(t_m)) &\cdot \gamma'(t_m)\big) \cdots \big(A(\gamma(t_{1})) \cdot \gamma'(t_{1})\big)\bigg] .
\end{equs}
Recalling the series expansion of the holonomy \eqref{eq:holonomy-series-expansion}, it suffices to show that for $m \geq 1$,
\begin{equs}
I_{2m-1} &= 0, \\
I_{2m} &= (-1)^m \sum_{\pi \in \mc{P}(2m)} N^{V(\pi) - E(\pi)} \int_{\Delta_{2m}(1)} dt \prod_{\{i, j\} \in \pi} G(\gamma(t_i), \gamma(t_j)) \gamma'(t_i) \cdot \gamma'(t_j).
\end{equs}
The first identity follows by parity, since $A \stackrel{d}{=} -A$. For the second, fix $t \in \Delta_{2m}(1)$. By combining Lemmas \ref{lemma:trace-product-and-tensor-product} and \ref{lemma:product-of-A-varep}, and then combining Corollary \ref{cor:trace-rho-sigma} and Lemma \ref{lemma:vertices-equals-cycles}, we have that
\begin{equs}
\E \Tr&\Big[\big(A(\gamma(t_{2m})) \cdot \gamma'(t_{2m})\big) \cdots \big(A(\gamma(t_{1})) \cdot \gamma'(t_{1})\big) \Big] \\
&= (-1)^{m} N^{-m} \sum_{\pi \in \mc{P}(2m)} \Tr(\rho(\pi) \rho(1 \cdots 2m)) \prod_{\{i, j\} \in \pi} G(\gamma(t_i), \gamma(t_j)) \gamma'(t_i) \cdot \gamma'(t_j) \\
&= (-1)^m \sum_{\pi \in \mc{P}(2m)} N^{V(\pi) - E(\pi)} \prod_{\{i, j\} \in \pi} G(\gamma(t_i), \gamma(t_j)) \gamma'(t_i) \cdot \gamma'(t_j),
\end{equs}
as desired.
\end{proof}

\subsection{Big loops and surfaces}

Instead of taking regularized Wilson loops, one may consider alternative observables which are well-defined for a $\frkg$-valued $1$-form GFF. We discuss such observables in this subsection. These observables will be gauge-invariant in the sense discussed in Section \ref{sec:gaugetransformations}. Throughout this section, we take $\manifold = \R^n$ or $\T^n$. 

To begin the discussion, note that we may reinterpret the holonomy \eqref{eq:holonomy-ode} in the following way. A smooth loop $\gamma \colon [0, 1] \ra \manifold$ defines a map $\Gamma \colon [0, 1] \ra \Omega^1 \schwartz'$ via the formula $\Gamma(t) := \delta_{\gamma(t)} \gamma'(t)$. I.e., $\Gamma(t)$ is the distribution on $1$-forms such that
\begin{equs}
(A, \Gamma(t)) = A(\gamma(t)) \cdot \gamma'(t), ~~ A \in \Omega^1 \schwartz.
\end{equs}
Given $A \in \Omega^1 \schwartz$, the holonomy ODE \eqref{eq:holonomy-ode} may then be written
\begin{equs}\label{eq:general-holonomy-ode}
h'(t) = h(t)(A, \Gamma(t)), ~~ h(0) = \groupid.
\end{equs}
More generally, we make the following definition.

\begin{definition}
Let $\ms{G}$ be the set of functions $\Gamma \colon [0, 1] \ra \Omega^1 \schwartz'$ satisfying
\begin{equs}
t \mapsto (A, \Gamma(t)) \in L^\infty([0, 1]) \text{ for all $A \in \Omega^1 \schwartz$},
\end{equs}
and $d^* \Gamma(t) = 0$ for a.e. $t \in [0, 1]$. We will refer to elements of $\ms{G}$ as ``big loops".
\end{definition}


\begin{definition}[Generalized holonomy and big loop observable]
For any $\Gamma \in \ms{G}$, and for any $A \in \Omega^1 \schwartz(\manifold; \frkg)$, define the ``holonomy of $A$ around $\Gamma$" by \eqref{eq:general-holonomy-ode}. Define the Wilson big loop observable $W_\Gamma$ by
\begin{equs}
W_\Gamma(A) := \Tr(h(1)),
\end{equs}
where $h$ is the holonomy of $A$ along $\Gamma$.
\end{definition}

Similar to the series expansion \eqref{eq:holonomy-series-expansion} for the holonomy, we also have the following series expansion for the holonomy of $A$ around a big loop:
\begin{equs}\label{eq:big-loop-holonomy-series-expansion}
h(t) = \sum_{m \geq 0} \int_{\Delta_m(t)} ds (A, \Gamma(s_m)) \cdots (A, \Gamma(s_1)).
\end{equs}
Note also that the gauge-invariance of $W_\Gamma$ follows immediately from the assumption that $d^* \Gamma(t) = 0$ for a.e. $t \in [0, 1]$, since then given a Schwartz gauge transformation $\lambda \in \schwartz(\R^n; \frkg)$, we have that
\begin{equs}
(A + d\lambda, \Gamma(t)) = (A, \Gamma(t)) \text{ for a.e. $t \in [0, 1]$.}
\end{equs}

\begin{remark}
We envision big loops and their holonomies as arising from limits of smooth loops (along with suitable rescalings of the $1$-forms) in the following manner. As a simple example, take $\manifold = \T^n$ and $\liegroup = \unitary(1)$, so that $A$ is a $\icomplex \R$-valued $1$-form. Suppose further that $A$ is smooth. Fix some vector $v \in \T^n$ such that the flow $t \mapsto tv$ equidistributes over $\T^n$. We can define a sequence of loops $\{\gamma_m\}_{m \geq 1}$ such that for each $m$, $\gamma_m(t) = m t v$ on the interval $[0, 1-\varep_m]$ for some $\varep_m$ very small, say $2^{-2^m}$. On the remaining part of the interval $[1-\varep_m, 1]$, $\gamma_m$ goes in a straight line from $\gamma_m(1 - \varep_m)$ to its starting point. Since $A$ is $\icomplex \R$-valued, its holonomy along any given loop is the exponential of its line integral along the loop. Since $v$ was chosen so that $t \mapsto tv$ equidistributes over $\T^n$, we approximately have that its line integral
\begin{equs}
\frac{1}{m} \int_0^1 A(\gamma_m(t)) \cdot \gamma_m'(t) \approx \int_{\T^n} A(x) dx,
\end{equs}
and thus the holonomy of $A$ along $\gamma_m$ is approximately given by
\begin{equs}
\exp\bigg(m \int_{\T^n} A(x) dx \bigg).
\end{equs}
In particular, if we rescale and define $A_m = \frac{1}{m} A$, then as $m \toinf$, we would expect that the holonomy of $A_m$ along $\gamma_m$ converges to
\begin{equs}
\exp\bigg(\int_{\T^n} A(x) dx \bigg).
\end{equs}
This is precisely the holonomy of $A$ along the big loop $\Gamma$ such that $\Gamma(t) \equiv 1$ for all $t \in [0, 1]$.

More generally, we may consider a sequence of $\{\gamma_m\}_{m \geq 1}$ such that each $\gamma_m$ has length $m$. Then, if $\gamma_m$ is suitably chosen, and $A$ is again rescaled to $A_m$, we would expect that the holonomy of $A_m$ along $\gamma_m$ converges to the holonomy of $A$ along some big loop $\Gamma$.
\end{remark}

The nice thing about big loop observables is that their expectations are well-defined, even if $A$ is a $1$-form GFF. Moreover, we have the following surface-sum representation for Wilson big loop expectations, analogous to our previous surface-sum representation for regularized Wilson loops (Proposition \ref{prop:approximate-regularized-wilson-loops-genus-expansion}). Again, we take $\liegroup = \unitary(N)$, and so $\frkg = \mathfrak{u}(N)$.

\begin{prop}\label{prop:big-loop-surface-sum-representation}
Let $\Gamma \in \ms{G}$ be a big loop such that $\sup_{t \in [0, 1]} \|\Gamma(t)\|_{\dot{H}^{-1}(\manifold)} < \infty$. Let $A$ be a $\mrm{u}(N)$-valued $\gff^1(\manifold)$. We have that
\begin{equs}
\E[W_\Gamma(A)] = \sum_{m=0}^\infty (-1)^m \sum_{\pi \in \mc{P}(2m)} N^{V(\pi) - E(\pi)} \int_{\Delta_{2m}(1)} dt \prod_{\{i, j\} \in \pi} (\Gamma(t_i), \Gamma(t_j))_{\Omega^1 \dot{H}^{-1}(\manifold)}.
\end{equs}
More generally, given big loops $\Gamma_1, \ldots, \Gamma_k$ such that $\sup_{t \in [0, 1]} \|\Gamma_j(t)\|_{\dot{H}^{-1}(\manifold)} < \infty$ for each $j \in [k]$, we have that
\begin{equs}
\E&\bigg[\prod_{\ell \in [k]} W_\Gamma(A)\bigg] \\
&= \sum_{m=0}^\infty (-1)^m \sum_{\substack{m_1, \ldots, m_k \geq 0 \\ m_1 + \cdots + m_k = 2m}} \sum_{\pi \in \mc{P}(2m)} N^{V(\pi, \mu) - E(\pi, \mu)} \int_{\Delta_{m_1}(1)} dt^1 \cdots \int_{\Delta_{m_k}(1)} dt^k \prod_{\{i, j\} \in \pi} (\Gamma(t_i), \Gamma(t_j))_{\Omega^1 \dot{H}^{-1}(\manifold)}.
\end{equs}
Here, the partition $\mu$ appearing in the sum is $\mu = (m_1, \ldots, m_k)$. Also, given $t^1 \in \Delta_{m_1}(1), \ldots, t^k \in \Delta_{m_k}(1)$, we concatenate $t = (t^1, \ldots, t^k) \in [0, 1]^{2m}$ to obtain a vector of length $2m$. The variables $t_i, t_j$ which appear in the above display are the coordinates of this vector.
\end{prop}

The proof of Proposition \ref{prop:approximate-regularized-wilson-loops-genus-expansion} is very similar to the proof of Proposition \ref{prop:approximate-regularized-wilson-loops-genus-expansion}. We first show the following preliminary lemma.

\begin{lemma}\label{lemma:matrix-wick-GFF}
Let $A$ be a $\mrm{u}(N)$-valued 1-form $\gff(\R^n)$. For $m \geq 1$ even and $\phi_1, \ldots, \phi_m \in \Omega^1 \dot{H}^{-1}(\manifold)$, we have that
\begin{equs}
\E\Big[(A, \phi_1) \otimes \cdots \otimes (A, \phi_m) \Big] = \frac{(-1)^{\frac{m}{2}}}{N^{\frac{m}{2}}}\sum_{\pi \in \mc{P}(m)} \rho(\pi) \prod_{\{i, j\} \in \pi} (\phi_i, \phi_j)_{\Omega^1 \dot{H}^{-1}(\manifold)}.
\end{equs}
\end{lemma}
\begin{proof}
By definition of the GFF, $(X_1, \ldots, X_m)$ is a jointly Gaussian collection of $\frkg$-valued random variables. We may expand $X_j = X_j^a E_a$, so that $(X_j^a, j \in [m], a \in [N^2])$ is a jointly Gaussian collection of $\R$-valued random variables. Thus by Wick's formula, we have that
\begin{equs}
\E\Big[X_1 \otimes \cdots \otimes X_m\Big] &= \E\Big[X_1^{a_1} \cdots X_m^{a_m}\Big] E_{a_1} \otimes \cdots \otimes E_{a_m} \\
&= \sum_{\pi \in \mc{P}(m)} \prod_{\{i, j\} \in \pi} \E[X_i^{a_i} X_j^{a_j}] E_{a_1} \otimes \cdots \otimes E_{a_m}.
\end{equs}
Note that $X^a_j = (A^a, \phi_j)$, where $A^a$ arises as $A = A^a E_a$, so that $(A^a, a \in [N^2])$ is a collection of independent $\R$-valued 1-form GFFs. As a consequence, we have that
\begin{equs}
\E[X_i^{a_i} X_j^{a_j}] &= \delta^{a_i a_j} \E\Big[(A^{a_i}, \phi_i) (A^{a_j}, \phi_j)\Big] \\
&= \delta^{a_i a_j} (\phi_i, \phi_j)_{\Omega^1 \dot{H}^{-1}(\manifold)}.
\end{equs}
Substituting this into our previous formula, we obtain
\begin{equs}
\E\Big[X_1 \otimes \cdots \otimes X_m\Big] &= \sum_{\pi \in \mc{P}(m)} \prod_{\{i, j\} \in \pi} \delta^{a_i a_j} (\phi_i, \phi_j)_{\Omega^1 \dot{H}^{-1}(\manifold)}  E_{a_1} \otimes \cdots \otimes E_{a_m} \\
&= \frac{(-1)^{\frac{m}{2}}}{N^{\frac{m}{2}}} \sum_{\pi \in \mc{P}(m)} \rho(\pi) \prod_{\{i, j\} \in \pi} (\phi_i, \phi_j)_{\Omega^1 \dot{H}^{-1}(\manifold)} , 
\end{equs}
where we applied Lemma \ref{lemma:casimir-result} in the second identity. 
\end{proof}

Using Lemma \ref{lemma:matrix-wick-GFF}, we may now prove Proposition \ref{prop:big-loop-surface-sum-representation}.

\begin{proof}[Proof of Proposition \ref{prop:big-loop-surface-sum-representation}]
We show the result for a single loop, as the general result follows by a straightforward generalization of the argument. The condition on the norm of $\Gamma$ ensures that the infinite series that arises is absolutely convergent. Recalling the series expansion of the holonomy \eqref{eq:big-loop-holonomy-series-expansion}, we have that
\begin{equs}
\E[W_\Gamma(A)] = \sum_{m=0}^\infty \int_{\Delta_m(1)} dt \Tr \E\Big[ (A, \Gamma(t_m)) \cdots (A, \Gamma(t_1))\Big]  
\end{equs}
When $m$ is odd, the expectation is zero by parity, since $A \stackrel{d}{=} - A$. When $m$ is even, we have by Lemmas \ref{lemma:trace-product-and-tensor-product}, \ref{lemma:matrix-wick-GFF}, \ref{cor:trace-rho-sigma}, and \ref{lemma:vertices-equals-cycles} that
\begin{equs}
\Tr \E\Big[ (A, \Gamma(t_m)) \cdots (A, \Gamma(t_1))\Big] &= \E\Big[ \Tr \Big((A, \Gamma(t_m)) \otimes \cdots \otimes (A, \Gamma(t_1)) \rho((1 \cdots m))\Big)\Big] \\
&=  \frac{(-1)^{\frac{m}{2}}}{N^{\frac{m}{2}}}\sum_{\pi \in \mc{P}(m)} \Tr \big(\rho(\pi) (1 \cdots m)\big)  \prod_{\{i, j\} \in \pi} (\Gamma(t_i), \Gamma(t_j))_{\Omega^1 \dot{H}^{-1}(\R^n)} \\
&= (-1)^{\frac{m}{2}} \sum_{\pi \in \mc{P}(m)} N^{V(\pi) - E(\pi)} \prod_{\{i, j\} \in \pi} (\Gamma(t_i), \Gamma(t_j))_{\Omega^1 \dot{H}^{-1}(\R^n)}.
\end{equs}
The desired result now follows.
\end{proof}

Note that if the big loop $\Gamma$ is of the form $\Gamma = \codif S$, where $S \colon [0, 1] \ra \Omega^2 L^2(\manifold)$ is what we call a ``big surface", then for any $t, s \in [0, 1]$, we have that (recall the projection $E$ from Definition \ref{def:E-E-star-projections}, as well as Lemma \ref{lemmma:projections-bounded-sobolev-space})
\begin{equs}
(\Gamma(t), \Gamma(s))_{\dot{H}^{-1(\manifold)}} = (\codif S(t), \codif S(s))_{\dot{H}^{-1(\manifold)}} = (S(t), d\codif (-\Delta)^{-1} S(s))_{L^2(\manifold)} = (S(t), E S(s))_{L^2(\manifold)}.
\end{equs}
Note that $\codif S = \codif E S$, and thus we may replace $S$ by $ES$, in which case we have
\begin{equs}
(\Gamma(t), \Gamma(s))_{\dot{H}^{-1}(M)} = (S(t), S(s))_{L^2(M)}.
\end{equs}
Thus if a big loop $\Gamma \in \ms{G}$ is of the form $\Gamma = \codif S$ where $S$ is such that $dS =0$, we then have the alternative formula for the Wilson big loop expectation:
\begin{equs}
\E[W_\Gamma(A)] = \sum_{m=0}^\infty (-1)^m \sum_{\pi \in \mc{P}(2m)} N^{V(\pi) - E(\pi)} \int_{\Delta_{2m}(1)} dt \prod_{\{i, j\} \in \pi} (S(t_i), S(t_j))_{L^2(\manifold)},
\end{equs}
and similarly for the case of multiple big loops. We may think of the above as a ``big surface expectation".






\section{Lattice Gaussian forms} \label{sec:latticeforms}


In this section, we discuss lattice versions of fractional Gaussian forms. We first introduce/review the needed lattice exterior calculus, and then define lattice Gaussian forms. Then, we discuss aspects of the scaling limit of lattice Gaussian forms to continuum Gaussian forms.

\subsection{Review of exterior calculus on the lattice}

In this subsection, we review the lattice analog of the continuum exterior calculus which was detailed in Section \ref{section:exterior-calculus}. At certain points, we will follow the discussion from \cite[Section 2.2]{GS2023}. See also \cite[Section 2]{chatterjee2020wilson} for another reference on lattice exterior calculus. In the following, let $\Lambda \sse \Z^n$ either be a finite box, or all of $\Z^n$. For $0 \leq k \leq n$, we consider the $k$-cells of $\Lambda$, which can be visualized as the $k$-dimensional hypercubes whose vertices are all contained in $\Lambda$. Each such hypercube gives rise to two $k$-cells, one positively oriented and one negatively oriented. See Figure \ref{figure:oriented_cell_examples} for some examples of oriented cells.

\begin{figure}
    \centering
    \includegraphics[width=.5\linewidth, page=1]{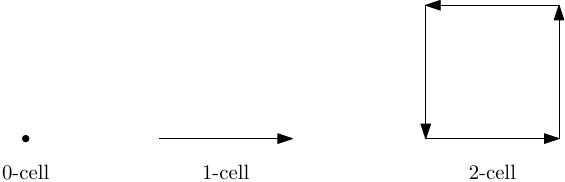}
    \caption{Examples of $0$, $1$, and $2$-cells (i.e. vertices, edges, and plaquettes).}
    \label{figure:oriented_cell_examples}
\end{figure}

For instance, the $0$-cells are the vertices of $\Lambda$. In an abuse of notation, given a vertex $x \in \Lambda$, we will denote the positively oriented cell corresponding to $x$ just by $x$, while if needed, we will denote the negatively oriented cell corresponding to $x$ by $-x$ (not to be confused with the lattice point whose coordinates are the negatives of the coordinates of $x$). The $1$-cells are oriented edges of $\Lambda$. The positively oriented edges of $\Lambda$ are of the form $(x, x + e_i)$ for some standard basis vector $e_i$, where both $x, x + e_i \in \Lambda$. The negatively oriented edges are given by $(x + e_i, x)$, which we will also write as $-(x, x + e_i)$. 

To discuss general oriented cells of arbitrary dimensions as well as their containment properties, we will need to set up an algebraic framework. We stress that it may help to have a geometric picture in mind, and in particular to think about how our ensuing discussion looks on concrete low-dimensional examples. Of course, once we go to high dimensions, it may be hard to visualize various things, and that is why it will help to have the algebraic framework. We note that the following discussion will be reminiscent of the discussion on differential forms from Section \ref{section:exterior-calculus}.

First, given a $k$-dimensional hypercube $c$, there exists a ``lower left corner" $x$, which is the vertex of $c$ which is smallest in the lexicographic order on the vertices of $\Z^n$, or equivalently is such that there exists $i_1, \ldots, i_k \in [n]$ such that the $2^k$ vertices of $c$ are given by
\begin{equs}
x + b_1 e_{i_1} + \cdots + b_k e_{i_k}, ~~ b_1, \ldots, b_k \in \{0, 1\}.
\end{equs}
Thus, each hypercube $c$ (which we think of as an unoriented $k$-cell) may be represented by $(x, i_1, \ldots, i_k)$, where $1 \leq i_1 <\cdots < i_k \leq n$. We also write the positively-oriented $k$-cell arising from $c$ by $(x, i_1, \ldots, i_k)$, and its negatively-oriented counterpart by $-(x, i_1, \ldots, i_k)$. More generally, given a permutation $\sigma \in \symgrp_k$, we identify
\begin{equs}
(x, i_{\sigma(1)}, \ldots, i_{\sigma(k)}) = (-1)^{\mrm{sgn}(\sigma)} (x, i_1, \ldots, i_k).
\end{equs}
In particular, the same oriented $k$-cell may be represented in multiple ways, e.g. $(x, i_1, i_2) = -(x, i_2, i_1)$. 

\begin{definition}[Oriented and unoriented cells]
Let $\overrightarrow{C}^k(\Lambda)$ (resp. $C^k(\Lambda)$) denote the oriented (resp. unoriented) $k$-cells of $\Lambda$. 
\end{definition}

When discussing cell containment (i.e. specifying which oriented $(k-1)$-cells are on the boundary of a given $k$-cell), it will help to have a more general representation of oriented cells, where the vertex $x$ is not necessarily the ``lower left corner" of the corresponding hypercube. With this in mind, given a vertex $x$, distinct indices $i_1, \ldots, i_k \in [n]$ as before, and additionally signs $s_1, \ldots, s_k \in \{\pm 1\}$, we denote by $(x, s_1 i_1, \ldots, s_k i_k)$ the oriented cell whose vertices are
\begin{equs}
x + b_1 s_1 e_{i_1} + \cdots + b_k s_k e_{i_k}, ~~ b_1, \ldots, b_k \in \{0, 1\}.
\end{equs}
To have a visual example in mind, it may help to look at Figure \ref{figure:oriented_cell_algebraic_framework} while reading the following discussion. There exists some vertex $y$ such that $(x, s_1 {i_1}, \ldots, s_k {i_k}) = s (y, {i_1}, \ldots, {i_k})$ for some $s \in \{\pm 1\}$, which are given by:
\[ y = x - \sum_{j=1}^k \ind(s_j = -1) e_{i_j}, ~~ s = \prod_{j=1}^k s_j. \]
Geometrically, each $s_j e_{i_j}$ for which $s_j = -1$ corresponds to a negatively oriented basis vector emanating out of $x$. To obtain the lexicographically smallest vertex $y$ that the cell $(x, s_1 {i_1}, \ldots, s_k {i_k})$ contains, we simply need to subtract from $x$ each basis vector which has a negative orientation. The orientation of $(x, s_1 {i_1}, \ldots, s_k {i_k})$ is $s$ times the orientation of $(y, {i_1}, \ldots, {i_k})$.

\begin{figure}
    \centering
    \includegraphics[width=.5\linewidth, page=1]{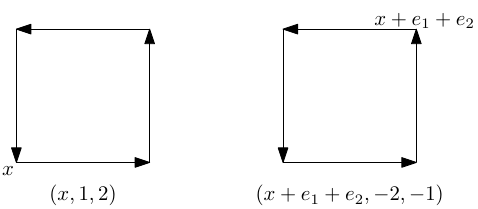}
    \caption{Different representations for the same oriented $2$-cell.}
    \label{figure:oriented_cell_algebraic_framework}
\end{figure}

We now discuss cell containment. While reading this definition, it may help to also look at Figure \ref{figure:oriented_cell_algebraic_framework_2}.

\begin{definition}[Cell containment]
Let $k \in [n]$. Let $c = s(x, s_{i_1} e_{i_1}, \ldots, s_{i_k} e_{i_k})$ be an oriented $k$-dimensional cell. The oriented $(k-1)$-dimensional cells which are contained in $c$ are:
\[ (-1)^j s (x, s_{i_1} {i_1}, \ldots, \widehat{s_{i_j} {i_j}}, \ldots, s_{i_k} {i_k}) , ~~ j \in [k], \]
where here the hat denotes omission. If $c'$ is contained in $c$, we write $c' \in c$ or $c \ni c'$.
\end{definition}

\begin{remark}
In this definition, we have to consider the general representation of an oriented cell $c = s (x, s_1 i_1, \ldots, s_k i_k)$, because not all $(k-1)$-dimensional cells on the boundary of contain the vertex $x$. As previously recommended, see Figure \ref{figure:oriented_cell_algebraic_framework_2}.
\end{remark}

For another example of cell containment, see Figure \ref{figure:oriented_cell_algebraic_framework_3}.

\begin{figure}
    \centering
    \hspace{10mm}\includegraphics[width=.5\linewidth, page=2]{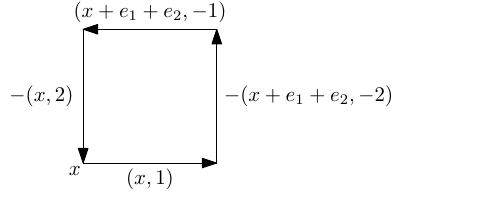}
    \caption{The same plaquette as in Figure \ref{figure:oriented_cell_algebraic_framework}, but now we label the oriented edges which are on the boundary of the plaquette. Two of the edges are obtained by starting with the representation $(x, 1, 2)$ and removing one of the two coordinate directions, while the other two edges are obtained by doing the same, but with the representation $(x + e_1 + e_2, -2, -1)$.}\label{figure:oriented_cell_algebraic_framework_2}

    \hspace{10mm}\includegraphics[width=.5\linewidth, page=3]{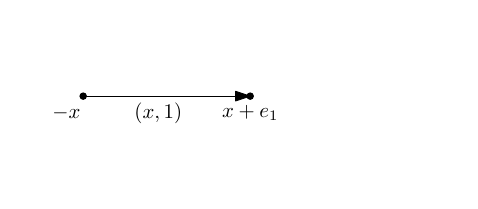}
    \vspace{-15mm}
    \caption{An oriented edge $(x, 1) = -(x+e_1, -1)$, with the two oriented $0$-cells that it contains: $x + e_1$ and $-x$. In particular, note the convention that the incoming vertex is positively oriented, while the outgoing vertex is negatively oriented.}\label{figure:oriented_cell_algebraic_framework_3}
    
\end{figure}


Next, we define the boundary map, which plays a fundamental role in lattice exterior calculus. For example, we will later define the lattice exterior derivative using this map.

\begin{definition}[Boundary map]
For $1 \leq k \leq n$, we define the boundary map $\ptl$ which maps oriented $k$-cells to formal sums of oriented $(k-1)$-cells. Algebraically, $\ptl$ is defined as
\begin{equs}\label{eq:boundary-operator}
\ptl c' := \sum_{c \in c'} c.
\end{equs}
In words, $\ptl$ associates a given $k$-cell $c$ to the formal sum of the oriented $(k-1)$-cells on the boundary of $c$. See Figures \ref{figure:boundary_examples_1} and \ref{figure:boundary_examples_2} for visual examples. We also define $\ptl$ on oriented $0$-cells to just map everything to zero. By linearity, we may also define $\ptl$ on the space of formal sums of oriented cells.
\end{definition}

\begin{figure}
    \centering
    \includegraphics[width=.65\linewidth, page=1]{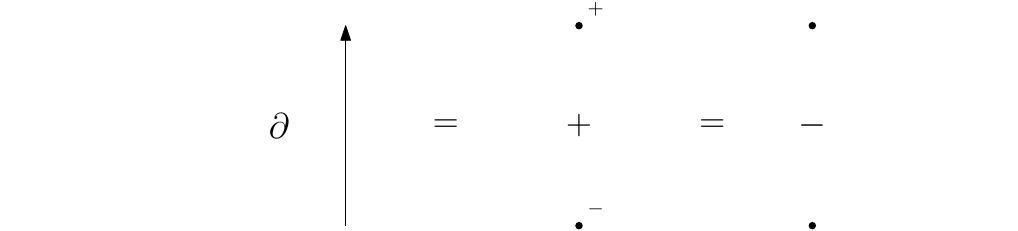}
    \caption{The boundary of an edge (i.e. $1$-cell) is the formal sum of its two vertices. By convention, the incoming vertex appears in its positive orientation, while the outgoing vertex appears in its negative orientation. In the middle, the orientations are denoted by the small $\pm$ symbols to the right and above each vertex. In the right, both vertices are positively oriented, and so we instead formally subtract the bottom vertex.}
    \label{figure:boundary_examples_1}
    \includegraphics[width=.65\linewidth, page=2]{images/boundary_examples.pdf}
    \caption{The boundary of a plaquette (i.e. $2$-cell) is the formal sum of its four edges.}
    \label{figure:boundary_examples_2}
\end{figure}


Having set up enough of the algebraic framework, we can now state the following key property of the boundary map $\ptl$. See Figure \ref{figure:oriented_cube} for a visual example to follow along.

\begin{lemma}\label{lemma:boundary-boundary-zero}
We have that $\ptl^2 = 0$.
\end{lemma}
\begin{proof}
It suffices to assume $2 \leq k \leq n$. Let $c$ be an oriented $k$-cell. By definition, we have that
\begin{equs}
\ptl^2 c = \ptl \bigg(\sum_{c' \in c} c' \bigg) = \sum_{c' \in c} \sum_{c'' \in c'} c''. 
\end{equs}
Fix an oriented $(k-2)$-cell $c''$ which appears in the above sum, i.e. such that there exists an oriented $(k-1)$-cell $c'$ such that $c'' \in c' \in c$. Suppose that $c''$ has the representation
\begin{equs}
c'' = (x,  {i_1}, \ldots, {i_{k-2}}).
\end{equs}
Then the assumption implies that there exists $i_{k-1}, i_k \in [n] \backslash \{i_1, \ldots, i_{k-2}\}$, $s_{k-1}, s_k, s \in \{\pm 1\}$ such that
\begin{equs}
c = s (x, i_1, \ldots, i_{k-2}, s_{k-1} i_{k-1}, s_k i_k).
\end{equs}
(Geometrically, this says that $c$ can be obtained from $c''$ by adding two more basis edges.) Define
\begin{equs}
c_1' := (-1)^{k-1} s(x, i_1, \ldots, i_{k-2}, s_k i_k), ~~ c_2' := (-1)^{k} s(x, i_1, \ldots, i_{k-2}, s_{k-1} i_{k-1}).
\end{equs}
By definition, $c_1', c_2' \in c$. Moreover, observe that 
\begin{equs}
c_1' \text{ contains } s (x, i_1, \ldots, i_{k-2}) \quad \text{ and } \quad  c_2' \text{ contains } -s (x, i_1, \ldots, i_{k-2}).
\end{equs}
Thus exactly one of $c_1', c_2'$ contains $c''$, while the other contains $-c''$. Moreover, any given oriented $(k-2)$-cell can only appear at most once in the sum for $\ptl^2 c$. Thus, any $c''$ which appears in the sum for $\ptl^2 c$ is cancelled out by a corresponding $-c''$ term. The desired result follows.
\end{proof}

\begin{figure}
    \centering
    \includegraphics[width=.35\linewidth, page=1]{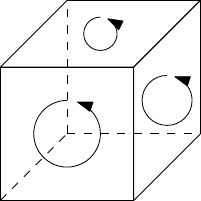}
    \caption{An oriented $3$-cell (i.e. cube) with orientations drawn on three of the plaquettes (the other three orientations are omitted for practical reasons). Lemma \ref{lemma:boundary-boundary-zero} is a general form of the classic fact that each edge of the cube appears exactly twice, once in each orientation. Thus each edge is ``cancelled" out, which is the visual way to say that $\ptl^2 = 0$.}\label{figure:oriented_cube}
\end{figure}

Having set up enough preliminaries about oriented cells, we now define lattice differential forms. 

\begin{definition}[Lattice differential forms]
Let $0 \leq k \leq n$. A $k$-form on $\Lambda$ is a function $f : \overrightarrow{C}^k(\Lambda) \ra \R$ such that $f(-c) = -f(c)$ for all $c \in \overrightarrow{C}^k(\Lambda)$. We denote the space of finitely-supported $k$-forms by $\Omega^k(\Lambda)$.
\end{definition}

Given a $k$-form $f \in \Omega^k (\Lambda)$, we may extend $f$ to formal sums of oriented $k$-cells just by linearity, i.e.:
\begin{equs}
f\bigg(\sum_{c \in \overrightarrow{C}^k(\Lambda)} \alpha(c) c \bigg) := \sum_{c \in \overrightarrow{C}^k(\Lambda)} \alpha(c) f(c).
\end{equs}
The boundary operator $\ptl$ directly leads to a definition of the lattice exterior derivative, as follows.

\begin{definition}[Lattice exterior derivative]
For $0 \leq k \leq n$, we define the lattice exterior derivative $d : \Omega^k(\Lambda) \ra \Omega^{k+1}(\Lambda)$ as follows. If $k = n$, then $d \equiv 0$. Otherwise, if $k \leq n-1$, we define
\begin{equs}\label{eq:exterior-derivative}
(df)(c) := f(\ptl c) = \sum_{c' \in c} f(c'), ~~ c \in \overrightarrow{C}^{k+1}(\Lambda).
\end{equs}
To be precise, if $c' \in c$ is such that $c' \not \sse \Lambda$, we set $f(c') := 0$ (equivalently, we extend $f$ to $\Z^n$ by zero).
One easily checks that $(df)(-c) = -df(c)$, so that we have indeed defined a lattice differential form.
\end{definition}

From the definition of $d$ and Lemma \ref{lemma:boundary-boundary-zero}, we directly obtain the following result.

\begin{lemma}
We have that $d^2 = 0$.
\end{lemma}

\begin{definition}
For $0 \leq k \leq n$, define an inner product on $\Omega^k(\Lambda)$ by
\begin{equs}
(f, g) := \sum_{c \in C^k(\Lambda)} f(c) g(c).
\end{equs}
\end{definition}

\begin{definition}[Lattice codifferential]
For $1 \leq k \leq n$, we define $d^* : \Omega^k(\Lambda) \ra \Omega^{k-1}(\Lambda)$ as the adjoint of $d$. I.e., it is the unique linear map which satisfies
\begin{equs}
(f, d^* g) = (df, g), ~~ f \in \Omega^{k-1}(\Lambda), g \in \Omega^k(\Lambda).
\end{equs}
For $k = 0$, we define $d^* g \equiv 0$ for $g \in \Omega^0(\Lambda)$.
\end{definition}

\begin{remark}
Our definition of $d^*$ is off by a sign compared to that of \cite{GS2023}. Thus our later definition of the lattice Laplacian will also be off by a sign compared to that of \cite{GS2023}. However, in both our definition of the Laplacian and that of \cite{GS2023}, the Laplacian will be negative definite.
\end{remark}

Note that $d^*$ has the following explicit form. Recall \eqref{eq:boundary-operator}. For $f \in \Omega^{k-1}(\Lambda), g \in \Omega^k(\Lambda)$, we have that
\begin{equs}
(df, g) = \sum_{c \in C^k(\Lambda)} (df)(c) g(c) = \sum_{c \in C^k(\Lambda)} \sum_{c' \in c} f(c') g(c) = \sum_{c' \in C^{k-1}(\Lambda)} f(c') \sum_{C^k(\Lambda) \ni c \ni c'} g(c).
\end{equs}
We thus see that
\begin{equs}\label{eq:codifferential-formula}
(d^* g)(c') = \sum_{C^k(\Lambda) \ni c \ni c'} g(c), ~~ c' \in \overrightarrow{C}^{k-1}(\Lambda).
\end{equs}
In particular, note that $d^*$ depends on $\Lambda$. We will usually omit this dependence in our notation.

\begin{definition}[Lattice Laplacian]
For $0 \leq k \leq n$, define $\Delta : \Omega^k(\Lambda) \ra \Omega^k(\Lambda)$ by $\Delta := -(dd^* + d^* d)$.
\end{definition}

As in the continuum, we have that
\begin{equs}\label{eq:lattice-laplacian-integration-by-parts}
(f, \Delta g) = (\Delta f, g) = - (df, dg) - (d^* f, d^* g),
\end{equs}
so that $\Delta$ is symmetric. If $\Lambda$ is finite, then the spaces $\Omega^k(\Lambda), 0 \leq k \leq n$ are finite dimensional. Thus, by the spectral theorem in finite dimensions, for each $0 \leq k \leq n$, there is an orthonormal basis of $\Omega^k(\Lambda)$ consisting of eigenvectors of $\Delta$. Moreover, by \eqref{eq:lattice-laplacian-integration-by-parts}, $\Delta$ is negative semidefinite, and thus every eigenvalue of $\Delta$ is non-positive. Another basic property of $\Delta$ is the following result.

\begin{prop}[Proposition 2.7 and Remark 2.8 of \cite{GS2023}]
Let $\Lambda$ be a finite box. For $0 < k \leq n$, $\Delta$ is strictly negative definite on $\Omega^k(\Lambda)$. In particular, $\Delta$ is invertible.
\end{prop}

Let $\Lambda$ be a finite box as in the proposition. Since $\Delta : \Omega^k(\Lambda) \ra \Omega^k(\Lambda)$ is invertible for $1 \leq k \leq n$, all of its eigenvalues are strictly less than zero, and thus combining this with the aforementioned spectral decomposition, we can make the following definition.

\begin{definition}[Fractional Laplacian]\label{def:lattice-fractional-laplacian}
Let $s \in \R$. For $1 \leq k \leq n$, we define the operator $(-\Delta)^s : \Omega^k(\Lambda) \ra \Omega^k(\Lambda)$ via the spectral decomposition of $-\Delta$. For $k = 0$, we define the operator $(-\Delta)^s$ via spectral decomposition on the subspace $\bigdot{\Omega}^0(\Lambda)$ spanned by those eigenvectors of $\Delta$ which correspond to a nonzero eigenvalue.
\end{definition} 

\begin{remark}
Note that $\bigdot{\Omega}^0(\Lambda)$ is also the subspace of mean-zero functions, since the kernel of $-\Delta : \Omega^0(\Lambda) \ra \Omega^0(\Lambda)$ is given by the constant functions (see \cite[Remark 2.4]{GS2023}).
\end{remark}

The space $\Omega^k(\Lambda)$ admits a Hodge decomposition, which is the analog of Theorem \ref{thm:hodge-decomposition}. For simplicity, we state the result for finite boxes $\Lambda$, but as mentioned in \cite[Remark 4.5]{GS2023}, with a little more effort one can obtain a decomposition in infinite volume. We first need some preliminary definitions and results.

\begin{definition}
For $1 \leq k \leq n$, define $\Omega^{k-1 \ra k} := d \Omega^{k-1} \sse \Omega^k$, and for $0 \leq k \leq n-1$, define $\Omega^{k +1 \ra k} := d^* \Omega^{k+1} \sse \Omega^k$.
\end{definition}

Next, we define the lattice analogs of the projections $E, E^*$ from Definition \ref{def:E-E-star-projections}. Here, we use slightly different notation, to align with \cite{GS2023}.

\begin{definition}
For $1 \leq k \leq n -1$, define $\pi_{k-1 \ra k} := dd^* (-\Delta)^{-1}$ and $\pi_{k+1 \ra k} := d^* d (-\Delta)^{-1}$.
\end{definition}

\begin{lemma}[Lemma 4.3 of \cite{GS2023}]\label{lemma:lattice-orthogonal-projections}
Let $\Lambda$ be a finite box. For $1 \leq k \leq n-1$, $\pi_{k-1 \ra k}$ and $\pi_{k+1 \ra k}$ are respectively the orthogonal projections of $\Omega^k$ onto $\Omega^{k-1 \ra k}$ and $\Omega^{k+1 \ra k}$.
\end{lemma}

We now state the lattice Hodge decomposition.

\begin{prop}[Lattice Hodge decomposition]\label{prop:lattice-hodge-decomposition}
Let $\Lambda$ be a finite box.
For $1 \leq k \leq n -1$, we have the orthogonal decomposition
\begin{equs}
\Omega^k(\Lambda) = \Omega^{k-1 \ra k}(\Lambda) \oplus \Omega^{k+1 \ra k}(\Lambda).
\end{equs}
When $k = n$, we have that
\begin{equs}
\Omega^n(\Lambda) = \Omega^{n-1 \ra n}(\Lambda).
\end{equs}
\end{prop}
\begin{proof}
For $k = n$, this is a direct consequence of \cite[Proposition 2.3]{GS2023}. Thus, we focus on the case $1 \leq k \leq n-1$. We have that
\begin{equs}
\pi_{k-1 \ra k} + \pi_{k+1 \ra k} = (dd^* + d^* d)(-\Delta)^{-1} = I,
\end{equs}
and thus for any $f \in \Omega^k(\Lambda)$, we have that
\begin{equs}
f = \pi_{k-1 \ra k} f + \pi_{k+1 \ra k} f \in \Omega^{k-1 \ra k}(\Lambda) \oplus \Omega^{k+1 \ra k},
\end{equs}
where the inclusion follows by Lemma \ref{lemma:lattice-orthogonal-projections}.
\end{proof}

Next, we discuss the lattice analog of equation \eqref{eq:hodge-laplacian-coordinatewise-formula}, which recall says that the Laplacian on $k$-forms reduces to the Laplacian applied to each of the coordinate functions. For simplicity, we will first focus on the case $\Lambda = \Z^n$. Afterwards, we briefly discuss the case of finite $\Lambda$, where the Laplacian looks slightly different on boundary terms (see \cite[Section 3.3]{GS2023}).

To begin, we give formulas for the operators $dd^*$ and $d^* d$. By \eqref{eq:boundary-operator} and \eqref{eq:exterior-derivative}, we have that for $f \in \Omega^k(\Z^n)$ with $1 \leq k \leq n$,
\begin{equs}\label{eq:d-codif-formula}
(dd^* f)(c) = \sum_{c' \in c} (d^* f)(c') = \sum_{c' \in c} \sum_{c'' \ni c'} f(c''), ~~ c \in \overrightarrow{C}^k(\Z^n).
\end{equs}
In words, the second sum above is over those oriented $k$-cells $c''$ such that there is an oriented $(k-1)$-cell $c'$ contained in both $c$ and $c''$. Note that $c$ itself appears a number of times, while each $c'' \neq c$ will appear at most once, because if such a $c'$ exists for $c, c''$, it is unique. Similarly, for $0 \leq k \leq n-1$ and $f \in \Omega^k(\Z^n)$, we have that
\begin{equs}\label{eq:codif-d-formula}
(d^* d f)(c) = \sum_{c' \ni c} (df)(c') = \sum_{c' \ni c} \sum_{c'' \in c'} f(c''), ~~ c \in \overrightarrow{C}^k(\Z^n).
\end{equs}
In words, the second sum above is over those oriented $k$-cells $c''$ such that there exists an oriented $(k+1)$-cell $c'$ which contains both $c$ and $c''$. See Figures \ref{figure:laplacian_examples_1}-\ref{figure:laplacian_examples_2} for visualizations of these two sums in the case $n = 2, k = 1$.

\begin{figure}
    \centering
    \includegraphics[width=.3\linewidth, page=1]{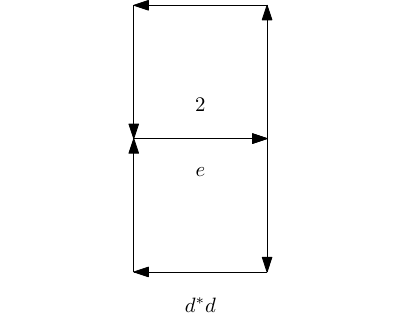}
    \caption{Here, we display the edges appearing in the sum for $(d^* d f)(e)$ for $n = 2, k = 1$. The label $2$ on the edge $e$ means that $e$ appears twice in the sum.}
    \label{figure:laplacian_examples_1}
    \hspace{5mm}
    \includegraphics[width=.3\linewidth, page=2]{images/laplacian_examples.pdf}
    \caption{Here, we display the edges appearing in the sum for $(d d^* f)(e)$ for $n = 2, k = 1$. The label $2$ on the edge $e$ means that $e$ appears twice in the sum.}
    \label{figure:laplacian_examples_2}
    \hspace{5mm}

    \includegraphics[width=.3\linewidth, page=3]{images/laplacian_examples.pdf}
    \caption{When computing $((dd^* + d^*d)f)(e)$, we observe that the contribution from any edge which is not parallel to $e$ gets cancelled out. What remains is a sum over all edges parallel to $e$.}
    \label{figure:laplacian_examples_3}
\end{figure}

Now, when we add the two sums in \eqref{eq:d-codif-formula} and \eqref{eq:codif-d-formula}, it turns out that there is a sizable cancellation and many of the terms are cancelled out. To express this cancellation, we first make the following definition.

\begin{definition}
Let $0 \leq k \leq n$ and let $c, c' \in C^k(\Z^n)$. We say that $c \sim c'$ if $c'$ can be obtained from $c$ by a shift of $\pm e_i$ for some $i \in [n]$. This defines a graph structure on $C^k(\Z^n)$, which is isomorphic to $\Z^n$.
\end{definition}

\begin{remark}
One should think of $c \sim c'$ as saying that $c, c'$ are neighboring $k$-cells which are ``parallel".
\end{remark}

In the following, we identify $k$-forms $f \in \Omega^k(\Z^n)$ with functions $f \colon C^k(\Z^n) \ra \R$ on the set unoriented $k$-cells. In particular, the Laplacian $\Delta$ is defined on such functions.

\begin{prop}\label{prop:lattice-laplacian-coordinatewise}
For $0 \leq k \leq n$ and $f \colon C^k(\Z^n) \ra \R$, we have that
\begin{equs}
(\Delta f)(c) = \sum_{c' \sim c} (f(c') - f(c)), ~~ c \in C^k(\Z^n).
\end{equs}
\end{prop}

We do not prove this proposition here; we just remark that the result is due to the sizable cancellation we previously mentioned. For a visual example when $n = 2, k = 1$, see Figure \ref{figure:laplacian_examples_3}.  Ultimately, the cancellation is due to the fact that any $c''$ appearing in the sums \eqref{eq:d-codif-formula} and \eqref{eq:codif-d-formula} which is not ``parallel" to $c$ in fact appears once in a positive orientation and once in a negative orientation, while any $c''$ which is parallel (so that $c'' \sim c$) appears exactly once across the two sums. At an algebraic level, the argument for this would be very similar to showing \eqref{eq:hodge-laplacian-coordinatewise-formula}, i.e. why the Laplacian on $k$-forms in $\R^n$ reduces to the Laplacian on each coordinate function. See \cite[Proposition 3.3]{GS2023} for a proof for general $n$ and $k = 1$ (in finite volume).

Next, we briefly discuss the lattice $k$-form Laplacian on a finite box $\Lambda$. For any $c \in C^k(\Z^n)$ which is sufficiently far away from the boundary (the needed condition is that every $c''$ appearing in the sums \eqref{eq:d-codif-formula} and \eqref{eq:codif-d-formula} is also in $\Lambda$), the formula for $(\Delta f)(c)$ is exactly as in Proposition \ref{prop:lattice-laplacian-coordinatewise}. The only difference appears when $c$ is close to the boundary. This is best explained with a pictorial example -- See Figure \ref{figure:laplacian_boundary_example}. 

\begin{figure}
    \centering
    \includegraphics[width=.5\linewidth, page=1]{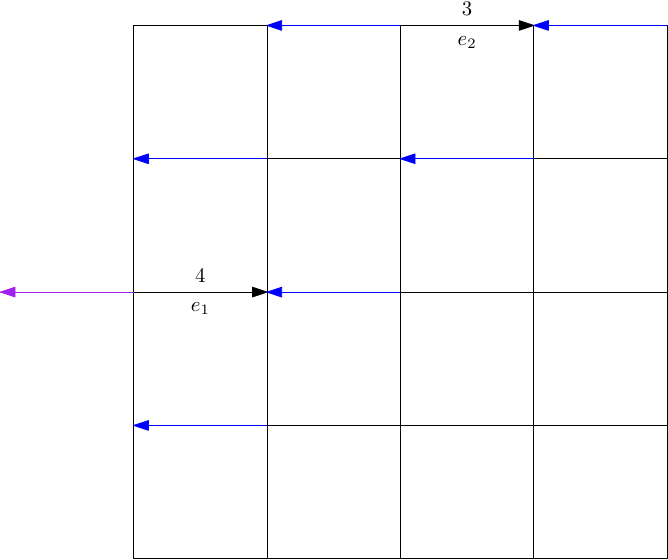}
    \caption{A finite box $\Lambda$ and two edges $e_1, e_2$ which touch the boundary of $\Lambda$. The edges $e_1, e_2$ are also labeled with the total number of times each edge appears in the sums \eqref{eq:d-codif-formula} and \eqref{eq:codif-d-formula}, where now the sums would be restricted to $c', c''$ lying in $\Lambda$. The blue edges are the cells $c'' \neq e_1, e_2$ which appear in the sums (each appears exactly once). Note that while $e_1$ appears four times, it only has three ``neighbors" (i.e. $c''$ such that $c'' \sim e_1$) in $\Lambda$. Thus we should really imagine that there is a fourth neighbor of $e_1$, which is the purple edge, and that $f$ is set to zero on this edge.}\label{figure:laplacian_boundary_example}
\end{figure}

As explained in Figure \ref{figure:laplacian_boundary_example}, for $k \geq 1$, the $k$-form Laplacian should be interpreted as having mixed free and Dirichlet boundary conditions. Essentially, one would define the $k$-cell ``boundary" of $\Lambda$ to be the set of $k$-cells $c \notin \Lambda$, for which there exists a $k$-cell $c' \in \Lambda$ with $c \sim c'$, and one would imagine that any given $k$-form is set to zero on this boundary. This is because the existence of such a cell $c \notin \Lambda$ for a given $c' \in \Lambda$ is precisely what causes the mismatch between the number of copies of $c'$ and its number of neighbors in $\Lambda$, as illustrated in Figure \ref{figure:laplacian_boundary_example} for the edge $e_1$. In the special case of $1$-forms, the boundary would be the set of edges not in $\Lambda$ but which have a vertex in $\Lambda$. This is explained further in \cite[Section 3]{GS2023}.

\begin{remark}[Boundary conditions]
In this subsection, we always worked with free boundary conditions for our $k$-forms. We could have also worked with zero boundary conditions, as also done in \cite{GS2023}. The discussion would have then been almost exactly the same -- see \cite[Section 2.2]{GS2023} for more details.
\end{remark}

In later discussion, we will need the following observation.

\begin{remark}\label{remark:horizontal 2-form}
Let $f \in \Omega^2 \Lambda$ be a $2$-form which is supported on plaquettes in the 1--2 plane. I.e., for all oriented $2$-cells $c = (x, i, j)$ where $\{i, j\} \neq \{1, 2\}$, we have that $f(c) = 0$. Then
\begin{equs}
\codif d f (c) = \sum_{\substack{c' \sim c \\ c', c \text{ not coplanar}}} (f(c') - f(c)).
\end{equs}
Here, by ``coplanar", we mean that the plaquettes $c, c'$ are not in the same plane, so that $c'$ is obtained from $c$ by a shift in the $i$th coordinate direction for some $i \neq 1, 2$. In the case of 3D, the two $c'$ appearing in the sum would be the plaquettes which are above and below $c$ (assuming that both plaquettes lie in $\Lambda$).
\end{remark}

\begin{remark}[General mesh size]\label{remark:general-mesh-size}
Given $\Lambda \sse \varep \Z^n$ for $\varep > 0$, we may also define oriented and unoriented $k$-cells $\overrightarrow{C}^k(\Lambda), C^k(\Lambda)$, $k$-forms $\Omega^k(\Lambda)$, exterior derivative, Laplacian, etc. Our convention will be that discrete differential operators are defined without any additional $\varep$ factors, e.g. to define the exterior derivative $df$, we use the exact same formula as in \eqref{eq:exterior-derivative}.
\end{remark}

\subsection{Lattice Gaussian forms}\label{section:lattice-gaussian-forms}

In this subsection, we define lattice fractional Gaussian forms. In the following, as before, we will identify $k$-forms with functions $f \colon C^k(\Z^n) \ra \R$ on the set of unoriented $k$-cells. Also, recall our definition of the lattice fractional Laplacian (Definition \ref{def:lattice-fractional-laplacian}).

\begin{definition}[Lattice fractional Gaussian $k$-form]\label{def:lattice-gaussian-k-form}
Let $\Lambda$ be a finite box, $\beta \geq 0$, and $s \in \R$. Let $1 \leq k \leq n$. We say that $A$ is a $k$-form $\fgf_s(\Lambda; \beta)$, also denoted $\fgf_s^k(\Lambda; \beta)$, if it is a random variable on $\Omega^k(\Lambda)$ with law
\begin{equs}
Z_{\Lambda, \beta}^{-1} \exp\Big(-\frac{\beta}{2}(A, (-\Delta)^s A)\Big) \prod_{c \in C^k(\Lambda)} dA_c.
\end{equs}
Here, $\prod_{c \in C^k(\Lambda)} dA_c$ is the product Lebesgue measure on $\R^{C^k(\Lambda)}$. For $k = 0$, we say that $A$ is a $0$-form $\fgf_s(\Lambda; \beta)$, also denoted $\fgf_s^0(\Lambda; \beta)$, if it is a random variable on $\bigdot{\Omega}^0(\Lambda)$ with law
\begin{equs}
Z_{\Lambda, \beta}^{-1} \exp\Big(-\frac{\beta}{2}(A, (-\Delta)^s A)\Big) \prod_{c \in C^k(\Lambda)} dA_c.
\end{equs}
In the special case $s = 0$, we will refer to $\fgf_0^k(\Lambda; \beta)$ as a $k$-form white noise, and when $s = 1$, we will refer to the $\fgf_1^k(\Lambda;\beta)$ as a $k$-form GFF. As usual in statistical mechanics, we will refer to $\beta$ as the inverse temperature.
\end{definition}

By considering the spectral decomposition of $-\Delta$, we may explicitly construct instances of $\fgf_s^k(\Lambda; \beta)$, which are analogous to that of the Fourier series representations of $\fgf_s^k(\T^n)$ from Remark \ref{remark:fgf-fourier-series}. Abstractly, for $0 \leq k \leq n$, let $0 < \lambda_1 \leq \lambda_2 \leq \cdots \leq \lambda_{M_k}$ be the eigenvalues of $-\Delta$ with corresponding eigenvectors $f_1, \ldots, f_{M_k} \in \Omega^k(\Lambda)$ (when $k = 0$, we restrict to the space $\bigdot{\Omega}^0(\Lambda)$ as usual). Here, $M_k$ is the dimension of $\Omega^k(\Lambda)$ ($\bigdot{\Omega}^0(\Lambda)$ if $k = 0$). Let $(Z_j, j \in [M_k])$ be a collection of i.i.d. $\N(0, 1)$ random variables. Then
\begin{equs}
A = \sqrt{\frac{1}{\beta}} \sum_{j \in [M_k]} \lambda_j^{-\frac{s}{2}} Z_j f_j 
\end{equs}
is an $\fgf_s^k(\Lambda; \beta)$. From this explicit representation, the following analog of Lemma \ref{lemma:laplacian-of-fgf} directly follows.

\begin{lemma}
Let $\Lambda$ be a finite box, $0 \leq k \leq n$, $s, s' \in \R$, $\beta \geq 0$. If $A$ is an $\fgf_s^k(\Lambda; \beta)$, then $(-\Delta)^{s'} A$ is an $\fgf_{s - 2s'}^k(\Lambda; \beta)$.
\end{lemma}

We make the following definition, which is the analog of Definition \ref{def:projection-fgf}.

\begin{definition}[Projections of lattice fractional Gaussian forms]
Let $\Lambda$ be a finite box, $0 \leq k \leq n$, $s \in \R$, $\beta \geq 0$. We say that $B$ is an $\big(\fgf^k_s(\Lambda; \beta)\big)_{d = 0}$ if $B$ has the same law as $\pi_{k-1 \ra k} A$, where $A$ is an $\fgf_s^k(\Lambda; \beta)$. Similarly, we say that $B$ is an $\big(\fgf_s^k(\Lambda; \beta)\big)_{\codif = 0}$ if it has the same law as $\pi_{k+1 \ra k} A$.
\end{definition}

Next we discuss what happens when we take $d$ or $\codif$ of a $\fgf_1^k(\Lambda; \beta)$, which will be analogous to Lemma \ref{lemma:projected-white-noise}.

\begin{lemma}\label{lemma:lattice-projected-white-noise}
Let $\Lambda$ be a finite box, $s \in \R, \beta \geq 0$. For $1 \leq k \leq n$, let $A$ be an $\fgf_s^{k-1}(\Lambda; \beta)$. Then $dA$ is an $\big(\fgf_{s-1}^k(\Lambda; \beta)\big)_{d=0}$. Similarly, for $0 \leq k \leq n-1$, let $A$ be $\fgf_s^{k+1}(\Lambda; \beta)$. Then $\codif A$ is an $\big(\fgf_{s-1}^k(\Lambda; \beta))_{\codif = 0}$.
\end{lemma}

We omit the proof of this result as it is essentially the same as the proof of Lemma \ref{lemma:projected-white-noise}. In words, this says that the exterior derivative or codifferential of a $k$-form GFF is in fact a projected version of $(k-1)$ or $(k+1)$-form white noise.

\begin{remark}
We remark that the exterior derivative $dA$ for a $1$-form $\gff(\Lambda)$ plays an important role in the work of \cite{GS2023}. There, they give $dA$ the name of ``gradient spin-wave", see \cite[Definition 2.13]{GS2023}. \cite[Proposition 4.4]{GS2023} is a special case of the previous lemma to the case $s = 1, k = 1$, which says that the exterior derivative of a $1$-form GFF is a projected version of $2$-form white noise.
\end{remark}

\begin{remark}
Similar to Remark \ref{remark:general-mesh-size}, given $\Lambda \sse \varep \Z^n$ for $\varep > 0$, we may also define lattice fractional Gaussian forms on $\Lambda$, in exactly the same way as was done on the unit lattice. 
\end{remark}

In Section \ref{sec:confinement}, we will further discuss various properties of the lattice $1$-form GFF.

\subsection{Scaling limit of lattice Gaussian forms}

In this subsection, we describe how to scale lattice Gaussian forms to obtain the continuum $\fgf_s^k(\R^n)$ in the limit as the mesh size of the lattice converges to zero. The discussion will be heuristic and will focus mostly on how to arrive at the correct scaling factor.

Let $\varep > 0$ and consider the lattice $\varep \Z^n$. First, we discuss discretizations of differential forms. Fix $0 \leq k \leq n$.  Let $\phi \in \Omega^k C^\infty_c(\R^n)$, and define the discretization $\phi_\varep : C^k(\varep \Z^n) \ra \R$ by
\begin{equs}\label{eq:discretization-k-form}
\phi_\varep(c) := \varep^k \phi_{i_1 \cdots i_k}(x), ~~ c = (x, i_1, \ldots, i_k) \in C^k(\varep \Z^n).
\end{equs}
The reason for the $\varep^k$ factor is because $c$ is a hypercube of volume $\varep^k$; in particular, we should really think of $\phi_\varep(c)$ for $c = (x, i_1, \ldots, i_k)$ as the value of the multilinear map associated to $\phi(x)$ applied to the tangent vectors $\varep e_{i_1}, \ldots, \varep e_{i_k}$ (recall Remark \ref{remark:k-form-alternating-multilinear-map}). Given $\phi, \psi \in \Omega^k C^\infty_c(\R^n)$, note by Riemann integral approximations we have that
\begin{equs}
\varep^{n-2k} (\phi_\varep, \psi_\varep) = \varep^n \sum_{\substack{x \in \varep \Z^n \\ 1 \leq i_1 < \cdots < i_k \leq n}} \phi_{i_1 \cdots i_k}(x) \psi_{i_1 \cdots i_k}(x) \ra (\phi, \psi) \text{  as $\varep \ra 0$.}
\end{equs}
The precise $\varep^{n-2k}$ factor will be important -- later on, when we discuss scaling limits of fractional Gaussian forms, we will take this $\varep^{n-2k}$ factor into account, and this will affect the way we are supposed to scale our fractional Gaussian form to expect convergence to a scaling limit. 

We first discuss the case of $k$-form white noise on the lattice. Fix $0 \leq k \leq n$. Let $\varep > 0$ and consider the lattice $\varep \Z^n$. Let $F_\varep$ be a $k$-form white noise on $\varep \Z^n$ with inverse temperature $\beta_\varep$, that is it assigns each positively oriented $k$-cell of $\varep \Z^n$ an independent $\N\big(0, \beta_\varep^{-1}\big)$ random variable. Given $\phi \in \Omega^k C^\infty_c(\R^n)$, we consider the following random variable:
\begin{equs}
\varep^{n-2k} (F_\varep, \phi_\varep) = \varep^{n-2k} \sum_{c \in C^k(\varep \Z^n)} F_\varep(c) \phi_\varep(c).
\end{equs}
Our objective is to find $\beta_\varep$ such that
$F_\varep$ converges in law to $k$-form white noise on $\R^n$, in the sense that
\begin{equs}
\varep^{n-2k} (F_\varep, \phi_\varep) \stackrel{d}{\ra} \N\big(0, \|\phi\|_{\Omega^k L^2}^2\big).
\end{equs}
Since each $(F_\varep, \phi_\varep)$ is a Gaussian random variable, it suffices to show that
the variances converge. Towards this end, we calculate
\begin{equs}
\mrm{Var}\big( \varep^{n-2k} (F_\varep, \phi_\varep)\big) = \varep^{2(n-2k)} \beta_\varep^{-1} \sum_{c \in C^k(\varep \Z^n)} \phi_\varep(c)^2.
\end{equs}
Since
\begin{equs}
\lim_{\varep \downarrow 0} \varep^{n} \sum_{c \in C^k(\varep \Z^n)} (\varep^{-k} \phi_\varep(c))^2 = \int_{\R^n} |\phi(x)|^2 dx = \|\phi\|_{\Omega^k L^2}^2,
\end{equs}
we see that we must take 
\begin{equs}
\beta_\varep = \varep^{n - 2k}.
\end{equs}
As a sanity check, in the case $k = 0$, $n = 1$, we see that $\beta_\varep = \varep$, which is consistent with the classical central limit theorem (recall that each $F_\varep(c)$ is distributed like $\N(0, \beta_\varep^{-1})$).

Next, we use these considerations to decide how one should scale the lattice $k$-form GFF to obtain the continuum $k$-form GFF in the scaling limit. Given $\Lambda_\varep \sse \varep \R^n$, let $A_\varep$ be a $k$-form $\gff(\Lambda_\varep)$ at inverse temperature $\beta_\varep$. As before, let $\phi \in \Omega^k C^\infty_c(\R^n)$, and let $\phi_\varep : C^k(\varep \Z^n) \ra \R$ be its discretization. As $\varep \downarrow 0$, we want that
\begin{equs}
\varep^{n-2k} (A_\varep, \phi_\varep) = \varep^{n-2k} \sum_{c \in C^k(\Lambda)} A_\varep(c) \phi_\varep(c) \stackrel{d}{\ra} \N\big(0, \|\phi\|_{\Omega^k \dot{H}^{-1}(\R^n)}^2\big).
\end{equs}
We want to use our previous discussion about white noise, at least heuristically. Thus, suppose that $\phi$ is of the form $\phi = d^* \psi$, for $\psi \in \Omega^{k+1} C^\infty_c(\R^n)$. Then 
\begin{equs}\label{eq:codifferential-approximation}
\phi_\varep = (\codif \psi)_\varep \approx \frac{1}{\varep^2} \codif \psi_\varep.
\end{equs}
Here, to be clear, in the middle term, the codifferential is applied in the continuum, while in the last term, the codifferential is applied to the lattice $(k+1)$-form $\psi_\varep$. The extra $\varep^{-2}$ factor arises from two sources. First, our convention is that lattice differential operations do not involve the mesh size -- recall Remark \ref{remark:general-mesh-size}. This accounts for one factor of $\varep^{-1}$. The other factor arises due to our definition of the discretization of continuum differential forms \eqref{eq:discretization-k-form}. Now, heuristically replacing $\phi_\varep$ by $\varep^{-2} \codif \psi_\varep$ as suggested by \eqref{eq:codifferential-approximation}, we then want to show that
\begin{equs}
\varep^{n-2(k+1)} (A_\varep, \codif \psi_\varep) = \varep^{n-2(k+1)} (d A_\varep, \psi_\varep) \stackrel{d}{\ra}  \N\big(0, \|\phi\|_{\Omega^k \dot{H}^{-1}(\R^n)}^2\big).
\end{equs}
By Lemma \ref{lemma:lattice-orthogonal-projections}, $dA_\varep$ is precisely a projected form of $(k+1)$-form white noise. If we assume that the same scaling rule should apply to projected white noise as for white noise itself, then this will lead us to take
\begin{equs}\label{eq:scaling-general-k-form}
\beta_\varep = \varep^{n - 2(k+1)} = \varep^{n-2k-2}.
\end{equs}
With this value of $\beta_\varep$, we would then hope to have that
\begin{equs}
\varep^{n-2(k+1)} (d A_\varep, \psi_\varep) \ra \N\big(0, \|E \psi\|_{\Omega^{k+1} L^2(\R^n)}^2\big).
\end{equs}
Next, note that
\begin{equs}
\|E \psi\|_{\Omega^{k+1} L^2(\R^n)}^2 &= (E \psi, E\psi) = (\psi, E \psi) = (\psi, d^* d(-\Delta)^{-1} \psi) \\
&= (d^* \psi, (-\Delta)^{-1} d^* \psi) = (\phi, (-\Delta)^{-1} \phi) = \|\phi\|_{\Omega^k \dot{H}^{-1}(\R^n)}^2.
\end{equs}
Thus, at least heuristically, we would have exactly what we wanted.

While this discussion only applies to $\phi$ of the form $\codif \psi$ (not to mention the various other heuristic steps), for us the main thing is that it already imposes the form of $\beta_\varep$. Note that the choice of $\beta_\varep$ is consistent with what we already know about the ($0$-form) GFF. First, note that the setting an inverse temperature of $\beta_\varep$ is equivalent to scaling a given Gaussian field of unit inverse temperature by $\beta_\varep^{-\frac{1}{2}}$. In the following discussion, we will assume that all fields are defined with a unit inverse temperature, and we will speak of scaling these fields by $\beta_\varep^{-\frac{1}{2}}$. In 1D, we must scale the GFF by $\varep^{\frac{1}{2}}$ (since in this case the GFF is a Brownian motion), in 2D, we shouldn't scale the GFF at all (and this is related to conformal invariance of the GFF in this dimension), and in 3D or higher dimensions, we must scale the GFF up by $\varep^{1-\frac{n}{2}}$, which is consistent with the fact that the infinite volume GFF in these dimensions is localized. 

Note that the general form of the scaling in terms of the dimension $n$ is $\varep^{1 - \frac{n}{2}} = \varep^H$, where $H = 1 - \frac{n}{2}$ is the Hurst parameter corresponding to $\fgf_{1}(\R^n)$. Since the Hurst parameter governs the regularity of the continuum GFF, one may think of this scaling as follows. By going to the $\varep$ lattice, we effectively impose a cutoff at distance scale $\sim \varep$. Another natural way to impose a cutoff at this scale is to average the continuum GFF on balls of radius $\sim \varep$. If we believe that these two ways of imposing a cutoff should give the same results, then the typical size of the lattice GFF $h_\varep$ at a given point $x$ should be of the same order as the average value of the continuum GFF $h$ on a ball $B_\varep(x)$ of radius $\varep$ around $x$. The latter quantity is precisely of order $\varep^H$ when $n \geq 3$ and $\sqrt{\log \varep^{-1}}$ when $n = 2$. Since the lattice GFF is localized when $n \geq 3$, any given point $h_\varep(x)$ is of size $\sim 1$, and thus we need to scale up by $\varep^H$ to match with the continuum. When $n = 2$, we don't need to scale at all, since the lattice GFF is delocalized.

In the case of the $1$-form $\gff$, \eqref{eq:scaling-general-k-form} says that we need to take $\beta_\varep = \varep^{n-4}$, or equivalently we need to scale by $\beta_\varep^{-\frac{1}{2}} = \varep^{2-\frac{n}{2}}$. This scaling coincides with the scaling needed to make $\unitary(1)$ lattice gauge theory converge to its continuum limit (which can either be viewed as a (projected) $1$-form GFF or a projected $2$-form white noise -- recall the discussion in Section \ref{section:U(1)-theory}) \cite{Gross1983, Driver1987}. This is also the same scaling used for the general non-Abelian Yang--Mills theory on the 2D torus -- see \cite[Example 5.2]{Chevyrev2019}. More generally, one might expect that $\beta_\varep = \varep^{n-4}$ should also be the right scaling for taking a scaling limit of 3D lattice gauge theory, however in 4D one might expect logarithmic corrections -- see the discussion in \cite[Section 3]{chatterjee2019yang}.

\begin{remark}
The fact that $\beta_\varep$ in \eqref{eq:scaling-general-k-form} depends on $k$ is somewhat unintuitive, because after all a $k$-form GFF is just a collection of independent $0$-form GFFs, and thus one might expect that the scaling for $k$-forms should just be the same as the scaling for $0$-forms. Ultimately, the $k$-dependence of $\beta_\varep$ arises because of our choice to always include the $\varep^{n-2k}$ prefactor when considering lattice observables. Ultimately, this choice is somewhat arbitrary, but if we had not done this, then the $\beta_\varep$ we obtain would not match with that from previous scaling limit results in lattice gauge theory (i.e. \cite{Gross1983, Driver1987, Chevyrev2019, CS23}).
\end{remark}

\section{Confinement and mass gap} \label{sec:confinement}

We begin by discussing confinement for lattice Gaussian $1$-forms. Let $\Lambda \sse \Z^n$ be a finite box. Let $\gamma = e_1 \cdots e_n$ be a loop in $\Lambda$. We will identity $\gamma$ with the $1$-form which maps each oriented edge $e \in \overrightarrow{C}^1(\Lambda)$ to the number of times $\gamma$ traverses $e$, counting orientation (thus if $\gamma$ traverses a given edge $e$ once, then it assigns $e$ to $1$ and $-e$ to $-1$). Since $\gamma$ is a loop, the associated $1$-form satisfies $\codif \gamma = 0$. By \cite[Proposition 2.3]{GS2023}, there exists a $2$-form $S \in \Omega^2(\Lambda)$ such that $\gamma = \codif S$. We think of $S$ as a surface. In the following discussion, we always take $\gamma$ to be a rectangular loop, in which case the surface $S$ can be taken as the rectangle itself.

Let $A$ be a $\fgf_1^1(\Lambda; \beta)$ as in Definition \ref{def:lattice-gaussian-k-form}. In this setting, the Wilson loop observable becomes
\begin{equs}
W_\gamma(A) = \exp(\icomplex (A, \gamma)) = \exp(\icomplex (dA, S)).
\end{equs}
Since $A$ is Gaussian, we have the exact calculation 
\begin{equs}
\E W_\gamma(A) = \exp\bigg(-\frac{1}{2\beta} (\gamma, (-\Delta)^{-1} \gamma)  \bigg) = \exp\bigg(-\frac{1}{2\beta} (S, E S)\bigg).
\end{equs}
The second equality follows because $dA \stackrel{d}{=} \pi_{1 \ra 2}\noise$ where $\noise$ is a $2$-form white noise (Lemma \ref{lemma:lattice-projected-white-noise}), or alternatively by the identity
\begin{equs}\label{eq:gamma-S-l2-relation}
(\gamma, (-\Delta)^{-1} \gamma) = (\codif S, (-\Delta)^{-1} \codif S) = (S, d \codif (-\Delta)^{-1} S) = (S, \pi_{1 \ra 2} S).
\end{equs}
Thus, for lattice $1$-form GFFs, the question of area vs. perimeter law reduces to the question of whether 
\begin{equs}\label{eq:area-vs-perim}
(\gamma, (-\Delta)^{-1} \gamma) = (S, \pi_{1 \ra 2} S) \sim \begin{cases} \mrm{area}(\gamma) \\ \mrm{perim}(\gamma) \end{cases}.
\end{equs}
To be precise, we would start on a finite lattice $\Lambda$ containing $\gamma$, and the implicit constants in the above ``$\sim$" should be uniform in all $\Lambda$ large enough. We do this to avoid discussing the existence of infinite-volume limits.

The question presented in \eqref{eq:area-vs-perim} depends on the dimension $n$ in the following manner:
\begin{equs}
(\gamma, (-\Delta)^{-1} \gamma) = (S, \pi_{1 \ra 2} S) \sim \begin{cases} \mrm{area}(\gamma) & n = 2 \\
\mrm{perim}(\gamma) \log \mrm{perim} (\gamma) & n = 3 \\
\mrm{perim}(\gamma) & n = 4
\end{cases}.
\end{equs}
In particular, note that in dimension $3$, the asymptotics falls strictly in between the two listed in \eqref{eq:area-vs-perim}. When $n = 2$, the asymptotics follows immediately from the fact that (since $d = 0$ on $2$-forms when $n = 2$)
\begin{equs}
\pi_{1 \ra 2} = d \codif (-\Delta)^{-1} = (d\codif  + \codif d) (-\Delta)^{-1} = I,
\end{equs}
and thus $(S, \pi_{1 \ra 2} S) = (S, S) = \mrm{area}(\gamma)$. The case $n = 4$ was shown in \cite[Proposition 6.1]{GS2023} by using properties of the lattice Green's function. In the same paper, the authors remarked (see \cite[Section 7.3]{GS2023}) that their same argument would work in $n = 3$, and would give the $\mrm{perim}(\gamma) \log \mrm{perim}(\gamma)$ asymptotics in this dimension. We briefly sketch their argument. 

We may decompose the $1$-form $\gamma = \gamma^h + \gamma^v$, i.e. into the horizontal and vertical parts of the loop. Since the $1$-form Laplacian acts coordinatewise (i.e. Proposition \ref{prop:lattice-laplacian-coordinatewise}, or more precisely a version of the lemma for a finite box), we may split
\begin{equs}
(\gamma, (-\Delta)^{-1} \gamma) = (\gamma^h, (-\Delta)^{-1} \gamma^h) + (\gamma^v, (-\Delta)^{-1} \gamma^v).
\end{equs}
Let the dimensions of $\gamma$ be $L \times H$, where $L$ is the length of one of the horizontal sides. In principle, there should be a condition on the relative sizes of $L$ and $H$ (see \cite[Proposition 6.1]{GS2023}), but we will ignore that here and keep the discussion heuristic. We focus on showing that
\begin{equs}
(\gamma^h, (-\Delta)^{-1} \gamma^h) \sim \begin{cases} L \log L & n = 3 \\
L & n = 4 \end{cases}.
\end{equs}
We may further decompose $\gamma^h = \gamma^{h, -} + \gamma^{h,+}$, where $\gamma^{h, -}$ is the lower horizontal part and $\gamma^{h, +}$ is the upper horizontal part. One may show that the cross term
\begin{equs}
(\gamma^{h, -}, (-\Delta)^{-1} \gamma^{h, +}) 
\end{equs}
is negligible, and so the main contribution comes from the terms
\begin{equs}
(\gamma^{h, -}, (-\Delta)^{-1} \gamma^{h, -}), ~~ (\gamma^{h, +} (-\Delta)^{-1}, \gamma^{h, +}).
\end{equs}
For simplicity, let us take $\Lambda = \Z^n$, in which case the two terms are equal by translation invariance. Then due to the explicit form of the lattice Laplacian, we have that
\begin{equs}
(\gamma^{h, -}, (-\Delta)^{-1} \gamma^{h, -}) = \sum_{1 \leq x, y \leq L} G(x e_1, y e_1),
\end{equs}
where $G$ is the Green's function for the (function) Laplacian on $\Z^n$. Since $G(x, y) \sim |x-y|^{2-n}$, we see that for any $1 \leq x \leq L$,
\begin{equs}
\sum_{1 \leq y \leq L} G(x e_1, y e_1) \sim \begin{cases} \log L & n = 3, \\ 1 & n = 4 \end{cases},
\end{equs}
and thus summing over $1 \leq x \leq L$, we obtain the stated asymptotics in the cases $n = 3, 4$ (in fact, also the case $n \geq 5$).

Combing back to the discussion, we observe that by the lattice Hodge decomposition (Proposition \ref{prop:lattice-hodge-decomposition}) the value of $(S, \pi_{1 \ra 2} S)$ can be cast as the following $L^2$ minimization problem:
\begin{equs}\label{eq:L2-minimization}
(S, \pi_{1 \ra 2} S) = \min_{\substack{ S' \in \Omega^2(\Lambda) \\ \codif S' = \gamma}} (S', S').
\end{equs}
This follows since for any $S' \in \Omega^2(\Lambda)$ with $\codif S' = \gamma$, we have that
\begin{equs}
(S', S') = (\pi_{1 \ra 2} S', \pi_{1 \ra 2} S') + (\pi_{3 \ra 2} S', \pi_{3 \ra 2} S') \geq (\gamma, (-\Delta)^{-1} \gamma),
\end{equs}
where the last inequality follows since $(\pi_{1 \ra 2} S', \pi_{1 \ra 2} S') = (\gamma, (-\Delta)^{-1} \gamma)$ by the same calculation as in \eqref{eq:gamma-S-l2-relation}.

Note that the surface which minimizes \eqref{eq:L2-minimization} is explicitly given by $\pi_{1 \ra 2} S = d \codif (-\Delta)^{-1} S$, where recall $S$ is the rectangular surface spanning $\gamma$. Geometrically, when $n > 2$, this surface is ``spread out", in the sense that it assigns nonzero value even to plaquettes which are far away from the rectangular surface $S$. By doing so, the squared $L^2$ norm can be made much smaller than $\mrm{area}(\gamma)$. When $n =2$, there is no room to spread out, which leads to area law in this dimension.

Based on this discussion, we would also expect area law to hold when $\Lambda \sse \Z^2 \times [-M, M]^{(n-2)}$ is a ``slab". Here, $M$ is fixed independently of $\Lambda$ (as before, we take $\Lambda$ finite but want estimates which are uniform for all $\Lambda$ large enough). This is because in such a slab, the $L^2$--minimizing surface has a limited amount of room to spread out, which makes this case qualitatively like the two-dimensional case. Following our sketch for the proof of \cite[Proposition 6.1]{GS2023}, we see that this claim essentially\footnote{As before we go to infinite volume for simplicity of discussion, whereas the needed estimates would be in finite volume. This is why we say ``essentially".} reduces to showing that for the Green's function $G_{\mrm{slab}}$ on $\Z^2 \times [-M, M]^{(n-2)}$, we have that
\begin{equs}
\sum_{1 \leq x, y \leq L} \big( G_{\mrm{slab}} (x e_1, ye_1) - G_{\mrm{slab}} (x e_1, ye_1 + H e_2) \big) = (\gamma^{h, -}, (-\Delta)^{-1} \gamma^h) \sim \mrm{area}(\gamma). 
\end{equs}
Here, we do not separate $\gamma^h = \gamma^{h, +} + \gamma^{h,-}$ as before, because in fact the cross term $(\gamma^{h, -}, (-\Delta)^{-1} \gamma^{h, +})$ is no longer negligible. Indeed, this is due to the fact that the Green's function on $\Z^2 \times [-M, M]^{(n-2)}$ is only defined up to additive constant (since random walk on the slab is recurrent), so we need to be careful to only ever deal with differences of Green's functions. As in the 2D case, one would expect that when $|x-y| \lesssim L$ and $H \gtrsim L$, we have that
\begin{equs}
G_{\mrm{slab}} (x e_1, ye_1) - G_{\mrm{slab}} (x e_1, ye_1 + H e_2) \sim \log \Big(\frac{\sqrt{|x-y|^2 + H^2}}{|x-y|}\Big) \sim 1.
\end{equs}
Since we sum over $1 \leq x, y \leq L$, this would mean that (assuming $H \sim L$)
\begin{equs}
(\gamma^h, (-\Delta)^{-1} \gamma^h) \sim L^2 \sim LH \sim \mrm{area}(\gamma),
\end{equs}
and thus we would have area law.

Next, we briefly discuss mass gap for lattice Gaussian $1$-forms. In this setting, the existence or not of a mass gap reduces to whether covariances of Wilson loop observables decay exponentially or not. In the following, fix two rectangular loops $\gamma_1, \gamma_2$ whose associated rectangles are disjoint. We want to understand how quickly 
\begin{equs}
\Cov(W_{\gamma_1}(A), W_{\gamma_2}(A))
\end{equs}
decays in the distance between $\gamma_1, \gamma_2$. By an explicit Gaussian calculation, we have that
\begin{equs}
\Cov(W_{\gamma_1}(A), W_{\gamma_2}(A)) = \E W_{\gamma_1}(A) \E W_{\gamma_2}(A) \bigg( \exp\bigg(-\frac{1}{\beta}(\gamma_1, (-\Delta)^{-1} \gamma_2)\bigg) - 1 \bigg).
\end{equs}
Thus, we are essentially asking for the decay of
\begin{equs}
(\gamma_1, (-\Delta)^{-1} \gamma_2) = (S_1, \pi_{1 \ra 2} S_2),
\end{equs}
where $S_1, S_2$ are the $2$-forms which correspond to the rectangles spanned by $\gamma_1, \gamma_2$. When $n = 2$, we have that $\pi_{1 \ra 2} S_2 = S_2$, and thus since the rectangles spanned by $\gamma_1, \gamma_2$ are assumed to be disjoint, we have that
\begin{equs}
(S_1, \pi_{1 \ra 2} S_2) = (S_1, S_2) = 0. 
\end{equs}
Thus, when $n = 2$, the covariance between $W_{\gamma_1}(A), W_{\gamma_2}(A)$ is just zero, and there is a mass gap. When $n > 2$, by explicit Green's function calculations similar to those from our sketch of the proof of \cite[Proposition 6.1]{GS2023}, one can see that $(\gamma_1, (-\Delta)^{-1} \gamma_2)$ decays only polynomially in the distance between $\gamma_1, \gamma_2$. Thus there is no mass gap when $n > 2$. 

To finish the discussion, we consider the slab $\Z^2 \times [-M, M]^{(n-2)}$. Before, we saw that the slab behaved qualitatively like 2D rather than higher dimensions, and based on this we might expect mass gap for this case. We discuss why this should be true, at least heuristically. Fix two plaquettes $p, p'$ which are in the 1--2 plane. We identify $p, p'$ with the $2$-forms which are respectively $1$ on $p$, $p'$, and zero everywhere else. Note that these plaquettes bound the loops $\ptl p, \ptl p'$.
The case of more general loops in the 12 plane can be reduced to this case by taking linear combinations of the appropriate $p, p'$ (using Stokes' theorem). We want to see why
\begin{equs}\label{eq:p-p-prime-exp-decay}
(p, d \codif (-\Delta)^{-1} p') \leq C \exp(-c \mrm{dist}(p, p')).
\end{equs}
This reduces to a concrete question about $\mrm{G}_{\mrm{slab}}$, the Green's function on the $\Z^2 \times [-M, M]^{(n-2)}$ which we previously encountered, as follows. We have that
\begin{equs}
d \codif (-\Delta)^{-1} p' = (-\Delta) (-\Delta)^{-1} p' - \codif d (-\Delta)^{-1} p',
\end{equs}
and thus
\begin{equs}
(p, d \codif (-\Delta)^{-1} p') = (p, p') - (p, \codif d (-\Delta)^{-1} p').
\end{equs}
Assuming that $p \neq p'$ (which is fine since we are primarily thinking of $p, p'$ as being very far away), and recalling from Remark \ref{remark:horizontal 2-form}, we further have that
\begin{equs}
(p, d \codif (-\Delta)^{-1} p') &= - \big(\codif d (-\Delta)^{-1} p'\big)(p) \\
&= - \sum_{\substack{\bar{p} \sim p' \\ \bar{p}, p' \text{ not coplanar}}} \Big( \big((-\Delta)^{-1}(p')\big)(\bar{p}) - \big((-\Delta)^{-1}(p')\big)(p)\Big)\\
&= -\sum_{\substack{\bar{p} \sim p \\ \bar{p}, p \text{ not coplanar}}} \big( G_{\mrm{slab}}(p', \bar{p}) - G_{\mrm{slab}}(p', p)\big) .
\end{equs}
One should think of the sum above as the Laplacian of $G_{\mrm{slab}}$ in the ``vertical directions", i.e. the last $n-2$ coordinate directions. In summary, to show \eqref{eq:p-p-prime-exp-decay}, it suffices to show that for plaquettes $p, \bar{p}$ which are a vertical translate of each other, we have that 
\begin{equs}\label{eq:slab-greens-fn-exp-decay}
\big| G_{\mrm{slab}}(p', \bar{p}) - G_{\mrm{slab}}(p', p)\big| \lesssim \exp(-c \mrm{dist}(p, p')).
\end{equs}
This is a classical question about simple random walk on the slab $\Z^2 \times [-M, M]^{(n-2)}$. In words, one wants to show that a change in the vertical direction of the starting point $x$ of a simple random walk barely affects the expected number of visits to a point $y$ which is very far away. More quantitatively, one wants to show that the effect is exponentially small in the distance between $x$ and $y$. Denoting $(S_n)_{n \geq 0}$ the simple random walk on the slab started at the origin, this would follow if for all $n$ large enough, and for all $x, y \in \Z^2 \times [-M, M]^{(n-2)}$ such that $x - y = \pm e_i$ for some $i \neq 1, 2$, we have that
\begin{equs}\label{eq:Sn-loses-memory}
d_{\mrm{TV}}(S_n + x, S_n + y) \leq C \exp(- c n).
\end{equs}
The point is that if we project $S_n$ to the vertical directions, we obtain a lazy random walk on $[-M, M]^{(n-2)}$, which loses memory of its starting point exponentially quickly in $n$. From this, one might expect that the original random walk on the slab itself loses memory of its vertical starting point exponentially quickly in $n$, which is \eqref{eq:Sn-loses-memory}. However, \eqref{eq:Sn-loses-memory} may not literally be true due to lattice effects, e.g. the fact that $\Z^2 \times [-M, M]^{(n-2)}$ is bipartite. On the other hand, some weaker version of \eqref{eq:Sn-loses-memory} should be true and sufficiently strong to imply \eqref{eq:slab-greens-fn-exp-decay}. Note that in the continuum, this becomes quite clear, since Brownian motion on $\R^2 \times [-M, M]^{(n-2)}$ is the Cartesian product of an independent Brownian motion on $\R^2$ and an independent Brownian motion on $[-M, M]^{(n-2)}$.

\section{Open Problems} \label{sec:openproblems}
It is easy to describe a discrete model that converges in law to $\textrm{FGF}^k_0(M)$, which is equivalent to the white noise $W_k$ on $k$-forms. Namely, one may construct a discrete $k$-form $A_\varep$ (on an $\varep$-width lattice) that assigns an i.i.d.\ random variable to each $k$-cell of the lattice. That raises the following question:

\begin{problem} \label{prob:heightmodels} For what examples (with $k \geq 1$) can one show that the above-mentioned discrete model {\em conditioned to} satisfy $d A_\epsilon=0$ converges (or does not converge) in law to  $\textrm{FGF}^k_0(M)_{d=0}$ w.r.t.\ some topology?  Equivalently, can one show that the $d^*$ pre-image of $A_\epsilon$ converges (or does not converge) in law to $\textrm{FGF}^{k-1}_1(M)_{d^*=0}$ in some topology?
\end{problem}

As a simple example, imagine assigning $\pm 1$ (each with probability $1/2$) to each directed edge of a large induced subgraph of $\mathbb Z^2$ and conditioning on the event that the directed sum around every unit square is zero: this is equivalent to a special case of the so-called {\em six vertex model} (a.k.a.\ {\em square ice}) and it remains open to show even in this simple case that the height function converges to the Gaussian free field in some topology---or equivalently that the one-form converges to the gradient of the Gaussian free field in some topology. See e.g.\ \cite{chandgotia2021delocalization,duminil2022logarithmic,wu2022central} for a sampling of recent work in this area and links to further references.
The upshot is that Problem~\ref{prob:heightmodels} is highly non-trivial even when $n=2$ and $k=1$. Nonetheless, there are closely related height function models (dimer height functions, discrete Gaussian height models, solid-on-solid models, Ginsburg-Landau $\nabla \phi$ models) that can be shown to converge to the Gaussian free field \cite{spencer1997scaling,naddaf1997homogenization, miller2011fluctuations, andres2021local, kenyon2000conformal, kenyon2001dominos}. Problem~\ref{prob:heightmodels} asks which of these successes can be reproduced for other $n$ and $k$ values. (Other cases with $k=n/2$ might be especially interesting.) We can formulate the problem about dimer models separately:

\begin{problem}\label{prob:dimer} If we interpret $\textrm{FGF}_1^1(M)_{d^*=0}$ as a random divergence-free vector field, can we show that this is the fine-mesh scaling limit of the $n$-dimensional dimer model when $n>2$?
\end{problem}
See the further discussion in \cite{chandgotia2023large}, where this question has also appeared, and note that Kenyon has already proved the analogous result for $n=2$ \cite{kenyon2000conformal,kenyon2001dominos}.

\begin{problem} \label{prob:twist}
If the answer to Problem~\ref{prob:dimer} is yes, can one prove further that if the law of the discrete model is appropriately {\em weighted} by the exponential of the twist parameter (see the discussion in \cite{chandgotia2023large} and the references therein) then one has convergence to a Chern-Simons-weighted model of the type discussed in Section~\ref{sec:introvariants}? Alternatively, can one check that a similar statement holds for some simple variant of the dimer model, perhaps a discrete Gaussian field with a twist weighting of some kind?
\end{problem}

Obviously, if $X$ is any finite-dimensional centered Gaussian random variable, weighting the law of $X$ by the exponential of a linear function of $X$ has the effect of ``shifting the mean'' of $X$, while weighting by the exponential of a centered quadratic function may ``change the covariance matrix'' of $X$. In the 2D dimer model, the ``overall average'' of the height function is a {\em linear} function that changes by fixed increments under local moves. 
What is interesting in 3D is that there is also a family of local moves that increment or decrement a {\em quadratic function} and that this happens to be somehow related to linking numbers and the celebrated Chern-Simons action.

For the next problem, recall that in two dimensions, there
is a so-called {\em Kosterlitz-Thousless roughening-transition} for models that may or may not converge in law to the gradient of the Gaussian free field, depending on some parameter. As above, we interpret the gradient of the GFF as ``white noise on $\Omega^1_-$'' or as ``a white-noise random $1$-form conditioned to be in the kernel of $d$'' and consider discrete models (like square ice or the dimer model) that somehow mimic that story on a discrete level. More generally, one can also consider random white noise forms in $\Omega^k_-$ conditioned to be in the kernel of $d$.

\begin{problem}
In what cases can one extend the Kosterlitz-Thouless phase transition results to models of random $k$-forms in dimension $n=2k$? See \cite{frohlich1982massless, GS2023} for a particular example involving a discrete Gaussian and the Kosterlitz-Thouless phase transition in dimension $n=4$ with $k=2$.
\end{problem}

\begin{problem}
Can the Coulomb gas ``defect'' behavior shown by Ciucu in  \cite{ciucu2009scaling} be extended to any models with other $k$ and $n$ values? For example, if we take $k=2$ and $n=4$, a loop (the boundary of a disk) can play the role of a pair of points (the boundary of a path).
\end{problem}

\begin{problem}
The sine-Gordon model is in some sense an off-critical limit of models at the smooth-rough interface in dimension $2$. See e.g.\ \cite{ bauerschmidt2021log,bauerschmidt2022maximum,mason2022two,gubinelli2024fbsde}.  Can this story be extended to any larger values of $k$ and $n=2k$?
\end{problem}

In two dimensions, there is a well known correspondence between uniform spanning trees and dimer model instances, and the sine-Gordon model is somehow related to a spanning tree on graph with some underlying ``drift'' that makes edges pointing one direction (say to the right) slightly more likely to occur than those pointing in the opposite direction (left). It is natural to wonder about similar problems in higher dimensions. Kozma proved that loop-erased random walk has a rotationally invariant scaling limit in 3D \cite{kozma2007scaling} but there is no obvious way to interpret the winding in terms of the Gaussian free field, as there is in two dimensions. We can instead ask a somewhat open-ended question.

\begin{problem}
Are uniform spanning trees in dimension $3$ or higher related to fractional Gaussian fields in {\em some} way? (See \cite{sun2013uniform} for a related relationship involving uniform spanning forests.)
\end{problem}

\begin{problem}
The results in \cite{cao2023random} provide a way to convert a $\unitary(N)$ Yang-Mills problem into a problem about sums over spanning surfaces, which we can interpret as sums over all discretized $2$-forms that have a given $1$-form as their $d^*$ boundary. What can be proved about the scaling limits of this ``dual'' random $2$-form (which is now an integer-valued $2$-form, rather than a Lie-algebra valued $2$-form)? The terms in the sums are signed (i.e.\ some are positive and some negative) and understanding the extent to which positive and negative terms cancel one another is part of the problem. Does {\em this} model have a type of Kosterlitz-Thouless transition and a non-Gaussian limit akin to the Sine-Gordon model?
\end{problem}

\begin{problem}
Can we find some use for the Lie-group version of Gaussian multiplicative chaos involving the logarithmically correlated ``$2$-form height function'' associated to a Gaussian random connection?
\end{problem}


\begin{appendix}

\section{Covariance kernel derivations}\label{section:covariance-kernel-proof}

In this section, we prove Proposition \ref{prop:fractional-laplacian-formula}. We closely follow \cite[Chapter 1]{landkof1972foundations}. The starting point is to note that by definition, 
\beq\label{eq:fractional-laplacian-def-formula} ((-\Delta)^{-\frac{\alpha}{2}} \phi)(x) = \big(\mc{F}^{-1}(|p|^{-\alpha} \mc{F}\phi)\big)(x) = \int_{\R^n} |p|^{-\alpha} \widehat{\phi}(p) e^{\icomplex p \cdot x} dp. \eeq
Fix $\phi \in \Phi$ and $x \in \R^n$. Letting $\alpha$ now be a complex variable, note that the right hand side above is entire as a function of $\alpha$ (since $\phi \in \Phi$ implies that $\widehat{\phi}$ has arbitrarily fast polynomial decay at $p = 0$; recall Lemma \ref{lemma:phi-psi-characterizations}). As previously noted, in the simplest case $\alpha \in (0, n)$, we also have that
\beq\label{eq:fractional-laplacian-given-by-convolution} ((-\Delta)^{-\frac{\alpha}{2}} \phi)(x) = \frac{1}{\gamma_n(\alpha)}\int_{\R^n} \phi(y) |x-y|^{\alpha - n} dy. \eeq
Thus to extend $((-\Delta)^{-\frac{\alpha}{2}} \phi)$ outside the region $\alpha \in (0, n)$, it suffices to find alternative expressions for the right hand side above that are analytic functions on various regions of $\C$. 

We split the proof of Proposition \ref{prop:fractional-laplacian-formula} by case. We first handle the easy cases.

\begin{proof}[Proof of Case (1) of Proposition \ref{prop:fractional-laplacian-formula}]
Recall that 
\[ \gamma_n(\alpha) = 2^\alpha \pi^{n/2} \frac{\Gamma(\alpha / 2)}{\Gamma((n-\alpha)/2)}. \]
Thus the function 
\[ \alpha \mapsto \frac{1}{\gamma_n(\alpha)} \int_{\R^n} \phi(y)|x-y|^{\alpha - n} dy\]
is analytic on $\{\alpha \in \C : \mrm{Re}(\alpha) > 0, \alpha \notin n + 2\N\}$ (and the singularities at $\alpha \in n +  2\N$ arise from $\gamma_n(\alpha)$, more specifically the term $\Gamma((n-\alpha)/2)$. This case now immediately follows upon recalling that the right hand side of equation \eqref{eq:fractional-laplacian-def-formula} is entire in $\alpha$.
\end{proof}

\begin{proof}[Proof of Case (4) of Proposition \ref{prop:fractional-laplacian-formula}]
This case follows simply because for $m \geq 0$, we have that
\[ \begin{split}
((-\Delta)^m \phi)(x) &= (\delta_x, (-\Delta)^m \phi) = ((-\Delta)^m \delta_x, \phi) \\
&= \int_{\R^n} \phi(y) ((-\Delta)^m \delta_x)(y) dy. \qedhere
\end{split}\]
\end{proof}

\begin{proof}[Proof of Case (2) of Proposition \ref{prop:fractional-laplacian-formula}]
Recalling that the right hand side of equation \ref{eq:fractional-laplacian-def-formula} is analytic in $\alpha$, and recalling the formula given by Case (1), we simply need to show that for integer $m \geq 0$,
\beq\label{eq:case-2-sufficient-claim}\begin{split}
\lim_{\alpha \ra n + 2m} &\frac{1}{\gamma_n(\alpha)} \int_{\R^n} \phi(y) |x-y|^{\alpha - n} dy = \\
&2\frac{(-1)^{m+1} 2^{-(n+2m)}\pi^{-n/2}}{ m! \Gamma((n+2m)/2)} \int_{\R^n} \phi(y) |x-y|^{2m} \log|x-y| dy.
\end{split}\eeq
To start, we write
\[\begin{split}
\frac{1}{\gamma_n(\alpha)} &= 2^{-\alpha} \pi^{-n/2} \frac{\Gamma((n-\alpha)/2)}{\Gamma(\frac{\alpha}{2})} \\
&= 2^{-\alpha} \pi^{-n/2} \frac{1}{\Gamma(\frac{\alpha}{2})} \big(\Gamma((n-\alpha)/2) (n - \alpha - 2m) \big)\frac{1}{n-\alpha - 2m}.
\end{split}\]
We have that
\[ \lim_{\alpha \ra n + 2m} \Gamma((n-\alpha)/2) (n - \alpha - 2m) = \frac{2 (-1)^m}{m!}. \]
Thus to show \eqref{eq:case-2-sufficient-claim}, it suffices to show that
\[ \lim_{\alpha \ra n + 2m} \frac{1}{n-\alpha - 2m} \int_{\R^n} \phi(y)|x-y|^{\alpha - n} dy = -\int_{\R^n} \phi(y)|x-y|^{2m} \log|x-y| dy. \]
To see this, note that $y \mapsto |x-y|^{2m}$ is a polynomial in $y$, and thus since $\phi$ integrates to zero against any polynomial (by the definition of the Lizorkin space $\Phi$), we have that
\[\begin{split}
\frac{1}{n-\alpha - 2m}  \int_{\R^n} \phi(y)|x-y|^{\alpha - n} dy &= \int_{\R^n} \phi(y) \frac{|x-y|^{\alpha - n} - |x-y|^{2m}}{n-\alpha - 2m}dy \\
&= \int_{\R^n} \phi(y) |x-y|^{2m} \frac{|x-y|^{\alpha - n - 2m} - 1}{n - \alpha - 2m} dy.
\end{split}\]
To finish, observe that for $y \neq x$, we have that 
\[ \lim_{\alpha \ra n + 2m} \frac{|x-y|^{\alpha - n - 2m} - 1}{n - \alpha - 2m} = -\log |x-y|, \]
and thus
\[\begin{split}
\lim_{\alpha \ra n + 2m} \int_{\R^n} \phi(y) |x-y|^{2m}& \frac{|x-y|^{\alpha - n - 2m} - 1}{n - \alpha - 2m} dy = \\
&-\int_{\R^n} \phi(y) |x-y|^{2m} \log|x-y| dy, 
\end{split}\]
as desired.
\end{proof}

The only case that remains is Case (3). For this case, the difficulty is that since $\alpha < 0$, the function $y \mapsto |x-y|^{\alpha - n}$ is no longer integrable. To fix this, we will need to subtract off Taylor approximations of $\phi$ to gain additional decay at $y = x$ to offset the singularity $|x-y|^{\alpha - n}$. Towards this end, we will need the Pizzetti formula, stated as follows. First, for $r \geq 0$, let $\bar{\phi}(r)$ be the average value of $\phi$ on the hypersphere of radius $r$ in $\R^n$, i.e.
\[ \bar{\phi}(r) := \frac{1}{\omega_n r^{n-1}} \int_{\ptl B(0, r)} \phi(y) dy, \text{ where } \omega_n := \frac{2\pi^{n/2}}{\Gamma(n/2)}\]
is the surface area of the unit $n$-dimensional hypersphere. Recall also the definition of $H_j$ in \eqref{eq:H-j-def}. The following lemma is proven in \cite[Chapter 1]{landkof1972foundations}.

\begin{lemma}[Pizzetti formula]
Let $m \geq 0$. We have that as $r \ra 0$,
\[ \bar{\phi}(r) = \sum_{j=0}^m H_j (\Delta^j \phi)(0) r^{2j} + O(r^{2m+2}).\]
\end{lemma}

\begin{proof}[Proof of Case (3) of Proposition \ref{prop:fractional-laplacian-formula}]
Let $m \geq 1$. Suppose first that $\alpha \in (0, n)$. Without loss of generality, we may assume $x = 0$, because we can always define the shift $\phi_x(y) = \phi(x + y)$ and apply the result to $\phi_x$. We will rewrite the right hand side of \eqref{eq:fractional-laplacian-given-by-convolution} in a way that is analytic on $\{\alpha \in \C : \mrm{Re}(\alpha) \in (-2m, n), \alpha \neq n + 2\Z\}$. To start, we split the integral into
\[\begin{split}
\int_{\R^n} \phi(y) |y|^{\alpha - n} dy = \int_{|y| < 1}& \bigg(\phi(y) - \sum_{j=0}^{m-1} H_j (\Delta^j \phi)(0) |y|^{2j}\bigg) |y|^{\alpha - n} dy  ~+ \\
&\omega_n \sum_{j=0}^{m-1} H_j \frac{(\Delta^j \phi)(0)}{\alpha + 2j} + \int_{|y| > 1} \phi(y) |y|^{\alpha - n} dy.
\end{split}\]
(To get the above, we added and subtracted the term $\sum_{j=0}^{m-1} H_j (\Delta^j \phi)(0) |y|^{2j}$ on the region $|y| < 1$, and computed one of the integrals by polar coordinates.) Converting the first integral on the right hand side above to polar coordinates, we have that
\[\begin{split}
\int_{|y| < 1}\bigg(\phi(y) - \sum_{j=0}^{m-1} H_j (&\Delta^j \phi)(0) |y|^{2j}\bigg) |y|^{\alpha - n} dy  = \omega_n \int_0^1 \bigg(\bar{\phi}(r) - \sum_{j=0}^{m-1} H_j (\Delta^j \phi)(0) r^{2j}\bigg) r^{\alpha - 1} dr.
\end{split}\]
By the Pizzetti formula, the right hand side above is analytic on $\{\alpha \in \C : \mrm{Re}(\alpha) > -2m\}$. Recalling that the right hand side of \eqref{eq:fractional-laplacian-def-formula} is entire in $\alpha$, we thus obtain that for $\alpha \in (-2m, 0)$, $\alpha \notin 2\Z$,
\[\begin{split}
((-\Delta)^{-\frac{\alpha}{2}} \phi)(0) = \frac{1}{\gamma_n(\alpha)} \bigg( &\int_{|y| < 1} \bigg(\phi(y) - \sum_{j=0}^{m-1} H_j (\Delta^j \phi)(0) |y|^{2j}\bigg) |y|^{\alpha - n} dy  ~+\\
&\omega_n \sum_{j=0}^{m-1} H_j \frac{(\Delta^j \phi)(0)}{\alpha + 2j} + \int_{|y| > 1} \phi(y) |y|^{\alpha - n} dy\bigg) .
\end{split}\]
(The additional restriction $\alpha \notin 2\Z$ is due to the fact that $\gamma_n(\alpha)$ has singularities at negative even integers coming from its $\Gamma(\alpha / 2)$ term.) Next, suppose further that $\alpha \in (-2m, -2(m-1))$. In this case, for each $0 \leq j \leq m-1$, we may write
\[ \frac{\omega_n}{\alpha + 2j} = - \omega_n\int_1^\infty r^{\alpha + 2j - 1} dr  = -\int_{|y| > 1} |y|^{\alpha + 2j - n} dy.\]
Combining this with the previous display, we thus obtain that when $\alpha \in (-2m, -2(m-1))$, 
\[ ((-\Delta)^{-\frac{\alpha}{2}} \phi)(0) = \frac{1}{\gamma_n(\alpha)} \int_{\R^n} \bigg(\phi(y) - \sum_{j=0}^{m-1} H_j (\Delta^j \phi)(0) |y|^{2j}\bigg) |y|^{\alpha - n} dy , \]
as desired.
\end{proof}

Next, we begin towards the proof of Lemma \ref{lemma:projection-laplace-inverse-formula}. This will follow directly from the next result.

\begin{lemma}\label{lemma:two-derivative-ineverse-laplacian-kernel}
Let $\phi \in \Phi$, $x \in \R^n$, $i, j \in [n]$, $\alpha > 0$. We have that $\big(\ptl_i \ptl_j (-\Delta)^{-(1+\frac{\alpha}{2})}\phi\big)(x)$ is equal to:
\begin{enumerate}
    \item $0 < \alpha < 2 $
    \begin{equs}
    \frac{\Gamma((n-\alpha+2)/2)}{2^\alpha \pi^{n/2} \Gamma((2+\alpha)/2)} \int_{\R^n} dy  |x-y|^{\alpha - 2 - n} (x_i - y_i)(x_j - y_j) \phi(y) - \frac{\delta_{ij}}{\alpha} \int_{\R^n} dy G^{\frac{\alpha}{2}}(x, y) \phi(y) .
    \end{equs}
    \item $\alpha = 2$
    \begin{equs}\label{eq:alpha-equals-two}
    \frac{1}{2\omega_n} \int_{\R^n} dy |x-y|^{-n} (x_i - y_i)(x_j - y_j) \phi(y) - \frac{\delta_{ij}}{2} \int_{\R^n} dy G^1(x, y) \phi(y).
    \end{equs}
    \item $\alpha > 2$
    \begin{equs}
    \frac{\Gamma((\alpha-2)/2)}{4 \Gamma((\alpha+2)/2)} \int_{\R^n} dy  G^{\frac{\alpha}{2}-1}(x, y) (x_i - y_i)(x_j - y_j) \phi(y) - \frac{\delta_{ij}}{\alpha} \int_{\R^n} dy G^{\frac{\alpha}{2}}(x, y) \phi(y) .
    \end{equs}
\end{enumerate}

\end{lemma}
\begin{remark}
As a sanity check, one may verify that starting from \eqref{eq:alpha-equals-two}, we have that
\begin{equs}
- \big((-\Delta)^{-1} \phi\big)(x) = \delta^{ij}\big(\ptl_i \ptl_j (-\Delta)^{-2} \phi\big)(x) = - \int_{\R^n} dy  G^1(x, y) \phi(y),
\end{equs}
which is consistent with Proposition \ref{prop:fractional-laplacian-formula}.
\end{remark}
\begin{proof}
First, for $\alpha > 0$ non-integer, by Proposition \ref{prop:fractional-laplacian-formula}\eqref{item:fractional-laplacian-positive-non-integer} we have that
\begin{equation}
\begin{aligned}
\big( \ptl_i \ptl_j (-\Delta)^{-(1+\frac{\alpha}{2})} \phi\big)(x) &= \frac{1}{\gamma_n(2+\alpha)} \ptl_i \ptl_j \int_{\R^n} dy  |x-y|^{2 + \alpha - n} \phi(y) \\
&= \frac{(2+\alpha - n)(\alpha - n)}{\gamma_n(2+\alpha)} \int_{\R^n} dy  |x-y|^{\alpha -2 - n} (x_i - y_i)(x_j - y_j)\phi(y) \\
&\quad \quad \quad + \delta_{ij} \frac{2+\alpha-n}{\gamma_n(2+\alpha)} \int_{\R^n} dy |x-y|^{\alpha - n} \phi(y) .
\end{aligned}
\end{equation}
Recalling the formula \eqref{eq:gamma-n-def} for $\gamma_n$, we have that
\begin{equs}
\frac{(2+\alpha - n)(\alpha - n)}{\gamma_n(2+\alpha)} = \frac{(2+\alpha - n)(\alpha -n)\Gamma((n-2-\alpha)/2)}{2^{2+\alpha} \pi^{n/2} \Gamma((2+\alpha)/2)} = \frac{\Gamma((n-\alpha+2)/2)}{2^\alpha \pi^{n/2} \Gamma((2+\alpha)/2)}.
\end{equs}
Similarly, we have that
\begin{equs}
\frac{2+\alpha-n}{\gamma_n(2+\alpha)} = -\frac{\Gamma((n-\alpha)/2)}{2^{1+\alpha} \pi^{n/2} \Gamma((2+\alpha)/2)} = -\frac{\Gamma((n-\alpha)/2)}{\alpha 2^\alpha \pi^{n/2} \Gamma(\alpha/2)} = -\frac{1}{\alpha \gamma_n(\alpha)}.
\end{equs}
From this, we obtain that for non-integer $\alpha > 0$,
\begin{equation}\label{eq:non-integer-alpha-formula}
\begin{aligned}
\big( \ptl_i \ptl_j (-\Delta)^{-(1+\frac{\alpha}{2})} \phi\big)(x) = &\frac{\Gamma((n-\alpha+2)/2)}{2^\alpha \pi^{n/2} \Gamma((2+\alpha)/2)} \int_{\R^n} dy  |x-y|^{\alpha - 2 - n} (x_i - y_i)(x_j - y_j) \phi(y) \\
&- \frac{\delta_{ij}}{\alpha \gamma_n(\alpha)} \int_{\R^n} dy |x-y|^{\alpha - n} \phi(y) .
\end{aligned}
\end{equation}
For general $\alpha > 0$, it is clear from the Fourier side that
\begin{equs}
\big(\ptl_i \ptl_j (-\Delta)^{-(1 + \frac{\alpha}{2})} \phi\big)(x) = \lim_{\delta \downarrow 0} \big(\ptl_i \ptl_j (-\Delta)^{-(1 + \frac{\alpha+\delta}{2})} \phi\big)(x).
\end{equs}
Note that when $0 < \alpha < 2$, we have that
\begin{equs}
\lim_{\delta \downarrow 0} \Gamma((n - \alpha - \delta + 2)/2) = \Gamma((n - \alpha +2)/2), \quad \text{and} \quad \lim_{\delta \downarrow 0} \gamma_n(\alpha + \delta) = \gamma_n(\alpha).
\end{equs}
The main point here is that this case of $\alpha$ avoids the singularities of the Gamma function. From these considerations and \eqref{eq:non-integer-alpha-formula}, we obtain the result when $0 < \alpha < 2$ (note from Corollary \ref{cor:fgf-covariance-kernel} that $\frac{1}{\gamma_n(\alpha)} |x-y|^{\alpha - n} = G^{\frac{\alpha}{2}}(x, y)$).

In the following, we focus on the case $\alpha \geq 2$. First, when $\alpha > 2$ is not of the form $n + 2m$ for some $m \geq 0$, the same considerations apply, since in this case $\alpha$ also avoids the singularity in the $\Gamma$ function appearing from $\Gamma((n-\alpha+2)/2)$ and $\Gamma((n-\alpha)/2)$ (note that the latter appears in the definition \eqref{eq:gamma-n-def} of $\gamma_n(\alpha)$). We thus have in this case that $(\ptl_i \ptl_j (-\Delta)^{-(1+\frac{\alpha}{2})} \phi)(x)$ is still given by \eqref{eq:non-integer-alpha-formula}. To place this in the desired form, note by Corollary \ref{cor:fgf-covariance-kernel} that when $\alpha > 2$ and $\alpha \neq n + 2m$, we have that
\begin{equs}
\frac{1}{\gamma_n(\alpha)} |x-y|^{\alpha - n} = G^{\frac{\alpha}{2}}(x, y), \quad \text{ and } \quad \frac{\Gamma((n-\alpha+2)/2)}{2^\alpha \pi^{n/2} \Gamma((2+\alpha)/2)} |x-y|^{\alpha - 2 - n} = \frac{\Gamma((\alpha - 2)/2)}{4\Gamma((\alpha+2)/2)} G^{\frac{\alpha}{2}-1}(x, y).
\end{equs}
From this, we obtain the desired result in the case $\alpha > 2, \alpha \neq n + 2m$ for some $m \geq 0$. If instead $2 = \alpha \neq n + 2m$, then we write
\begin{equs}
\frac{\Gamma((n-\alpha+2)/2)}{2^\alpha \pi^{n/2} \Gamma((2+\alpha)/2)} |x-y|^{\alpha - 2 - n}  = \frac{\Gamma(n/2)}{2^2 \pi^{n/2}} |x-y|^{-n} = \frac{1}{2\omega_n} |x-y|^{-n},
\end{equs}
where we used that $\omega_n = 2\pi^{n/2}/\Gamma(n/2)$ (recall \eqref{eq:omega-n}). This gives the desired result when $2 = \alpha \neq n + 2m$.

Next, suppose that $\alpha \geq 2$ and $\alpha = n$. In this case,
\begin{equs}
\lim_{\delta \downarrow 0} \Gamma((n-\alpha + \delta + 2)/2)  = \Gamma(1) = 1.
\end{equs}
However, $\gamma_n(n + \delta)$ exhibits a singularity as $\delta \downarrow 0$. To fix this singularity, we note that (recall $\phi \in \Phi$ so that it is mean zero)
\begin{equs}
\frac{1}{(n+\delta) \gamma_n(n+\delta)} \int_{\R^n} dy \phi(y) |x-y|^{\delta} &= \frac{\Gamma(-\delta/2)}{(n+\delta) 2^{n+\delta} \pi^{n/2} \Gamma((n+\delta)/2)} \int_{\R^n} dy \phi(y) \big( |x-y|^{\delta} - 1\big) \\
&= \frac{\delta \Gamma(-\delta/2)}{(n+\delta) 2^{n+\delta} \pi^{n/2} \Gamma((n+\delta)/2)} \int_{\R^n} dy \phi(y) \frac{1}{\delta}\big( |x-y|^{\delta} - 1\big).
\end{equs}
As $\delta \downarrow 0$, the above converges to
\begin{equs}
-\frac{2}{n 2^n \pi^{n/2} \Gamma(n/2)} \int_{\R^n} dy \phi(y) \log|x-y| = \frac{1}{n} \int_{\R^n} dy G^{\frac{n}{2}}(x, y) \phi(y),
\end{equs}
where we used Corollary \ref{cor:fgf-covariance-kernel} for the identity. This shows that when $\alpha = n \geq 2$, we have that
\begin{equation}\label{eq:alpha-equals-n}
\begin{aligned}
\big( \ptl_i \ptl_j (-\Delta)^{-(1+\frac{\alpha}{2})} \phi\big)(x) = &\frac{1}{2^n \pi^{n/2} \Gamma((n+2)/2)} \int_{\R^n} dy  |x-y|^{-2} (x_i - y_i)(x_j - y_j) \phi(y) \\
&- \frac{\delta_{ij}}{n} \int_{\R^n} dy G^{\frac{n}{2}}(x, y) \phi(y).
\end{aligned}
\end{equation}
If further $\alpha = n > 2$, then we may write
\begin{equs}
\frac{1}{2^n \pi^{n/2} \Gamma((n+2)/2)} |x-y|^{-2} = \frac{\Gamma((n-2)/2)}{4\Gamma((n+2)/2)} \frac{1}{\gamma_n(n-2)} |x-y|^{-2} = \frac{\Gamma((\alpha-2)/2)}{4\Gamma((\alpha+2)/2)} G^{\frac{\alpha}{2}-1}(x, y).
\end{equs}
Combining this with \eqref{eq:alpha-equals-n}, we obtain the desired result when $\alpha = n > 2$. If $\alpha = n = 2$, then we have that
\begin{equs}
\frac{1}{2^n \pi^{n/2} \Gamma((n+2)/2)} = \frac{1}{4 \pi} = \frac{1}{2\omega_n},
\end{equs}
and thus combining this with \eqref{eq:alpha-equals-n}, we also obtain the desired result when $\alpha = n = 2$.

The remaining case to handle is $\alpha = n + 2m \geq 2$ for some $m \geq 1$. First, note that since $m \geq 1$, this implies that $\alpha > 2$. We write
\begin{equs}
&~\frac{\Gamma((n-\alpha - \delta+2)/2)}{2^{\alpha+\delta} \pi^{n/2} \Gamma((2+\alpha+\delta)/2)} \int_{\R^n} dy \phi(y) |x-y|^{\alpha+\delta  - 2 - n}(x_i - y_i)(x_j - y_j) \\
&= \frac{\Gamma(-(m-1) - \delta/2)}{2^{n+2m+\delta} \pi^{n/2} \Gamma((n + 2(m+1) + \delta)/2)}\int_{\R^n} dy \phi(y) |x-y|^{2(m-1)+\delta} (x_i -y_i)(x_j - y_j) \\
&= \frac{\Gamma(-(m-1) - \delta/2)}{2^{n+2m+\delta} \pi^{n/2} \Gamma((n + 2(m+1) + \delta)/2)} \int_{\R^n} dy \phi(y) (x_i - y_i)(x_j - y_j)|x-y|^{2(m-1)} \big(|x-y|^\delta - 1),
\end{equs}
where in the second identity we used the fact that $\phi \in \Phi$ so that it integrates all polynomials to zero. Taking $\delta \downarrow 0$, and using that the residue of $\Gamma$ at $-(m-1)$ is $(-1)^{m-1}/(m-1)!$, we obtain that the above converges to
\begin{equs}
~&-2 \frac{(-1)^{m-1}}{(m-1)!} \frac{1}{2^{n+2m} \pi^{n/2} \Gamma((n+2(m+1))/2)} \int_{\R^n} dy \phi(y) (x_i - y_i)(x_j - y_j) |x-y|^{2(m-1)} \log |x-y| \\
&= \frac{\Gamma((\alpha-2)/2)}{4\Gamma((\alpha+2)/2)} \int_{\R^n} G^{\frac{\alpha}{2}-1}(x, y) (x_i - y_i)(x_j - y_j) \phi(y),
\end{equs}
where we applied Corollary \ref{cor:fgf-covariance-kernel} in the identity. Similarly, we may also show that
\begin{equs}
~&\lim_{\delta \downarrow 0} \frac{1}{(\alpha+\delta)\gamma_n(\alpha+\delta)} \int_{\R^n} dy \phi(y) |x-y|^{\alpha + \delta - n} \\
&= \frac{2(-1)^{m+1}}{(n+2m) m! 2^{n+2m} \pi^{n/2} \Gamma((n+2m)/2)} \int_{\R^n} dy \phi(y) |x-y|^{2m} \log|x-y| \\
&= \frac{1}{\alpha} \int_{\R^n} dy G^{\frac{\alpha}{2}}(x, y) \phi(y),
\end{equs}
where in the last identity we used Corollary \ref{cor:fgf-covariance-kernel}. The desired result in the case $\alpha = n + 2m$ for some $m \geq 1$ now follows by putting the previous few identities together.
\end{proof}

\begin{proof}[Proof of Lemma \ref{lemma:projection-laplace-inverse-formula}]
We have that $E = d \codif(-\Delta)^{-1}$, and so $E (-\Delta)^{-s} \phi = d\codif (-\Delta)^{-(1+s)} \phi$, or in coordinates,
\begin{equs}
\big(E (-\Delta)^{-s} \phi\big)_i = -\ptl_i \ptl_j (-\Delta)^{-(1+s)} \phi_j.
\end{equs}
The formula for $E$ now follows by Lemma \ref{lemma:two-derivative-ineverse-laplacian-kernel}. The formula for $E^*$ follows because $E^* = I - E$.
\end{proof}

Next, we begin towards the proof of Lemma \ref{lemma:iterated-riesz-kernel}. As a warmup, we first show the following result, which expresses the Riesz operators $\mc{R}_i$ as singular integral operators. The result is classical, but we still show it because the argument is easily adapted to show Lemma \ref{lemma:iterated-riesz-kernel}.

\begin{lemma}
For $\phi_1, \phi_2 \in \Phi$, $i \in [n]$, we have that
\[ (\phi_1, \mc{R}_i \phi_2) = \frac{\Gamma(\frac{n+1}{2})}{\pi^{\frac{n+1}{2}}} \lim_{\varep \downarrow 0} \int \int_{|x-y| > \varep} \phi_1(x) \phi_2(y) \frac{x_i - y_i}{|x-y|^{n+1}} dx dy. \]
\end{lemma}
\begin{proof}
To start, note it is clear from the Fourier side that 
\[ \lim_{\delta \downarrow 0}~ (\ptl_i (-\Delta)^{-\frac{1+\delta}{2}} \phi)(x) = (\ptl_i (-\Delta)^{-\frac{1}{2}} \phi)(x). \]
On the other hand, for $\delta > 0$ we have that
\[\begin{split}
(\ptl_i(-\Delta)^{-\frac{1+\delta}{2}}  \phi)(x) &= \ptl_i \frac{1}{\gamma_n(1+\delta)} \int_{\R^n} \phi(y) |x-y|^{1+\delta - n} dy \\
&= \frac{1+\delta - n}{\gamma_n(1+\delta)} \int_{\R^n} \phi(y) |x-y|^{\delta - n - 1} (x_i - y_i) dy. 
\end{split}\]
We thus want to take a limit as $\delta \downarrow 0$ in the right hand side above. The problem is that $|x-y|^{-n -1} (x_i - y_i)$ is singular near $x$, and thus we need to regularize. Towards this end, split
\[ \int_{\R^n} = \int_{|x-y| < \varep} + \int_{|x-y| > \varep}, \]
and write (using that $\int_{|x-y| < \varep} |x-y|^{\delta - n - 1} (x_i - y_i) dy = 0$ by parity)
\[ \int_{|x-y| < \varep} \phi(y) |x-y|^{\delta - n - 1} (x_i - y_i) dy = \int_{|x-y| < \varep} (\phi(y) - \phi(x)) |x-y|^{\delta - n - 1} (x_i - y_i) . \]
For any $\varep > 0$, the limit as $\delta \downarrow 0$ of the right hand side above exists. From this, we obtain that for any $\varep > 0$,
\[\begin{split}
(\ptl_i (-\Delta)^{-1/2} \phi)(x) = \frac{1-n}{\gamma_n(1)} \bigg(\int_{|x-y| < \varep} &(\phi(y) - \phi(x)) \frac{x_i - y_i}{|x-y|^{n+1}} dy ~+ \\
&\int_{|x-y| > \varep} \phi(y) \frac{x_i - y_i}{|x-y|^{n+1}} dy \bigg). 
\end{split}\]
Next, observe that
\[\begin{split}
\bigg|\int_{|x-y| < \varep} (\phi(y) - \phi(x)) \frac{x_i - y_i}{|x-y|^{n+1}} dy \bigg| &\lesssim \|\phi\|_{C^1} \int_{|x-y| < \varep} |x-y|^{-n + 1} dy \\
&\lesssim \varep \|\phi\|_{C^1} .
\end{split}\]
From this, we obtain (noting that $-(1-n)/\gamma_n(1) = \frac{\Gamma((n+1)/2)}{\pi^{(n+1)/2}}$)
\[ \lim_{\varep \downarrow 0} \sup_{x \in \R^n} \bigg|(\mc{R}_i \phi)(x)- \frac{\Gamma(\frac{n+1}{2})}{\pi^{\frac{n+1}{2}}} \int_{|x-y| > \varep} \phi(y) \frac{x_i - y_i}{|x-y|^{n+1}} dy\bigg| = 0.  \]
Due to this uniform convergence, we have that
\[\begin{split}
(\phi_1, \mc{R}_i \phi_2) &= \frac{\Gamma(\frac{n+1}{2})}{\pi^{\frac{n+1}{2}}} \int_{\R^n} dx \phi_1(x)  \lim_{\varep \downarrow 0} \int_{|x-y| > \varep} dy \phi(y) \frac{x_i - y_i}{|x-y|^{n+1}}  \\
&= \lim_{\varep \downarrow 0} \frac{\Gamma(\frac{n+1}{2})}{\pi^{\frac{n+1}{2}}} \int \int_{|x-y| > \varep} dx dy \phi_1(x) \phi_2(y) \frac{x_i - y_i}{|x-y|^{n+1}},
\end{split}\]
as desired.
\end{proof}

\begin{proof}[Proof of Lemma \ref{lemma:iterated-riesz-kernel}]
We have that
\[ (\mc{R}_{ij} \phi)(x) = \lim_{\delta \downarrow 0} ~(\ptl_i \ptl_j (-\Delta)^{-(1 + \frac{\delta}{2})} \phi)(x). \]
For fixed $\delta > 0$, we have that
\[\begin{split}
(\ptl_i \ptl_j (-\Delta)^{-(1 + \frac{\delta}{2})} \phi)(x) &= \ptl_i \ptl_j \frac{1}{\gamma_n(2 + \delta)} \int_{\R^n} \phi(y) |x-y|^{2 + \delta - n} dy \\
&= \frac{1}{\gamma_n(2+\delta)}\ptl_i \int_{\R^n} \phi(y) (2+\delta - n) |x-y|^{\delta - n} (x_j - y_j) dy, \end{split}\]
and this is further equal to
\[\begin{split}
\frac{1}{\gamma_n(2+\delta)} \bigg((2+\delta - n)(\delta - n) \int_{\R^n} \phi(y)& \frac{(x_i - y_i)(x_j - y_j)}{|x-y|^{n+2-\delta}} dy ~+ \\
&\delta_{ij} (2+\delta - n) \int_{\R^n} \phi(y) |x-y|^{\delta - n} dy \bigg). 
\end{split}\]
As usual, we need to first regularize before taking the limit $\delta \downarrow 0$. Let $\varep > 0$. We may write
\[\begin{split}
\int_{|y-x| < \varep} \phi(y) \frac{(x_i - y_i)(x_j - y_j)}{|x-y|^{n+2-\delta}} &dy = \int_{|x-y| < \varep} (\phi(y) - \phi(x)) \frac{(x_i - y_i)(x_j - y_j)}{|x-y|^{n+2-\delta}} dy ~+ \\
& \phi(x) \int_{|y-x| < \varep} \frac{(x_i - y_i)(x_j - y_j)}{|x-y|^{n+2-\delta}} dy.
\end{split} \]
We may compute the integral
\[ \int_{|y-x| < \varep} \frac{(x_i - y_i)(x_j - y_j)}{|x-y|^{n+2-\delta}} dy = \delta_{ij}  \frac{\omega_n}{n} \int_0^\varep r^2 r^{\delta - n - 2} r^{n-1} dr = \delta_{ij}  \frac{\omega_n \varep^\delta}{n\delta}.\]
Similarly, we have that
\[ \int_{|x-y| < \varep} \phi(y)|x-y|^{\delta - n} dy = \int_{|x-y| < \varep} (\phi(y) - \phi(x))|x-y|^{\delta - n} dy + \phi(x) \frac{\omega_n \varep^\delta}{\delta}.  \]
Combining, we thus have that
\[ (\ptl_i \ptl_j (-\Delta)^{-1} \phi)(x) = I_{< \varep}(\delta) + I_{> \varep}(\delta) + C(\varep, \delta), \]
where
\[\begin{split}
I_{< \varep}(\delta) := \frac{(2+\delta - n)(\delta - n)}{\gamma_n(2+\delta)}& \int_{|x-y| < \varep} (\phi(y) - \phi(x)) \frac{(x_i - y_i)(x_j - y_j)}{|x-y|^{n+2-\delta}} dy ~+ \\
&\delta_{ij}\frac{2+\delta - n}{\gamma_n(2+\delta)} \int_{|x-y| < \varep} (\phi(y) - \phi(x))|x-y|^{\delta - n} dy,
\end{split} \]
\[\begin{split}
I_{> \varep}(\delta) := \frac{(2+\delta - n)(\delta - n)}{\gamma_n(2+\delta)}& \int_{|x-y| > \varep} \phi(y) \frac{(x_i - y_i)(x_j - y_j)}{|x-y|^{n+2-\delta}}dy ~+ \\
&\delta_{ij}\frac{2+\delta - n}{\gamma_n(2+\delta)} \int_{|x-y| > \varep} \phi(y) |x-y|^{\delta - n} dy,
\end{split} \]
\[\begin{split}
C(\varep, \delta) &:= \delta_{ij} \phi(x) \frac{2+\delta - n}{\gamma_n(2+\delta)} \frac{\omega_n \varep^\delta}{\delta} \bigg( \frac{\delta - n}{n} + 1\bigg) \\
&= \delta_{ij} \phi(x) \frac{2+\delta - n}{\gamma_n(2+\delta)} \frac{\omega_n \varep^\delta}{n}.
\end{split}\]
We have that
\begin{equation}\label{eq:limit-volume-ball}
\lim_{\delta \downarrow 0} \frac{2+\delta-n}{\gamma_n(2+\delta)} = \frac{2-n}{\gamma_n(2)} = - \frac{1}{\omega_n},
\end{equation}
and thus for any $\varep > 0$, we have that
\[ \lim_{\delta \downarrow 0} C(\varep, \delta) = \delta_{ij} \phi(x) \frac{2-n}{\gamma_n(2)} \frac{\omega_n}{n} =- \frac{1}{n} \delta_{ij} \phi(x).\]
We thus have that for any $\varep > 0$,
\[ (\mc{R}_{ij} \phi)(x) = I_{< \varep}(0) + I_{> \varep}(0) - \frac{1}{n} \delta_{ij} \phi(x). \]
Next, observe that 
\[ |I_{< \varep}(0)| \lesssim \|\phi\|_{C^1} \int_0^\varep r r^{-n} r^{n-1} dr \lesssim \varep \|\phi\|_{C^1}. \]
It thus follows that
\[ \lim_{\varep \downarrow 0} \Big\|\mc{R}_{ij} \phi - I_{> \varep}(0) + \frac{1}{n} \delta_{ij} \phi \Big\|_{L^\infty} = 0. \]
From this, we may obtain
\[\begin{split}
(\phi_1, \mc{R}_{ij} \phi_2)=~ &-\frac{\delta_{ij}}{n} (\phi_1, \phi_2) ~- \\
&\frac{(2-n)n}{\gamma_n(2)} \lim_{\varep \downarrow 0} \int \int_{|x-y| > \varep} \phi(x) \phi(y) \frac{(x_i - y_i)(x_j - y_j)}{|x-y|^{n+2}} ~+ \\
&\delta_{ij} \frac{2-n}{\gamma_n(2)} \lim_{\varep \downarrow 0} \int \int_{|x-y| > \varep} \phi(x) \phi(y) |x-y|^{-n} dx dy.  
\end{split}\]
Using \eqref{eq:limit-volume-ball}, we have that
\[\begin{split}
(\phi_1, \mc{R}_{ij} \phi_2)=~ &-\frac{\delta_{ij}}{n} (\phi_1, \phi_2) ~+ \\
&\frac{n}{\omega_n} \lim_{\varep \downarrow 0} \int \int_{|x-y| > \varep} \phi_1(x) \phi_2(y) \frac{(x_i - y_i)(x_j - y_j)}{|x-y|^{n+2}} dx dy ~- \\
&\frac{\delta_{ij}}{\omega_n}\lim_{\varep \downarrow 0} \int \int_{|x-y| > \varep} \phi_1(x) \phi_2(y) |x-y|^{-n} dx dy,
\end{split}\]
as desired.
\end{proof}

In the following, for $\phi \in \schwartz$ recall $\psi_\phi$ defined in \eqref{eq:psi-phi}.

\begin{lemma}\label{lemma:psi-phi-l2-norm}
For $\phi, \eta \in \schwartz(\R^n)$, we have that
\begin{equs}
(\psi_\phi, \psi_\eta) = \int_{\R^{n-1}} dx' \int_\R dx_n \int_\R dy_n \phi(x', x_n) \eta(x', y_n) \big(\ind(x_n, y_n > 0) + \ind(x_n, y_n < 0)\big) \min(|x_n|, |y_n|).
\end{equs}
\end{lemma}
\begin{proof}
We write
\begin{equs}
(\psi_\phi, \psi_\eta) = \int_{\R^{n-1}} dx' \int_\R du &\bigg(\ind(u > 0) \int_u^\infty dx_n \phi(x', x_n) - \ind(u < 0) \int_{-\infty}^u dx_n \phi(x', x_n)\bigg) ~\times \\
&\bigg(\ind(u > 0) \int_u^\infty dy_n \eta(x', y_n) - \ind(u < 0) \int_{-\infty}^u dy_n \eta(x', y_n)\bigg)
\end{equs}
Expanding out the product and exchanging integration, we obtain that the above is further equal to 
\begin{equs}
~&\int_{\R^{n-1}} dx' \int_\R dx_n \int_\R dy_n \phi(x', x_n) \eta(x', y_n) \int_\R du \big(\ind(0 < u < x_n, y_n) + \ind(x_n, y_n < u < 0)\big) \\
&= \int_{\R^{n-1}} dx' \int_\R dx_n \int_\R dy_n \phi(x', x_n) \eta(x', y_n) \big(\ind(x_n, y_n > 0) + \ind(x_n, y_n < 0)\big) \min(|x_n|, |y_n|),
\end{equs}
as desired.
\end{proof}

Given a kernel $K : \R^n \times \R^n \ra \R$, define a new kernel $\Psi_K : \R^n \times \R^n \ra \R$ by
\begin{equs}
\Psi_K(x, y) := \int_\R du \int_\R dv K((x', u), (y', v)) \big(\ind(0 < u < x_n) - \ind(x_n < u < 0)\big) \big(\ind(0 < v < y_n) - \ind(y_n < v < 0)\big).
\end{equs}
In the above, $x = (x', x_n)$, $y = (y', y_n)$.

\begin{lemma}\label{lemma:psi-phi-kernel}
Let $K : \R^n \times \R^n \ra \R$ be bounded. For $\phi, \eta \in \schwartz(\R^n)$, we have that
\begin{equs}
\int_{\R^n} \int_{\R^n} dx dy \psi_\phi (x) K(x, y) \psi_\eta(y) = \int_{\R^n} \int_{\R^n} dx dy \phi(x) \Psi_K(x, y) \psi(y).
\end{equs}
\end{lemma}
\begin{proof}
We have that
\begin{equs}
~&\int_{\R^{n-1}} dx' \int_{\R^{n-1}} dy' \int_\R du \int_\R dv \psi_\phi(x', u) K((x', u), (y', v)) \psi_\eta(y', v) = \\ &\int_{\R^{n-1}} dx' \int_{\R^{n-1}} dy' \int_\R du \int_\R dv K((x', u), (y', v))\bigg(\ind(u > 0) \int_u^\infty dx_n \phi(x', x_n) - \ind(u < 0) \int_{-\infty}^u dx_n \phi(x', x_n)\bigg) ~\times \\
&\quad \quad \quad \quad \quad\quad \quad \quad \quad\quad \quad \quad \quad\quad \quad \quad \quad\quad \quad\quad \quad  \bigg(\ind(v > 0) \int_v^\infty dy_n \eta(y', y_n) - \ind(v < 0) \int_{-\infty}^v dy_n \eta(y', y_n)\bigg).
\end{equs}
To finish, expand out the product and exchange the integrations in $x_n, y_n$ with the integrations in $u, v$.
\end{proof}

\end{appendix}

\bibliographystyle{alpha}
\bibliography{formsurvey}

\end{document}